\documentclass[12pt,reqno]{amsart}
\makeatletter

\def\chaptermark#1{}

\def\chapter{%
  \if@openright\cleardoublepage\else\clearpage\fi
  \thispagestyle{plain}\global\@topnum\z@
  \@afterindenttrue \secdef\@chapter\@schapter}

\def\@chapter[#1]#2{\refstepcounter{chapter}%
  \ifnum\c@secnumdepth<\z@ \let\@secnumber\@empty
  \else \let\@secnumber\thechapter \fi
  \typeout{\chaptername\space\@secnumber}%
  \def\@toclevel{0}%
  \ifx\chaptername\appendixname \@tocwriteb\tocappendix{chapter}{#2}%
  \else \@tocwriteb\tocchapter{chapter}{#2}\fi
  \chaptermark{#1}%
  \addtocontents{lof}{\protect\addvspace{10\p@}}%
  \addtocontents{lot}{\protect\addvspace{10\p@}}%
  \@makechapterhead{#2}\@afterheading}
\def\@schapter#1{\typeout{#1}%
  \let\@secnumber\@empty
  \def\@toclevel{0}%
  \ifx\chaptername\appendixname \@tocwriteb\tocappendix{chapter}{#1}%
  \else \@tocwriteb\tocchapter{chapter}{#1}\fi
  \chaptermark{#1}%
  \addtocontents{lof}{\protect\addvspace{10\p@}}%
  \addtocontents{lot}{\protect\addvspace{10\p@}}%
  \@makeschapterhead{#1}\@afterheading}
\newcommand\chaptername{Chapter}

\def\@makechapterhead#1{\global\topskip 7.5pc\relax
  \begingroup
  \fontsize{\@xivpt}{18}\bfseries\centering
    \ifnum\c@secnumdepth>\m@ne
      \leavevmode \hskip-\leftskip
      \rlap{\vbox to\z@{\vss
          \centerline{\normalsize\mdseries
              \uppercase\@xp{\chaptername}\enspace\thechapter}
          \vskip 3pc}}\hskip\leftskip\fi
     #1\par \endgroup
  \skip@34\p@ \advance\skip@-\normalbaselineskip
  \vskip\skip@ }
\def\@makeschapterhead#1{\global\topskip 7.5pc\relax
  \begingroup
  \fontsize{\@xivpt}{18}\bfseries\centering
  #1\par \endgroup
  \skip@34\p@ \advance\skip@-\normalbaselineskip
  \vskip\skip@ }
\def\appendix{\par
  \c@chapter\z@ \c@section\z@
  \let\chaptername\appendixname
  \def\thechapter{\@Alph\c@chapter}}

\newcounter{chapter}
\newif\if@openright

\makeatother

\makeatletter
\@addtoreset{section}{chapter}
\makeatother  
\usepackage{amsmath, amssymb, amsthm, mathrsfs}
\usepackage{stmaryrd} 
\usepackage{amsfonts, latexsym, amscd,color} 

\usepackage[all]{xy}
\usepackage[makeroom]{cancel}
\usepackage{soul}

\usepackage{bbm} 

\usepackage{hyperref, color}
\usepackage{leftidx} 
\usepackage{mathabx} 
\usepackage{mathtools} 
\usepackage{rotating} 
\usepackage{stmaryrd} 

\usepackage{youngtab,young}

\usepackage{todonotes}
\usepackage{arydshln}
\makeatletter
\renewcommand*\env@matrix[1][*\c@MaxMatrixCols c]{%
  \hskip -\arraycolsep
  \let\@ifnextchar\new@ifnextchar
  \array{#1}}
\makeatother
\usepackage{pgf, tikz}
\usetikzlibrary{arrows, positioning, calc, chains}
\tikzset{
	ch/.style={circle,draw,on chain,inner sep=2pt},
	chj/.style={ch,join},
	every path/.style={shorten >=4pt,shorten <=4pt}
	}

\numberwithin{equation}{section}

\newcommand{\A}{\mathfrak o}

\newcommand{\ap}{{ap}}
\newcommand{\asl}{\widehat{\mathfrak{sl}}}

\newcommand{\bH}{\mbf H}
\newcommand{\bh}{\mbf h}
\newcommand{\bj}{\mathbf{j}}
\newcommand{\bJ}{\mbf J}
\newcommand{\bK}{\mbf K}
\newcommand{\bk}{\mbf k}

\newcommand{\bU}{\mathbf{U}}

\newcommand{\cA}{\mathcal{A}}

\newcommand{\C}{\mathfrak{c}}

\newcommand{\co}{\textup{co}}

\newcommand{\dep}{\mrm{dep}}

\newcommand{\E}{\mathbf E}
\newcommand{\e}{\mathbf e}
\newcommand{\End}{\textup{End}}

\newcommand{\f}{\mathbf f}
\newcommand{\F}{\mbf F}
\newcommand{\fa}{\mathfrak{a}}

\newcommand{\fc}{\mathfrak{c}}

\newcommand{\gl}{\mathfrak{gl}}
\newcommand{\glh}{\widehat{\mathfrak{gl}}}
\newcommand{\GL}{\mrm{GL}}

\newcommand{\K}{\dot{\mbf K}}

\newcommand{\KK}{\mbf{K}}

\newcommand{\mbb}{\mathbb}
\newcommand{\mbf}{\mathbf}

\newcommand{\mm}{\eta}   
\newcommand{\MX}{\Xi}  
\newcommand{\mrm}{\mathrm}

\newcommand{\nn}{\mathfrak{n}}

\newcommand{\ro}{\mrm{ro}}

\newcommand{\Sb}{\mbf{iS}}

\newcommand{\slh}{\widehat{\mathfrak{sl}}}
\newcommand{\sll}{\mathfrak{sl}}

\newcommand{\Sj}{{\mbf S}^{\fc}_{n,d}}

\newcommand{\smxylabel}[1]{{\text{\small$#1$}}}
\newcommand{\SP}{\mrm{Sp}}

\newcommand{\U}{\mbf U}
\newcommand{\Ua}{\dot \bU^{\fc}_{n, \text{alg}}}
\newcommand{\Ub}{\mbf{iU}}
\newcommand{\Ubd}{\mbf{i}\dot{\mbf{U}} }

\newcommand{\ve}{\varepsilon}

\newcommand{\X}{\mathcal X}

\newcommand{\Y}{\mathcal Y}

\newcommand{\Z}{\mathcal Z}

\newcommand{\ZZ}{\mathbb{Z}}

\usepackage{enumitem }

\theoremstyle{definition}
\newtheorem{Def}{Definition}[section] 

\newtheorem{example}[Def]{Example}

\newtheorem{rem}[Def]{Remark}

\theoremstyle{plain}
\newtheorem{prop}[Def]{Proposition}
\newtheorem{thm}[Def]{Theorem}
\newtheorem{lem}[Def]{Lemma}

\newtheorem{cor}[Def]{Corollary}

\newtheorem{thrm}{Theorem}

\newcommand{\ji}{\jmath \imath}
\newcommand{\Dji}{\Delta^{\ji}}
\newcommand{\Eji}{\check{\mbf E}}
\newcommand{\Fji}{\check{\mbf F}}
\newcommand{\Kji}{\check{\mbf K}}
\newcommand{\Hji}{\check{\mbf H}}
\newcommand{\eji}{\check{\mbf e}}
\newcommand{\fji}{\check{\mbf f}}
\newcommand{\kji}{\check{\mbf k}}
\newcommand{\hji}{\check{\mbf h}}
\newcommand{\tji}{\check{\mbf t}}
\newcommand{\iji}{{\jiw}}
\newcommand{\Kcji}{\dot{\mbf K}^{\ji}_{\nn}}

\newcommand{\jj}{\jmath \jmath}

\newcommand{\Kc}{\dot{\mbf K}^{\fc}}

\renewcommand{\ij}{\imath \jmath}
\newcommand{\Dij}{\Delta^{\ij}}
\newcommand{\Eij}{\hat{\mbf E}}
\newcommand{\Fij}{\hat{\mbf F}}
\newcommand{\Kij}{\hat{\mbf K}}
\newcommand{\Hij}{\hat{\mbf H}}
\newcommand{\eij}{\hat{\mbf e}}
\newcommand{\fij}{\hat{\mbf f}}
\newcommand{\kij}{\hat{\mbf k}}
\newcommand{\hij}{\hat{\mbf h}}
\newcommand{\tij}{\hat{\mbf t}}
\newcommand{\iij}{{\ijw}}
\newcommand{\Kcij}{\dot{\mbf K}^{\ij}_{\nn}}
\newcommand{\Sij}{{\mbf S}}
\newcommand{\Sji}{{\mbf S}}
\newcommand{\Sii}{{\mbf S}}

\newcommand{\ii}{\imath \imath}
\newcommand{\Dii}{\Delta^{\ii}}
\newcommand{\Eii}{\tilde{\mbf E}}
\newcommand{\Fii}{\tilde{\mbf F}}
\newcommand{\Kii}{\tilde{\mbf K}}
\newcommand{\Hii}{\tilde{\mbf H}}
\newcommand{\eii}{\tilde{\mbf e}}
\newcommand{\fii}{\tilde{\mbf f}}
\newcommand{\kii}{\tilde{\mbf k}}
\newcommand{\hii}{\tilde{\mbf h}}
\newcommand{\tii}{\tilde{\mbf t}}
\newcommand{\iii}{\iiw}  
\newcommand{\Kcii}{\dot{\mbf K}^{\ii}_{\mm}}

\newcommand{\jjw}{\mbox{$\jmath$\kern-1.3pt$\jmath$}}
\newcommand{\jiw}{\mbox{$\jmath$\kern-0.8pt$\imath$}}
\newcommand{\ijw}{\mbox{$\imath$\kern-1.8pt$\jmath$}}
\newcommand{\iiw}{\mbox{$\imath$\kern-1.0pt$\imath$}}

\title[Affine flag varieties and quantum symmetric pairs]{Affine flag varieties and quantum symmetric pairs}    

\author[Z. Fan, C. Lai, Y. Li, L. Luo, and W. Wang]{Zhaobing Fan,  Chun-Ju Lai, Yiqiang Li,  Li Luo,
and Weiqiang Wang}
\address{School of science, Harbin Engineering University, Harbin, China 150001}
    \email{fanz@math.ksu.edu  (Fan)}
\address{Department of Mathematics, University of Virginia, Charlottesville, VA 22904}
    \email{cl8ah@virginia.edu (Lai) }
\address{Department of Mathematics, University at Buffalo, SUNY, Buffalo, NY 14260}
    \email{yiqiang@buffalo.edu (Li)}
\address{ 
    Department of mathematics, East China Normal University, Shanghai, China 200241}
\email{lluo@math.ecnu.edu.cn (Luo), ww9c@virginia.edu (Wang)}


\begin{document}

\begin{abstract}    
The quantum groups of finite and affine type $A$ admit geometric realizations in terms of partial flag varieties of finite and affine type $A$.
Recently, the quantum group associated to partial flag varieties of finite type $B/C$ is shown to be a coideal subalgebra of the quantum group of finite type $A$.
In this paper we study the structures of Schur algebras  and Lusztig algebras associated to (four variants of) partial flag varieties of affine type $C$. We  
show that the quantum groups arising from Lusztig algebras and Schur algebras via stabilization procedures are (idempotented) coideal subalgebras 
of quantum groups of affine $\mathfrak{sl}$ and $\mathfrak{gl}$ types, respectively. 
In this way, we provide geometric realizations of eight quantum symmetric pairs of affine types. 
We construct monomial and canonical bases of all these quantum (Schur, Lusztig, and coideal) algebras. 
For the idempotented coideal algebras of affine $\mathfrak{sl}$ type, 
we establish the positivity properties of the canonical basis with respect to multiplication, comultiplication and a bilinear pairing. 
In particular, we obtain a new and geometric construction of the idempotented quantum affine $\mathfrak{gl}$ and its canonical basis. 
\end{abstract}

\maketitle

\setcounter{tocdepth}{1}
\tableofcontents  

\chapter{Introduction}

\section{Background}  
  \label{sec:finiteABC}

\subsection{}

Iwahori \cite{I64}  provided a geometric realization of Iwahori-Hecke algebras $\mbf H_W^{\text{fin}}$ 
as convolution algebras on pairs of (finite type) complete flags over a finite field.
Iwahori-Matsumoto \cite{IM65} have subsequently realized the affine Hecke algebras using pairs of complete flags of affine (or $p$-adic) type
over a local field. 
These works are foundational for geometric representation theory. 

The Drinfeld-Jimbo  quantum groups \cite{Dr86, Jim86} have played important roles in many areas of mathematics.
Beilinson, Lusztig and MacPherson \cite{BLM90} provided a geometric realization of quantum group $\bU(\gl_n)$ of finite type $A$. 
The BLM construction utilizes the $n$-step flag varieties in an ambient space of dimension $d$, and the convolution algebra
on pairs of $n$-step flags is shown to be the quantum Schur algebra $\mbf S_{n,d}^{\text{fin}}$; this can be viewed as 
a generalization of Iwahori's construction of Hecke algebras in finite type $A$. 

Beilinson, Lusztig and MacPherson \cite{BLM90} further established multiplication formulas in $\mbf S_{n,d}^{\text{fin}}$
with  divided powers of Chevalley generators, which allows them to observe some remarkable stabilization phenomenon as $d\to \infty$. 
A suitable limit construction gives rise to the idempotented quantum group $\dot{\bU}(\gl_n)$ and its (stably) canonical basis. 
The construction is easily modified further to produce variants such as $\bU(\gl_n)$, $\bU(\sll_n)$, and the idempotented form $\dot{\bU}(\sll_n)$.
The idempotented form $\dot{\bU}(\sll_n)$ also has a canonical basis (cf. \cite{Lu93}, \cite{K94}), 
in analog with the Kazhdan-Lusztig bases for Iwahori-Hecke algebras ~\cite{KL79}. 
A modification of the above construction \cite{GL92}  provides a geometric realization of the Schur-Jimbo duality \cite{Jim86}. 

There have been some generalizations of the BLM-type construction using the $n$-step (partial) flag varieties of affine type $A$ earlier on;
see Ginzburg-Vasserot \cite{GV93} and Lusztig \cite{Lu99, Lu00} (also cf. \cite{VV99, Mc12, P09} for further developments).
However, there is a major difference between affine and finite type $A$, which was first made clear by Lusztig. 
He showed that a natural homomorphism from $\bU (\slh_n)$ to the  affine quantum Schur algebra ${\mbf S}_{n,d}$ is no longer surjective
(the image is denoted by $\bU_{n,d}$ and called {\em Lusztig algebra} in this paper). 
Alternatively, one could characterize $\bU_{n,d}$ as the proper subalgebra of ${\mbf S}_{n,d}$
generated by 
the Chevalley generators.  

There has been a new (algebraic) approach recently developed by \cite{DF13, DF14} (see also \cite{G99}) which allows one to construct 
a larger algebra $\dot\bU(\glh_n)$  (called the idempotented quantum affine $\gl_n$ in this paper;
also known as the quantum loop algebra of $\gl_n$), from BLM-type stabilization
of the affine Schur algebras ${\mbf S}_{n,d}$. 

\subsection{}

Since the constructions of Iwahori and Iwahori-Matsumoto are valid for flag varieties of any
finite and affine type, it is a natural question since the work of \cite{BLM90} in 1990 to ask for 
generalization of the above  type $A$ constructions  to other, say classical, types. 
The progress in this direction has been made only in recent years.  Motivated by  \cite{BW13}, 
Bao, Kujawa, and two of the authors \cite{BKLW14, BLW14} provided
a geometric construction of Schur-type algebras  $\Sb_{n,d}^{\text{fin}}$ 
(denoted therein by $\mbf S^{\jmath}$ for $n$ odd and $\mbf S^{\imath}$ for $n$ even)
in terms of $n$-step flag varieties of type $B_d$ (or $C_d$). 

The authors of \cite{BKLW14, BLW14} further established multiplication formulas in the Schur algebras  $\Sb_{n,d}^{\text{fin}}$ 
with  divided powers of Chevalley generators, which again enjoy some remarkable stabilization properties as $d\mapsto \infty$.
They showed the quantum algebra arising from the stabilization procedure is a coideal subalgebra $\Ub(\gl_n)$  of $\bU(\gl_n)$
(this coideal subalgebra was denoted in {\em loc. cit.} as $\bU^{\jmath}$ for $n$ odd and $\bU^{\imath}$ for $n$ even);  the pair $(\bU(\gl_n), \Ub(\gl_n))$ forms
a quantum symmetric pair in the sense of Letzter \cite{Le02} and Kolb \cite{Ko14}. 
An  $(\Ub(\gl_n), \mbf H_{C_d}^{\text{fin}})$-duality (called iSchur duality) is also realized in \cite{BKLW14} by using mixed pairs of $n$-step flags and complete flags of type B/C.
The iSchur duality was discovered algebraically and categorically in \cite{BW13} as a crucial ingredient for a new approach to Kazhdan-Lusztig theory of classical type. 

A new canonical basis (called $\imath$canonical basis) was constructed in \cite{BW13} for various tensor product modules of $\Ub(\gl_n)$, and
the $\imath$canonical basis for the idempotented form $\Ubd(\gl_n)$ was subsequently constructed in \cite{BKLW14}. 
It has been shown in \cite{FL14} that coideal like algebras together with their $\imath$canonical bases arise from partial flag varieties of type $D$. 
There has been a further geometric realization of the idempotented coideal subalgebra  $\Ubd(\sll_n)$ of $\bU(\sll_n)$ \cite{LW15} and its canonical basis. 

(To distinguish from many other different quantum coideal subalgebras in the literature, the quantum coideal subalgebras appearing in quantum symmetric pairs
could be called {\em $\imath$quantum groups}, where $\imath$ stands for involution or isotropic.)

For canonical bases, there is a major difference between $\dot{\bU}(\gl_n)$ and $\dot{\bU}(\sll_n)$, or between idempotented coideal subalgebras of $\gl$ and $\sll$ type:
the canonical basis of $\dot{\bU}(\sll_n)$ admits remarkable positivity properties with respect to
multiplication and a bilinear pairing \cite{SV00, Mc12, LW15} and so does the canonical basis of $\Ubd(\sll_n)$ \cite{LW15}.
It is recently shown in \cite{FL15} that the canonical bases of idempotented quantum (affine) $\sll_n$ and idempotented coideal algebra $\Ubd(\sll_n)$ admit
positivity property with respect to the comultiplication. 
In contrast, the canonical bases of $\dot{\bU}(\gl_n)$ and of $\Ubd(\gl_n)$ both fail to exhibit a positivity property with respect to multiplication; see \cite{LW15}. 

\section{The goal: affine type $C$}

\subsection{}

The goal of this paper is to initiate the study of the Schur algebras and quantum groups arising from partial flag varieties of classical affine type beyond type $A$,
generalizing  the constructions in finite type $B/C$ in Section~\ref{sec:finiteABC}.

In this paper, we  focus on the affine type $C$. 
As we shall see, the affine type $C$ setting already provides a more challenging and much richer setting than the finite type $C$ and the affine type $A$. 
For each of the two  type $A$ quantum affine algebras  (of level zero)
$\bU(\slh_n)$ and $\bU(\glh_n)$, we shall provide geometric realizations of four different (idempotented)
coideal subalgebras and their canonical bases.  
(The four cases are denoted by $\jj, \ji, \ijw, \ii$, respectively; we also write $\fc \equiv \jj$.)
The corresponding four Dynkin diagrams with involutions are depicted in 
Figures~\ref{figure:jj}, \ref{figure:ji}, \ref{figure:ij}, and \ref{figure:ii}, respectively, as follows. 
Therefore, in total we have provided a geometric realization of eight distinct quantum symmetric pairs of affine type. 
\begin{figure}[ht!]
\caption{Dynkin diagram of type $A^{(1)}_{2r+1}$ with involution of type $\jmath\jmath \equiv \C$.}
 \label{figure:jj}
\begin{tikzpicture}
\matrix [column sep={0.6cm}, row sep={0.5 cm,between origins}, nodes={draw = none,  inner sep = 3pt}]
{
	\node(U1) [draw, circle, fill=white, scale=0.6, label = 0] {}; 
	&\node(U2)[draw, circle, fill=white, scale=0.6, label =1] {};
	&\node(U3) {$\cdots$};  
	&\node(U4)[draw, circle, fill=white, scale=0.6, label =$r-1$] {}; 
	&\node(U5)[draw, circle, fill=white, scale=0.6, label =$r$] {};
\\
	&&&&&
\\
	\node(L1) [draw, circle, fill=white, scale=0.6, label =below:$2r+1$] {};  
	&\node(L2)[draw, circle, fill=white, scale=0.6, label =below:$2r$] {};
	&\node(L3) {$\cdots$};  
	&\node(L4)[draw, circle, fill=white, scale=0.6, label =below:$r+2$] {}; 
	&\node(L5)[draw, circle, fill=white, scale=0.6, label =below:$r+1$] {};
\\
};
\begin{scope}
\draw (L1) -- node  {} (U1);
\draw (U1) -- node  {} (U2);
\draw (U2) -- node  {} (U3);
\draw (U3) -- node  {} (U4);
\draw (U4) -- node  {} (U5);
\draw (U5) -- node  {} (L5);
\draw (L1) -- node  {} (L2);
\draw (L2) -- node  {} (L3);
\draw (L3) -- node  {} (L4);
\draw (L4) -- node  {} (L5);
\draw (L1) edge [color = blue,<->, bend right, shorten >=4pt, shorten <=4pt] node  {} (U1);
\draw (L2) edge [color = blue,<->, bend right, shorten >=4pt, shorten <=4pt] node  {} (U2);
\draw (L4) edge [color = blue,<->, bend left, shorten >=4pt, shorten <=4pt] node  {} (U4);
\draw (L5) edge [color = blue,<->, bend left, shorten >=4pt, shorten <=4pt] node  {} (U5);
\end{scope}
\end{tikzpicture}
\end{figure}
\begin{figure}[ht!]
\caption{Dynkin diagram of type $A^{(1)}_{2r}$ with involution of type $\jmath\imath$.}
   \label{figure:ji}
\begin{tikzpicture}
\matrix [column sep={0.6cm}, row sep={0.5 cm,between origins}, nodes={draw = none,  inner sep = 3pt}]
{
	\node(U1) [draw, circle, fill=white, scale=0.6, label = 0] {}; 
	&\node(U2)[draw, circle, fill=white, scale=0.6, label =1] {};
	&\node(U3) {$\cdots$};  
	&\node(U5)[draw, circle, fill=white, scale=0.6, label =$r-1$] {};
\\
	&&&& 
	\node(R)[draw, circle, fill=white, scale=0.6, label =$r$] {};
\\
	\node(L1) [draw, circle, fill=white, scale=0.6, label =below:$2r$] {};  
	&\node(L2)[draw, circle, fill=white, scale=0.6, label =below:$2r-1$] {};
	&\node(L3) {$\cdots$};  
	&\node(L5)[draw, circle, fill=white, scale=0.6, label =below:$r+1$] {};
\\
};
\begin{scope}
\draw (U1) -- node  {} (U2);
\draw (U2) -- node  {} (U3);
\draw (U3) -- node  {} (U5);
\draw (U5) -- node  {} (R);
\draw (U1) -- node  {} (L1);
\draw (L1) -- node  {} (L2);
\draw (L2) -- node  {} (L3);
\draw (L3) -- node  {} (L5);
\draw (L5) -- node  {} (R);
\draw (R) edge [color = blue,loop right, looseness=40, <->, shorten >=4pt, shorten <=4pt] node {} (R);
\draw (L1) edge [color = blue,<->, bend right, shorten >=4pt, shorten <=4pt] node  {} (U1);
\draw (L2) edge [color = blue,<->, bend right, shorten >=4pt, shorten <=4pt] node  {} (U2);
\draw (L5) edge [color = blue,<->, bend left, shorten >=4pt, shorten <=4pt] node  {} (U5);
\end{scope}
\end{tikzpicture}
\end{figure}
\begin{figure}[ht!]
\caption{Dynkin diagram of type $A^{(1)}_{2r}$ with involution of type $\imath\jmath$.}
   \label{figure:ij}
\begin{tikzpicture}
\matrix [column sep={0.6cm}, row sep={0.5 cm,between origins}, nodes={draw = none,  inner sep = 3pt}]
{
	&\node(U1) [draw, circle, fill=white, scale=0.6, label = 1] {}; 
	&\node(U3) {$\cdots$};  
	&\node(U4)[draw, circle, fill=white, scale=0.6, label =$r-1$] {}; 
	&\node(U5)[draw, circle, fill=white, scale=0.6, label =$r$] {};
\\
	\node(L)[draw, circle, fill=white, scale=0.6, label =0] {}; 
	&&&&&
\\
	&\node(L1) [draw, circle, fill=white, scale=0.6, label =below:$2r$] {};  
	&\node(L3) {$\cdots$};  
	&\node(L4)[draw, circle, fill=white, scale=0.6, label =below:$r+2$] {}; 
	&\node(L5)[draw, circle, fill=white, scale=0.6, label =below:$r+1$] {};
\\
};
\begin{scope}
\draw (L) -- node  {} (U1);
\draw (U1) -- node  {} (U3);
\draw (U3) -- node  {} (U4);
\draw (U4) -- node  {} (U5);
\draw (U5) -- node  {} (L5);
\draw (L) -- node  {} (L1);
\draw (L1) -- node  {} (L3);
\draw (L3) -- node  {} (L4);
\draw (L4) -- node  {} (L5);
\draw (L) edge [color = blue, loop left, looseness=40, <->, shorten >=4pt, shorten <=4pt] node {} (L);
\draw (L1) edge [color = blue,<->, bend right, shorten >=4pt, shorten <=4pt] node  {} (U1);
\draw (L4) edge [color = blue,<->, bend left, shorten >=4pt, shorten <=4pt] node  {} (U4);
\draw (L5) edge [color = blue,<->, bend left, shorten >=4pt, shorten <=4pt] node  {} (U5);
\end{scope}
\end{tikzpicture}
\end{figure}
\begin{figure}[ht!]
\caption{Dynkin diagram of type $A^{(1)}_{2r-1}$ with involution of type $\imath\imath$.}
   \label{figure:ii}
\begin{tikzpicture}
\matrix [column sep={0.6cm}, row sep={0.5 cm,between origins}, nodes={draw = none,  inner sep = 3pt}]
{
	&\node(U1) [draw, circle, fill=white, scale=0.6, label = 1] {}; 
	&\node(U2) {$\cdots$};  
	&\node(U3)[draw, circle, fill=white, scale=0.6, label =$r-1$] {}; 
\\
	\node(L)[draw, circle, fill=white, scale=0.6, label =0] {}; 
	&&&&
	\node(R)[draw, circle, fill=white, scale=0.6, label =$r$] {};
\\
	&\node(L1) [draw, circle, fill=white, scale=0.6, label =below:$2r-1$] {};  
	&\node(L2) {$\cdots$};  
	&\node(L3)[draw, circle, fill=white, scale=0.6, label =below:$r+1$] {}; 
\\
};
\begin{scope}
\draw (L) -- node  {} (U1);
\draw (U1) -- node  {} (U2);
\draw (U2) -- node  {} (U3);
\draw (U3) -- node  {} (R);
\draw (L) -- node  {} (L1);
\draw (L1) -- node  {} (L2);
\draw (L2) -- node  {} (L3);
\draw (L3) -- node  {} (R);
\draw (L) edge [color = blue, loop left, looseness=40, <->, shorten >=4pt, shorten <=4pt] node {} (L);
\draw (R) edge [color = blue,loop right, looseness=40, <->, shorten >=4pt, shorten <=4pt] node {} (R);
\draw (L1) edge [color = blue,<->, bend right, shorten >=4pt, shorten <=4pt] node  {} (U1);
\draw (L3) edge [color = blue,<->, bend left, shorten >=4pt, shorten <=4pt] node  {} (U3);
\end{scope}
\end{tikzpicture}
\end{figure}

In summary, the quantum algebras behind the various kinds of  flag varieties are listed in the following table for comparison. 
\begin{center}
\begin{tabular}{| l | l | l |}
\hline
Flag variety: & Complete flag & Partial flag \\
\hline
&  & Type A: quantum $\mathfrak{gl}_n$, $\mathfrak {sl}_n$  \\ \cline{3-3}
of finite type &Iwahori-Hecke algebra & Type B/C/D: coideal subalgebras \\
&& \hfill of quantum $\mathfrak{gl}_n$, $\mathfrak {sl}_n$ \\
\hline
 & & Type A: affine quantum $\mathfrak{gl}_n$, $\mathfrak {sl}_n$ \\ \cline{3-3}
of affine type& Affine Iwahori-Hecke algebra& Type C:  coideal subalgebras\\
&& \hfill of affine quantum $\mathfrak{gl}_n$, $\mathfrak {sl}_n$\\
\hline
\end{tabular}
\end{center}

To help the reader to follow and digest this long paper, we organize various chapters in three parts. Here is a brief summary. 
\begin{itemize}
\item
Part~\ref{part1} contains the basic constructions of 
the affine Schur algebra $\mbf S_{n,d}^{\fc}$ and its distinguished Lusztig subalgebra $\bU_{n,d}^{\fc}$, as well as their $\ji, \ij, \ii$-variants.
Then we study in depth the multiplicative and coideal like comultiplicative structures of these algebras. 

\item
In Part~\ref{part2} we study the structures of the family of Lusztig algebras $\bU_{n,d}^{\fc}$  (and their $\ji, \ij, \ii$-siblings), and
show that they lead to quantum coideal subalgebras $\bU^{\fc} (\slh_n)$ of $\bU(\slh_n)$.
The corresponding idempotented forms $\dot\bU^{\fc} (\slh_n)$ (and their $\ji, \ij, \ii$-siblings) are shown to admit canonical bases with positivity.  

\item
Part~\ref{part3} is focused on the study of the stabilization properties of
the family of Schur algebras $\mbf S_{\nn,d}^{\fc}$ (and their $\ji, \ij, \ii$-siblings), leading to stabilization
algebras which  are identified as  idempotented coideal subalgebras  $\bU^{\fc} (\glh_n)$ of quantum affine $\gl_n$;
these stabilization algebras are shown to admit canonical bases (without positivity). 
\end{itemize}

The following diagram is a brief road map of some main constructions (there are 4 distinct cases where $\fc$ can be replaced by $\jjw, \ji, \ijw, \iiw$): 
\begin{equation}
\label{cd:master}
 \begin{CD}
\bU^\fc_{n,d} @> \text{Stabilization} > d \mapsto \infty > \lim\limits_{\longleftarrow} \bU^{\fc}_{n,d} @> \approx >>  \dot\bU^{\fc} (\slh_n)
\\
@VVV \\  
\Sj
@>  \text{Stabilization}> d \mapsto \infty > \lim\limits_{\longleftarrow} \Sj   @> \approx >> \dot \bU^{\fc} (\glh_n)
\end{CD}
\end{equation}

\subsection{}

While the quantum algebras arising from partial flags of classical  types (except type $A$) are not of Drinfeld-Jimbo quantum groups,
they are meaningful and significant generalizations of the type $A$ quantum groups because of their geometric origin.
There has been an intimately related category $\mathcal O$ interpretation and an application of 
canonical bases arising from quantum symmetric pairs of finite type \cite{BW13} (also cf. \cite{ES13, Bao16} for type $D$).

It is expected that the quantum symmetric pairs of affine type (and their categorifications) will play a fundamental role in modular representations of 
algebraic groups and quantum groups of classical type.
We also expect a Langlands dual picture of the constructions of this paper, 
realizing the coideal algebras of affine type in terms of Steinberg-type varieties of finite type 
(cf. \cite{CG97} for some earlier instances of such dual pictures).

\section{An overview}

\subsection{An overview of Part~\ref{part1}}

\subsubsection{}

Most of the geometric constructions in \cite{BKLW14, BLW14} (and also \cite{LW15, FL15}) in finite type $B/C$
were treated in two separate cases, depending on the parity of $n$, even though the statements are uniform.
The results for $\Sb_{n,d}^{\text{fin}}$ and $\Ubd(\gl_n)$ with $n$ odd
are established first,  and then the subtler even $n$ case is settled by relating to the  odd $n$ case.

Before proceeding to the affine type, it is instructive for us to explain informally some of the main ideas of \cite{BKLW14} (and \cite{BLW14}). 
We shall fix an even positive integer $n$ and set $\nn =n+1$ (which is odd) in this section. We shall write
 $\Sb_{\nn,d}^{\text{fin}} =\mbf S_{\nn,d}^{\jmath,\text{fin}}$, $\Ub(\gl_\nn) =\bU^{\jmath} (\gl_\nn)$, 
 $\Sb_{n,d}^{\text{fin}} =\mbf S_{n,d}^{\imath,\text{fin}}$,   $\Ub(\gl_n) =\bU^{\imath} (\gl_n)$, and use similar notations for the idempotented forms.

The Schur algebra $\mbf S_{\nn,d}^{\jmath,\text{fin}}$ 
is most naturally realized via pairs of $\nn$-step type $B$ flags. 
Even though the geometric realization for $\mbf S_{n,d}^{\imath,\text{fin}}$ 
could naturally use $n$-step type $C$ flags, \cite{BKLW14}  instead chose to work with 
 $\nn$-step type $B$ flags subject to a maximal isotropic condition on the middle subspaces of flags. 
This approach of using the type $B$ geometry alone allows one to 
relate the Schur algebras as well as the coideal algebras with indices $n, \nn$ of different parities. 

The Dynkin diagram automorphism of type $\gl_{\nn}$  has no fixed point as $\nn$ is odd,  
which is Figure (\ref{figure:jj}) with vertices $0$ and $2r+1$ removed, 
while it has a fixed point for type $\gl_{n}$, which is Figure (\ref{figure:ji}) with vertices $0$, $2r$ removed. 
Working with flags subject to maximal isotropic middle constraints can be loosely understood as giving rise to  
the Schur algebras and coideal algebras with a fixed point;
the imbedding of such flags into a variety of flags without maximal isotropic constraints is a way of resolving such a fixed point,
and this is how we succeeded in understanding $\mbf S_{n,d}^{\imath,\text{fin}}$ (and respectively, $\dot\bU^{\imath}(\gl_{n})$)
through its relation to $\mbf S_{\nn,d}^{\jmath,\text{fin}}$ (and respectively, $\dot\bU^{\jmath}(\gl_{\nn})$). 

As a preparation toward affine type $C$, we reformulate the main geometric constructions of \cite{BKLW14, FL15} in the framework of finite type $C$ flags
in Appendix~\ref{chap:finiteC},  expanding the outline in \cite[\S6]{BKLW14}. 
Recall  that $\mbf S_{n,d}^{\imath,\text{fin}}$ can be realized  
using $n$-step type $C$ flags (note the middle subspace in such a flag is automatically maximal isotropic). 
 To realize $\mbf S_{\nn,d}^{\fc,\text{fin}}$   (recall $\nn =n+1$),
we employ $\nn$-step type $C$ flags, and then identify an $n$-step flag as an $\nn$-step flag
subject to a maximal isotropic condition on the middle subspace. 
Then all type $B$ constructions in \cite{BKLW14, BLW14, FL15} can be repeated in such a finite type $C$ setting. (This might be regarded a
manifestation of Langlands duality philosophy.)

\subsubsection{}

Let us return to the affine cases.
There is a lattice presentation of the complete and $n$-step flag varieties of affine type $A$ due to Lusztig; see Chapter~\ref{chap:A}. 
Such a lattice presentation can be adapted to affine type $C$, on which the symplectic loop group $\mrm{Sp}_F(2d)$
acts (where $F=k((\epsilon))$); 
cf. Sage \cite{Sa99} for complete flags and its variant for the $n$-step partial flag variety $\X^{\fc}_{n,d}$ 
which is formulated in this paper, for $n$ even. 

However, for our purpose we need to define such a $\X^{\fc}_{n,d}$
in a somewhat delicate way, keeping in mind the lesson we learned from finite type $B/C$. That is,  
 $\X^{\fc}_{n,d}$ is defined to avoid ``maximal isotropic" constraints  and (as shown later)
 it will  give rise to Schur algebras associated to the affine Dynkin diagram automorphism {\em without} fixed points in Figure \ref{figure:jj};
 the most  obvious candidate of $n$-step flag variety of affine type $C$ will not do. 

The orbits for the product $\X^{\fc}_{n,d} \times \X^{\fc}_{n,d}$ under the diagonal action of the group $\mrm{Sp}_F(2d)$
can be parameterized by
the set $\MX_{n,d}$ of $\ZZ\times \ZZ$-matrices with entries in $\mbb N$ 
satisfying certain natural periodicity and centro-symmetry conditions.
Denote by $\MX_{n,d}^{\ap}$ the set of aperiodic matrices in $\MX_{n,d}$ 
(recall the notion of aperiodic matrix was introduced in \cite{Lu99} in the affine type $A$ setting).

The Schur algebra ${\mbf S}_{n,d}^{\fc}$ is by definition the (generic) convolution algebra of pairs of flags in $\X^{\fc}_{n,d}$.
It admits a canonical basis (IC basis) which enjoys a positivity with respect to multiplication.
We formulate a subalgebra $\bU^{\fc}_{n,d}$ of ${\mbf S}_{n,d}^{\fc}$ generated by the Chevalley generators. 
We caution that the Chevalley generators do not form a generating set for the algebra ${\mbf S}_{n,d}^{\fc}$,
that is, $\bU^{\fc}_{n,d}$ is a proper subalgebra of ${\mbf S}_{n,d}^{\fc}$ in general. 
Our first main result is the following. 

\begin{thrm} [Theorem~\ref{CB-Udn}]
  \label{CB-Udn:Intr}
The algebra  $\bU^{\fc}_{n,d}$ admits a monomial basis $\{\zeta_A  \vert A \in \MX_{n,d}^{\ap} \}$ 
and  a canonical basis  $\{ \{A\}_d \vert A \in \MX_{n,d}^{\ap} \}$, which are compatible with
the corresponding bases in $\Sj$ under the inclusion $\bU^{\fc}_{n,d} \subset \Sj$.
\end{thrm}

\subsection{An overview of Part~\ref{part2}}

Generalizing the constructions in affine type $A$ and finite type $C$ \cite{FL15} (see also ~\cite{Lu00}),
we introduce a comultiplication-like homomorphism 
$\Delta^{\fc}  =\Delta^{\fc}_{d',d''}: {\mbf S}_{n,d}^{\fc} \rightarrow  {\mbf S}_{n,d'}^{\fc} \otimes  {\mbf S}_{n,d''}$,
for a composition $d=d' + d''$.
This further leads to a transfer map of affine type $C$ (which is  an algebra homomorphism)
$\phi^{\fc}_{d, d-n}: \Sj \rightarrow {\mbf S}^{\fc}_{d -n , n}$, which is shown to preserve the Chevalley generators. 
Both homomorphisms $\Delta^{\fc}_{d',d''}$ and $\phi^{\fc}_{d, d-n}$ make sense on the level of Schur algebras instead of Lusztig algebras.

The algebra $\bU^{\fc}_n$ is  by definition a suitable subalgebra  of the projective limit of the projective system $\{(\bU^{\fc}_{n,d}, \phi^{\fc}_{d, d-n})\}_{d\ge 1}$,
just as $\bU_n$ is a limit algebra for a similar affine type $A$ projective system. 
Recall by Proposition~\ref{Un=slhn} (due to Lusztig) we have an algebra isomorphism $\bU_n \cong \bU (\slh_n)$.
We show that the family of homomorphisms $\{\Delta^{\fc}_{d',d''} \}$ gives rise to a homomorphism
$\Delta^{\fc}: \bU^{\fc}_n \rightarrow \bU^{\fc}_n \otimes \bU_n$ and an injective homomorphism $\jmath_n: \bU^{\fc}_n \to \bU_n$,
whose images on the Chevalley generators are explicitly given.

\begin{thrm} [Theorem~\ref{thm:QSP}]
 \label{thm:QSP:intr}
The algebra $\bU^{\fc}_n$ is a coideal subalgebra of $\bU (\slh_n)$, 
and the pair $( \bU (\slh_n), \bU^{\fc}_n)$ forms a quantum symmetric pair of affine type in the sense of Letzter and Kolb \cite{Ko14}.
(The relevant involution is illustrated in Figure~\ref{figure:jj}.)
\end{thrm}

Thanks to Theorem~\ref{thm:QSP:intr}, it makes sense to denote $\bU^{\fc}_n =\bU^\C (\slh_n)$. 
One can also formulate an idempotented form of $\bU^{\fc}_n$, denoted by $\dot{\bU}^{\fc}_n$ or $\dot \bU^\C (\slh_n)$, which is analogous to 
the idempotented quantum groups as formulated in \cite{BLM90, Lu93}. 
Following the approach of \cite{Mc12} in the affine type $A$ setting and \cite{LW15} in the finite type $B$ setting,
we construct canonical basis for $\dot{\bU}^{\fc}_n$ and establish its positivity with respect to the multiplication 
and a bilinear pairing of geometric origin.
Following \cite{FL15} in the finite type $B$ setting, we establish the positivity of the canonical basis for $\dot{\bU}^{\fc}_n$ with respect to the comultiplication. 

\begin{thrm} [Theorem~\ref{thm:iCB-Unc}, Theorem~ \ref{thm:positivity-Unc}]
  \label{thm:positivity:Intr}
The algebra $\dot \bU^{\fc}_n$ admits a canonical basis $\dot{\mbf B}^{\fc}_n$. 
The structure constants of the canonical basis $\dot{\mbf B}^{\fc}_n$ with respect to the multiplication and comultiplication are all positive,
that is, they lie in $\mbb N[v, v^{-1}]$  and so do they with respect to  
the bilinear pairing, that is, they lie in $\mbb N [[v^{-1}]]$.
\end{thrm}

Recall in the finite type $C$ setting, there are geometric realizations of two quantum symmetric pairs (with superscripts $\jmath$ and $\imath$), 
the superscript $\jmath$ corresponds to the Dynkin diagram involution without fixed point 
and $\imath$ to the involution with a fixed point. 
The involution for $\slh_n$ (where $n$ is even) in Figure~\ref{figure:jj} has no fixed point. 
In this paper we construct three more variants of quantum symmetric pairs
arising from the affine type $C$ flags.  The remaining three cases
are labelled by superscripts $\ji, \ij$, $\ii$ and they correspond to involutions which are illustrated in Figures~\ref{figure:ji}, \ref{figure:ij}
and \ref{figure:ii}, respectively (the superscript $\fc$  for the algebras above could be denoted by $\jmath\jmath$). 

{\bf In each of the three new variants, we have counterparts of Theorems~\ref{CB-Udn:Intr}, \ref{thm:QSP:intr} and \ref{thm:positivity:Intr}.}
The proofs are sometimes more difficult, as it is already clear in the finite rank $\imath$-version \cite{BLW14, LW15, FL15}. 


There is also a totally different, purely algebraic, construction \cite{BW16} of canonical bases for general quantum symmetric pairs  (cf. \cite{Ko14}). 
That approach does not establish the positivity of canonical bases. 

\subsection{An overview of Part~\ref{part3}}

In contrast to the finite types,  the Schur algebra $\Sj$ is not generated by the Chevalley generators in general, 
that is, $\bU_{n,d}^{\fc}$ is a proper subalgebra of $\Sj$
(this phenomenon already happens in affine type $A$ \cite{Lu99}). 
The next goal (Part 3) is to understand  the limit algebra $\Kc_n$ arising
from the family of Schur algebras $\{\Sj\}_{d\ge 1}$ as well as its $\ji, \ij, \ii$-variants. 
One key difficulty we encounter here is that the Schur algebras $\Sj$ do not have any obvious (finite) generating set to start with,
and this makes it tricky to understand the stabilization. 

To that end, we introduce a new idea by imbedding $\Sj$ into the Lusztig algebra $\U_{\breve n, d}^\C$ (with $\breve n =n+2$).
The imbedding $\Sj \to \U_{\breve n, d}^\C$ is constructed as an imbedding $\Sj \to \Sji_{\breve n, d}^\C$ 
(in a way similar to the embedding $\Sji_{\nn,d}^{\ji} \to \Sj$ earlier) which factors through $\U_{\breve n, d}^\C$.
As Lusztig algebras have a nice set of Chevalley generators and they are well understood in Part~\ref{part1} and Part~\ref{part2}, 
we gain insights about $\Sj$ this way. 

One first result which we obtain via such an imbedding is to establish a (bar invariant) monomial basis $\{ \f_A | A\in \MX_{n,d}\}$ for $\Sj$,
and we see that $\Sj$ is generated by the standard basis elements $[A]_d$ with $A$ tridiagonal. 
(In affine type $A$, it was first shown \cite{DF13} that the Schur algebra is generated by the standard basis elements ${}^{\fa} [A]_d$ for $A$ bidiagonal.)
In our affine type $C$ setting, thanks to  the centrosymmetry condition of the matrices $A$ parametrizing the basis of $\Sj$,
the appearance of tridiagonal matrices parametrizing a generating set is perhaps not surprising.
It does make any possible multiplication formula in affine type $C$ with $[A]$ for $A$ tridiagonal enormously complicated. 

The imbedding $\Sj \to \U_{\breve n, d}^\C$ and the monomial basis for $\Sj$ further 
allow us to study fruitfully the stabilization as $d$ goes to infinity of the multiplication,
comultiplication, and bar involution on $\Sj$.   
The stabilization properties for $\Sj$ allow us to introduce a limit algebra $\Kc_n$ and establish  its main properties.

\begin{thrm} [Theorems~\ref{Kj-bases}, \ref{thm:Psi:c}]
  \label{thm:Kbasis:intr}
The algebra $\K^{\C}_n$ admits a standard basis $\{[A] | A\in \widetilde{\Xi}_n\}$, 
a monomial basis $\{ \f_A | A\in \widetilde{\Xi}_n\}$,
and a stably canonical basis
$\{\{A \}| A\in  \widetilde{\Xi}_n\}$.
Moreover, there is a natural surjective algebra homomorphism $\Psi_{n, d} : \K^{\C}_n \rightarrow  \Sj$
which sends each stably canonical basis element to a canonical basis element or zero. 
\end{thrm}

In a completely analogous way and as a byproduct, we can formulate the stabilization properties of the family of Schur algebras $\mbf S_{n,d}$ of affine type $A$
and introduce its stabilization algebra $\K_n$, and prove a theorem for $\K_n$ analogous to Theorem~\ref{thm:Kbasis:intr}.
Such results in affine type $A$ were first obtained in \cite{DF13, DF14} by a completely different and algebraic approach,
and they also identify $\K_n$ as the idempotented quantum affine $\glh_n$.
Our geometric approach here offers a shortcut to some main results in  {\em loc. cit.} 
and obtains new results on the comultiplication structure. 

The stabilization property of the comultiplication on $\Sj$ leads to the following.

\begin{thrm}[Propositions ~\ref{prop-cop-alg-hom}, \ref{prop-cop-coass}, Remark~\ref{rem:KnQSP}]
\label{thm:KnQSP:intr}
The pair $(\K_n, \K^{\C}_n)$ forms a quantum symmetric pair (in an idempotented form).
\end{thrm}

Similarly, the other families of Schur algebras $\{{\mbf S}^{\ji}_{\nn,d} \}_d, \{{\mbf S}^{\ij}_{\nn,d} \}_d,$ and $\{{\mbf S}^{\ii}_{\mm,d} \}_d$
admit similar stabilizations which lead to limit algebras $\Kcji, \Kcij, \Kcii$, respectively. 
We also establish the counterparts of Theorems~\ref{thm:Kbasis:intr} and \ref{thm:KnQSP:intr}  for the algebras $\Kcji, \Kcij, \Kcii$.
In the process, we actually establish the following interrelations in Section~\ref{sec:diag}  (where one finds 
the precise definition of subquotients)
among the algebras  $\K^\C_n, \Kcji, \Kcij, \Kcii$ in a conceptual way.

\begin{thrm}[Proposition ~\ref{prop:sq}, Theorems ~\ref{Kji-sq}, \ref{Kij-bases}, \ref{Kii-bases}]
\label{thm:K4diagram:intr}
We have the following diagram of subquotient constructions ($\mathfrak{sq}$ stands for subquotients):
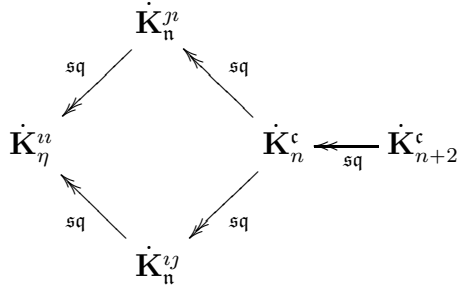
\[
\xymatrix{
&\K^{\ji}_{\nn}  \ar@{->>}[dl]_{\mathfrak{sq}} 
  &&\\
\K^{\ii}_{\eta}  && \K^{\C}_{n}  \ar@{->>}[ul]_{\mathfrak{sq}}  \ar@{->>}[dl]^{\mathfrak{sq}} & \K^{\C}_{n+2} \ar@{->>}[l]^{\mathfrak{sq}} \\
& \K^{\ij}_{\nn} \ar@{->>}[ul]^{\mathfrak{sq}} &&
}
\]
Moreover, all the subquotient constructions are compatible with the stably canonical bases.
\end{thrm}

We have been developing a Hecke-algebraic approach in a companion paper \cite{FLLLW} simultaneously, which
redevelops some of the main results of Part 3 of this paper in a completely different way.  

\section{The organization}

The paper is divided into three parts. 
Part~\ref{part1} consists of Chapters~\ref{chap:A}-\ref{chap:schur}, and it deals with the Schur algebras and Lusztig algebras
arising from convolution algebras on pairs of partial flags of affine type $C$. 
Part~~\ref{part2} consists of Chapters~\ref{chap:coideal}-\ref{chap:coideal34}, and it studies the limit algebras of each of the four families of Lusztig algebras
and identifies them as (idempotented) coideal subalgebras of the quantum affine $\sll$. 
Part~~\ref{part3} consists of Chapters~\ref{chap:Kc}-\ref{chap:Kc:ji}, and it treats the 
stabilization algebras arising from the four families of Schur algebras,
and identify them as (idempotented) coideal subalgebras of the quantum affine $\gl_n$. 

In the somewhat preliminary Chapter~\ref{chap:A}, which is exclusively on affine type $A$, 
we review the constructions of \cite{ Lu99} in affine type $A$ and set up the type $A$ notations.
We formulate Lusztig algebra $\bU_{n,d}$
as the (proper) subalgebra of the Schur algebra ${\mbf S}_{n,d}$ generated by Chevalley generators. 
A new result in this Chapter is a geometric construction of a monomial basis for $\bU_{n,d}$ and then for $\dot \bU(\asl_n)$.
This makes our approach here and further generalization in affine type $C$ below quite different from those in \cite{Mc12, SV00}.
In particular, the approach here does not rely on the crystal basis theory of Kashiwara and Ringel-Hall algebras. 

Before proceeding to the remaining chapters, we recommend the reader to browse Appendix~\ref{chap:finiteC}.
In Appendix~\ref{chap:finiteC}, we review and expand the geometric constructions from \cite{BKLW14, FL15} in finite type $C$.
Recall most of the results in {\em loc. cit.} were formulated in detail in the geometric setting of finite type $B$.

From now on we take $n$ to be a positive {\bf even} integer. 

In Chapter~\ref{chap:latticeC}, 
we present lattice models for the  variety $\Y^{\fc}$ of complete flags of affine type $C$, following 
\cite{Sa99}, ~\cite{H99} and ~\cite{Lu03}.
We also formulate a variety $\X^{\fc}_{n,d}$ of $n$-step flags of affine  type $C$.
Then we classify the orbits of products  $\X^{\fc}_{n,d} \times \Y^{\fc}$ and $\X^{\fc}_{n,d} \times \X^{\fc}_{n,d}$
under the diagonal action of the loop symplectic group. 

In Chapter~\ref{chap:formula}, 
we study the Schur algebra $\Sj$ arising from the convolution algebra of pairs of $n$-step flags of affine type $C$.
 We present multiplication formulas in $\Sj$ with the Chevalley generators and with their divided powers. 
We then specify some general scenarios where these multiplication formulas produce a leading term with coefficient $1$. 
The results in this chapter are local in the sense that they are analogous to the results in finite types $A$ and $C$.

In Chapter~\ref{chap:schur},
we introduce the Lusztig algebra $\bU_{n,d}^{\fc}$
as the (proper) subalgebra of the Schur algebra $\Sj$ generated by Chevalley generators. 
We then introduce a coideal algebra type structure which involves both Schur algebras (and respectively, Lusztig algebras) of affine types $C$ and $A$. 
This leads to an imbedding  $\jmath_{n,d}$ from $\Sj$ to ${\mbf S}_{n,d}$, and also from $\bU_{n,d}^{\fc}$ to ${\mbf U}_{n,d}$. 
The canonical basis and monomial basis are shown to be
compatible with the inclusion $\bU_{n,d}^{\fc} \subset \Sj$. 

In Chapter~\ref{chap:coideal}, 
we introduce the transfer maps $\phi^{\fc}_{d, d-n}$ on Schur algebras $\Sj$ and Lusztig algebras $\bU_{n,d}^{\fc}$.
We then construct algebras $\bU_n^{\fc}$ (or $\dot{\bU}_n^{\fc}$) from the projective system of algebras 
$\{(\bU_{n,d}^{\fc}, \phi^{\fc}_{d, d-n})\}_{d \ge 0}$. We show that
$\bU_n^{\fc}$ (or $\dot{\bU}_n^{\fc}$)  is isomorphic to an (idempotented) coideal subalgebra of $\bU(\slh_n)$, and $(\bU(\slh_n), \bU_n^{\fc})$ forms
an affine quantum symmetric pair. 
The canonical basis of $\dot{\bU}_n^{\fc}$ is established and shown to admit positivity with respect to
multiplication, comultiplication, and a bilinear pairing.

In the remainder of the Introduction we set $\nn =n-1$ (which is odd) and $\mm =n-2$ (which is even). 

In Chapter~\ref{chap:coideal2} and Chapter~\ref{chap:coideal34}, we present several more projective systems
$\{(\bU_{\nn,d}^{\ji}, \phi^{\ji}_{d, d- \nn})\}_{d \ge 0}$, $\{(\bU_{\nn,d}^{\ij}, \phi^{\ij}_{d, d-\nn})\}_{d \ge 0}$, and $\{(\bU_{\mm,d}^{\ii}, \phi^{\ii}_{d, d-\mm})\}_{d \ge 0}$.
We emphasize that each of these Lusztig algebras arises from convolution algebras of geometric origin. 
We obtain the limit algebras $\bU_{\nn}^{\ji}$, $\bU_{\nn}^{\ij}$, $\bU_{\mm}^{\ii}$ and their idempotented counterparts.
We show that $\bU_{\nn}^{\ji}$ (respectively, $\bU_{\nn}^{\ij}$, or $\bU_{\mm}^{\ii}$)  is isomorphic to a coideal subalgebra of $\bU(\slh_{\nn})$
(respectively, $\bU(\slh_{\nn})$ or $\bU(\slh_{\mm})$).
The monomial and canonical bases of $\bU_{\nn,d}^{\ji}, \bU_{\nn,d}^{\ij},$ and $\bU_{\mm,d}^{\ii}$ are established by relating to
their counterparts for $\bU_{n,d}^{\C}$.
The canonical bases of $\dot{\bU}_{\nn}^{\ji}$, $\dot{\bU}_{\nn}^{\ij}$ and $\dot{\bU}_{\mm}^{\ii}$ are established and shown to admit positivity with respect to
multiplication, comultiplication, and a bilinear pairing.

In Chapter~\ref{chap:Kc}, 
we study the stabilization properties of the family of Schur algebras $\Sj$ (as $d$ varies).
To overcome the difficulty of working with the Schur algebra $\Sj$ which does not have a good finite generating set, 
we study $\Sj$ via an imbedding into a Lusztig algebra of higher rank. 
This allows us to understand generating sets, monomial bases, multiplication, comultiplication and bar operators of the Schur algebras and their stabilization properties
in  a conceptual way and lift all these structures to a stabilization algebra $\K^{\C}_n$. 
We show that $\Kc_n$ admits a stably canonical basis, and 
the pair $(\K_n, \Kc_n)$ forms a quantum symmetric pair in an idempotented form,  
where $\K_n$ is isomorphic to the idempotented quantum affine $\gl_n$.

In Chapter~\ref{chap:Kc:ji}, we formulate the main results for the stabilizations of the remaining 3 families 
of Schur algebras of types $\ji, \ij, \ii$,
following the blueprints in Chapter~\ref{chap:Kc}. 
Moreover, we establish interrelations among all the stabilization algebras 
$\Kc_n$, $\Kcij$, $\Kcij$, and $\Kcii$ of types $\jj, \ji, \ij, \ii$, and among their stably canonical bases.

Notation: $\mbb N =\{0, 1, 2, \ldots\}$. 


\vspace{.3cm}
\noindent {\bf Acknowledgement.} 
We thank Huanchen Bao, Jie Du and Xuhua He for helpful discussions.  
We thank the following institutions whose support and hospitality help to facilitate the progress and completion of this project:
East China Normal University (Shanghai), Institute of Mathematics, Academia Sinica (Taipei), University of Virginia.
The last author is partially supported by the NSF grant DMS-1405131. 

\newpage

\part{Affine flag varieties, Schur algebras, and Lusztig algebras}
  \label{part1}

\chapter{Constructions in affine type $A$}
\label{chap:A}

This chapter is preliminary in nature. Most of it has been well known \cite{Lu99, Lu00, SV00, Mc12} (also cf. \cite{DF14}).
However we present a new geometric construction of a monomial basis (and hence canonical basis) 
for  the modified quantum group
$\dot \bU(\asl_n)$, in analogy to the one in ~\cite[Proposition 3.9]{BLM90}, 
without use of crystal basis ~\cite{K91} and the theory of Ringel-Hall algebras ~\cite{R90}, 
(see also ~\cite{Sch06, DDPW08, DDF12}, ~\cite{VV99, LL15}).

\section{Lattice presentation of affine flag varieties of  type $A$}

Let $k$ be a finite field of $q$ elements, where $q$ is a prime power.
Let $F=k((\ve))$ be the field of formal Laurent series over $k$ and $\A=k[[\ve]]$ the ring of formal power series.
Let $d$ be a positive integer.
Let $\GL_F(d)$ (respectively, $\GL_{\A}(d)$, $\GL_k(d)$)  be the invertible $d\times d$ matrices with coefficients in $F$ (respectively, $\A$, $k$).
Consider a reduction mod-$\ve$ map
$
\text{ev}|_{\ve=0}: \GL_{\A}(d) \to \GL_k(d), \quad \ve \mapsto 0.
$
The $parahoric$ $subgroups$ of $\GL_F(d)$ are inverse images of parabolic subgroups of $\GL_k(d)$ under $\text{ev}|_{\ve=0}$,
and the parahoric subgroups which are inverse images of Borel subgroups are called $Iwahori$ $subgroups$.
The affine partial flag of type $A$ is then defined to be the homogeneous space $\GL_F(d)/ P$ where $P$ is a parahoric subgroup.

Let  $V$ be an $F$-vector space of dimension $d$.
A free $\A$-submodule $\mathcal L$ of $V$ of rank $d$  is called a $lattice$ in $V$. 
Let $\Y^{\mathfrak a}$ be the set of all lattice chains $L = (L_i)_{i\in \ZZ}$ where each $L_i$ is a lattice in $V$, such
that $L_{i-1} \subset  L_i$ and $L_{i-d} = \ve L_i$ for all $i \in \ZZ$.

We fix a basis $\{ e_1,\ldots, e_d\}$ for $V$, and
we set
\[
e_m= {\ve}^{-s} e_i, \quad \mbox{if } m= sd+i \text{ for }  i\in [1,d].
\]
Then we have a total order for $(e_m)_{m\in\mbb Z}$ as follows: 
\[
\ldots, \ve e_1, \ldots, \ve e_d, e_1, \ldots, e_d, \ve^{-1} e_1, \ldots, \ve^{-1} e_d, \ldots.
\]

Clearly,
\[
\L_0 = \A e_1\oplus \cdots \oplus \A e_d
\]
is a lattice in $V$. More generally, for $m= s d +i$ with $1\leq i\leq d$, we define the lattice
\begin{align*}
\L_m &=[e_{m+1},\ldots, e_{m+d}]_{\A}\\
&=   \A \ve^{-s} e_{i+1} \oplus \cdots \oplus \A \ve^{-s} e_d \oplus \A \ve^{-s-1} e_1\oplus \cdots \oplus \A \ve^{-s-1}  e_i.
\end{align*}

We set $\L = (\L_m| m\in \mbb Z)$ to be the standard lattice chain.
There  exists a surjective map
\[
\GL_F(d) \to \Y^{\mathfrak a},\quad g\mapsto g. \L.
\]
It is clear that the stabilizer $\mbf  I^{\mathfrak a}$ of $\L$ in $\GL_F(d)$ consists  exactly  of the mod-$\ve$ upper triangular matrices.  
Thus $\mbf I^{\mathfrak a}$ is  an Iwahori subgroup of $\GL_F(d)$. We thus have the identification of affine flag variety of type $A$:
\begin{equation}
\label{flag-A}
\GL_F(d) / \mbf  I^{\mathfrak a} \longrightarrow  \Y^{\mathfrak a}.
\end{equation}
There are similar lattice chain models for the partial flag varieties of type $A$. 

\section{Monomial basis for quantum affine $\mathfrak{sl}_n$}
\label{sec:MB-A}

In this section, we shall construct an explicit monomial basis for quantum affine $\mathfrak{sl}_n$
(the construction here will be generalized in latter chapters).

For the partial flag cases, the treatment is similar.
More generally, we consider the set $\X_{n,d}$ of $n$-periodic lattice chains in $V$. Here $\dim_F V=d$ and 
 a sequence $ L = ( L_i)_{i\in \mbb Z}$  of lattices in $V$ is called an $n$-periodic  lattice chain if  $L_i \subseteq  L_{i+1}$ and $L_i = \ve L_{i+n}$ for all $i\in \mbb Z$.
The group $\GL_F(d)$ acts naturally on $\X_{n,d}$ from the left, and then  acts on the product $\X_{n,d} \times \X_{n',d}$ diagonally,
for a pair $(n, n')$  of positive integers.

Let $\Theta_{n|n',d}$ be the set of all  matrices
$A=(a_{ij})_{i,j \in \mbb Z}$ with non-negative integer entries
satisfying the following conditions:
\begin{equation} 
   \label{Theta:dn}
(i) \; a_{ij}=a_{i+n, j+n'} \; (\forall i, j\in \mbb Z); \quad (ii) \; \sum_{i=i_0}^{i_0+n-1} \sum_{j\in \mbb Z} a_{ij}=d, \text{ for each (or for all) }i_0 \in \mbb Z.
\end{equation}
The condition ($ii$) can be equivalently replaced by ($ii'$) below:
\begin{itemize}
\item [$(ii')$] For any $j_0\in \mbb Z$, $\sum_{j=j_0}^{j_0+n'-1} \sum_{i\in \mbb Z} a_{ij}=d$.
\end{itemize}
A matrix $A$ in $\Theta_{n|n',d}$ automatically satisfies that, for any $i\in \mbb Z$, the sets $\{j\in \mbb Z| a_{ij}\neq 0\}$
                and $\{ j\in \mbb Z| a_{ji}\neq 0\}$ are finite. 

Following \cite{Lu99}, the $\GL_F(d)$-orbits in $\X_{n,d} \times \X_{n', d}$ are parametrized by the set  
$\Theta_{n|n',d}$.
More precisely, to a pair of $n$-periodic lattices $(L, L')$, we define a matrix $A = (a_{ij})_{i, j \in \mbb Z}$ where 
\[
a_{ij} = \dim_k L_i \cap L'_j / (L_{i-1} \cap L'_j + L_i \cap L'_{j-1}),
\quad (\forall i, j\in \mbb Z).
\]
This defines a bijection  $\GL_F(d) \backslash \X_{n,d} \times \X_{n',d} \leftrightarrow \Theta_{n|n',d}.$
Let $\mathcal O_A$ denote the associated $\GL_F(d)$-orbit indexed by $A$. 
We are mostly interested in the case when $n'=n$, and we shall write 
\[\Theta_{n,d} =\Theta_{n|n,d}.
\]               

We set
\begin{equation}
  \label{Land}
\Lambda_{n,d} =
 \big\{ \mbf{\lambda}=(\lambda_i)_{i\in \mbb Z} \in \mbb{N}^{\mbb Z} | \lambda_i = \lambda_{i+n}, \forall i\in \mbb Z; \sum_{1\leq i\leq n} \lambda_i=d \big\}.
\end{equation}
To each matrix $A \in \Theta_{n,d}$, we define its row/column sum vectors $\ro(A) =(\ro(A)_i)_{i\in \ZZ}$ and $\co(A) =(\co(A)_i)_{i\in \ZZ}$
in $\Lambda_{n,d}$ by
\[
\ro(A)_i = \sum_{j\in \mbb Z} a_{ij},
\quad
\co(A)_j = \sum_{i\in \mbb Z} a_{ij}\; (\forall i, j \in \mbb Z).
\]

Let $A, B, C \in \Theta_{n,d}$, we fix $L, L' \in \X_{n,d}$ such that $\dim_k L_i/L_{i-1} = \ro(A)_i$  and $\dim_k L'_j/L'_{j-1} = \co(B)_j$ 
for all $i, j \in \mbb Z$. 
We set 
\[
g^C_{A, B} (\sqrt{q}) =\# \big\{ \tilde L \in \X_{n,d} | (L, \tilde L) \in \mathcal O_A, (\tilde L, L') \in \mathcal O_B, (L, L') \in \mathcal O_C \big\}.
\]
By ~\cite{Lu99},  $g^C_{A, B}(\sqrt{q})$ is independent of the choices of $L, L'$ 
and is the specialization of a polynomial $g^C_{A, B}(v) \in \mbb Z[v, v^{-1}]$ at $v = \sqrt q$.
Note that $g^{C}_{A, B}=0$ for all but finitely many $C$. 

We set $\cA=\ZZ[v,v^{-1}]$.
The  affine Schur $\cA$-algebra of type $A$, denoted by ${\mbf S}_{n, d; \cA}$,  is by definition the (generic) 
convolution algebra  ${\cA}_{\mrm{GL}_F(d)} (\X_{n,d} \times \X_{n,d})$.  
Denote by $e_A$ the characteristic function of the orbit $\mathcal O_A$, for $A \in \Theta_{n,d}$.
Then the algebra ${\mbf S}_{n, d; \cA}$ is a free $\cA$-module with an $\cA$-basis $\{e_A \vert A \in \Theta_{n,d} \}$,
with multiplication given by
$e_A * e_B = \sum_{C} g^C_{A, B}(v) e_C$. We then set 
\begin{equation}
  \label{SchurA}
  {\mbf S}_{n,d} =\mbb Q(v) \otimes_{\cA}  {\mbf S}_{n, d; \cA}.
\end{equation}

To $A \in \Theta_{n,d}$, we define 
\[
d_A^{\mathfrak a} = \sum_{1\leq i \leq n,  i\geq k, j < l} a_{ij} a_{kl},
\]
and
\[
[A] = v^{- d_A^{\mathfrak a}} e_A.
\]
The set $\{ [A] | A\in \Theta_{n,d}\}$ is the standard basis of ${\mbf S}_{n,d}$.
Let $\{ \{A\}_d | A \in \Theta_{n, d}\}$ be the canonical basis of $\mbf S_{n, d}$.

Given $i,j \in \ZZ$, let $E^{ij}$ be the $\mbb Z\times \mbb Z$ matrix whose  $(k, \ell)$th entries are $1$, for all $(k, \ell) \equiv (i, j)$ 
(mod $n$), and $0$ otherwise;
that is, 
\begin{equation}
\label{Eij}
E^{ij} = (\mathcal E_{k,\ell})_{k, \ell \in \ZZ}, \quad \text{ where } {\mathcal E}_{k,\ell} =1 \text{ if } (k, \ell) \equiv (i, j) 
(\text{mod } n), \text{ otherwise } 
{\mathcal E}_{k,\ell} =0.
\end{equation}

\begin{Def}
\label{def:LA}
The subalgebra of ${\mbf S}_{n,d}$ generated by the standard basis elements $[X]$ 
such that either $X$ or $X - E^{i, i+1}$ or $X - E^{i+1, i}$ is diagonal,
is denoted by $\bU_{n,d}$ and called {\em Lusztig} {\em algebra (of affine type $A$)}.
\end{Def}

Let  $\bU_{n, d; \cA}$ be the subalgebra of ${\mbf S}_{n, d; \cA}$ generated 
by the standard basis element $[X]$ such that either $X - R E^{i, i+1}$ or $X - R E^{i+1, i}$ is diagonal,
for various $R \in {\mathbb N}$. 
For each $\lambda \in \Lambda_{n, d}$, let $D_{\lambda}$ be the diagonal matrix  in $\Theta_{n, d}$ whose diagonal is $\lambda$.
For each $R\in \mbb N$, $i\in \mbb Z$, we set
\begin{align}
\begin{split}
\E_i^{(R)} &= \sum [X], \quad
\F_i^{(R)}  = \sum [X],  \quad 
\mbf H^{\pm 1}_i = \sum_{\lambda \in \Lambda_{n, d}} v^{\pm \lambda_i} [D_{\lambda}], \quad
\mbf K^{\pm 1}_i  = \mbf H^{\pm 1}_{i+1} \mbf H_i^{\mp 1},
\end{split}
\end{align}
where the first and second sums run over all $X$ such that $X- RE^{i, , i+1}$ and $X - R E^{i+1, i}$ are diagonal, respectively.
Clearly, we have $\E_i^{(R)} = \E_j^{(R)}$, $\F_i^{(R)} = \F_j^{(R)}$, $\mbf H^{\pm 1}_i = \mbf H^{\pm 1}_j$ and $\mbf K^{\pm 1}_i =\mbf K^{\pm 1}_j$ for all $i \equiv j$ (mod $n$).
For convenience, we also set $1_{\lambda} = [D_{\lambda}]$, 

It is known from \cite{Lu99} that $\bU_{n, d;\cA}$ is an $\cA$-lattice of $\bU_{n,d}$ and generated by $\E_i^{(R)}$, $\F_i^{(R)}$ and $\mbf K_i^{\pm 1}$ for all $i$ and $R\in \mbb N$.
Recall  a $\ZZ\times \ZZ$-matrix $A=(a_{ij})$ 
is $aperiodic$ if
\begin{equation}
\label{aperiod}
\text{ for any $p\in \ZZ- \{0\}$ there exists
$k\in \ZZ$ such that $a_{k,k+p} = 0.$}
\end{equation}
We denote by $\Theta_{n,d}^{\ap}$ the set of all aperiodic matrices in $\Theta_{n,d}$. 
Lusztig \cite{Lu99} showed that
$\bU_{n,d}$ is a proper subalgebra of ${\mbf S}_{n,d}$ and further the subset $\{ \{A\} | A \in \Theta^{\ap}_{n, d}\}$ of the canonical basis of $\mbf S_{n, d}$ form a canonical basis $\{A\}_d$ of $\bU_{n, d}$.
Note that the latter result is completely nontrivial since the standard basis element $[A]$ for $A$ aperiodic is not in $\bU_{n, d}$ in general.

For $a \in \mathbb{Z}$ and $b \in \mathbb N$, we define 
\begin{equation}  
  \label{eq:binomial}
\begin{bmatrix}
a\\
b
\end{bmatrix}
=\prod_{1\leq i\leq b} \frac{v^{2(a-i+1)}-1}{v^{2i}-1}, \quad \text{ and } \quad [a] = \begin{bmatrix}
a\\
1
\end{bmatrix}.
\end{equation}

We define two partial orders ``$\leq_{\text{alg}}$'' and ``$\leq$'' on $\Theta_{n,d}$ as follows. 
For any $A =(a_{ij}), A' =(a_{ij}')  \in \Theta_{n,d}$,  let
\begin{align} 
\label{partial-alg}
A \leq_{\text{alg}} A'  \Longleftrightarrow
&  \sum_{k \leq i, l \geq j} a_{kl} \leq \sum_{k \leq i, l \geq j} a'_{kl}, \quad \forall i < j,   \\
& \sum_{k \geq i, l \leq j} a_{kl} \leq \sum_{k \geq i, l \leq j} a'_{kl}, \quad \forall i > j. \nonumber\\
\label{partial}
A \leq A' \Longleftrightarrow 
& A\leq_{\text{alg}} A', \ro (A) = \ro (A'), \co (A) = \co (A').
\end{align}
We further say that ``$A <_{\text{alg}} A'$'', (respectively, ``$A< A'$'') if $A\leq_{\text{alg}} A'$ (respectively, $A\leq A'$) and $A\neq A'$.
For convenience, we write ``$[A] +$ lower terms'' to stand for ``$[A]$ plus a linear sum of various $[B]$ with $B < A$''.

The following lemma is a slightly stronger affine version of ~\cite[Lemma3.8]{BLM90}, which is used to obtain an affine analogue of ~\cite[Proposition 3.9]{BLM90} for quantum affine $\mathfrak{sl}_n$.

\begin{lem}
\label{lem2}
Let $A, B, C \in \Theta_{n,d}$ and $R$ be a positive integer.

\begin{enumerate}
\item  
Assume that  $B-RE^{h, h+1}$ is diagonal for some $ h \in [1, n]$ and $\co(B) = \ro(A)$.
Assume further that  $R= R_0 + \cdots + R_l$ and the matrix $A$ satisfies   the following conditions:
\begin{align*}
a_{h j}=0, \  \forall j\geq k; a_{ h +1,k+i}= R_i,\ i\in [1,l],\  a_{ h +1,k} \geq  R_0, \  a_{h +1,j}=0,\  \forall j>k+l.
\end{align*}
Then  we have
\[
[B] * [A] = [ A+ \sum_{i=0}^l R_i(E^{h, k+i} - E^{h +1, k+i})] + \mbox{lower terms}.
\]

\item 
Assume  that $C-R E^{h+1,h} $ is diagonal for some $h\in [1,n]$ and $\co(C ) =\ro(A)$.
Assume further  that $R=R_0 + \cdots + R_l$ and   $A$ satisfies  the following conditions:
\begin{align*}
 a_{hj}=0,\  \forall j < k, \ a_{h,k+i}= R_i, i\in [0,l-1],\ a_{h,k+l} \geq  R_l; \  a_{h+1,j} =0,\  \forall j\leq k+l.
\end{align*}
Then we have
\[
[C] * [A] = [ A- \sum_{i=0}^{l} R_i (E^{h, k+i} - E^{h+1,k+i}) ] + \mbox{lower terms}.
\]
\end{enumerate}
\end{lem}

\begin{proof}
By \cite[Section 3]{Lu99}, we have
\begin{equation}
\label{B*A}
[B] * [A] = \sum_t v^{ \beta(t) } \prod_{u \in \mbb Z}
\overline{\begin{bmatrix} a_{hu} + t_u \\ t_u \end{bmatrix}}
[A + \sum_{u\in \mbb Z} t_u (E^{hu} - E^{h+1,u} )],
\end{equation}
where $  \beta(t) =
\sum_{j \geq  u} a_{h j} t_u - \sum_{j > u} a_{h+1, j} t_u + \sum_{j < u} t_j t_u$.
Here the bar is the involution on $\mbb Q(v)$ defined by $\bar v = v^{-1}$.
Observe that 
$A +  \sum_{i=0}^l R_i (E^{h, k}- E^{h +1, k})$ is the leading term for the right hand side of (\ref{B*A}).

We shall show its coefficient is 1.
Note that the leading term is determined by
 \[t_{k+i}=  R_i,\ t_j=0,\ \forall i \in [0,l],\ j \not \in  [k, k+l].
 \]
 In this case, we have $\prod_{u \in \mbb Z}
\overline{\begin{bmatrix} a_{hu} + t_u \\ t_u \end{bmatrix}} =1$ and
\[\beta(t) = \sum_{j \geq  u}(t_j- a_{h+1, j}) t_u 
=   
\sum_{i=0}^l \sum_{j > k+i} (t_{j}-a_{h+1, j}) R_i =0.
\]
This shows (1).
Part (2) can be proved similarly.
\end{proof}

A product of standard basis elements $[G_1] * [G_2] * \cdots *[G_m]$ in ${\mbf S}_{n,d}$ is called an $aperiodic$ $monomial$ if for each $i$,   
either $G_i - R E^{j, j+1}$ or $G_i - RE^{j+1, j}$ is diagonal for some $R\in \mbb N$ and $j\in \mbb Z$. 
The following proposition is a missing piece  in the affine generalization of ~\cite{BLM90}, 
corresponding to Proposition 3.9 in the {\it loc. cit}.
We refer to ~\cite{DD05} and the references therein for early treatments using Ringel-Hall algebras and generic extension.

\begin{prop} \label{Astandard-basis}
For any $A \in \Theta_{n,d}^{\ap}$,
there exists (and we shall fix) an aperiodic  monomial $\zeta^{\mathfrak a}_A$ such that 
$\zeta^{\mathfrak a}_A = [A] + \mbox{lower terms}.$
Moreover,  the set $\{ \zeta^{\mathfrak a}_A \mid A \in \Theta_{n,d}^{\ap} \}$ is a basis for $\bU_{n,d}$. 
\end{prop}

\begin{proof}
Recall \cite{Lu99} that $\{ \{A\} \vert A \in \Theta_{n,d}^{\ap} \}$ forms a canonical basis for $\bU_{n,d}$. 
Assuming the first statement on  the existence of such  $\zeta^{\mathfrak a}_A$,
we then have $\zeta^{\mathfrak a}_A =\{A\} +$lower terms in $\bU_{n,d}$, and hence 
$\{ \zeta^{\mathfrak a}_A \mid A \in \Theta_{n,d}^{\ap} \}$ forms a basis for $\bU_{n,d}$.

It remains to prove the existence of such 
an aperiodic  monomial $\zeta^{\mathfrak a}_A$.
Let us fix some notations. 
Given a matrix $A=(a_{ij}) \in \Theta_{n,d}$, we define a matrix 
$$f_{k; s,t}(A) = A - \sum_{s\leq j\leq t} a_{k-1,j} (E^{k-1,j} - E^{k,j})  \in \Theta_{n,d}.$$ 
Let $\Psi (A) = \sum_{i\in [1,n]} |j-i|a_{ij}$. 
It is clear that $\Psi(f_{k; s,t} (A)) \leq \Psi(A)$ for all $k, s$ and $t$ with $k \leq s \leq t$, 
where the equality holds if and only if 
\begin{equation}
  \label{condition1}
  a_{k-1, j} = 0,\ \forall s \leq j \leq t.
\end{equation}

We are now ready to prove the existence of such 
an aperiodic  monomial $\zeta^{\mathfrak a}_A$ by induction on $\Psi(A)$.
If $\Psi (A) =0$, then $A$ is a diagonal matrix, and $\zeta^{\mathfrak a}_A = [A]$.

We now assume that $\Psi(A) \geq 1$ and that the existence of such $\zeta^{\mathfrak a}_{A'}$ for all aperiodic matrices $A'$ with $\Psi(A') < \Psi(A)$.
Set  $m ={\rm min} \{l \in \mbb N | a_{ij} =0\ {\rm for\ all}\ |i-j| > l\}$.
If  there exists $k \in \mbb Z$ such that
$a_{k, k+m}=0$ and $a_{k-1, k -1 + m} \neq 0$.
By (\ref{condition1}), we have $\Psi(f_{k; s, t} (A)) < \Psi(A)$ for all $k \leq s \leq t$.
  
Let $u = \max \{s  \leq k+m -1 \mid f_{k; s, k+m-1} (A) \ \text{is aperiodic}\}$.
We have $a_{kl}=0$ for all $l > u$.
(Otherwise, there exists $j > u$ such that $a_{kj} \neq 0$.
Then $f_{k; j, k+m-1}(A)$ is aperiodic, which contradicts with the definition of $u$.)
By Lemma \ref{lem2} (1), we have
\begin{equation}
\label{case-1}
 [B]  * [f_{k, u, k+m-1}(A)] = [A] + {\rm lower\ terms},
\end{equation}
where $B$ is the matrix such that $\co(B) = \ro(f_{k, u, k+m-1}(A))$  and
$B -  \sum_{l=u}^{k+m-1} a_{k-1,l} E^{k-1,k}$ is diagonal.

If  there exists $k \in \mbb Z$ such that
$a_{k, k - m} \neq 0$ and $a_{k-1, k -1 - m} = 0$, 
we can prove a statement similar to (\ref{case-1}) by using Lemma ~\ref{lem2}(2).
By induction on $\Psi(A)$, the existence of  $\zeta^{\mathfrak a}_A$ follows.
\end{proof}

\begin{example}
Let $n=2$. Let $A$ be a lower triangular matrix whose nonzero entries  are 
located at $(5, 5)$, $(6, j)$  (mod 2), for $2\leq j \leq 6$, which are
\[
a_{55} =1, a_{62} = 2, a_{63} = 3, a_{64}=2, a_{65}=1, a_{66}=2.
\]
Let $A'=(a'_{ij}) $ be the lower triangular matrix whose nonzero entries are specified by
\[
a'_{52}=2, a'_{53} =3, a'_{54} = 2, a'_{55}=2, a'_{66}=2.
\]
Applying the algorithm in the proof,  
$
\E^{(8)}_1 * [A'] = [A] + \mbox{lower terms}. 
$
Inductively, we have
\[
\E^{(8)}_1 * \E^{(7)}_0 * \E^{(5)}_1*  \E^{(2)}_0 * 1_{\co(A)} = [A] + \mbox{lower terms}. 
\]
\end{example}

\section{Algebras $\bU_n$ and $\dot{\bU}_n$}
  \label{sec:Un-A}

Recall a transfer map $\phi_{d, d-n}: \bU_{n,d} \rightarrow \bU_{n,d-n}$ was introduced in \cite{Lu00} by sending the generators $\E^{(R)}_i$, $\F^{(R)}_i$ and $\mbf K_i^{\pm 1}$ to the respective generators.
Let us define a partial order $\leq_n$ on $\mbb{N}$ by declaring that
\[
a \leq_n b \ \mbox{iff} \ b-a = p n \quad \ \mbox{for some} \ p \geq 0.
\]
Then   $\{(\bU_{n,d}, \phi_{d, d-n})\}_{d\in \mbb{N}}$ form a projective system over  the poset $(\mbb{N}, \leq_n)$.
We shall consider its projective limit:
\[
\bU_{n,\infty} \equiv
\varprojlim_{d } \bU_{n,d}
= \Big \{ x \equiv  (x_d)_{d\in \mbb{N}}  \in \prod_{d\in \mbb{N}} \bU_{n,d}  \Big\vert \phi_{d, d-n}(x_d) = x_{d-n} \quad \forall d \Big \}.
\]
The bar involution on $\bU_{n,d}$ induces a bar involution
$
\bar \ : \bU_{n,\infty} \to \bU_{n,\infty},
$
since it commutes with the transfer map \cite{Lu00}.
Similarly, we have an integral version:
$\bU_{n, \infty;\cA} =
\varprojlim_{d}  \bU_{n, d; \cA}.$
Since $\mbb Q(v) \otimes_{\mathcal A} \bU_{n, d; \cA} = \bU_{n,d}$ for all $d$, we have
$\mbb Q(v) \otimes_{\mathcal A} \bU_{n, \infty;\cA} = \bU_{n, \infty}$.

As we  deal with all $d \in \mbb{N}$ simultaneously, we will write
\[
1_{\lambda, d}, \E_{i, d}, \F_{i, d}, \mbf K^{\pm 1}_{i, d}, \quad \forall 1\leq i \leq n
\]
for the generators in $\bU_{n,d}$, which are denoted  without $d$ in the subscript previously.
Since the transfer map sends generators $\E_{i, d}$, $\F_{i, d}$ and $\mbf K^{\pm 1}_{i, d}$ to the respective generators, we can define
elements $\E_i, \F_i$ and $\mbf K^{\pm 1}_i$ for all $1\leq i \leq n$  in $\bU_{n,\infty}$ by
declaring that their $d$-th component is $\E_{i, d}$, $\F_{i, d}$ and $\mbf K^{\pm 1}_{i, d}$, respectively.
(Similarly, we can define the $a$-th divided power $\E_i^{(a)}$ and $\F_i^{(a)}$.)

\begin{Def}
Let $\bU_n$ be the subalgebra of $\bU_{n,\infty}$ generated by $\E_i, \F_i$ and $\mbf K^{\pm 1}_i$ for all $1 \leq i \leq n$.
\end{Def}

Clearly, the restriction of the natural projection $\phi_d: \bU_{n,\infty} \to \bU_{n,d}$ gives us a surjective algebra homomorphism:
\[
\phi_d : \bU_{n} \longrightarrow \bU_{n,d}.
\]



We set
\begin{equation}
  \label{Zn}
\mbb Z_n =\{ \lambda=(\lambda_i)_{i\in \ZZ}  \mid  \lambda_i \in \ZZ, \lambda_i = \lambda_{i+n}, \forall i\}.
\end{equation}
We define an equivalence relation $\sim$ on $\mbb Z_n$ by
\[
\lambda  \sim \mu   \Leftrightarrow \lambda - \mu = (\ldots, p, p, p, \ldots), \ \mbox{for some} \ p\in \mbb Z.
\]
Let $\mbb Z_n/\sim$ be the set of equivalence classes and $\bar{\lambda}$ be the equivalence class of $\lambda$.
Let
\[
\mbb X = \mbb Z_n/ \sim,
\qquad
\mbb Y = \{\nu \in \mbb Z_n | \sum_{1\leq i\leq n} \nu_i =0\}.
\]
Then the standard dot product on $\mbb Z_n$  induces a  pairing
$
\cdot: \mbb Y \times \mbb X \longrightarrow  \mbb Z.
$
Set $I=\{1,\ldots, n\}$. We define two injective maps
$I \to \mbb Y, \quad I \to \mbb X,$
by letting 
\[
i \mapsto - \epsilon_i + \epsilon_{i+1}, \; i \mapsto -\bar{\epsilon}_i + \bar{\epsilon}_{i+1}, \quad \forall 1\leq i\leq n,
\]
respectively, where $\epsilon_i$ is the $i$-th  standard basis element in $\mbb Z_n$, that is $(\epsilon_i)_j = \delta_{\bar i, \bar j}$.
We thus obtain a root datum of affine type $A_{n-1}$  in  ~\cite[2.2]{Lu93}.  

For each $\bar \lambda \in \mbb X$, we define an element $1_{\bar \lambda} $ in $\bU_{n,\infty}$ by
setting $(1_{\bar \lambda})_d = 0$ unless $|\lambda |: =\sum_{1\leq i\leq n} \lambda_i  = d$ mod $n$, and in which case
$(1_{\bar \lambda})_d = 1_{\mu, d}$ where $\mu \in \bar \lambda$ and $|\mu| =d$.
We define $\dot{\bU}_n$ to be the $\bU_n$-bimodule in $\bU_{n,\infty}$ generated by $1_{\bar \lambda}$ for all $\bar \lambda \in \mbb X$.
It is clear then that $\dot{\bU}_n$ is naturally a subalgebra in $\bU_{n,\infty}$.
The algebra $\dot{\bU}_n$ admits a decomposition
\[
\dot{\bU}_n = \bigoplus_{\bar \lambda \in \mbb X} \bU_n 1_{\bar \lambda} =
 \bigoplus_{\bar \mu, \bar \lambda \in \mbb X} \ _{\bar \mu} (\dot{\bU}_n)_{\bar \lambda},
\]
where $_{\bar \mu} (\dot{\bU}_n)_{\bar \lambda} = \ 1_{\bar \mu} \dot{\bU}_n 1_{\bar \lambda}$.

Let $\bU(\asl_n)$ be the affine quantum group of type $A_{n-1}$ (of level zero) attached to the above  root datum.
Let $\dot \bU(\asl_n)$ be its modified form.
The following result is due to Lusztig ~\cite[Proposition~3.5]{Lu00} (more precisely, the first one was explicitly written down therein,
while the second one is folklore as it follows in the same way as in the finite type $A$ case \cite{BLM90}.)

\begin{prop}
 \label{Un=slhn}
We have the algebra isomorphisms $\dot \bU(\asl_n) \simeq \dot{\bU}_n$, and 
$\bU(\asl_n) \simeq \bU_n$.
\end{prop}

\begin{proof}
We regard $\dot \bU_n$ as the left modules of $\bU_n$ and $\bU(\asl_n)$.
Then we have two algebra homomorphisms
$
\bU_n \to \End ( \dot \bU_n),$
and
$\bU(\asl_n) \to \End(\dot \bU_n).
$
Both maps are injective and have obviously the same image, so they must be isomorphic.
(In short, $\bU_n$ and $\bU(\asl_n)$ act faithfully on $\dot\bU_n$.)
\end{proof}

Therefore the geometric pair $(\bU_n, \dot \bU_n)$ is identified with the algebraic pair $(\bU(\asl_n), \dot \bU(\asl_n))$.

\chapter{Lattice presentation of affine flag varieties of type $C$}
  \label{chap:latticeC}

We present lattice models for the  variety $\Y^{\fc}$ of affine complete flags and a variety $\X^{\fc}_{n,d}$ of $n$-step flags in an $F$-vector space $V$ of affine  type $C$, for $n$ even.
Then we classify the ${\SP}_F(V)$-orbits on $\X^{\fc}_{n,d} \times \X^{\fc}_{n,d}$, $\X^{\fc}_{n,d} \times \Y^{\fc}$, and $\Y^{\fc} \times \Y^{\fc}$. 

\section{Affine complete flag varieties of type $C$}

Recall $k$ is a finite field of odd $q$ elements, $F=k((\ve))$ is the field of formal Laurent series over $k$, and $\A=k[[\ve]]$ the ring of formal power series.
Let $d$ be a positive integer.  Let
\begin{align} \label{J}
\begin{split}
J =
\begin{pmatrix}
0 & 0  & \cdots & 0 & 1\\
0 & 0  & \cdots & 1 & 0 \\
. & . & \cdots & . & . \\
1 & 0 & \cdots & 0 & 0
\end{pmatrix}_{d \times d},
\qquad
M  = M_{2d} =
\begin{pmatrix}
0 & J\\
- J & 0
\end{pmatrix}.
\end{split}
\end{align}
Let $V =F^{2d}$ be a symplectic vector space over $F$ with a symplectic form
$(,): V \times V \longrightarrow F$ specified by $M$.
Let $^t g$ be the transpose of a matrix $g$.
We define the symplectic group with coefficients in $F$
\begin{align}
\SP_F(2d) = \{ g \in \GL_F(2d) |  g = M \ \! ^t \!g^{-1} M^{-1} \}.
\end{align}
We also define $\SP_{\A}(2d)$ and $\SP_k(2d)$ similarly.
By our choice of $M$, we see that
$P\cap \SP_F(2d)$ is parahoric if $P$ is parahoric in $\GL_F(2d)$.
In particular,  $\mbf I^\fc = \mbf I^{\mathfrak a} \cap \SP_F(2d)$ is Iwahori and it is the stabilizer of the standard lattice chain $\L$ in $\SP_F(V)$.
Therefore, we have the bijection
\begin{align}
\SP_F(2d) / \mbf I^\fc \simeq \SP_F(2d) . \L \equiv  \tilde \Y^\fc.
\end{align}
So the lattice presentation of affine flag variety of type $C$ is reduced to a description of $\tilde  \Y^\fc$.

For any lattice $\mathcal L$  of $V$, we set
\[
\mathcal L^{\#}=\{ v\in V \vert (v, \mathcal L)\subset \A\}.
\]
Then the $\A$-module $\mathcal L^{\#}$ is again a lattice of $V$ and $(\mathcal L^{\#})^{\#}= \mathcal L$.
We shall use freely the following properties: for any two lattices $\mathcal L$ and $\mathcal M$
\[
(\mathcal L + \mathcal M)^{\#} = \mathcal L^{\#} \cap \mathcal M^{\#}, \qquad
(\mathcal L \cap \mathcal M)^{\#} = \mathcal L^{\#} + \mathcal M^{\#}.
\]

Following Sage ~\cite{Sa99}, we call a lattice  $alternating$ if $\mathcal L\subseteq \mathcal L^{\#}$ or $\mathcal L\supseteq \mathcal L^{\#}$.
An alternating  lattice $\mathcal L$ is called $sympletic$ if $\mathcal L$ or $\mathcal L^{\#}$ 
is homothetic to a lattice $\Lambda$, i.e., $\mathcal L$  or $\mathcal L^{\#}$ is equal to $\ve^{a} \Lambda$ for some $a\in \mbb Z$, such that
\begin{align}
\label{SP-1}
&\ve \Lambda  \subseteq \Lambda^{\#}  \subseteq  \Lambda.
\end{align}
Clearly $\L_z$ are symplectic for $z\in \mbb Z$.
The following proposition can be found in ~\cite{H99, Sa99, Lu03}.

\begin{prop}  \label{propSa99}
The set  $ \tilde \Y^{\fc}$ is the set of all collections $L=(L_z)_{z\in \mathbb Z}$ of symplectic  lattices in $V$ subject to the following conditions:
\begin{align*}
L_z\subset L_{z+1},\;\; \dim_k L_{z+1}/L_z =1,\; \; L_z= \ve L_{z+2d},\;\;  L_{z}^{\#} = L_{ -z}  \quad ( \forall z\in \mbb Z).
\end{align*}
\end{prop}

For our purpose later, we define a variant of the set $\tilde \Y^{\fc}$ as follows. 
Let $\Y^{\fc}$ be the set of all chains $L=(L_z|z\in \mbb Z)$ of symplectic lattices subject to the following conditions:
\begin{align}
 \label{Yc}
\begin{split}
& \dim_k L_{z+1}/L_z =
\begin{cases}
0, & \mbox{if} \ z\equiv  -1, d \ \mbox{mod} \  2d+2, \\
1, &  \mbox{otherwise}; 
\end{cases}
  \\
&L_z\subset L_{z+1},\quad L_{z}= \ve L_{z+2d +2},\quad L_{z}^{\#} = L_{ -z -1} \quad (\forall z\in \mbb Z).
\end{split}
\end{align}
Clearly, we have a natural bijection:
$\tilde \Y^{\fc}\simeq \Y^{\fc}.$  
 
Via the identification $\SP_F(2d)/\mbf I^{\fc} \cong \Y^{\fc}$, there is a left action
of $\SP_F(2d)$ on $\Y^{\fc}$ which is transitive. 
Let $\SP_F(2d)$ act on the product $\Y^{\fc} \times \Y^{\fc}$ diagonally. 
We shall describe the $\SP_F(2d)$-orbits in $\Y^{\fc} \times \Y^{\fc}$. 

Recall a set $\Theta_{n|n',d}$ was defined in \eqref{Theta:dn} for any positive integers $d, n, n'$.
Let ${}^{\C}\Sigma_d$ be the following subset of $\Theta_{2d+2|2d+2, 2d}$ of matrices with entries being $0$ or $1$:
\begin{align}
 \label{Sigma}
 \begin{split}
{}^{\C}\Sigma_d = \Big\{ A \in & \text{Mat}_{\ZZ\times \ZZ} (\{0,1\}) \Big \vert 
 a_{-i, -j} = a_{i j} =a_{i+2d+2,j+2d+2} \;  (\forall i, j \in \mbb Z), 
 \\
& \text{the $0$th and $(d+1)$st row/column  are zero},
 \\
& \exists \text{ exactly one nonzero entry per rows/columns } i \in [0, 2d+1] \backslash \{0,d+1\}
\Big\}. 
\end{split}
\end{align}

We define a map from the set  of $\SP_F(2d)$-orbits in $\Y^{\fc} \times \Y^{\fc}$ to ${}^{\C}\Sigma_d$:
\begin{align} 
\label{IM}
\varphi: \SP_F(2d) \backslash \Y^{\fc} \times \Y^{\fc} \longrightarrow {}^{\C}\Sigma_d
\end{align}
by
sending the orbit $\SP_F(2d) . (L, L') $ to $A=(a_{ij})_{i, j \in \mbb Z}$ where 
\[
a_{ij} = \dim_k \frac{L_{i-1}+L_i\cap L_j'}{L_{i-1}+L_i\cap L_{j-1}'}.
\]
By the definition of $a_{i j}$, we have
\begin{equation*}
\begin{split}
a_{-i, -j}
& = \dim_k \frac{L_{-i-1}+L_{-i}\cap L'_{-j}}{L_{-i-1}+L_{-i}\cap L'_{-j-1}} 
 = \dim_k \frac{L^{\#}_{i}+L^{\#}_{i-1}\cap L'^{\#}_{j-1}}{L^{\#}_{i}+L^{\#}_{i-1}\cap L'^{\#}_{j}} \\
& =\dim_k \frac{(L_{i}\cap (L_{i-1}+ L'_{j-1}))^{\#}}{(L_{i}\cap (L_{i-1}+ L'_{j}))^{\#}} 
 =\dim_k \frac{L_{i}\cap (L_{i-1}+ L'_{j})}{L_{i}\cap (L_{i-1}+ L'_{j-1})}\\
& = \dim_k \frac{L_{i-1}+  L_{i} \cap  L'_{j} }{L_{i-1} + L_{i} \cap  L'_{j-1}}
= a_{ij}.
\end{split}
\end{equation*}
So the map $\varphi$ is well defined. The following proposition can be found in ~\cite{H99}, see also ~\cite{Lu99}.

\begin{prop} \cite[Proposition 2.6]{H99} 
\label{decom-lem}
Let $A=(a_{ij})_{i, j\in \mbb Z}$ be  the associated matrix of $(L, L')$ under $\varphi$. Then we can decompose $V$ into 
$V=\oplus_{i, j \in \mbb Z} V_{ij}$ as $k$-vector spaces satisfying that 
$\dim_k V_{ij} = a_{ij}$, 
\begin{align}
\label{decom-lem-a}
L_i =\bigoplus_{k, l\in \mbb Z:  k\leq i} V_{kl}, \quad L_j' = \bigoplus_{k, l\in \mbb Z: l\leq j} V_{kl}, \quad \forall i, j\in \mbb Z.
\end{align}
Moreover,  there exists a basis $\{ e^m_{ij} | 1\leq m \leq a_{ij}\}$ of $V_{ij}$ such that 
\begin{align}
\label{decom-lem-b}
\begin{split}
e^{m}_{i, j} & = \ve e^m_{i+ 2d+2, j+2d+2}, \quad \forall i, j\in \mbb Z, 1\leq m \leq a_{ij}, \\
(e^m_{ij} , e^{m'}_{kl} ) & = - (e^{m'}_{kl} , e^m_{ij}) , \quad   \forall i, j, k, l \in \mbb Z, 1\leq m \leq a_{ij}, 1\leq m' \leq a_{kl}, \\
(e^m_{ij} ,  e^{m'}_{kl} ) & = \ve (e^m_{ij} , e^{m'}_{k +  (2d+2), l + (2d+2)}), 
\quad  \forall i, j, k, l \in \mbb Z , 1\leq m \leq a_{ij}, 1\leq m' \leq a_{kl}, \\
(e^m_{ij} , e^{m'}_{kl} )& = \delta_{m, 1} \delta_{m', 1} \ve^{-2}, \quad \forall 1 \leq i < k \leq 2d+2, i+k = 2d+2, j+l = 2d+2.
\end{split}
\end{align} 
\end{prop}

From the above proposition, we have  the Iwahori-Bruhat decomposition  for the group $\SP_F(V)$.

\begin{prop}
\label{IB-prop}
The map $\varphi: \SP_F(2d) \backslash \Y^{\fc} \times \Y^{\fc} \longrightarrow {}^{\C}\Sigma_d$ in (\ref{IM}) is a bijection.
\end{prop}

\begin{proof}
By Proposition ~\ref{decom-lem}, $\varphi$ is clearly surjective.
Assume now that the associated matrix of two pairs $(L, L')$ and $(\tilde L, \tilde L')$ of symplectic lattice chains  is the same matrix, say $A$.
By Proposition ~\ref{decom-lem}, we can find bases $\{e^m_{ij}\}$ and $\{ f^m_{ij}\}$ for the pairs $(L, L')$ and $(\tilde L, \tilde L')$, respectively, subject to the conditions
(\ref{decom-lem-a}) and (\ref{decom-lem-b}). 
We define a map $g: V \to V$ by sending $e^m_{ij}$ to $f^m_{ij}$ for all $i, j\in \mbb Z$ and $1\leq m \leq a_{ij}$. Then we have $g\in \SP_F(2d)$ and $g(L, L') = (\tilde L, \tilde L')$.
So $\varphi$ is injective. The proposition is proved.
\end{proof}

\section{Affine partial flag varieties of type $C$}

Now we fix an even positive integer
\[
n=2r+ 2, \quad \mbox{for some} \ r \in \mbb{N}.
\]
Let $\X^{\fc}_{n,d}$ be  the set of all chains $L=(L_z)_{z\in \mathbb Z}$ of symplectic lattices   in $V$ subject to the following conditions:
\begin{align}
 \label{Lz}
 L_z \subset L_{z+1},  \quad  L_{z} = \ve L_{z+ n}, \quad L_{z}^{\#} = L_{-z-1} \quad  (\forall z\in \mbb Z).
\end{align}

\begin{rem}
The shift by $-1$ in the condition $L_{z}^{\#} = L_{-z-1}$ in definition of 
$\X^{\fc}_{n,d}$
(see 
\eqref{Lz}) 
allows  the valuation at $L_0$ to vary. In contrast the valuation at $L_0$ is always zero in the case of $\tilde \Y^{\fc}$.
\end{rem}

The group $\SP_F(2d)$ acts from the left on $\X^{\fc}_{n,d}$  in a standard way.
Let $\SP_F(2d)$ act diagonally on the products $\X^{\fc}_{n,d} \times \X^{\fc}_{n,d} $ and $\X^{\fc}_{n,d} \times \Y^{\fc}$.
Let $\Pi_{n,d}$ be the subset of $\Theta_{n|2d+2, 2d}$ (for $\Theta_{n|n',d}$ see \eqref{Theta:dn}), 
which consists of all matrices $A=(a_{ij}) \in \text{Mat}_{\ZZ\times \ZZ} (\mbb N)$
such that
\begin{align}
\begin{split}
a_{-i, -j } =a_{ij}  = a_{i+n, j+2d+2}   \; (\forall i, j \in \mbb Z),
\quad
\sum_{l\in \mbb Z} a_{lj}  = 
\begin{cases} 
0, & \forall j \equiv 0, d+1 \ (\mbox{mod}\ 2d+2) \\
1, & \mbox{otherwise}.
\end{cases}
\end{split}
\end{align}
Similar to (\ref{IM}), we have a map
\begin{align}
\label{IM-tensor}
\SP_F(2d)\backslash \X^{\fc}_{n,d} \times \Y^{\fc} \longrightarrow \Pi_{n,d}.
\end{align}
More generally, let ${}^{\C}\Xi_{n,d}$ be the subset  of $\Theta_{n,2d}$ given by
\begin{align}
 \label{Xidn}
 \begin{split}
{}^{\C}\Xi_{n,d} = \big\{ (a_{ij} ) \in \text{Mat}_{\ZZ\times \ZZ} (\mbb N)  & \big \vert \;
a_{-i, -j } =a_{ij}  = a_{i+n, j+n}, \; (\forall i, j), \\ 
&\sum_{1\leq i \leq n} \sum_{j\in \mbb Z} a_{ij} = 2d, \;
a_{00}, a_{r+1, r+1}  \in 2\mbb Z \big \}. 
\end{split}
\end{align}
Similar to (\ref{IM}) again, we have a map
\begin{align}
\label{IM-schur}
\SP_F(2d)\backslash \X^{\fc}_{n,d} \times \X^{\fc}_{n,d} \longrightarrow {}^{\C}\Xi_{n,d}.
\end{align}

\begin{prop}
The  maps in  \eqref{IM-tensor} and \eqref{IM-schur}  are bijective.
\end{prop}

\begin{proof}
Note that the first bijection \eqref{IM-tensor} is a special case of the second bijection \eqref{IM-schur}. 
So we only need to prove the second bijection, and we shall deduce it from Proposition ~\ref{IB-prop} as follows. 
For a given matrix  $A \in {}^{\C}\Xi_{n,d}$, we can delete all its zero rows and zero columns. 
Let us denote the resulting matrix by $\mrm{dlt}(A)$, which is essentially determined by the stripe $[1, a]\times \mbb Z$ of $\mrm{dlt}(A)$  for some $a \leq 2d$. 
It is then possible to find a (nonunique) matrix $w \in {}^{\C}\Sigma_d$ (see \eqref{Sigma})
such that $A$ can be obtained from $w$ by adding consecutive rows between $[1, 2d+2]$.
Now pick a representative, say $(L, L')$, in the orbit $\mathcal O_w$.
We can construct a pair of partial flags by removing subspaces in $L$ and $L'$ corresponding to the summations of consecutive rows, 
whose associated matrix is $\mrm{dlt}(A)$. 
This shows that the map \eqref{IM-schur} is surjective.

Now if there are two pairs, say $x$, $y$, of flags whose associated matrix is $A$, we fix a matrix $w \in \ ^{\mathfrak{c}}\Sigma$ 
such that it can be merged to $A$, 
and use the above process to find two pairs, $x'$, $y'$ of flags in $\mathcal O_w$ such that they can reach $x$, $y$, respectively, by throwing away certain steps.
Moreover, $w$ can be chosen to be the one obtained from $A$ by blowing up the entries in $A$ of value strictly greater than $1$ to an identity matrix locally. 
For example, if $a_{ij} =2$ and $x=(L, L')$,  we can find a vector $u \in L_i \cap L'_j - (L_{i-1} \cap L'_j + L_i \cap L'_{j-1})$ such that $L_{i-1} + \A u$ and $L'_{j-1} + \A u$ are symplectic lattices.
We expand $L$ by plugging the lattice $L_i + \A u$ in between $L_i $ and $L_{i+1}$. Similarly, we can expand $L'$.
Then the matrix of the resulting pair will be the one by blowing up $A$ at $(i, j)$ to be a $2\times 2$ identity matrix locally.
By repeating the above process, we have the desired pair $x'$ for $x$ whose matrix is $w$. 
Since $x'$ and $y'$ are in the same orbit, there is a $g\in \SP_F(2d)$ such that $g.x'=y'$, which induces that $g.x=y$. So $x$ and $y$ are in the same orbit. 
Therefore the map \eqref{IM-schur} is injective, and hence  a bijection. The proposition is proved.
\end{proof}

\section{Local property at $L_0$}\label{seclocal}

\begin{lem}
Let  $L=(L_z)_{z\in \ZZ} \in \X^{\fc}_{n,d}$. We have
\begin{align*}
v(L_r) \in [-d, 0], \quad
v(L_{r + 1} ) \in [-2d, - d], \quad \text{ and } \quad
v(L_r) + v(L_{r + 1} ) = -2d.
\end{align*}
\end{lem}

\begin{proof}
Note that the valuation\footnote{Following~\cite{Sa99}, we define the valuation of a lattice.
More precisely, let $v: F^\times \to \mbb Z$ be the discrete valuation of the field $F$ so that the associated discrete valuation ring is $\mathfrak o$.
Let us fix a basis, say $e_1,\cdots, e_d$, of $V$. For any lattice $L$ of $V$ with  a  given $\mathfrak o$-basis, say $f_1,\cdots, f_d$, of $L$,  we can write
$f_1\wedge \cdots \wedge f_d = a e_1\wedge \cdots \wedge e_d$ for some $a\in F$.  We define $v(L)= v(a)$. The definition is independent of the choice of bases of $L$.
}
of $L_0$ is non-positive because $L_{-1}=L_0^{\#}$.
So the valuation  $v(L_r) \leq 0$.

By definition, for any lattice $\mathcal L$, we have
\[
v(\mathcal L^{\#}) = - v(\mathcal L), \quad
v(\ve \mathcal L) = 2d + v(\mathcal L).
\]
So we have
\begin{align*}
v(L_{r + 1})
 =  v( \ve^{-1} L _{-(r + 1)}) 
 =  -2d + v(L_{-(r + 1)})
 = - 2d + v( L^{\#}_r)
 = - 2d - v(L_r).
\end{align*}
Since $v(L_{r + 1}) \leq v(L_r)$, we have
$
v(L_r) \geq  -d,
$
and
$
-d \geq v(L_{r + 1}) \geq -2d.
$
\end{proof}
More generally, we have
$
v(L_i) \in [-d, 0], v(L_i) + v(L_{n-1 -i}) = - 2d, \quad \forall i\in [0, r]$,
by the same type of argument above. But we do not need this general fact.  

Then, we can find a `maximal isotropic' lattice $L_{r^\dagger}$ isometric to
\[
\L_d =[e_{d+1},\ldots, e_{2d}, \ve^{-1} e_1,\ldots, \ve^{-1} e_d]_{\A},
\]
such that
\[
L_r \subseteq L_{r^\dagger} \subseteq L_{r + 1}.
\]
Here the basis $\{ e_i| 1\leq i\leq 2d\}$ is chosen such that the associated matrix of the symplectic form on $V$ is given by (\ref{J}).
The lattice $\L_d$ satisfies the following properties: 
\[
(\L_d, \L_d) = \ve^{-1} \A, \quad
(\L_d, \ve \L_d) = \A,  \quad
(\ve \L_d, \L_d)=\A.
\]
So the map
\begin{align*}
(-, - )_{r^\dagger} &: L_{r^\dagger}/ \ve L_{r^\dagger} \times L_{r^\dagger} / \ve L_{r^\dagger} \to k
 \\
(\bar x, \bar y)_{r^\dagger} &= \mathrm{ev}|_{\ve=0} \ \ve (x, y)  
\end{align*}
is a non-degenerate symplectic form on $ L_{r^\dagger}/ \ve L_{r^\dagger}\simeq k^{2d}$.

Moreover, $L_{-1}/ \ve L_{r^\dagger}$ and $L_0/ \ve L_{r^\dagger}$ are orthogonal  complements  to each other
with respect to the above form $(-, -)_{r^\dagger}$ on $L_{r^\dagger}/ \ve L_{r^\dagger}$.

\begin{lem}
We have the following bijection
\begin{align} \label{L'}
\begin{split}
& \{\text{lattices } L' \text{ in } V \big \vert \; L_{-1} \subseteq (L')^{\#} \subseteq L' \overset{1}{\subset} L_0\} \stackrel{\sim}{\longrightarrow}  \\
&\{ k\text{-subspaces } W  \ \text{of} \ L_0/ \ve L_{r^\dagger}
\big \vert \; L_{-1}/ \ve 
L_{r^\dagger} \subseteq W \overset{1}{\subset} L_0/ \ve L_{r^\dagger}, W^{\perp} \subseteq W \},
\end{split}
\end{align}
which sends $L'$ to $L'/ \ve L_{r^\dagger}$.
(Here and below $\overset{1}{\subset}$ denotes subspaces of codimension $1$.)
\end{lem}

Therefore, the computation at $L_0$ is exactly the same as the computation at $L_{r}$. 
In particular, we have the following lemma which we shall use freely. 

\begin{lem}
 \label{lem:LL}
\begin{enumerate}
\item
Suppose that $L$ is a lattice such that
$L_{-1} \subset L \subset L_0$ and $\dim_k L_0/ L =1$, then the lattice $L$ is symplectic and $L^{\#}\subset L$.

\item 
If the pair $(L_{-1}, L_0)$ is replaced by $(L_r, L_{r + 1})$ such that $\dim_k L/L_r=1$, then $L$ is symplectic and $ L \subset L^{\#}$.
\end{enumerate}
\end{lem}

\chapter{Multiplication formulas for Chevalley generators} 
 \label{chap:formula}
 
 In this chapter, we study the convolution algebra $\Sj$ of pairs of $n$-step flags of affine type $C$.
 We present multiplication formulas in $\Sj$ with (the divided powers of) Chevalley generators. 
We then specify some general scenarios when these multiplication formulas produce a leading term with coefficient $1$.

\section{Some dimension computation}

Fix $L \in \X^{\fc}_{n,d}$. 
For $A\in {}^{\C}\Xi_{n,d}$ (which was defined in \eqref{Xidn}), we define
\begin{equation}
  \label{XAL}
X_A^L =\{ L' \in \X^{\fc}_{n,d} | (L, L')\in \mathcal O_A\}.
\end{equation}
This is an orbit of the stabilizer subgroup $\mbox{Stab}_{\SP(V)} (L)$ of $\SP(V)$, and one can associate to it a 
structure of quasi-projective algebraic variety.
We are interested in computing its dimension $d^{\C}_A$ (in order to define the standard basis element $[A]$ later on).
We have the following affine type $C$ analogue of ~\cite[Lemma 4.3]{Lu99}.

\begin{lem}  \label{dimension}
Fix $L \in \X^{\fc}_{n,d}$. 
For $A\in {}^{\C}\Xi_{n,d}$, the dimension of  $X^{L}_A$ is given by
\begin{align*}
d_A^{\C}
& = \frac{1}{2} \Big (\sum_{\substack{i\geq k, j<l \\ i\in [1, n]} } a_{ij} a_{kl}
+  \sum_{i\geq 0  >j } a_{ij}
+ \sum_{i\geq  r + 1 >j } a_{ij}
\Big ).
\end{align*}
\end{lem}

\begin{proof}
The proof is similar to that of  in \cite[Lemma 4.3]{Lu99}.
Indeed, we can fix a decomposition $V=\oplus_{i, j\in \mbb Z} V_{ij}$ as $k$-vector spaces such that $\dim_k V_{ij} = a_{ij}$.
We can further assume that $V_{ij}$ admits a $k$-basis $\{ e_{ij}^m | 1\leq m\leq a_{ij}\}$ satisfying
\[
\ve e_{ij}^m = e_{i - n, j- n}^m, \quad \forall i, j\in \mbb Z, m\in [1, a_{ij}].
\]
We define a symplectic  $F$-form on $V$ by,  for any $i, k \in [1, n]$, $j, l\in \mbb Z$,
\[
( e_{ij}^m, e_{kl}^{m'} ) =
\begin{cases}
\delta_{m,m'} \ve^{-2},  & \mbox{if} \ i+k=n, j+l=n, i < r + 1,\\
- \delta_{m,m'} \ve^{-2} ,  & \mbox{if} \ i+k=n, j+l=n, i > r + 1,\\
\delta_{m,m'} \ve^{-2},  & \mbox{if} \ i+k=n, j+l=n, i = r + 1, j < r + 1,\\
- \delta_{m,m'} \ve^{-2},  & \mbox{if} \ i+k=n, j+l=n, i = r + 1, j > r + 1,\\
\delta_{m, a_{r + 1, r + 1} +1-m'} \ve^{-2}, & \mbox{if} \ (i,j) =(k,l) = (r + 1, r + 1), m\leq a_{r + 1,r + 1}/2,\\
- \delta_{m, a_{r + 1, r + 1} +1-m'} \ve^{-2}, & \mbox{if} \ (i,j) =(k,l) = (r + 1, r + 1), m\geq a_{r + 1,r + 1}/2 +1 ,\\
\ve  ( e_{ij}^m, e_{k - n,l - n}^{m'}).
\end{cases}
\]
Now set $L=( L_i)_{i\in \mbb Z}$ and $L'= (L'_j)_{j\in \mbb Z}$, where
\[
L_i =\bigoplus_{k, l\in \mbb Z:  k\leq i} V_{kl}, \quad L_j' = \bigoplus_{k, l\in \mbb Z: l\leq j} V_{kl}, \quad \forall i, j\in \mbb Z.
\]
We see that
$
(L, L')\in \mathcal O_A.
$
Let
\[
X =\{ x\in \mathfrak{sp}(V) | x(L)\subseteq L\}, \quad
X' = \{ x\in \mathfrak{sp}(V) | x(L)\subseteq L, x(L') \subseteq L'\}.
\]
We have
\[
\dim X^L_A = \dim_k X/ X'.
\]
Now $x=(x_{(i,j), (k,l)}: V_{ij} \rightarrow V_{kl} ) \in \mathfrak{sp}(V)$ if and only if the following conditions are satisfied:
\begin{align} 
  \label{epsilon-action}
 \begin{split}
x_{(n+i, n+j), (n+ k, n+l)} (u) &= \ve^{-1} x_{(i, j), (k, l)} ( \ve u), \quad \forall u \in V_{n-i, n-j},
   \\
(x(u), u') + (u, x(u')) &=0, \quad \forall u, u' \in V.
\end{split}
\end{align}
The second condition in \eqref{epsilon-action} is equivalent to
\begin{align} \label{form-matrix}
^t x_{(i, j), (k, l)} M + M x_{ (s n -k, s n -l), (s n -i, s n -j)} =0, \quad \forall i, j, k, l, s\in \mbb Z,
\end{align}
where $M$ is  a certain matrix associated to the symplectic form $(\cdot, \cdot)$.

In particular, if
$
i+k \neq sn$
or
$
 j+l \neq s n,  \forall s\in \mbb Z,
$
then the linear map $x_{(i,j), (k, l)}$ completely determines
$x_{ (s n -k, s n -l), (s n -i, s n -j)}$ for all $s\in \mbb Z$.
So the contribution for these linear maps in $\dim X/X'$ is
\begin{align} \label{ik=n}
\frac{1}{2}
\sum_{
\substack{ i\geq k, j <  l,  i\in [0, n - 1] \\ i+k \neq s n
\
\small{or}
\
 j+l \neq s n }}
a_{ij} a_{kl} .
\end{align}
If we have
$
i+k = sn$ or 
$
 j+l = s n$ for some $s\in \mbb Z,$
then the equation (\ref{form-matrix}) becomes
$
^t x_{(i, j), (k, l)} M + M x_{ (i, j), (k, l)} =0.
$
By (\ref{epsilon-action}), the collection of linear maps
$x_{(i,j),(k,l)}$ such that $i+k =s n$ and $j+l=s n$ for some $s\in \mbb Z$ determines
the collection of linear maps
$x_{(i,j),(k,l)}$ such that $i+k =(s+2) n$ and $j+l=(s+2) n$.
So they are   determined by the following two subsets:
\[
\{ x_{(ij), (kl)} | i+k = 0, j+l =0\},
\quad
\{x_{(ij), (kl)}| i+k = n, j+l = n\}.
\]
So the contribution of these kind of linear maps to $\dim X/X'$ is
\begin{align} \label{ikneqn}
\frac{1}{2}
\sum_{
\substack{ i\geq k, j < l, i\in [0, n-1] \\ i+k = s n,
 j+l = s n }}
a_{ij} a_{kl}
+
\frac{1}{2}
\sum_{i\geq 0 > j} a_{ij}
+
\frac{1}{2} \sum_{i\geq r + 1 > j} a_{ij}.
\end{align}
The lemma follows by summing up  (\ref{ik=n}) and (\ref{ikneqn}).
\end{proof}


\section{Standard and canonical bases of Schur algebras}

It turns out a ``type $B$" parametrization in place of the ``type $C$" parametrization via ${}^{\C}\Xi_{n,d}$ is more natural,
for the $\SP(V)$-orbits in $\X^{\fc}_{n,d} \times \X^{\fc}_{n,d}$ and then for bases of the Schur algebras later on.
(Such a phenomenon already occurred in the finite type; cf. \cite{BKLW14, FL15}.)
We introduce the ``type $B$" parametrization  set
\begin{equation}
\label{Mdn}
\MX_{n,d} =\{ A + E^{00} + E^{r + 1, r + 1} \vert A\in {}^{\C}\Xi_{n,d} \}.
\end{equation}
That is, $\MX_{n,d}$ is the set of matrices $A \in \text{Mat}_{\ZZ\times \ZZ} (\mbb N)$ subject to the following conditions: 
\begin{align}
\begin{split}
 a_{ij} = a_{-i, -j} = a_{i - n, j- n} &\;  (\forall i, j \in \mbb Z),
 \quad  a_{00}, a_{r + 1, r + 1} \in 2\mbb Z +1, \\
  \sum_{i=i_0+1}^{i_0+n} \sum_{j\in \mbb Z} a_{ij}=2d+2, & \quad \text{for one (or each) } i_0 \in \mbb Z.
\end{split}
\end{align}
By definition we have a bijection  
\begin{equation}
  \label{bijection}
 {}^{\C}\Xi_{n,d} \longleftrightarrow \MX_{n,d},
 \qquad A \mapsto A + E^{00} + E^{r + 1, r + 1}. 
\end{equation}

From now on, we shall switch to the indexing set  $\MX_{n,d}$ for the rest of the paper.
Hence by abuse of notations, an $\SP_F(2d)$-orbit on $\X^{\fc}_{n,d} \times \X^{\fc}_{n,d}$
is  denoted by $\mathcal O_A$ and the set  in \eqref{XAL} is denoted by $X_A^L$,
now for $A \in \MX_{n,d}$. 

\begin{lem}
The dimension of  $X_A^L$ for $A \in \MX_{n,d}$ is given by
\begin{align} \label{eqda}
d_A
& = \frac{1}{2} \Big(\sum_{\substack{i\geq k, j<l \\ i\in [0, n-1 ]} } a_{ij} a_{kl}
-  \sum_{i\geq 0  >j } a_{ij}
- \sum_{i\geq  r + 1 >j } a_{ij}
\Big).
\end{align}
\end{lem}

\begin{proof}
Let $A =(a_{ij}) \in {}^{\C}\Xi_{n,d}$. 
We shall denote $A' = (a_{ij}') \in \MX_{n,d}$ corresponding to $A$ via the bijection \eqref{bijection}.
Thus we have
\begin{equation}\label{dA'}
  a_{ij}'  = a_{ij} + \delta_{ij} \sum_{k\in \ZZ} \delta_{0, i+kn} + \delta_{ij} \sum_{k\in \ZZ} \delta_{r+1, i+kn}.
\end{equation}
It follows from Lemma \ref{dimension} that
\begin{align*}
  d_{A'} =d_{A}^{\C} 
 &= \frac{1}{2} \Big (\sum_{\substack{i\geq k, j<l \\ i\in [1, n]} } a_{ij} a_{kl}
+  \sum_{i\geq 0  >j } a_{ij}
+ \sum_{i\geq  r + 1 >j } a_{ij}
\Big )
\\
&= \frac{1}{2} \Big ( \sum_{\substack{i\geq k, j<l \\ i\in [1, n]} } a_{ij}' a_{kl}' - \sum_{l > 0 \geq k} a_{kl}' - \sum_{l>r+1 \geq k}a_{kl}'
    - \sum_{\substack{i,j,k: i \geq kn >j \\ i\in [1, n]} } a_{ij}' \\
    &\qquad - \sum_{\substack{i,j,k: i \geq r+1+kn >j \\ i\in [1, n]} } a_{ij}' 
    +\sum_{i \geq 0 > j} a_{ij}' +\sum_{i\geq r+1 > j}a_{ij}'
    \Big )\\
    &= \frac{1}{2} \Big ( \sum_{\substack{i\geq k, j<l \\ i\in [1, n]} } a_{ij}' a_{kl}'
-  \sum_{i\geq 0  >j } a_{ij}'
- \sum_{i\geq  r + 1 >j } a_{ij}'
\Big ).
\end{align*} 
The lemma is proved. 
\end{proof}

We also introduce
\begin{align}
 \label{SigmaB}
 \begin{split}
 \Sigma_d = \Big\{ A \in  \text{Mat}_{\ZZ\times \ZZ} (\{0,1\}) \Big \vert 
 & a_{-i, -j} = a_{i j} =a_{i+2d+2,j+2d+2} \;  (\forall i, j \in \mbb Z), 
 \\
& \exists \text{ exactly one nonzero entry per row/column} 
\Big\}. 
\end{split}
\end{align}
Note the description of $\Sigma_d$ is much cleaner than ${}^{\C}\Sigma_d$, cf. \eqref{Sigma}. 
Nevertheless, the bijection \eqref{bijection} induces a natural bijection
\begin{equation}
  \label{bijection01}
 {}^{\C}\Sigma_d \longleftrightarrow \Sigma_d,
 \qquad A \mapsto A + E^{00} + E^{r + 1, r + 1}. 
\end{equation}
 The bijection $\varphi: \SP_F(2d) \backslash \Y^{\fc} \times \Y^{\fc} \longrightarrow {}^{\C}\Sigma_d$ in Proposition~\ref{IB-prop}
can be reformulated using $\Sigma_d$ in place of ${}^{\C}\Sigma_d$.
 
Recall the Schur algebra of affine type $A$, ${\mbf S}_{n, d; \cA}$, was defined in Section~\ref{sec:MB-A}.
The $\cA$-algebra ${\mbf S}^{\fc}_{n, d; \cA}$ is defined in the same way, now 
as the (generic) convolution algebra  ${\cA}_{\SP_F(2d)} (\X_{n,d}^{\fc} \times \X_{n,d}^{\fc})$ attached to the variety $\X_{n,d}^{\fc}$ introduced in the previous chapter.  
We then set 
\begin{equation}
  \label{SchurC}
\Sj =\mathbb Q(v) \otimes_{\cA}  {\mbf S}_{n, d; \cA}^{\fc}.
\end{equation}
The algebras ${\mbf S}^{\fc}_{n, d; \cA}$ and $\Sj$ are called the {\em Schur algebras (of affine type $C$)}. 
Denote by $e_A$ the characteristic function of the orbit $\mathcal O_A$, for $A \in  \MX_{n,d}$. Then
$\{e_A \vert A \in  \MX_{n,d} \}$ forms a basis for ${\mbf S}_{n, d; \cA}^{\fc}$ and $\Sj$.
Set 
\begin{equation}
 \label{standard}
 [A] = v^{-d_A} e_A, \quad \text{for } A \in  \MX_{n,d}.
\end{equation}

\begin{rem}
We have
\begin{align}
  \label{dAdA}
d_A - d_{\ \! ^t \! A}
& = \frac{1}{4} \Big( \sum_{1\leq i \leq n} \left ({\rm ro}(A)_i^2-{\rm co}(A)_i^2 \right )
- \left ({\rm ro}(A)_{0}-{\rm co}(A)_{ 0 } \right )
- \left ({\rm ro}(A)_{r + 1}-{\rm co}(A)_{ r + 1} \right ) \Big). 
\end{align}
Hence  the assignment $[A] \mapsto [\ \! ^t\! A]$
defines an anti-involution $\Psi: \Sj \to  \Sj$.  
Note  from ~\cite[1.6(a)]{Lu99} that
\[
\sum_{i\geq - (r + 1) > j} a_{ij} = \dim_k \frac{L'_{-r -2}}{ L_{-r -2} \cap L'_{-r -2}},
\quad
\sum_{i\geq r + 1 > j} a_{ij} = \dim_k \frac{L'_{r }}{ L_{r} \cap L'_{r}},
\]
for any $(L, L') \in \mathcal O_A$.
\end{rem}

Recall the partial orders $\leq_{\text{alg}}$ and $\leq$ on $\Theta_{n, d}$ from (\ref{partial-alg}) and (\ref{partial}). 
These two partial orders restrict to similar ones on $\MX_{n,d}$, still denoted by the same notations.
Since any matrix $A$ in $\MX_{n, d}$ satisfies that $a_{ij} = a_{-i, -j}$ for all $i, j \in \mbb Z$. 
The two conditions in (\ref{partial-alg}) are equivalent to each other. Hence,
the partial order $\leq_{\text{alg}}$ on $\MX_{n,d}$ can be simplified as follows.
Given any $A =(a_{ij}), A' =(a_{ij}')  \in \MX_{n,d}$, one has
\begin{align} \label{comb-ord}
A \leq_{\text{alg}} A' \Longleftrightarrow
  \sum_{k \leq i, l \geq j} a_{kl} \leq \sum_{k \leq i, l \geq j} a'_{kl}, \quad \forall i < j.
\end{align}
Since the Bruhat order of affine type $C$ is compatible with the Bruhat order of affine type $A$,
we see that the partial order ``$\leq$'' is compatible with (though possibly weaker than) the Bruhat order of affine type $C$.

Assume for now that the ground field is $\overline{\mbb F}_q$.
Let $IC_{A}$ be the intersection cohomology complex of the closure $\overline{X^L_{A}}$ of $X^L_A$, 
taken in certain ambient algebraic variety over $\overline{\mbb F}_q$,
such that the restriction of  the stratum $IC_{A}$ to $X^L_{A}$ is the constant sheaf on $X^L_{A}$.
We refer to ~\cite{BBD82} for the precise definition of intersection complexes.
The restriction of the $i$-th cohomology sheaf $\mathscr H_{X^L_{B}}^i(IC_{A})$
of $IC_{A}$ to $X^L_{B}$ for $B\leq A$ is a trivial local system,
 whose rank is denoted by
$n_{B,A,i}$.
We set
\begin{equation}
 \label{eq:cba}
  \{A\}_d =\sum_{B \leq A} P_{B,A}[B],\quad \mbox{where}\quad
   P_{B,A}=\sum_{i\in \mbb Z} n_{B  , A,i}v^{i-d_A+d_B}.
\end{equation}
The polynomials $P_{B, A}$ satisfy
\begin{equation}
P_{A,A}=1\quad {\rm and}\quad P_{B,A}\in v^{-1}\mbb Z[v^{-1}]\ {\rm for \ any }\ B < A.
\end{equation}
Recall $\{[A] \big \vert  A\in \MX_{n,d}\}$ forms an $\mathbb Q(v)$-basis of $\Sj$. 
In light of ~\cite{BBD82, L97},  we have the following.

\begin{prop}
 \label{CB+c}
The set $\{\{A\}_d \big \vert A \in \MX_{n,d}\}$ forms an $\mathcal A$-basis of ${\mbf S}_{n, d; \cA}^{\fc}$
and  a $\mbb Q(v)$-basis of ${\mbf S}_{n, d}^{\fc}$ (called the \it{canonical basis}).
Moreover, the structure constants of $\Sj$ with respect to the canonical basis 
are in ${\mathbb N} [v, v^{-1}]$.
\end{prop}
%
%
%
%
%

\section{Some multiplication formulas }
\label{Mult}

Recalling $E^{ij}$ from \eqref{Eij}, we set
\begin{equation}
  \label{Eij:theta}
E^{ij}_{\theta} = E^{ij} + E^{-i, -j}.
\end{equation}
Note that we have
\[
E^{00}_{\theta} = 2 E^{00},
\quad
E^{r + 1 , r + 1}_{\theta} = 2 E^{r + 1, r + 1}.
\]

We have the following affine analogue of ~\cite[Lemma 3.2]{BKLW14},
whose proof also explains why the formula therein is the same as those in \cite{BLM90}.

\begin{lem} \label{raw-mult-form}
Assume that $i\in \mbb Z$ and  $A, B, C\in \MX_{n,d}$.

\begin{enumerate}
\item
If  $ \ro(A) = \co(B) $ and $B - E^{i, i+1}_{\theta} $ is diagonal, then we have
\begin{align}
e_B * e_A = \sum_{\substack{p\in \mbb Z \\ a_{i+1, p} \geq (E^{i+1,p}_\theta)_{i+1,p} }}
v^{ 2 \sum_{j > p} a_{ij} } \frac{v^{2 (1+a_{ip})} -1}{v^2 - 1}  e_{A + E^{ip}_{\theta} - E^{i+1,p}_{\theta}}.
\end{align}

\item
If $\ro(A) = \co(C)$ and $C - E^{i+1, i}_{\theta} $ is diagonal, then we have
\begin{align}
e_C * e_A = \sum_{\substack{p\in \mbb Z \\ a_{i, p} \geq (E^{i,p}_\theta)_{i,p} } }
v^{ 2 \sum_{j < p} a_{i+1, j} } \frac{v^{2 (1+a_{i+1, p})} -1}{v^2 - 1}  e_{A - E^{ip}_{\theta} + E^{i+1,p}_{\theta}}.
\end{align}
\end{enumerate}
\end{lem}

\begin{proof}
The proof is essentially the same as that of ~\cite[Proposition 3.5]{Lu99}.  
Obtaining the structure constant in the first formula is reduced to computing the orders of the following two sets:
\[
\{ U \ \mbox{symplectic lattice}  | L_{i-1} + (L_i \cap L'_{p-1}) \subseteq U \subseteq L_i  \dim_k L_i / U =1\},
\]
\[
\{  U \ \mbox{symplectic lattice}  | L_{i-1} + (L_i \cap L'_{p}) \subseteq U \subseteq L_i , \dim_k L_i / U =1\}.
\]
Since the lattices $U$ such that $L_{i-1} \subseteq U \overset{1}{\subset} L_i$ 
are automatically symplectic (and $U^{\#} \subseteq U$) by Lemma ~\ref{lem:LL},
 the computations in {\em loc. cit.} still work and we have the first formula.

For the second formula, it is reduced to computing the difference of the orders of the following two sets:
\[
\{ U \ \mbox{symplectic lattice}  \vert  L_{i} \subseteq U \subseteq L_i + (L_{i+1} \cap L_p') , \dim_k U/L_i =1\}, 
\]
\[
\{ U \ \mbox{symplectic lattice}  \vert  L_{i} \subseteq U \subseteq L_i + (L_{i+1} \cap L_{p-1}') , \dim_k U/L_i =1\}.
\]
And again in this case, the lattices involved are automatically symplectic and thus the computations in {\it loc. cit.}   work here again. The second formula is obtained.
\end{proof}

We now generalize Lemma~\ref{raw-mult-form} to a multiplication formula by ``divided powers" of Chevalley generators. 

\begin{lem} \label{raw-R}
Assume that  $A, B, C \in \MX_{n,d}$  and $R\in \mbb N$.
\begin{enumerate}
\item
If $\ro(A) = \co(B)$ and $B - R E^{i, i+1}_{\theta}$ is diagonal for some $i\in [1, r]$, then we have
\begin{align*}
e_B * e_A
= \sum_t v^{2 \sum_{j > u} a_{i j} t_u} \prod_{u \in \mbb Z} \begin{bmatrix} a_{iu} + t_u \\ t_u \end{bmatrix}
e_{A + \sum_{u\in \mbb Z} t_u (E^{iu}_{\theta} - E^{i+1,u}_{\theta} )},
\end{align*}
where the sum is over all sequences $t =(t_u|u\in\mbb Z)$ such that $t_u\in \mbb N$ and $\sum_{u\in \mbb Z} t_u= R$.

\item
If $\ro(A) = \co(B)$ and $B - R E^{0, 1}_{\theta}$ is diagonal, then we have
\begin{align*}
\begin{split}
e_B * e_A
& = \sum_t v^{ 2 \sum_{j > u} a_{0 j} t_u + 2 \sum_{j < u < - j} t_j t_u + \sum_{u < 0} t_u (t_u -1) } \\
& \quad \quad
\cdot \prod_{u > 0} \begin{bmatrix} a_{0 u} + t_u + t_{-u} \\ t_u \end{bmatrix}
\prod_{u < 0} \begin{bmatrix} a_{0 u} + t_u \\ t_u \end{bmatrix}
\prod_{i=0}^{t_0 -1} \frac{[a_{00} +1 + 2i]} { [i+1] }
e_{A + \sum_{u\in \mbb Z} t_u ( E^{0u}_{\theta} - E^{1 u}_{\theta} )}.
\end{split}
\end{align*}

\item
If $ \ro(A) = \co(C)$ and  $C- R E^{i+1, i}_{\theta}$   is diagonal for some $i\in [0, r-1]$, then  we have
\begin{align*}
e_C * e_A
= \sum_{t} v^{ 2 \sum_{j<u} a_{i+1, j} t_u} \prod_{u\in \mbb Z}
\begin{bmatrix}
a_{i+1, u} + t_u\\
 t_u
\end{bmatrix} \,
e_{A- \sum_{u \in \mbb Z}  t_u( E_{\theta}^{iu} - E_{\theta}^{i+1,u})}.
\end{align*}

\item
If $ \ro(A) = \co(C)$ and  $C- R E^{r + 1, r}_{\theta}$   is diagonal, then  we have
\begin{equation*}
\begin{split}
e_C * e_A &= \sum_{t} v^{ 2 \sum_{j<u} a_{r + 1, j} t_u + 2\sum_{n-j<u<j} t_ut_j + \sum_{u>r + 1} t_u(t_u-1)} \prod_{u<r + 1}
\begin{bmatrix}
a_{r + 1, u} + t_u\\
 t_u
\end{bmatrix} \\
& \cdot \prod_{u>r + 1}
\begin{bmatrix}
a_{r + 1,u} +t_u+t_{n-u}\\
 t_u
\end{bmatrix}
 \prod_{i=0}^{t_{r + 1}-1} \frac{[a_{r + 1,r + 1}+1+2i]}{[i+1]} \, e_{A- \sum_{u \in \mbb Z}  t_u( E_{\theta}^{r u} - E_{\theta}^{r + 1,u})},
\end{split}
\end{equation*}
where the sum is taken over $t=(t_u | u\in \mbb Z)$ such that $t_u\in \mbb N$ and  $\sum_{u \in \mbb Z}  t_u=R$.
\end{enumerate}
\end{lem}

\begin{proof}
The proofs of (1), (3) and (4)  are essentially the same as that of  ~\cite[Proposition~ 3.3]{BKLW14}, while the proof of (2) is similar to the proof of (4).
Let us give a proof of (2) and skip (1), (3) and (4).
We shall prove by induction on $R$. When $R=1$, we have (2) by Lemma ~\ref{raw-mult-form}.
Write $B_R$ for $B$ in order to keep track of the $R$, and $A_t$ for the matrix
$A + \sum_{u\in \mbb Z} t_u ( E^{0u}_{\theta} - E^{1 u}_{\theta} )$ associated with $t$.
Let $G_{A, t}$ denote the coefficient of $e_{A_t}$ in (2).
By Lemma ~\ref{raw-mult-form}, we have
$
e_{B_1} *  e_{B_R} = [R+1] e_{B_{R+1}}.
$
So
\[
e_{B_{R+1}} * e_A
= \frac{1}{[R+1]} \sum_{ t, \underline p} G_{A, t} G_{A_t, \underline p}  e_{A_{t + \underline p}},
\]
where $\underline p\in \mathbb{N}^{\mathbb{Z}}$ is the sequence whose nonzero entry is $1$ at the position $p$.

It  suffices to show that
\[
\frac{1}{[R+1]} \sum_{ t, \underline p: t +  \underline p = s} G_{A, t} G_{A_t, \underline p}  = G_{A, s}
\]
for any sequence $s \in \mbb N^{\mbb Z}$ such that $\sum_{u\in \mbb Z}  s_u = R+1$.
By Lemma ~\ref{raw-mult-form}, the coefficient $G_{A_t, \underline p}$ is equal to
$v^{2 \sum_{j > p} a_{0 j} + 2 \sum_{j > p} (t_j + t_{ - j} ) } [ a_{0 p} + t_p + t_{-p} +1].$
The $v$-power terms of $G_{A, t}$ and $G_{A_t, \underline p}$ together yield
the $v$-power term of $G_{A, s}$ multiplying with $v^{ 2\sum_{ j > p} t_j}$.
The $v$-binomial coefficients  of $G_{A, t}$ and $G_{A_t, \underline p}$ yield
the $v$-binomial coefficient of $G_{A, s}$ multiplying with $[s_p]$.
So we have
\[
\frac{1}{[R+1]} \sum_{ t +  \underline p = s} G_{A, t} G_{A_t, \underline p}
=
G_{A, s}
\frac{1}{[R+1]}
\sum_{p\in \mbb Z}
v^{ 2\sum_{ j > p} t_j}
[r_p]
=
G_{A, s}.
\]
By induction, we have proved  (2).
\end{proof}

Lemma ~\ref{raw-R} can be rewritten in terms of the standard basis $[A]$ as follows.
Recall that we have a bar involution $\bar \empty : \mbb Q(v) \to \mbb Q(v)$ defined by $\bar v = v^{-1}$.

\begin{prop}\label{prop-stand-mult}
Assume that  $A, B, C \in \MX_{n,d}$  and $R\in \mbb N$.
\begin{enumerate}
\item
 If $\ro(A) = \co(B)$ and $B - R E^{i, i+1}_{\theta}$ is diagonal for some $i\in [1, r]$, then we have
\begin{align*}
[B] * [A]
& = \sum_t v^{ \beta_t } \prod_{u \in \mbb Z}
\overline{\begin{bmatrix} a_{iu} + t_u \\ t_u \end{bmatrix}}
[A + \sum_{u\in \mbb Z} t_u (E^{iu}_{\theta} - E^{i+1,u}_{\theta} )],\\
&  \beta_t =
 \sum_{j \geq  u} a_{i j} t_u - \sum_{j > u} a_{i+1, j} t_u + \sum_{j < u} t_j t_u
 + \frac{1} {2} \delta_{i, r}  \left ( \sum_{j + u < n} t_j t_u + \sum_{j < r + 1} t_j\right ),
\end{align*}
where the sum is over all sequences $t =(t_u|u\in\mbb Z)$ such that $t_u\in \mbb N$ and $\sum_{u\in \mbb Z} t_u= R$.

\item
If $\ro(A) = \co(B)$ and $B - R E^{0, 1}_{\theta}$ is diagonal, then we have
\begin{align*}
\begin{split}
& [B] * [A]
 \\
 &= \sum_t v^{ \beta'_t   }
 \prod_{u > 0}
\overline{\begin{bmatrix} a_{0 u} + t_u + t_{-u} \\ t_u \end{bmatrix} }
\prod_{u < 0}
\overline {\begin{bmatrix} a_{0 u} + t_u \\ t_u \end{bmatrix}}
\prod_{i=0}^{t_0 -1}
 \frac{\overline{[a_{00} +1 + 2i]} } { \overline{ [i+1]} }
\left [ A + \sum_{u\in \mbb Z} t_u ( E^{0u}_{\theta} - E^{1 u}_{\theta} ) \right ],
\end{split}
\end{align*}
where
the sum is over all sequences $t =(t_u|u\in\mbb Z)$ such that $t_u\in \mbb N$ and $\sum_{u\in \mbb Z} t_u= R$,
and
\[
\beta'_t
=
\sum_{j \geq u} a_{0j} t_u - \sum_{j > u} a_{1j} t_u
 + \sum_{j<u, j+u \leq 0} t_j t_u -  \sum_{j>0} \frac{t_j^2 - t_j}{2}
+ \frac{R^2 - R}{2}.
\]

\item
 If $ \ro(A) = \co(C)$ and  $C- R E^{i+1, i}_{\theta}$   is diagonal for some $i\in [0, r -1]$, then  we have
\begin{align*}
[C] * [A]
= \sum_{t} v^{ \gamma_t } \prod_{u\in \mbb Z}
\overline{
\begin{bmatrix}
a_{i+1, u} + t_u\\
 t_u
\end{bmatrix}
} \,
\left [A- \sum_{u \in \mbb Z}  t_u( E_{\theta}^{iu} - E_{\theta}^{i+1,u}) \right ], 
\end{align*}
where
the sum is over all sequences $t =(t_u|u\in\mbb Z)$ such that $t_u\in \mbb N$ and $\sum_{u\in \mbb Z} t_u= R$,
and
\[
\gamma_t = \sum_{j \leq u} a_{i+1, j} t_u - \sum_{j < u} a_{i j} t_u + \sum_{j < u} t_j t_u
+ \frac{1}{2} \delta_{i,  0}
\left (
\sum_{j + u > 0} t_j t_u + \sum_{j > 0} t_j
\right ).
\]

\item
 If $ \ro(A) = \co(C)$ and  $C- R E^{r + 1, r}_{\theta}$   is diagonal, then  we have
\begin{equation*}
\begin{split}
[C] * [A]
&= \sum_{t} v^{ \gamma'_t}
\prod_{u<r + 1}
\overline{
\begin{bmatrix}
a_{r + 1, u} + t_u\\
 t_u
\end{bmatrix}
}
\prod_{u>r + 1}
\overline{
\begin{bmatrix}
a_{r + 1,u} +t_u+t_{n-u}\\
 t_u
\end{bmatrix}
}\\
&
 \prod_{i=0}^{t_{r + 1}-1} \frac{\overline{[a_{r + 1,r + 1}+1+2i]}}{\overline{[i+1]}} \,
\left  [A- \sum_{u \in \mbb Z}  t_u( E_{\theta}^{r u} - E_{\theta}^{r + 1,u}) \right ],
\end{split}
\end{equation*}
where
\[
\gamma'_t =
\sum_{j \leq u} a_{r + 1, j} t_u - \sum_{j < u} a_{r j} t_u + \sum_{j < u, j+ u \geq n} t_j t_u -\sum_{ j < r + 1} \frac{t_j^2 - t_j }{2} + \frac{R^2 - R}{2},
\]
and $t=(t_u | u\in \mbb Z)$ such that $t_u\in \mbb N$ and  $\sum_{u \in \mbb Z}  t_u=R$.
\end{enumerate}
\end{prop}

\begin{proof}
Let us prove (1). By definition, we have $d_B = R b_{ii} = \sum_{j, u} a_{ij} t_u.$
Let us denote 
$(k, l)$-th entry in $A + \sum_{u\in \mbb Z} t_u (E^{iu}_{\theta} - E^{i+1,u}_{\theta}) $ by $^ia_{kl}$.
A lengthy calculation yields
\[
\sum \ \! ^i a_{ij} \ \! ^i a_{kl} - \sum a_{ij} a_{kl}
=
2 \sum_{j < u} a_{i j} t_u - 2 \sum_{j > u} a_{i+1, j} t_u  + 2 \sum_{j < u} t_j t_u
+ \delta_{i, r}
\sum_{j + u < n} t_j t_u,
\]
where the sums on the left-hand side run over all $(i, j, k, l)$ such that $i \geq k$, $j < l$ and $i\in [0, n-1]$.
We also have
\begin{align*}
\sum_{i \geq 0 > j} \ \! ^i a_{ij}
& = \sum_{i \geq 0  > j} a_{ij},
\\
\sum_{i \geq r + 1 > j} \ \! ^i a_{ij}
& =
\sum_{i \geq r + 1 > j}  a_{ij}
-
\delta_{i r}  \sum_{j < r + 1} t_j.
\end{align*}
Putting the above computations together, we have
\[
d_{A_t} - d_A =
\sum_{j < u} a_{ij} t_u - \sum_{j > u} a_{i+1, j} t_u + \sum_{j < u} t_j t_u
+
\frac{1}{2} \delta_{i, r}
\left (
\sum_{j +u  < n} t_j t_u + \sum_{j < r + 1} t_j
\right ),
\]
where $A_t = A + \sum_{u\in \mbb Z} t_u (E^{iu}_{\theta} - E^{i+1,u}_{\theta} )$. Now from Lemma ~\ref{raw-R}(1), we have
\[
\beta_t = - d_B + d_{A_t} - d_A +  2 \sum_{j > u} a_{ij} t_u + 2 \sum_{u\in \mbb Z} a_{i u} t_u.
\]
The above calculations give rise to the formula for $\beta_t$, and  (1) follows.

We now prove (2).  We set
$A_{0, t} = A + \sum_{u\in \mbb Z} t_u ( E^{0u}_{\theta} - E^{1 u}_{\theta} )$ and write its $(i, j)$-entry by $^0a_{ij}$.
We have
\[
\beta'_t
=  - d_B - d_A + d_{A_{0,t}} + 2 \sum_{j \geq u} a_{0j} t_u + 2 \sum_{ j < u, j + u \leq 0} t_j t_u + \sum_{u \leq 0} t_u (t_u -1).
\]
By definition, we have
$
d_B = \sum_{j, u} a_{0 j} t_u + \frac{R^2 - R}{2}.
$
Moreover, we have 
\[
d_{A_{0, t}} - d_A =
\sum_{j < u} a_{0j} t_u - \sum_{j > u} a_{1j} t_u + \sum_{j < u} t_j t_u + \frac{1}{2} \left ( \sum_{j+u >0} t_j t_u - \sum_{j >0} t_j \right ).
\]
Thus,
\begin{align*}
\beta'_t
& = \sum_{j \geq u} a_{0j} t_u - \sum_{ j > u} a_{1j} t_u - \frac{R^2 - R}{2} + \sum_{j < u} t_j t_u  
\\
&\qquad \qquad + 2 \sum_{j < u, j+u \leq 0} t_j t_u + \sum_{u \leq 0} t_u (t_u-1)
+ \frac{1}{2} (\sum_{j + u >0} t_j t_u - \sum_{j > 0} t_j )\\
& = \sum_{j \geq u} a_{0j} t_u - \sum_{ j > u} a_{1j} t_u + \frac{R^2 - R}{2} + \sum_{j < u, j+u \leq 0} t_j t_u - \frac{1}{2} \sum_{j > 0} t_j^2 - t_j.
\end{align*}
So we have proved (2).

For (3), we have
$d_C = \sum_{j, u} a_{i+1, j} t_u,$
and
\[
d_{A- \sum_{u \in \mbb Z}  t_u( E_{\theta}^{iu} - E_{\theta}^{i+1,u})} - d_A
=
\sum_{j > u} a_{i+1, j} t_u - \sum_{j < u} a_{ij} t_u + \sum_{j < u} t_j t_u
+ \frac{1}{2} \delta_{i, 0}
\left ( \sum_{j+u > 0} t_j t_u + \sum_{j > 0} t_j \right ).
\]
So we have the formula for $\gamma_t$ in (3).

For (4), we have $d_C = \sum_{j, u} a_{r + 1, j} t_u + \frac{R^2 - R}{2},$ and
\[
d_{A- \sum_{u \in \mbb Z}  t_u( E_{\theta}^{ru} - E_{\theta}^{r + 1,u})} - d_A
=
\sum_{j > u} a_{r + 1, j} t_u - \sum_{j < u} a_{rj} t_u + \sum_{j < u} t_j t_u
+ \frac{1}{2}
\left (
\sum_{j + u < n} t_j t_u - \sum_{j < r + 1} t_j
\right ).
\]
So we have the formula for $\gamma'_t$ in (4).
\end{proof}

\section{The leading term}

We have the following affine generalization of \cite[Lemma 3.9]{BKLW14}.

\begin{lem}
\label{BKLW3.9}
Let $A, B, C \in \MX_{n,d}$. Let $R$ be a positive integer.

\begin{enumerate}
\item Assume that  $B-RE_{\theta}^{h, h+1}$ is diagonal for some $ h \in [0, r]$ and $\co(B) = \ro(A)$.
Assume further that the matrix $A$ satisfies  one of  the following conditions:
\begin{align*}
& a_{0 j} =0 , \; \forall j \geq k; \ a_{1 k} = R, a_{1 j}  =0, \; \forall j > k, \quad \mbox{if} \ h =0, k \geq 0; \mbox{or}\\
& a_{h j}=0, \; \forall j\geq k; \; a_{ h +1,k}=R, a_{h +1,j}=0,\; \forall j>k, \quad \mbox{if} \; h\in [1, r-1]; \;\mbox{or}\\
& a_{r j}=0, \; \forall j\geq k; \; a_{r + 1,k}=R, a_{r + 1, j}=0,\; \forall j> k,\quad \mbox{if} \;  h = r, k >  r + 1; \; \mbox{or}\\
& a_{r j}=0, \; \forall j\geq r + 1; \; a_{r + 1,r + 1} \geq 2R, a_{r + 1, j}=0,\; \forall j> r + 1,\quad \mbox{if} \; h = r, k = r + 1.
\end{align*}
Then  we have
$
[B] * [A] = [ A+ R(E^{h, k}_{\theta}- E^{h +1, k}_{\theta}) ] + \mbox{lower terms}.
$

\item 
Assume  that $C-R E_{\theta}^{h+1,h} $ is diagonal for some $h\in [0, r]$ and $\co (C ) =\ro(A)$.
Assume further  that  $A$ satisfies one of the following conditions:
\begin{align*}
& a_{1j} =0, \quad \forall j \leq k; \ a_{0k} = R, a_{0j} = 0, \quad \forall j < k,  \;  \quad \mbox{if} \ h =0, k < 0; \ \mbox{or}\\
& a_{1j} =0, \quad \forall j \leq k; \ a_{0k} \geq 2R, a_{0j} = 0, \quad \forall j < k, \quad \mbox{if} \ h =0, k = 0; \ \mbox{or}\\
& a_{hj}=0,\quad \forall j < k; \ a_{hk}=R; \;a_{h+1,j} =0,\quad \forall j\leq k,\quad\mbox{if} \; h\in [1,r-1];\; \mbox{or}\\
& a_{r j}=0,\quad\forall j< k; \ a_{r k}=R; \; a_{r + 1,j}=0,\quad\forall j\leq k, \quad \mbox{if}\; h=r, k \leq r.
\end{align*}
Then we have
$
[C] * [A] = [ A- R(E^{h, k}_{\theta}- E^{h+1,k}_{\theta}) ] + \mbox{lower terms}.
$
\end{enumerate}
\end{lem}

\begin{proof}
We prove (1).
Set 
\begin{align*}
M &= A+ R(E^{h, k}_{\theta}- E^{h +1, k}_{\theta}),
\qquad 
M'  =  A+ \sum_{u \in \mbb Z} t_u ( E_{\theta}^{h u} - E_{\theta}^{h+1, u}), 
\end{align*} 
with $\sum_{u\in \mbb Z} t_u =R$.
By an argument similar to the proof of   ~\cite[Lemma 3.8]{BLM90}, it is enough to show that
$M' \leq_{\text{alg}} M$.
Assume that $h\in [1, r-1]$.
By definition, the $(r, s)$-th entry $m_{rs}$ of $M$ is
\begin{align*}
 m_{rs} = a_{rs}
 +  R \sum_{l\in \mbb Z} \delta_{s, k+ l n} (\delta_{r, h+ln} - \delta_{r, h+1+ln} ) 
 + R \sum_{l\in \mbb Z} \delta_{s, n - k + ln}  ( \delta_{r, n-h+ln} - \delta_{r, n-1-h+ln}).
\end{align*}
Observe that
\begin{align}
\label{M-a}
\sum_{r \leq i, s \geq j} & R \sum_{l\in \mbb Z} \delta_{s, k+ l n} (\delta_{r, h+ln} - \delta_{r, h+1+ln} ) \\
& =
\begin{cases}
R, & \mbox{if} \ i = h + l_1 n, j \leq k + l_1n, \ \mbox{for some}\ l_1 ,  \\
0, & \mbox{otherwise}.
\end{cases}
\nonumber \\
\label{M-b}
\sum_{r \leq i, s \geq j}  &+ R \sum_{l\in \mbb Z} \delta_{s, n - k + ln}  ( \delta_{r, n-h+ln} - \delta_{r, n-1-h+ln}) \\
& =
\begin{cases}
- R, & \mbox{if} \ i = n-1- h + l_1 n , j \leq n -k+l_1n,\ \mbox{for some} \ l_1,\\
0, & \mbox{otherwise}.
\end{cases}
\nonumber
\end{align}
On the other hand, the $(r, s)$-th entry $m'_{r,s}$ of $M'$ is equal to
\begin{align*}
m'_{r s} = a_{rs}
+  \sum_{l \in \mbb Z} t_{s - ln}  ( \delta_{r, h+ln} - \delta_{r, h+1+ln})  
 +  \sum_{l \in \mbb Z} t_{n - s + l n} (\delta_{r, n-h+ln} - \delta_{r, n-1-h+ln}).
\end{align*}
Notice that
\begin{align}
\label{M-c}
\sum_{r \leq i, s \geq j} & \sum_{l \in \mbb Z} t_{s - ln}  ( \delta_{r, h+ln} - \delta_{r, h+1+ln})\\
&=
\begin{cases}
\sum_{s + l_1n \geq j} t_{s}, & \mbox{if} \ i = h+l_1 n, \ \mbox{for some} \ l_1,\\
0, & \mbox{otherwise}.
\end{cases}
\nonumber \\
\label{M-d}
\sum_{r \leq i, s \geq j} & \sum_{l \in \mbb Z}  t_{n-s+ln} ( \delta_{r, n-h+ln} - \delta_{r, n-1-h+ln}) \\
&=
\begin{cases}
- \sum_{n-s+l_1n \geq j}  t_s, & \mbox{if} \ i= n-1-h+ l_1n, \\
0, & \mbox{otherwise}.
\end{cases}
\nonumber
\end{align}
To show that $M' \leq_{\text{alg}} M$ when $A$ is subject to the second condition, it suffices to show that
(\ref{M-a}) $\geq $ (\ref{M-c}) and (\ref{M-b}) $\geq$ (\ref{M-d}) when $i < j$.
Indeed, since $A$ satisfies the second condition, we have $t_u =0$ unless $u \leq k$.
If $i < j$ and $i = h+l_1n$ for some $l_1$, we have
$\sum_{s + l_1n \geq j} t_{s}
\leq
\sum_{ s > h} t_s \leq R.$
If, moreover, $h \geq k$, then $\sum_{s > h } t_s \leq \sum_{s > k} t_s =0$. From these data, we see that (\ref{M-c}) $\leq $ (\ref{M-a}) when $i < j$.
When $i < j$, we see that (\ref{M-b}) is equal to $-R$ when $j \leq n -k+l_1n$, and in this case
(\ref{M-d}) is also equal to $-R$. So we have (\ref{M-d}) $\leq $ (\ref{M-b}) when $i < j$. Therefore we have (1) when $A$ is subject to the second condition.

For $A$ subject to either of the remaining conditions, the proof of (1) is entirely similar and is left to the readers.

We now prove (2) for $h\in [1, r -1]$, i.e., when $A$ is  subject to the second condition. Suppose that
$M = A - R (E^{h k}_{\theta} - E^{h+1, k}_{\theta} )$
and
$M'= A- \sum_{u \in \mbb Z} t_u ( E_{\theta}^{h u} - E_{\theta}^{h+1, u})$ with $\sum_{u\in \mbb Z} t_u =R$.
It suffices to show that $M' \leq_{\text{alg}} M$. Similar to the proof of (1), it is reduced to show that
(\ref{M-a}) $\leq$ (\ref{M-c}) and (\ref{M-b}) $\leq $ (\ref{M-d}) when $i < j$.
By assumption, we see that $t_u =0$ unless $u \geq k$.
When (\ref{M-a}) takes value $R$, then $j \leq k + l_1n$, which implies that
$\sum_{s+ l_1 n \geq j}  t_s= \sum_{ s\geq k} t_s =R.$
Hence (\ref{M-a}) $\leq $ (\ref{M-c}) in this case.
When (\ref{M-b}) takes value $0$, we have either $j > n - k + l_1 n$ for some $l_1$ or $i \neq n-1-h+l_1n$ for any $l_1$. 
For the latter case, (\ref{M-d}) is always zero. For the former,
we have
$\sum_{n-s+l_1n} t_s =\sum_{s < k } t_s=0.$
Thus we have (\ref{M-b}) $\leq$ (\ref{M-d}). Therefore we have proved (2) if $A$ satisfies the second condition.

For the remaining cases, the proof of (2) is again similar and skipped.
\end{proof}

The following lemma is the counterpart of Lemma ~\ref{lem2}.
 
\begin{lem} \label{BKLW3.9-b}
  Let $A, B, C \in \MX_{n,d}$. Let $R$ be a positive integer.

\begin{enumerate}
\item Assume that  $B-RE_{\theta}^{h, h+1}$ is diagonal for some $ h \in [0, r]$ and $\co(B) = \ro(A)$.
Assume further that $R= R_0 + \cdots + R_l$ and  the matrix $A$ satisfy  one of  the following conditions:
\begin{align*}
\begin{cases}
a_{0 m} =0 ,  \ a_{1, k+i} = R_i,\ a_{1k} \geq R_0 ,\  a_{1j}  =0, &   \mbox{if} \ h =0,\ k \geq 1;\\
a_{h m}=0, \    a_{ h +1,k+i}= R_i, \ a_{h+1,k} \geq R_0  ,\  a_{h +1,j}=0,&  \mbox{if} \ h\in [1,r-1];\\
 a_{r m}=0, \   a_{r + 1,k+i}= R_i, \ a_{r + 1,k} \geq R_0  ,\ a_{r + 1, j}=0,&  \mbox{if} \  h =r,\ k >  r+1;\\
 a_{r m}=0, \   a_{r + 1,k+i}=R_i,\ a_{r + 1,k} \geq 2 R_0 ,\ a_{r + 1, j}=0, & \mbox{if} \ h = r,\ k = r + 1
\end{cases}
\end{align*}
for all $m \geq k$, $i\in [1,l]$ and $j > k+l$.
Then  we have
\[
[B] * [A] = [ A+ \sum_{i=0}^l R_i(E^{h, k+i}_{\theta}- E^{h +1, k+i}_{\theta}) ] + \mbox{lower terms}.
\]

\item 
Assume  that $C-R E_{\theta}^{h+1,h} $ is diagonal for some $h\in [0, r]$ and $\co(C ) =\ro(A)$.
Assume further  that $R=R_0 + \cdots + R_l$ and   $A$ satisfy one of the following conditions:
\begin{align*}
\begin{cases}
 a_{1m} =0,  \ a_{0,k+i} = R_i,\ a_{0,k+l} \geq R_l  ,\ a_{0j} = 0,  &    \mbox{if} \ h =0,\ k + l < 0;\\
 a_{1m} =0, \  a_{0,k+i} = R_i,\ a_{00} \geq 2 R_l  ,\ a_{0j} = 0, &  \mbox{if} \ h =0, k +l = 0;\\
 a_{h+1,m} =0, \ a_{h,k+i}= R_i,\ a_{h,k+l} \geq R_l , \ a_{hj} =0,  & \mbox{if} \ h\in [1,r-1];\\
 a_{r+1,m} =0, \ a_{r, k+i}= R_i,\ a_{r, k+l} \geq R_l, \ a_{r j}=0,&   \mbox{if}\; h=r,\ k < r.
\end{cases}
\end{align*}
for all $m \leq k+l$, $ i\in [0,l-1]$ and $j < k$.
Then we have
\[
[C] * [A] = [ A-\sum_{i=0}^{l} R_i(E^{h, k+i}_{\theta}- E^{h+1,k+i}_{\theta}) ] + \mbox{lower terms}.
\]
\end{enumerate}
\end{lem}

\begin{proof}
  We show (1).
  By a similar argument as that for Lemma \ref{BKLW3.9}, the leading term is
  $[ A+ \sum_{i=0}^l R_i(E^{h, k+i}_{\theta}- E^{h +1, k+i}_{\theta})  ]$. 
  It remains to show that its coefficient is 1.
  In this case, we have
  \[ t_{k+i} = R_i, \forall i\in [0, l] , \quad {\rm and}\quad t_j =0,\  \forall j\not \in [k, k+l].\]
  By Proposition \ref{prop-stand-mult}, we have
  \begin{equation*}
    \begin{cases}
    \prod_{u > 0}
\overline{\begin{bmatrix} a_{0 u} + t_u + t_{-u} \\ t_u \end{bmatrix} }
\prod_{u < 0}
\overline {\begin{bmatrix} a_{0 u} + t_u \\ t_u \end{bmatrix}}
\prod_{i=0}^{t_0 -1}
 \frac{\overline{[a_{00} +1 + 2i]} } { \overline{ [i+1]} }=1,& {\rm if}\ h=0,\\
 \prod_{u\in \mbb Z} \overline{\begin{bmatrix} a_{h u} + t_u \\ t_u \end{bmatrix}} =1,&{\rm if}\  h\neq 0.
    \end{cases}
  \end{equation*}
  Moreover, we have
  
  \begin{equation*}
    \begin{cases}
      \beta_t = \sum_{j> u} (t_j-a_{h+1,j}) t_u  + \frac{1}{2}(\sum_{j+u < n+1} t_jt_u-\sum_{j<r+1} t_j)  =0,&{\rm if}\  h\neq 0,\\
      \beta_t'= -\sum_{j> u} a_{1j}t_u -\sum_{j>0} \frac{t_j^2-t_j}{2} +\frac{R^2-R}{2} 
      =\frac{1}{2}( R^2 -\sum_j t_j^2 - 2\sum_{j>u} t_jt_u) =0,& {\rm if}\ h=0.
    \end{cases}
  \end{equation*}
  In each case, the leading coefficient is 1, and whence (1).
  A similar proof of (2) is skipped.
\end{proof}

\chapter{Coideal algebra type structures of  Schur algebras and Lusztig algebras}
  \label{chap:schur}
  
In this chapter, we formulate a coideal algebra type structure which involves
Schur algebras of both affine type $C$ and $A$, and its behavior on the Chevalley generators. 
 This leads to an imbedding  $\jmath_{n,d}$ from $\Sj$ to ${\mbf S}_{n,d}$ (Schur algebra of affine type $A$). 
We show the comultiplication and $\jmath_{n,d}$ behave well when replacing Schur algebras by Lusztig subalgebras. 
The canonical bases and monomial bases are shown to be
compatible under the inclusion $\bU_{n,d}^{\fc} \subset \Sj$.

\section{The Lusztig algebra $\bU_{n,d}^{\fc}$}
We now set
$$n=2r +2,\qquad (r \in \mbb N).$$
Recall the Schur algebra $\Sj$ from \eqref{SchurC}. 
Let $\bU_{n,d}^{\fc}$ be the subalgebra of $\Sj$ generated by all elements $[B]$ such that $B$, $B - E^{h, h+1}_{\theta}$ 
or $B- E^{h+1, h}_{\theta}$ is diagonal for various $h$.
Let $\bU_{n, d; \mathcal A}^{\fc}$ denote the $\cA$-subalgebra of $\Sj$ generated by all elements $[B]$ 
such that $B$, $B - RE^{h, h+1}_{\theta}$ or $B- RE^{h+1, h}_{\theta}$ is diagonal for various $h$ and $R \in {\mathbb N}$.
Following the affine type $A$ setting, we make the following definition.

\begin{Def}
The algebra  $\bU_{n,d}^{\fc}$ is called the {\em Lusztig algebra (of affine type $C$)}.
\end{Def}

For $i\in [0, r]$ and $a\in [-1, r+1]$, we define the following functions
(with the notation $\overset{1}{\subset}, \overset{1}{\supset}$ denoting inclusions of codimension $1$ 
and $|W|$ for the dimension of a $k$-vector space):
for any $L = (L_i)_{i \in \mathbb{Z}}, L' = (L'_i)_{i \in \mathbb{Z}} \in \mathcal{X}^{\mathfrak{c}}_{n,d}$
\begin{align}
\mathbf e_i (L, L') &=
\begin{cases}
v^{-|L'_{i+1}/L'_i | - \delta_{i, r}}, &\mbox{if}\; L_i \overset{1}{\subset} L_i', L_j = L_{j'},\forall j\in [0, r]\backslash \{i\}; \\
0, &\mbox{otherwise}.
\end{cases}\\
\mathbf f_i (L, L') &=
\begin{cases}
v^{-|L'_i/L'_{i -1}| - \delta_{i, 0} }, &\mbox{if}\; L_i \overset{1}{\supset} L_i', L_j=L_j',\forall j\in [0, r]\backslash \{i\}; \\
0, &\mbox{otherwise}.
\end{cases}\\
\mathbf h_a^{\pm 1} (L, L') & =
v^{\pm(|L_a'/L_{a - 1}'| +\delta_{a, 0} + \delta_{a, r+1} )} \delta_{L, L'}. \\
\mbf k_i & =\mbf h_{i+1} \mbf h_i^{-1}.
\end{align}
It follows by the definition that for $i \in [0,r]$ and $a\in [0, r+1]$,
$$
\e_i = \sum [A], \quad
\mbf f_i =\sum [A], \quad 
\mbf h_a = \sum_{\lambda\in \Lambda^{\C}_{n, d}}  v^{\lambda_a} 1_{\lambda}, \quad 
\mbf k_i =\sum_{\lambda\in \Lambda^{\C}_{n, d}}  v^{\lambda_{i+1} -\lambda_i} 1_{\lambda},
$$ 
where  the first two sums run over all $A \in \Xi_{n, d}$ such that 
$A - E^{i+1, i}_{\theta}$ and $A - E^{i, i+1}_{\theta}$ are  diagonal, respectively,
and $1_{\lambda}$ stands for the standard basis element of a diagonal matrix whose diagonal is $\lambda$.
So we have $\mbf e_i, \mbf f_i, \mbf k_i^{\pm 1}, \mbf h_a^{\pm 1} \in \U^{\C}_{n, d}$.



By the local property of $L_0$ in Section ~\ref{seclocal}, 
one can obtain the following relations (\ref{i=0}) by using 
a similar argument in \cite[Proposition 3.1]{BKLW14} for the relations related to the generators $\e_r$, $\f_r$, $\mbf h_r$ and $\mbf h_{r+1}$
for $r\geq 1$.
Note that the generators $\e_0$, $\f_0$, $\mbf h_0$ and $\mbf h_1$ 
play the roles of the respective generators $\f_r$, $\e_r$, $\mbf h_{r+1}$ and $\mbf h_r$
in the argument, by comparing items (1) and (2) in Lemma ~\ref{lem:LL}.
For $r\geq 1$, we have 

\begin{align}
\label{i=0}
\begin{split}
\mathbf{h}_0  \mathbf f_0 & = v^{2} \mathbf  f_0 \mathbf{h}_0, \quad
\mathbf{h}_0  \mathbf e_0  = v^{-2} \mathbf  e_0 \mathbf{h}_0,  \\ 
\mathbf e_0^2 \mathbf f_0 + \mathbf f_0 \mathbf e^2_0
& = (v + v^{-1}) \big(\mathbf  e_0 \mathbf f_0 \mathbf e_0 - ( v\mathbf {h}_1^{-1} \mathbf{h}_0  + v^{-1}\mathbf{h}_1 \mathbf{h}_0^{-1} ) \mathbf  e_0 \big), 
\\
\mathbf f_0^2 \mathbf  e_0 + \mathbf e_0 \mathbf f_0^2
& = (v + v^{-1})  \big(\mathbf f_0\mathbf  e_0 \mathbf f_0 - \mathbf  f_0 (v\mathbf{h}_1^{-1}\mathbf{h}_0 + v^{-1}\mathbf{h}_1 \mathbf{h}_0^{-1}) \big).  
\end{split}
\end{align}

For $i, j \in [0, r]$, we denote the Cartan integers by
\begin{equation} 
  \label{aij}
  \texttt{c}_{ij} = 2 \delta_{ij} - \delta_{i, j+1} - \delta_{i, j-1}.
\end{equation} 

\begin{prop}
 \label{relation:Schurc}
Let $r\geq 1$. 
The elements$\mathbf e_i, \mathbf f_i$, and $\mathbf k^{\pm 1}_i$  for $i\in [0, r]$ satisfy the following relations\footnote{Remove the relation $\mbf k_0   (\mbf k^2_1 \cdots \mbf k^2_{r-1} ) \mbf k_r  = 1$.}
in $\bU_{n,d}^{\fc}$, for all $i, j \in [0, r]$:
\begin{align*}
\mbf k_0 &  (\mbf k^2_1 \cdots \mbf k^2_{r-1} ) \mbf k_r  = 1,\\
\mathbf k_i \mathbf k_i^{-1} & = 1, \quad \mathbf k_i \mathbf k_j   = \mathbf k_j \mathbf k_i, \\
\mathbf k_i \mathbf e_j \mathbf k_i^{-1}& =  v^{\texttt{c}_{ij} + \delta_{i,0} \delta_{j,0} + \delta_{i, r} \delta_{j, r}}  \mathbf e_j , \\
\mathbf k_i \mathbf f_j \mathbf k_i^{-1} & =  v^{-\texttt{c}_{ij} - \delta_{i,0} \delta_{j,0} - \delta_{i, r} \delta_{j, r} }  \mathbf f_j , \\
\mathbf e_i \mathbf e_j & = \mathbf e_j \mathbf e_i , \quad 
\mathbf f_i \mathbf f_j  = \mathbf f_j \mathbf f_i , \quad \forall |i-j|> 1, \\
\mathbf e_i^2 \mathbf e_j  + \mathbf e_j \mathbf e_i^2 &= (v + v^{-1}) \mathbf e_i \mathbf e_j \mathbf e_i, \quad \forall |i-j|=1,\\
\mathbf f_i^2 \mathbf f_j  + \mathbf f_j \mathbf f_i^2 &= (v + v^{-1}) \mathbf f_i \mathbf f_j \mathbf f_i, \quad \forall |i-j|=1,\\
\mathbf e_i \mathbf f_j - \mathbf f_j \mathbf e_i & = \delta_{ij} \frac{\mathbf k_i - \mathbf k_i^{-1}}{v - v^{-1}} , \quad \forall (i, j) \neq (0,0) , (r,r),\\
\mathbf e_0^2 \mathbf f_0 + \mathbf f_0 \mathbf e_0^2 
&= (v + v^{-1})( \mathbf e_0 \mathbf f_0 \mathbf e_0 - (v \mathbf k_0 + v^{-1} \mathbf k_0^{-1}) \mathbf e_0),\\
\mathbf e_r^2 \mathbf f_r + \mathbf f_r \mathbf e_r^2 
&= (v + v^{-1})( \mathbf e_r \mathbf f_r \mathbf e_r - \mathbf e_r (v \mathbf k_r + v^{-1} \mathbf k_r^{-1})),\\
\mathbf f_0^2 \mathbf e_0   + \mathbf e_0 \mathbf f_0^2 
& = (v + v^{-1})  ( \mathbf f_0 \mathbf e_0 \mathbf f_0 - \mathbf f_0 (v \mathbf k_0 + v^{-1} \mathbf k_0^{-1})),\\
\mathbf f_r^2 \mathbf e_r   + \mathbf e_r \mathbf f_r^2 
& = (v + v^{-1})  ( \mathbf f_r \mathbf e_r \mathbf f_r - (v \mathbf k_r + v^{-1} \mathbf k_r^{-1})\mathbf f_r).
\end{align*}
\end{prop}

\begin{proof}
In light of 
\eqref{i=0} the verification of the relations is essentially reduced to the finite type computations, which is given in \cite[Proposition~3.1]{BKLW14}.
We skip the detail. 
\end{proof}

The following lemma is an analogue of \cite[Corollary~3.13]{BKLW14} which follows by a standard  Vandermonde determinant type argument. 
\begin{lem}
The algebra $\bU_{n,d}^{\fc}$ is generated by $\mathbf e_i$, $\mathbf f_i$ and $\mbf k^{\pm 1}_i$ for all $i\in [0, r]$.
\end{lem}
We will refer to the generators of the algebra $\bU_{n,d}^{\fc}$ given by the above lemma as Chevalley generators. 


\section{A raw comultiplication}
\label{affine-DC}

In this section, we shall give the definition of a raw comultiplication $\widetilde \Delta^{\fc}$
(and it will be modified to be a comultiplication $\Delta^{\fc}$ in Section~\ref{sec:refined}). 
The construction is an affine analogue of \cite{FL15}. 

We fix some notations to begin with.
Let $k=\mbb F_q$,
$F = k ((\varepsilon))$, and  $\mathfrak o =k[[\varepsilon]]$ where $q$ is odd.
Let $V$ be a symplectic  $F$-vector space of dimension $2d$ with the form $(-,-)$.
Let $V''$ be an isotropic $F$-subspace of $V$ of dimension $d''$, and so $V'=V''^{\perp}/V''$ is a symplectic space of dimension $2d'$  with 
its symplectic form induced from $V$; note that $d' =d-d''$.

Given a periodic chain $L$ in $\X_{n,d}^{\fc}$, we can define
a periodic chain $L'' :=\pi''(L)  \in \X_{n,d''}$ (of affine type $A$) by setting
$L''_i = L_i \cap V''$ for all $i$. 
We can also define a periodic chain $L' =\pi^{\natural} (L)  \in \X_{n, d'}^{\fc}$  
by setting $L'_i = (L_i \cap V''^{\perp} + V'')/{V''}$ for all $i$.
Given any pair $(L', L'') \in \X_{n,d'}^{\fc} \times \X_{n,d''}$, we set
\[
\Z^{\C}_{L', L''} = \{ L\in \X_{n,d}^{\fc} \vert  \pi^{\natural}(L) = L', \pi''(L) = L'' \}.
\]
We can define a map
\begin{align}
\label{tDj}
\widetilde \Delta^{\fc} : {\mbf S}_{n,d}^{\fc} \longrightarrow  {\mbf S}_{n,d'}^{\fc} \otimes  {\mbf S}_{n,d''},   \quad  \forall d' + d''=d,
\end{align}
such that, when specializing the parameter $v$ at $v = \sqrt{q}$, it is given by
\begin{align}
\label{tDj-formula}
\widetilde \Delta^{\fc}  (f) (L', \tilde L', L'', \tilde L'') = \sum_{\tilde L \in \Z_{\tilde L', \tilde L''}^{\fc}} f(L, \tilde L),
\quad \forall L', \tilde L' \in \X^{\fc}_{n,d'}, L'', \tilde L'' \in \X_{n, d''},
\end{align}
where $L$ is a fixed element in $\Z_{L', L''}^{\fc}$.
Note the appearance of ${\mbf S}_{n,d''}$ in \eqref{tDj}, which is an Schur algebra of affine type $A$ defined in \eqref{SchurA}.

By applying Proposition ~\ref{decom-lem}, we have the following analogue of ~\cite[Lemma 1.3]{Lu00}. 

\begin{lem}
\label{cop-auxi}
Suppose that $V''$ is an isotropic subspace of the symplectic space $V$ and $L=(L_i)_{i\in \mbb Z}  \in \X^{\fc}_{n,d}$. 
Then we can find a pair $(T, W)$  of subspaces in $V$ such that
\begin{enumerate}
\item  $V= V''\oplus T\oplus W$, $(V'')^{\perp} = V'' \oplus T$,
\item $W$ is isotropic,  $(T, W)=0$,
\item  There exist bases $\{z_1,\ldots,z_s\}$ and $\{w_1,\ldots, w_s\}$ of $V''$ and $W$, respectively, such that $ (z_i,w_j)=\delta_{ij}$ for any $i, j\in [1, s]$,
\item $L_i = (L_i\cap V'')\oplus (L_i\cap T)\oplus (L_i\cap W)$, for any $i\in \mbb Z$.
\end{enumerate}
\end{lem}

We can now show that the definition
(\ref{tDj-formula}) is well defined (i.e., it  is independent of the choice of $L$),
following the argument  in ~\cite[1.2]{Lu00}; see also ~\cite[3.2]{FL15}.
For fixed $L \in \X^{\fc}_{n,d}$, let  $\tilde \Z^{\fc}_L$ be the set of all pairs $(T, W)$ satisfying the first three conditions  in Lemma ~\ref{cop-auxi}.
Note that given a pair $(T, W)$ in $\tilde \Z^{\fc}_L$, we have an isomorphism $\pi: T \to V''^{\perp}/V''$.
Now if $L \in \Z^{\fc}_{L', L''}$, we define a map
\[
\psi: \tilde \Z^{\fc}_L \to \Z^{\fc}_{L', L''}, \quad (T, W) \mapsto L^{T, W},
\]
where
\[
L_i^{T,W} = L''_i\oplus \pi^{-1} (L_i') \oplus (L''_{-i-1})^{\#}_W,   \quad 
(L''_{-i-1})^{\#}_W =\{ w\in W| (w,L''_{-i - 1}) \in \A \},
\quad\forall i\in \mbb Z.
\]
By Lemma ~\ref{cop-auxi}, the map $\psi$ is surjective. 

Let $P_{V''}$ be the stabilizer of the flag $V'' \subseteq V''^{\perp}$ in $\SP_F(V)$. 
Let $\mathcal U=\mathcal U_{V''}$ be its unipotent radical, i.e., the set of all $g \in \SP_F(V)$ such that
$g(x) = x$ for all $x \in V''$ and $g(y) - y \in V''$ for all $y \in V''^{\perp}$. 
The $\SP_F(V)$-actions on $\X^{\fc}_{n,d}$ and $V$ restrict to the $\mathcal U$-actions on $\tilde \Z^{\fc}_{L}$ and $\Z^{\fc}_{L', L''}$, respectively.
Clearly,  $\psi$ is $\mathcal U$-equivariant and $\mathcal U$ acts transitively on $\tilde Z^{\fc}_L$,
and so $\mathcal U$ acts transitively on $\Z^{\fc}_{L', L''}$. 
This means that if $\hat L \in \Z^{\fc}_{L', L''}$,  there is $g\in \mathcal U$ such that $g \hat L = L$. 
From this, we have for all $\hat L \in \Z^{\fc}_{L', L''}$, 
\begin{align}
\sum_{\tilde L \in \Z^{\fc}_{\tilde L', \tilde L''}} f( \hat L, \tilde L)
= \sum_{\tilde L \in \Z^{\fc}_{\tilde L', \tilde L''}} f(L, g^{-1} \tilde L)
=\sum_{\tilde L \in \Z^{\fc}_{\tilde L', \tilde L''}} f(L, \tilde L).
\end{align}
Therefore the definition of (\ref{tDj-formula}) and hence $\tilde \Delta^{\fc}$ is independent of the choice of $L$. 

Following the argument of ~\cite[Proposition 1.5]{Lu00}, which is formal and not  reproduced here, we have the following proposition.

\begin{prop}
The map $\widetilde \Delta^{\fc}$ is an algebra homomorphism.
\end{prop}
The map $\widetilde \Delta^{\fc}$ is  referred to as  the {\em raw} comultiplication, which is a key component in 
the construction of a refined comultiplication in Section~\ref{sec:refined}.

Now we determine how the map $\widetilde \Delta^{\fc}$ acts on the generators.
Recall from Chapter~\ref{chap:A} the Chevalley generators $\mbf H_i$, $\mbf E_i$ and $\mbf F_i$ for Lusztig algebra $\bU_{n,d}$ of affine type $A$
(a subalgebra of the Schur algebra ${\mbf S}_{n,d}$ of affine type $A$), and 
 that $\mbf H_{n+i} = \mbf H_i$, $\mbf E_{n+i} =\mbf E_i$ and $\mbf F_{n+i} =\mbf F_i$.

\begin{prop}
\label{tDj-form}
For any $i\in [0, r]$, we have
\begin{align*}
\begin{split}
\widetilde \Delta^{\fc} ( \e_i)
&= \e_i' \otimes \mbf H_{i+1}'' \mbf H''^{-1}_{n -1  -i} + \mbf h'^{-1}_{i+1} \otimes \mbf E_i''  \mbf H''^{-1}_{n -1 -i} 
+  \mathbf h'_{i+1} \otimes \mbf F''_{n -1 -i} \mbf H''_{i+1}.  \\
\widetilde \Delta^{\fc} (\f_i)
 & = \f'_i \otimes \mbf H''^{-1}_{i} \mbf H''_{n -i} + \mbf h'_i\otimes \mbf F''_i \mbf H''_{n - i} 
 + \mbf h'^{-1}_{i} \otimes \mbf E''_{n -1 -i} \mbf H''^{-1}_{i}. \\
\widetilde \Delta^{\fc} (\mbf k_i) & = \mbf k'_i \otimes \mbf K''_i \mbf K''^{-1}_{n -1 -i}.
\end{split}
\end{align*}
Here the superscripts $'$ and $''$ indicate that the underlying Chevalley generators 
lie in $\mbf S^{\C}_{n, d'}$ and $\mbf S_{n, d''}$, respectively.
\end{prop}

\begin{proof}
For any $L\in \X^{\C}_{n,d}$, we have
$$
|L_{i+1}/L_i | = |L'_{i+1}/L'_i| + |L''_{i+1}/L''_i| + |L''_{n-1-i}/L''_{n- 2-i}|.
$$
The proposition in the cases for $i\in [1, r]$ follows directly from Proposition ~\ref{finiteD}   for the finite type; also cf. \cite{FL15}. 
The case for $i=0$ follows  
from a similar argument to that of 
the case for $i=r$. 
Note that when $r=0$, one uses the non-degenerate symplectic form on $L_1/L_0= \ve^{-1} L_0^{\#}/L_0$, inherited from that of $V$ (see ~\cite{Lu03}).
\end{proof}

\section{The comultiplication $\Delta^{\C}$}
\label{sec:refined}

Recall $\Lambda_{n,d}$ from \eqref{Land}, and we denote 
\begin{align}
\label{LandC}
\Lambda^{\fc}_{n,d} =
\Big\{
(a_i)_{i\in \mbb Z}  \Big \vert a_i\in \mbb N,  a_i = a_{-i}, a_i = a_{n-i}, \sum_{1\leq i \leq n}  a_i = 2d+2, a_0, a_{r+1} \text{ odd}
\Big \}.
\end{align}
The set $\X_{n,d}$ can be decomposed as follows: 
\begin{equation}
 \label{eq:conn-A}
\X_{n,d} = \bigsqcup_{\mbf a=(a_i) \in \Lambda_{n,d} } \X_{n,d} (\mbf a), \quad \text{ where }
\X_{n,d} (\mbf a) =\{ V\in \X_{n,d} \big\vert \, | V_i/V_{i-1} | = a_i, \forall 1\leq i\leq n\}.
\end{equation}
Similarly the set $\X^{\fc}_{n,d}$ admits the following decomposition:
\begin{equation}
 \label{eq:conn-C}
\X^{\fc}_{n,d} = \bigsqcup_{\mbf a =(a_i) \in \Lambda^{\fc}_{n,d} } \X^{\fc}_{n,d}(\mbf a), \quad \text{ where }
\X^{\fc}_{n,d}(\mbf a) =\{ V\in \X^{\fc}_{n,d} \big\vert \,  |V_i/V_{i-1}|  = a_i, \forall 1\leq i\leq n\}.
\end{equation}

Given $\mbf a, \mbf b \in \Lambda^{\fc}_{n,d}$, let 
$\Sj (\mbf b, \mbf a)$ be the subspace of $\Sj$ spanned  by 
the standard basis elements $[A]$ such that $\ro(A) = \mbf b$ and $\co(A)= \mbf a$.
Similarly, for $\mbf a, \mbf b \in \Lambda_{n,d}$, we define the affine type $A$ counterpart ${\mbf S}_{n,d} (\mbf b, \mbf a)$.
Let 
$\widetilde \Delta^{\fc}_{ \mbf b', \mbf a', \mbf b'', \mbf a''}$ be the component of $\widetilde \Delta^{\fc}$ from
$\Sj (\mbf b, \mbf a)$ to $\mbf S^{\C}_{n, d'} (\mbf b', \mbf a') \otimes {\mbf S}_{n,d''} (\mbf b'', \mbf a'')$ such that
$b_i = b'_i + b''_i + b''_{- i },   a_i = a'_i + a''_i +  a''_{-i},$ for $i\in \mbb Z.$
We set
\[
s(\mbf b', \mbf a', \mbf b'', \mbf a'') = \sum_{1\leq k\leq j\leq n} b'_k b''_j - a'_k a''_j,
\]
and
\[
u(\mbf b'', \mbf a'') =\frac{1}{2} \Big (
 \sum_{\substack{1\leq k, j \leq n-1 \\ k+j \geq n}} b''_k b''_j -a''_k a''_j + \sum_{n-1\geq k \geq r+1} a''_k - b''_k
\Big ),
\]
for all $\mbf b', \mbf a' \in \Lambda^{\fc}_{n,d'}$ and $\mbf b'', \mbf a'' \in \Lambda_{n,d''}$.
We renormalize the raw comultiplication $\widetilde \Delta^{\fc}$
to be $\widetilde \Delta^{\fc \dagger}$   by letting
\begin{align}
\label{Dagger}
\begin{split}
 \Delta^{\fc\dagger}_{\mbf b', \mbf a', \mbf b'', \mbf a''}
 &= v^{s(\mbf b', \mbf a', \mbf b'', \mbf a'') + u(\mbf b'', \mbf a'')}    \widetilde \Delta^{\fc}_{\mbf b', \mbf a', \mbf b'', \mbf a''},  
  \\
 \Delta^{\fc \dagger} &= \bigoplus_{\mbf b', \mbf a', \mbf b'', \mbf a''} \Delta^{\fc\dagger}_{\mbf b', \mbf a', \mbf b'', \mbf a''}
 : \Sj \to {\mbf S}_{n, d'}^{\fc} \otimes {\mbf S}_{n, d''}. 
\end{split}
\end{align}

\begin{prop}
Let $d=d' + d''$.
For all $i\in [0, r]$, we have
\begin{align*}
\Delta^{\fc\dagger} (\mbf e_i)
& = v^{\delta_{i, 0} d''} \mbf e'_i \otimes \mbf K''_i + 1 \otimes v^{- \delta_{i,0} (2d'+2)} \mbf E''_i
+ \mbf k'_i \otimes   v^{\delta_{i,0} (d''-1) + \delta_{i,0}} \mbf F''_{n -1 -i}  \mbf K''_i . \\
\Delta^{\fc\dagger} (\mbf f_i)
& = v^{-\delta_{i,0} d''} \mbf f'_i \otimes \mbf K''_{n-1-i} + \mbf k'^{-1}_i \otimes v^{\delta_{i,0} (2d'+2)-\delta_{i,0}}  \mbf K''_{n-1-i} \mbf F''_i  
+ 1\otimes v^{-\delta_{i,0} (d''-1)} \mbf E''_{n-1 -i}.\\
\Delta^{\fc\dagger} (\mbf k_i) &= \mbf k'_i \otimes \mbf K''_i \mbf K''^{-1}_{n-1-i}.
\end{align*}
Here the superscripts follow the same convention as in Proposition ~\ref{tDj-form}.
\end{prop}

\begin{proof}

The third formula on $\Delta^{\fc\dagger} (\mbf k_i)$ is clear.

Suppose that the quadruple  $(\mbf b', \mbf a', \mbf b'', \mbf a'')$ satisfies the following conditions:
\[
b'_k = a'_k - \delta_{k, i} + \delta_{k, i+1} + \delta_{k, n-1-i} - \delta_{n-i},  \quad b''_k=a''_k, \quad \forall k, \ \text{some}\ i\in [0,r].
\]
Then we have
$
s(\mbf b', \mbf a', \mbf b'', \mbf a'') = -a'_i + a''_{n-1-i} + \delta_{i, 0} d''
 \mbox{ and }  
u(\mbf b'', \mbf a'') =0.
$

Suppose that the quadruple $(\mbf b', \mbf a', \mbf b'', \mbf a'')$ satisfies the following conditions:
\[
b'_k = a'_k, \quad
b''_k = a''_k - \delta_{k,i} + \delta_{k, i+1}, \forall k \ \text{some} \ i\in [0, r].
\]
Then we have
$
s(\mbf b', \mbf a', \mbf b'', \mbf a'') = a'_{i+1} - \delta_{i, 0} ( 2d'+2),
  \mbox{ and }  
u(\mbf b'', \mbf a'') =a''_{n-1 -i} .
$

Suppose that the quadruple  $(\mbf b', \mbf a', \mbf b'', \mbf a'')$ satisfies the following conditions:
\[
b'_k = a'_k, \quad b''_k = a''_k + \delta_{k, n-1-i} - \delta_{k, n-i}, \quad \forall k, \ \text{some} \ i \in [0, r].
\]
Then we have
$
s(\mbf b', \mbf a', \mbf b'', \mbf a'') = - a'_i,
  \mbox{ and }  
u(\mbf b'', \mbf a'') = - a''_{i} + \delta_{i, 0} d'' .
$
The above computations lead to the first formula on $\Delta^{\fc\dagger} (\mbf e_i)$.

The second formula on $\Delta^{\fc\dagger} (\mbf f_i)$ follows from the following computations.
Suppose that the quadruple  $(\mbf b', \mbf a', \mbf b'', \mbf a'')$ satisfies the following conditions:
\[
b'_k = a'_k + \delta_{k, i} - \delta_{k, i+1} - \delta_{k, n-1-i} + \delta_{n-i},  \quad b''_k=a''_k, \quad \forall k, \ \text{some}\ i\in [0,r].
\]
Then we have
$
s(\mbf b', \mbf a', \mbf b'', \mbf a'') = a'_i - a''_{n-1-i} - \delta_{i, 0} d''
 \mbox{ and }  
u(\mbf b'', \mbf a'') =0.
$

Suppose that the quadruple $(\mbf b', \mbf a', \mbf b'', \mbf a'')$ satisfies the following conditions:
\[
b'_k = a'_k, \quad
b''_k = a''_k + \delta_{k,i} - \delta_{k, i+1}, \forall k,  \ \text{some} \ i\in [0, r].
\]
Then we have
$
s(\mbf b', \mbf a', \mbf b'', \mbf a'') = - a'_{i+1} + \delta_{i, 0} ( 2d'+2),
 \mbox{ and }  
u(\mbf b'', \mbf a'') =- a''_{n-1 -i} .
$

Suppose that the quadruple  $(\mbf b', \mbf a', \mbf b'', \mbf a'')$ satisfies the following conditions.
\[
b'_k = a'_k, \quad b''_k = a''_k - \delta_{k, n-1-i} + \delta_{k, n-i}, \quad \forall k, \ \text{some} \ i \in [0, r].
\]
Then we have
$
s(\mbf b', \mbf a', \mbf b'', \mbf a'') =  a'_i,
 \mbox{ and }
u(\mbf b'', \mbf a'') =  a''_{i} - \delta_{i, 0} (d'' -1) .
$

The proposition is proved. 
\end{proof}

The above formulas are indeed compatible with the ones in the finite type case for $i\in [1,r]$; cf. \cite{FL15} and Proposition ~\ref{finiteD}.
Recall $\xi_{d,i,c}: {\mbf S}_{n,d} \to {\mbf S}_{n,d}$ in affine type $A$ from ~\cite{FL15}. We generalize it to the affine type $C$ 
as 
\begin{align}
\label{rescale}
\begin{split}
\xi^{\fc}_{d,i,c}: & {\mbf S}^{\fc}_{n,d} \longrightarrow {\mbf S}^{\fc}_{n,d}, \qquad \text{for } i \in [0, r], c\in \mbb Z,
 \\
\xi^{\fc}_{d, i, c}  ([A]) &= v^{ c \varepsilon_i(A)} [A],  \qquad  \text{for } A \in \MX_{n,d}, 
\end{split}
\end{align}
where 
\begin{equation}
\label{ep:i}
\varepsilon_i(A)= \sum_{r \leq i < s} a_{rs} - \sum_{r > i \geq s} a_{rs}.
\end{equation}
In particular, we have
\[
\xi^{\fc}_{d, i, c} ( \e_j )= v^{-c \delta_{i, j}} \e_j, \quad
\xi^{\fc}_{d, i, c} ( \f_j ) =v^{c \delta_{i, j}} \f_j, \quad
\xi^{\fc}_{d, i, c} ( \mbf k_j )=\mbf k_j.
\]
We define the algebra homomorphism
(which is a refined comultiplication from the raw multiplication $\widetilde \Delta^{\fc}$)
\begin{align}
\label{Dj}
\Delta^{\fc} \equiv \Delta^{\fc}_{d',d''} =
(\xi^{\fc}_{d', 0, d''} \otimes \xi_{d'', 0, -(2d'+2)} \xi_{d'', n-1, -(d''-1)}) \circ
\Delta^{\fc \dagger}: {\mbf S}_{n,d}^{\fc} \longrightarrow  {\mbf S}_{n,d'}^{\fc} \otimes  {\mbf S}_{n,d''}.
\end{align}

\begin{prop}
\label{xiC}
For all $i\in [0, r]$ and $A \in \MX_{n,d}$, we have
$\xi^{\fc}_{d, i, a} ( \{A\}_d ) = v^{a \varepsilon_i(A)} \{A\}_d$.
\end{prop}

\begin{proof}
By  the definition  \eqref{ep:i} and using $a_{rs}=a_{-r, -s}$, we have
\begin{align*}
\varepsilon_i (A)
& = \Big ( \sum_{r \leq i < s} - \sum_{r < -i \leq s} \Big ) a_{rs}  
= \Big( \sum_{\substack{-i \leq r \leq i \\ i < s }} - \sum_{\substack{-i \leq s \leq i\\ r < -i}} \Big ) a_{rs}   \\
& 
= \frac{1}{2} \Big (
\sum_{\substack{-i \leq r \leq i \\ i< s}} +
\sum_{\substack{-i \leq r \leq i \\ s < - i }}
- \sum_{\substack{-i \leq s \leq i \\ r< -i}}
- \sum_{\substack{-i \leq s \leq i \\ r > i}}
\Big ) a_{rs}\\
& =  \frac{1}{2} \Big (  \sum_{-i \leq r \leq i} - \sum_{-i \leq s \leq i} \Big ) a_{rs} 
 = \frac{1}{2} \sum_{-i \leq s \leq i} \ro(A)_s - \co(A)_s.
\end{align*}
Now if the polynomial $P_{A, B}$ in (\ref{eq:cba}) is not zero, then $\ro(B)=\ro(A)$ and $\co(B) = \co(A)$, and hence
$\varepsilon_i(A) = \varepsilon_i (B)$.
Therefore, we have
\begin{align*}
\xi^{\fc}_{d, i, a} ( \{A\}_d ) 
&=
\xi^{\fc}_{d, i, a} \big( \sum_{B\leq A} P_{A,B} [B] \big)
=\sum_{B\leq A} P_{A, B} v^{c\varepsilon_i(B)} [B]
 \\
&= v^{c\varepsilon_i(A)} \sum_{B\leq A} P_{A,B} [B]
= v^{c\varepsilon_i(A)}  \{A\}_d.
\end{align*}
The proposition is proved.
\end{proof}

Let $\mbf a, \mbf b \in \Lambda^{\fc}_{n,d}$.
Fix $L \in \X^{\C}_{n,d}(\mbf b)$ (which was defined in \eqref{eq:conn-C}), and let $P_{\mbf b} = {\rm Stab}_{{\SP}_F(2d)}(L)$.
We have a natural embedding
\[\iota_{\mbf b, \mbf a}: \X_{n,d}^{\C}(\mbf a) \longrightarrow \X_{n,d}^{\C}(\mbf b)\times \X_{n,d}^{\C}(\mbf a),
\qquad L' \mapsto (L, L'). \]
It is well known that $\iota_{\mbf b, \mbf a}$ induces the following isomorphism of $\mathcal A$-modules:
\[\iota^*_{\mbf b, \mbf a}: \mathcal A_{{\SP}_F(2d)}(\X^{\C}_{n,d}(\mbf b) \times \X^{\C}_{n,d}(\mbf a)) \longrightarrow \mathcal A_{P_{\mbf b}}(\X^{\C}_{n,d}(\mbf a)).
\]
Let 
\[
\X^{\C}_{\mbf a,\mbf a', \mbf a''} =\{L\in \X^{\C}_{n,d}(\mbf a) \big \vert  \pi^{\natural} (L) \in \X^{\C}_{n, d'}(\mbf a'),\ \pi''(L) \in \X_{n, d''} (\mbf a'')\}.
\]
Then we have the following diagram
\[
\xymatrix{
\X^{\C}_{n,d}(\mbf a) && \X^{\C}_{\mbf a, \mbf a', \mbf a''} \ar[ll]_{\iota} \ar[rr]^-{\pi} && \X^{\C}_{n,d'} (\mbf a') \times \X_{n,d''} (\mbf a''),
}
\]
where $\iota$ is the imbedding and $\pi(L) = (\pi^{\natural}(L), \pi''(L))$.
By identifying $\mathcal A_{P_{\mbf b'} \times P_{\mbf b''}}(\X^{\C}_{n, d'}(\mbf a') \times \X_{n,d''}(\mbf a'')) =
\mathcal A_{P_{\mbf b'}}(\X^{\C}_{n,d'}(\mbf a')) \times \mathcal A_{P_{\mbf b''}}(\X_{n,d''}(\mbf a''))$, 
we have the following linear map
\begin{equation*}
  \xymatrix{\pi_! \iota^*: \mathcal A_{P_{\mbf b}}(\X^{\C}_{n,d}(\mbf a)) \longrightarrow 
  \mathcal A_{P_{\mbf b'}}(\X^{\C}_{n, d'}(\mbf a')) \times \mathcal A_{P_{\mbf b''}}(\X_{n, d''}(\mbf a'')).}
\end{equation*}
By a similar argument as for  \cite[Lemma 1.3.5]{FL15}, the following diagram commutes:
\begin{equation}
  \label{comm-coproduct}
  \xymatrixrowsep{.5in}
\xymatrixcolsep{.8in}
\xymatrix{
\mathcal A_{{\SP}_F(2d)}(\X^{\C}_{n,d}(\mbf b) \times \X^{\C}_{n,d}(\mbf a))
\ar@{->}[r]^-{\iota^*_{\mbf b, \mbf a}}
\ar@{->}[d]_{\widetilde \Delta^{\C}_{\mbf b', \mbf a', \mbf b'', \mbf a''}}
&
\mathcal A_{P_{\mbf b}} (\X^{\C}_{n,d}(\mbf a))
\ar@{->}[d]^{ \pi_! \iota^*}
\\
\smxylabel{
\substack{
\mathcal A_{{\SP}_F(2d')} (\X^{\C}_{n,d'}(\mbf b') \times \X^{\C}_{n,d'} (\mbf a') ) \\
\otimes \\
\mathcal A_{{\SP}_F(2d'')} (\X_{n,d''}(\mbf b'') \times \X_{n,d''} (\mbf a'') )
}
}
\ar@{->}[r]^-{ \iota^*_{\mbf b', \mbf a'} \otimes \iota^*_{\mbf b'', \mbf a''}}
&
\smxylabel{
\mathcal A_{P_{\mbf b'} } ( \X_{n,d'}^{\C} (\mbf a') )
\otimes
\mathcal A_{ P_{\mbf b''}} ( \X_{n,d''} (\mbf a'')).
}
}
\end{equation}

Recall $\Delta^{\C}: {\mbf S}_{n,d}^{\fc} \longrightarrow  {\mbf S}_{n,d'}^{\fc} \otimes  {\mbf S}_{n,d''}$ from \eqref{Dj}.
\begin{prop}
\label{comult+S}
For $A \in \MX_{n,d}$,  write
$$\Delta^{\C}( \{A\}_d ) = \sum_{A' \in \MX_{n,d'}, A''\in \Theta_{n,d''}} h^{A', A''}_{A} \{A'\}_{d'} \otimes \ ^{\mathfrak a} \{A''\}_{d''}.$$
Then $h^{A', A''}_A \in \mbb N[v,v^{-1}]$ for all $A, A' $ and $A''$.
\end{prop}

\begin{proof}
By Proposition~\ref{xiC}, the proof is reduced to showing the same type of positivity with respect to $\widetilde{\Delta}^{\C}$.
By an argument similar to ~\cite[Section 2.4]{FL15} and (\ref{comm-coproduct}), 
the positivity for  $\widetilde{\Delta}^{\C}$ follows from ~\cite[Theorem 8]{Br03}. The proposition is proved. 
\end{proof}



Now let us  study the restriction of $\Delta^{\fc}$ to ${\bU}_{n,d'}^{\fc}$. 

\begin{prop}
\label{Dc}
Let $d=d' + d''$. We have a  homomorphism $\Delta^{\fc} : {\bU}_{n,d}^{\fc} \longrightarrow  {\bU}_{n,d'}^{\fc} \otimes  {\bU}_{n,d''}.$
More precisely, for all $i\in [0, r]$, we have 
\begin{align*}
\Delta^{\fc} (\mbf e_i)
& = \mbf e'_i \otimes \mbf K''_i + 1 \otimes  \mbf E''_i
+ \mbf k'_i \otimes v^{\delta_{i,0}}   \mbf F''_{n-1-i}  \mbf K''_i,  \\
\Delta^{\fc} (\mbf f_i)
& =  \mbf f'_i \otimes \mbf K''_{n-1-i} + \mbf k'^{-1}_i \otimes  v^{-\delta_{i,0}} \mbf K''_{n-1-i} \mbf F''_i
+ 1\otimes  \mbf E''_{n-1-i}, \\
\Delta^{\fc} (\mbf k_i) &= \mbf k'_i \otimes \mbf K''_i \mbf K''^{-1}_{n-1-i}.
\end{align*}
\end{prop}

Recall the comultiplication $\Delta$ in the affine type $A$  from ~\cite{FL15} (see also ~\cite{Lu00} for a related construction).
This is an algebra homomorphism  
$$
\Delta: {\mbf S}_{n,d} \longrightarrow {\mbf S}_{n, d'} \otimes {\mbf S}_{n, d''}
$$
defined by
\begin{align}
\label{A-coass}
\begin{split}
\Delta ( \mbf E_i) & = \mbf E'_i \otimes \mbf K_i'' + 1 \otimes \mbf E''_i, \\
\Delta (\mbf F_i) & = \mbf F'_i \otimes 1+ \mbf K'^{-1}_{i} \otimes \mbf F_i'', \\
\Delta (\mbf K_i) & = \mbf K_i' \otimes \mbf K_i'', \hspace{2cm} \forall 0\leq i \leq n-1.
\end{split}
\end{align}
Here the superscripts follow the same convention in Proposition ~\ref{tDj-form}.

\begin{prop}
\label{coassociative}
The following coassociativity  holds  on $\bU^{\fc}_{n,d}$: 
$$
(1\otimes \Delta) \Delta^{\fc} = (\Delta^{\fc} \otimes 1) \Delta^{\fc}.
$$
\end{prop}

\begin{proof}
Beyond type $A$ or finite type $B/C$ we only need to check the desired identity when acting on $\mbf e_0$,  $\mbf f_0$ and $\mbf k_0^{\pm}$. This can be verified directly.
\end{proof}

Now setting $d'=0$, we have
$
\mbf e'_i=0, \mbf f'_i=0, \mbf k'_i = v^{-\delta_{i, 0} + \delta_{i, r}}$ in ${\mbf S}^{\fc}_{n, 0},$
and  $\Delta^{\fc}$ becomes the following algebra homomorphism
\begin{align}
\label{je}
\begin{split}
\jmath_{n,d} &: \Sj \longrightarrow {\mbf S}_{n,d}
\\
\jmath_{n,d} (\e_i) & = \mbf E_i + v^{-\delta_{i,0}} \mbf K_i \mbf  F_{n-1-i}, \\
\jmath_{n,d} (\f_i ) & =  \mbf E_{n-1-i} + v^{ \delta_{i,0}} \mbf F_i \mbf K_{n-1-i},\\
\jmath_{n,d} (\mbf k_i) & = v^{-\delta_{i,0} + \delta_{i, r}} \mbf  K_i \mbf K_{n-1-i}^{-1}, \quad \forall i \in [0, r].
\end{split}
\end{align}
It follows by restriction that we have also a homomorphism $\jmath_{n,d}: \bU^{\fc}_{n,d} \rightarrow \bU_{n,d}.$
Thanks to Propositions ~\ref{Astandard-basis} and ~\ref{monomial}, 
the same argument as in finite type $B/C$ \cite{FL15} gives us the following.

\begin{prop}
\label{j-inj}
The homomorphism $\jmath_{n,d}: \Sj \rightarrow {\mbf S}_{n,d}$ 
(and $\jmath_{n,d}: \bU^{\fc}_{n,d} \rightarrow \bU_{n,d}$)
is injective.
\end{prop}

Proposition~\ref{comult+S} in our setting of $d'=0$ gives us the following.
\begin{prop}
\label{j+}
 The map $\jmath_{n,d}$ sends a canonical basis element in $\Sj$  to a sum of canonical basis elements of
 ${\mbf S}_{n,d}$ with coefficients in  $\mbb{N} [v, v^{-1}]$.
\end{prop}

\section{Monomial and canonical bases of $\bU^{\fc}_{n,d}$}

Recall  $\MX_{n,d}$ from \eqref{Mdn} and the notion of aperiodic matrices from~\eqref{aperiod}. We denote
\begin{equation}
\label{Mdn:ap}
 \MX_{n,d}^{\ap} =\{ A \in \MX_{n,d} \vert A \text{ is aperiodic} \}.
\end{equation}

A product of standard basis elements $[G_1] * [G_2] * \cdots *[G_m]$ in $\Sj$ is called an $aperiodic$ $monomial$ if for each $i$,   
$G_i - R E^{j, j+1}_\theta$ 
is diagonal for some $R\in \mbb N$ and $j\in \mbb Z$. 
The following aperiodic monomial is an analogue of  $\zeta^{\mathfrak a}_A$ for $\bU_{n,d}$ (see Proposition~ \ref{Astandard-basis}).

\begin{prop}
\label{monomial}
For any $A \in \MX_{n,d}^{\ap}$,
there exists an aperiodic monomial $\zeta_A \in \bU_{n,d}^{\fc}$ such that 
\begin{align}
\label{zA}
\zeta_A = [A] + \mbox{lower terms}.
\end{align}
\end{prop}

\begin{proof}
With the help of Lemma \ref{BKLW3.9-b}, the proof is the same as that for Proposition~ \ref{Astandard-basis}.
\end{proof}
While the aperiodic monomial $\zeta_A$ with \eqref{zA} is not unique, we shall fix one for each $A$.

The following type $C$ aperiodicity follows from two kinds of positivity properties and 
the highly nontrivial affine type $A$ aperiodicity in ~\cite[Proposition 6.5]{Lu99}.

\begin{prop} \label{monomial-aperiodic}
Let $\texttt{M}$ be an aperiodic  monomial in  ${\mbf S}^{\fc}_{n,d}$.
Suppose that $\texttt{M} =\sum c_A \{ A\}_d$ where $c_A\in \mbb Z[ v, v^{-1}]$. If $c_A \neq 0$, then $A$ must be aperiodic.
\end{prop}

\begin{proof}
We shall denote by $^{\fa} \{A\}_d$ the canonical basis elements in ${\mbf S}_{n,d}$ (and in $\bU_{n,d}$) \cite{Lu99}.
Recall that $\bU_{n, d}$ is generated by $\E_i, \F_i$ and $\mbf K_i^{\pm 1}$ for all $1\leq i\leq n$.
By (\ref{je}), we have
$\jmath_{n,d} (\texttt{M}) \in \bU_{n,d}$. 
By  ~\cite[Proposition 6.5]{Lu99}, we see that
\begin{align}
\label{j-ap-1}
\jmath_{n,d}(\texttt{M})=\sum_{B \ \text{aperiodic}}  g_B \ ^{\fa}\{B\}_d,
\quad \mbox{where} \ g_B \in \mbb{N} [v, v^{-1}].
\end{align}

For $A=(a_{ij}) \in \Xi_{n, d}$, 
we set
\begin{align}
\label{XiA}
\Xi_{n, d} (A) = \{ B=(b_{ij}) \in \Theta_{n, d} | b_{ij} =0, \forall i < j, b_{ij}  = a_{ij}, \forall i > j, \co (B) \models \co (A) \},
\end{align}
where the notation `$\mbf b \models \mbf a$' stands for  $b_i + b_{- i} + \delta_{i, 0} + \delta_{i, r+1}= a_i$ for all $1\leq i\leq n$.
In particular, if $A$ is aperiodic, so is any matrix in $\Xi_{n, d}(A)$.
Since $\zeta_A = \{ A\}_d +$ lower terms by Proposition~\ref{monomial}, it implies that
$$
\jmath_{n,d}(\{A\}_d) = \sum_{A^-\in \Xi_{n, d} (A)}  \ ^{\fa}\{A^-\}_d + 
 \sum_{A^-\in \Xi_{n, d} (A)} \sum_{B  < A^-} h_{A^-,B} \ ^{\fa}\{B\}_d + R_A,
\quad h_{A^-,B} \in \mbb{N}[v, v^{-1}],
$$
where $R_A$ is a linear combination of   $^{\fa}\{B\}_d$ over $\mbb{N} [v, v^{-1}]$ for those $B$ not lower triangular. 
Indeed, this can be proved by induction on the length of the monomial $\zeta_A$ and 
utilizing the fact that the action of  the Chevalley generators $\F_i^{(a)}$  
on a standard basis element of a lower triangular matrix $A$ gives rise to a linear combination of 
standard basis element of either 
lower triangular  matrices $A' < A$ or non-lower-triangular matrices.
The latter is an observation from the multiplication formula for the Chevalley generator $\F_i^{(a)}$ in (\ref{B*A}).
So
\begin{align}
\label{j-ap-2}
\begin{split}
\jmath_{n,d} (\texttt{M}) & = \jmath_{n,d} (\sum_A  c_A \{A\}_d)  \\
&= \sum_A  \sum_{A^-\in \Xi_{n, d} (A)}  c_A\  ^{\fa}\{A^-\}_d +  \sum_A \sum_{A^-\in \Xi_{n, d} (A)} \sum_{B  < A^-}  c_A h_{A^-,B} \ ^{\fa}\{B\}_d  + \sum_A c_A R_A.
\end{split}
\end{align}
Observe also that $c_A \in \mbb{N}[v, v^{-1}]$ due to the geometric interpretation of $\texttt{M}$.
This implies that the coefficient of $^{\fa} \{A^-\}_d$ in (\ref{j-ap-2})  is $c_A$ plus some terms  in $\mbb{N}[v, v^{-1}]$
since $h_{A^-, B} \in \mbb{N}[v, v^{-1}]$, hence  nonzero.
By comparing (\ref{j-ap-1}) and (\ref{j-ap-2}), we see that $A^-$ are all aperiodic. Therefore $A$ is aperiodic. The proposition is proved.
\end{proof}

\begin{thm}
\label{CB-Udn}
The set $
\{ \{A\}_d \vert A \in \MX_{n,d}^{\ap} \}$ forms a basis (called the canonical basis) of $\bU^{\fc}_{n,d}$. Also, 
the set $\{\zeta_A  \vert A \in \MX_{n,d}^{\ap} \}$  forms a  basis (called a monomial basis) of $\bU^{\fc}_{n,d}$.
\end{thm}

\begin{proof}
 For $A \in \MX_{n,d}^{\ap}$, we have $\zeta_A = [A] +\text{lower terms}$ by Proposition~\ref{monomial}, and so
$\zeta_A \in \{A\}_d +\sum_{A' <  A} \cA \{A'\}_d$; this sum can be additionally restricted to $A' \in \MX_{n,d}^{\ap}$
by Proposition~\ref{monomial-aperiodic}. Hence by an induction on $A$ by the partial ordering, we conclude that
$\{A\}_d \in \bU^{\fc}_{n,d}$. 
Since $\{ \{A\}_d \vert A \in \MX_{n,d}^{\ap} \}$ is clearly linearly independent and it forms a 
spanning set of $\bU^{\fc}_{n,d}$ by Proposition~\ref{monomial-aperiodic}, it is a basis  of $\bU^{\fc}_{n,d}$.

Since the transition matrix from $\{\zeta_A  \vert A \in \MX_{n,d}^{\ap} \}$ to the canonical basis is uni-triangular,
$\{\zeta_A  \vert A \in \MX_{n,d}^{\ap} \}$ forms a basis as well. 
\end{proof}

The next proposition follows from Propositions~\ref{comult+S}, \ref{Dc} and Theorem~\ref{CB-Udn}.
\begin{prop}
\label{comult+U}
For $B \in \MX_{n,d}^{\ap}$, write $\Delta^{\fc} (\{B\}_d) = \sum_{C  \in \MX_{n, d'}^{ap}, A \in \Theta_{n,d''}^{ap}}
 \hat m^{C,A}_B \{C\}_{d'} \otimes \ ^{\mathfrak a} \{A\}_{d''}$. Then we have $\hat m^{C,A}_B \in \mbb{N} [v, v^{-1}]$.
\end{prop}
We also have the following corollary of Proposition~\ref{j+} and Theorem~\ref{CB-Udn}.

\begin{cor}
The image of $\jmath_{n,d}$ of a canonical basis element  in $\bU^{\fc}_{n,d}$ is a sum of canonical basis elements of
$\bU_{n,d}$ with coefficients in  $\mbb{N} [v, v^{-1}]$.
\end{cor}

\newpage

\part{Lusztig algebras and coideal  subalgebras of  $\bU (\widehat{\mathfrak{sl}}_n)$}
  \label{part2}

\chapter{Realization of the idempotented coideal  subalgebra $\dot{\bU}^{\fc}_n$ of  $\bU(\slh_n)$} 
  \label{chap:coideal}

In this chapter we introduce the transfer maps $\phi^{\fc}_{d, d-n}$ on Schur algebras $\Sj$ and Lusztig algebras $\bU_{n,d}^{\fc}$.
We then construct algebras $\bU_n^{\fc}$ (or $\dot{\bU}_n^{\fc}$) from the projective system of algebras 
$\{(\bU_{n,d}^{\fc}, \phi^{\fc}_{d, d-n})\}_{d \ge 0}$. We show that
$\bU_n^{\fc}$ (or $\dot{\bU}_n^{\fc}$)  is isomorphic to an (idempotented) coideal subalgebra of $\bU(\slh_n)$, and $(\bU(\slh_n), \bU_n^{\fc})$ forms
an affine quantum symmetric pair. 
The canonical basis of $\dot{\bU}_n^{\fc}$ is established and shown to admit positivity with respect to
multiplication, comultiplication, and a bilinear pairing.

\section{The coideal subalgebra $\bU^{\fc}_n$ of $\bU_n$}
\label{sec:Unc}
 
Recall \cite{Lu00} there exists a homomorphism $\chi_n: {\mbf S}_{n, n} \rightarrow \mbb Q(v)$ such that
\[
\chi_n (\mbf E_i)=\chi_n (\mbf F_i)=0,
\chi_n (\mbf H_i) =v.
\]
Following Lusztig \cite{Lu00}, we introduce the transfer map of affine type $C$, 
$$
\phi^{\fc}_{d, d-n}: \Sj \rightarrow {\mbf S}^{\fc}_{n,d -n},
$$
which is by definition the composition of the following homomorphisms
(for $d\ge n$)
\begin{equation}
  \label{transfer}
\phi^{\fc}_{d, d-n}:
\Sj
\overset{\widetilde \Delta^{\fc}}{\longrightarrow}
{\mbf S}^{\fc}_{n, d -n  }  \otimes {\mbf S}_{n, n}
\overset{1\otimes \chi_n}{\longrightarrow}
{\mbf S}^{\fc}_{n,d -n}.
\end{equation}
The following can be proved similarly to \cite{Lu00} in affine type $A$ and \cite{FL15} in type $B/C$.
\begin{prop}
\label{phi-gene}
For  $i\in [0, r]$, we have
$\phi^{\fc}_{d, d-n} (\mbf e_i) = \mbf e'_i$,
$\phi^{\fc}_{d, d-n} (\mbf f_i) = \mbf f'_i$,
$\phi^{\fc}_{d, d-n} (\mbf k_i) = \mbf k'_i$.
\end{prop}

Now we consider  the projective system $\{(\bU^{\fc}_{n,d}, \phi^{\fc}_{d, d-n})\}_{d \ge 0}$ and its projective limit:
\[
\bU^{\fc}_{n,\infty} :=
\varprojlim_{d} \bU^{\fc}_{n,d}
= \Big \{ x \equiv  (x_d)_{d\in \mbb{N}}  \in \prod_{d\in \mbb{N}} \bU^{\fc}_{n,d}  \Big \vert   \phi^{\fc}_{d, d-n}(x_d) = x_{d-n} \quad \forall d \Big \}.
\]
Denote by $\phi_d^{\fc}: \bU_{n,\infty}^{\fc} \to \bU^{\fc}_{n,d}$ the natural projection.
The bar involution on $\bU^{\fc}_{n,d}$ induces a bar involution
$
\bar \ : \bU^{\fc}_{n,\infty} \to \bU^{\fc}_{n,\infty},
$
since it commutes with the transfer map \eqref{transfer}.
Similarly, we have an integral version:
$ \bU^{\fc}_{n,\infty; \cA} = \varprojlim\limits_{d} \bU^{\fc}_{n,d; \cA}.$
Since $\mbb Q(v) \otimes_{\mathcal A} \bU^{\fc}_{n,d; \cA} = \bU^{\fc}_{n,d}$ for all $d$, we have
$\mbb Q(v) \otimes_{\mathcal A} \bU^{\fc}_{n,\infty; \cA} = \bU^{\fc}_{ n, \infty}$.

Recall from Section~\ref{sec:Un-A} the counterparts of the above constructions in the affine type $A$ setting, 
where we drop the superscript $\fc$. 
We have the following commutative diagram
\[
\begin{CD}
\bU^{\fc}_{n,d} @> \jmath_{n,d} >> \bU_{n,d} \\
@V \phi^{\fc}_{d, d-n} VV @VV\phi_{d, d-n} V\\
\bU^{\fc}_{n,d -n} @> \jmath_{n,d -n} >> \bU_{n,d -n}
\end{CD}
\]
That is,  $\phi_{d, d-n} \circ \jmath_{n,d} = \jmath_{n,d -n}  \circ \phi^{\fc}_{d,d-n}.$
Thus by the universality of $\bU_{n,\infty}$, we have a unique algebra homomorphism
\[
\jmath_{n} : \bU^{\fc}_{n,\infty} \longrightarrow \bU_{n,\infty},
\]
such that $\phi_d \circ \jmath_{n} = \jmath_{n,d} \circ \phi^{\fc}_d.$

We define elements  $\e_i$, $\f_i$ and $\mbf k_i^{\pm 1}$ for all $0 \leq i \leq r$ in $\bU^{\fc}_{n,\infty}$ by
\[
(\e_i)_d = \e_{i, d},
(\f_i)_d = \f_{i, d},
(\mbf k_i^{\pm 1})_d = \mbf k^{\pm 1}_{i, d}, \quad \forall d \in \mbb{N},
\]
where the $d$ in the subscript of $\e_{i, d}$ etc. indicates $\e_{i, d}$ is a copy of the Chevalley generator $\mbf e_i$ in $\bU^{\fc}_{n,d}$.
Let $\bU^{\fc}_n$ be the subalgebra of $\bU^{\fc}_{n,\infty}$ generated by (the Chevalley generators) $\e_i$, $\f_i$ and $\mbf k^{\pm 1}_i$ for all $0 \leq i \leq r$.
Since $\jmath_{n,d}$ is injective for all $d$, so is $\jmath_{n}: \bU^{\fc}_{n,\infty} \rightarrow \bU_{n,\infty}$.
It follows by (\ref{je})  that the image of $\bU^{\fc}_n$ under $\jmath_{n}$ lies in $\bU_n$. Summarizing, we have obtained the following.

\begin{prop}
 \label{prop:inj-j}
There is a unique algebra imbedding $\jmath_n: \bU^{\fc}_n \to \bU_n$ such that
\begin{align}
\label{je-limit}
\begin{split}
\jmath_{n} (\e_i) & = \mbf E_i + v^{-\delta_{i,0}} \mbf K_i \mbf  F_{n-1-i}, \\
\jmath_{n} (\f_i ) & =  \mbf E_{n-1-i} + v^{ \delta_{i,0}} \mbf F_i \mbf K_{n-1-i},\\
\jmath_{n} (\mbf k_i) & = v^{-\delta_{i,0} + \delta_{i, r}} \mbf  K_i \mbf K_{n-1-i}^{-1}, \quad \forall i \in [0, r].
\end{split}
\end{align}
\end{prop}
Recall from Proposition~\ref{Un=slhn} that $\bU_n \cong \bU (\slh_n)$. 
At the $v=1$ limit, the images of the generators under $\jmath_n$ are in the fixed point subalgebra by an involution $\theta^{\jj}$ of $\slh_n$
(which switches $\mbf E$'s and $\mbf F$'s); 
for an illustration of $\theta^{\jj}$ see Figure~\ref{figure:jj}. 

Recall $\Delta^{\fc}$ from \eqref{Dj}. 
We have the following commutative diagram
\[
\begin{CD}
\bU^{\fc}_{n, d' + d''} @> \Delta^{\fc} >> \bU^{\fc}_{n, d'} \otimes \bU_{n,d''} \\
@V \phi^{\fc}_{d' + d'', d'+ d'' - (a+b) n} VV @VV \phi^{\fc}_{d', d' - an} \otimes \phi_{d'', d''-bn} V \\
\bU^{\fc}_{n, d'+d''-(a+b)n} @>\Delta^{\fc} >> \bU^{\fc}_{n, d'-an} \otimes \bU_{n, d''-bn}
\end{CD}
\]
for any $a, b\in \mbb{N}$.
By universality, these $\Delta^{\fc}$ (for $d', d'', n$) induce an algebra homomorphism
\[
\Delta^{\fc} : \bU^{\fc}_{n,\infty} \longrightarrow \bU^{\fc}_{n,\infty} \otimes \bU_{n,\infty}.
\]
Moreover, the image of $\bU^{\fc}_n$ under $\Delta^{\fc}_{n}$ is contained in $\bU^{\fc}_n \otimes \bU_n$ by Proposition ~\ref{Dc}.
Summarizing, we have the following.

\begin{prop}
 \label{prop:Unc-coideal}
There is a unique algebra homomorphism
$
\Delta^{\fc}: \bU^{\fc}_n \longrightarrow \bU^{\fc}_n \otimes \bU_n
$
such that, for all $i\in [0, r]$,
\begin{align}
\Delta^{\fc} (\mbf e_i)
& = \mbf e_i \otimes \mbf K_i + 1 \otimes  \mbf E_i
+ \mbf k_i \otimes v^{\delta_{i,0}}   \mbf F_{n-1-i}  \mbf K_i. \\
\Delta^{\fc} (\mbf f_i)
& =  \mbf f_i \otimes \mbf K_{n-1-i} + \mbf k^{-1}_i \otimes  v^{-\delta_{i,0}} \mbf K_{n-1-i} \mbf F_i
+ 1\otimes  \mbf E_{n-1-i}.\\
\Delta^{\fc} (\mbf k_i) &= \mbf k_i \otimes \mbf K_i \mbf K^{-1}_{n-1-i}.
\end{align}
\end{prop}
This algebra homomorphism is coassociative by Proposition ~\ref{coassociative} in the sense that
\begin{equation}
 \label{coass:UnC}
(1\otimes \Delta) \Delta^{\fc} = (\Delta^{\fc} \otimes 1) \Delta^{\fc}.
\end{equation}
As a degenerate case for \eqref{coass:UnC}, we also have
$$ 
\Delta  \circ \jmath_n =(\jmath_n \otimes 1) \circ \Delta^{\fc}.
$$ 

Summarizing the results from Propositions~\ref{prop:inj-j} and \ref{prop:Unc-coideal}, we have proved the following. 
\begin{thm}
  \label{thm:QSP}
The algebra $\bU^{\fc}_n$ is a coideal subalgebra of $\bU_n$, 
and $(\bU_n, \bU^{\fc}_n)$ forms an affine quantum symmetric pair in the sense of Kolb-Letzter \cite{Ko14}. (see Figure~\ref{figure:jj} for the relevant involution.)
\end{thm}
 
The following is a variant of \cite[Theorem 7.1]{Ko14} in our setting and our notation. 
\begin{prop}
 \label{present:Ujj}
For $n=2r+2$ with $r\geq 1$, the $\mbb Q(v)$-algebra $\bU_n^{\fc}$ has a presentation with generators
$\mathbf e_i, \mathbf f_i$, and $\mathbf k^{\pm 1}_i$  for $i\in [0, r]$ and relations given in Proposition~\ref{relation:Schurc}. 
\end{prop}
Note that the first relation in Proposition~\ref{relation:Schurc} (which is not present in \cite{Ko14})
simply reflects the fact that  various  quantum affine algebras arising from geometry in this paper are always of level zero. 
\footnote{Replace this sentence  by ``Note that  the generators $\mbf k^{\pm 1}_i$ for $i\in [r+1,n-1]$ can be defined similarly, but we will have $\mbf k^{\pm 1}_i = \mbf k^{\mp 1}_{n-1-i}$, this  simply reflects the fact that various  quantum affine algebras arising from geometry in this paper are always of level zero.''}

For $n=2$ (i.e., $r=0$), the imbedding $\jmath_2: \bU^{\C}_2 \to \bU_2 = \bU(\widehat{\mathfrak{sl}_2})$ in \eqref{prop:inj-j} is defined by
\[
\e_0 \mapsto \E_0 + v^{-1}\KK_0 \F_1,  \quad 
\f_0 \mapsto \E_1 + v^{-1}\KK_1 \F_0, \quad 
\bk_0 \mapsto\KK_0\KK_1^{-1}.
\] 
We shall give a presentation for $\bU_2^{\fc}$, which was excluded from Proposition~\ref{present:Ujj} above.

\begin{prop}
\label{present:Ujj-r=0}
The $\mbb Q(v)$-algebra $\bU^{\C}_2$ has a presentation with generators
$\mathbf e_0, \mathbf f_0$, and $\mathbf k^{\pm 1}_0$ and the following relations.
\begin{align}
\bk_0 \bk_0^{-1} =1, \quad \quad \bk_0 \e_0  = v^4 \e_0 \bk_0, & \qquad
\bk_0 \f_0 = v^{-4} \f_0 \bk_0, 
\label{kekf} \\
\label{UC-2-d}
\e^{3}_0 \f_0 - \llbracket 3\rrbracket \e^2_0 \f_0 \e_0 + \llbracket 3\rrbracket \e_0 \f_0 \e^{2}_0 -\f_0 \e^3_0 
& = \llbracket 3\rrbracket ! (v-v^{-1}) \e_0  (\bk_0 - \bk^{-1}_0) \e_0, \\
\label{UC-2-e}
\f^{3}_0 \e_0 - \llbracket 3\rrbracket \f^2_0 \e_0 \f_0 +\llbracket 3\rrbracket\f_0 \e_0 \f^{2}_0 -\e_0 \f^3_0 
& = - \llbracket 3\rrbracket ! (v-v^{-1}) \f_0   (\bk_0 - \bk^{-1}_0) \f_0.
\end{align}
Here $\llbracket i\rrbracket = \frac{v^i - v^{-i}}{v - v^{-1}} $ and $\llbracket a\rrbracket ! = \prod_{1\leq i\leq a}\llbracket i\rrbracket$.
\end{prop}

\begin{proof}
Note that $\bU_2$ is of level zero, so we have $\KK_0\KK_1 =1$. Thus $\jmath_n(\bk_0) =\KK_0^2$.
From this, we have the identity \eqref{kekf}.

We now prove the identity (\ref{UC-2-d}).
Since $\jmath_n$ is injective, it suffices to show that (\ref{UC-2-d}) holds after applying $\jmath_n$.
In other words, it suffices to prove the identity in $\bU_2$.
Let $S(\e_0, \f_0)$ denote the left hand side in (\ref{UC-2-d}). 
We define $S(\e_0, \E_1)$ and $S(\e_0, v^{-1}\KK_1 \F_0)$ in a similar fashion.
By a lengthy calculation involving $4\times 2^4 = 64$ terms, we have
\begin{align}
\label{UC-2-da}
S(\e_0, \E_1) =\llbracket 3\rrbracket ! (v - v^{-1}) \e_0 \bk_0 \e_0.
\end{align}
Similarly, we have
\begin{align}
\label{UC-2-db}
S(\e_0, v^{-1}\KK_1 \F_0) = - \llbracket 3\rrbracket ! (v - v^{-1}) \e_0  \bk^{-1}_0 \e_0.
\end{align}
So the relation (\ref{UC-2-d}) follows  by adding (\ref{UC-2-da}) and (\ref{UC-2-db}).
Similarly, one can show Eq.~(\ref{UC-2-e}) and we leave the detail to the reader.

Now we invoke \cite[Theorem 7.1]{Ko14}, which says no additional relations are needed. This finishes the proof. 
\end{proof}

\section{The algebra $\dot{\bU}^{\fc}_n$ and its monomial basis}
  \label{MA}
  
Let
\begin{equation}
 \label{Znc}
\mbb Z^{\fc}_n =\{ \lambda = (\lambda_i)_{i\in \mbb Z} \vert \lambda_i \in \ZZ,  
\lambda_i = \lambda_{i+n}, \lambda_i = \lambda_{-i}, \forall i, \lambda_0, \lambda_{r+1} \ \mbox{odd} \}.
\end{equation}
Let $|\lambda| = \lambda_1 + \ldots + \lambda_n$.
Define an equivalence relation $\approx$ on $\mbb Z^{\fc}_n$ by letting $\lambda \approx \mu $ if and only 
if $\lambda - \mu = (\dots, p, p, p, \dots)$, for  some even integer $p$.
Let $\mbb Z^{\fc}_n  / \approx$ be the set of equivalence classes with respect to the equivalence relation $\approx$;
and let $\widehat \lambda$ be the equivalence class of $\lambda$.

Fix $\widehat \lambda \in \mbb Z^{\fc}_n/\approx$, we define the element $1_{\widehat \lambda} \in \bU^{\fc}_{n,\infty} $ as follows.
$(1_{\widehat \lambda})_d = 0$ if $d \not \equiv |\lambda|$ (mod $2n$).  If $d = |\lambda| + pn$ for some even integer $p$,
we have $(1_{\widehat \lambda})_d = 1_{\lambda + (\ldots, p, p, p, \ldots)}$. Here $1_{\lambda + (\ldots, p, p, p, \ldots)} \in \bU^{\fc}_{n,d}$ is understood to be zero if there is a negative entry in $ \lambda + (\ldots, p, p, p, \ldots)$.

\begin{Def}
Let $\dot \bU^{\fc}_{n}$ be the $\bU^{\fc}_n$-bimodule in $\bU^{\fc}_{n,\infty}$ generated by $1_{\widehat \lambda}$
for all $\widehat \lambda \in \mbb Z^{\fc}_n/\approx$.
\end{Def}

It is clear that $\dot \bU^{\fc}_n$ is a subalgebra of $\bU^{\fc}_{n,\infty}$ generated by
$1_{\widehat \lambda}$, $\e_i 1_{\widehat \lambda}$ and $\f_i 1_{\widehat \lambda}$ for all $i\in [0, r]$ and $\widehat \lambda \in \mbb Z^{\fc}_n / \approx$.
Similarly, we define the $\mathcal A$-subalgebra $_{\mathcal A}\dot  \bU^{\fc}_n$ of $\bU^{\fc}_{n,\infty}$  generated by
$\e_i^{(a)} 1_{\widehat \lambda}$ and $\f_i^{(a)} 1_{\widehat \lambda}$, for all $i\in [0, r]$ and $a\in \mbb{N}$.
So we have
$\mbb Q(v) \otimes_{\mathcal A} \ _{\mathcal A} \dot \bU^{\fc}_{n} = \dot \bU^{\fc}_{ n}.$
The bar involution on $\bU^{\fc}_{n,\infty}$ induces a bar involution on $\dot \bU^{\fc}_n$, which we denote by
$
\bar \ : \dot \bU^{\fc}_n \longrightarrow \dot \bU^{\fc}_n.
$
Note that it leaves the elements $\e_i^{(a)} 1_{\widehat \lambda}$ and $\f_i^{(a)} 1_{\widehat \lambda}$ fixed, and hence we have 
$\bar \ : \ _{\mathcal A}  \dot \bU^{\fc}_n \longrightarrow \ _{\mathcal A} \dot \bU^{\fc}_n.$

We denote
\begin{align}
  \label{eq:Mtilde}
  \begin{split}
  \widetilde{\MX}_{n}
  &= \big\{A=(a_{ij}) \in \text{Mat}_{\mbb Z \times \mbb Z}(\mbb Z) \big \vert \,
a_{0, 0}, a_{r+1, r+1} \in 2 \mbb Z +1, \;
\\
&\qquad\qquad\qquad\qquad \qquad
a_{ij}\ge 0 \; (i\neq j),\;
a_{ij} = a_{-i, -j} = a_{i+n, j+n} (\forall i, j)
\big \},
\\
\widetilde{\MX}^{ap}_{n}
&= \{A \in \widetilde{\MX}_{n} \big \vert A \text{ is aperiodic} \}.
\end{split}
\end{align}
For $A \in \widetilde{\MX}_{n}$, we shall denote by 
$$
|A|=d
$$ 
if  $\sum_{i=i_0+1}^{i_0+n} \sum_{j\in \mbb Z} a_{ij}=2d+2$ for some (or each) $i_0 \in \ZZ$.
We set, for $d\in  \ZZ$, 
\begin{equation}
\label{tXid}
  \widetilde{\MX}_{n,d}
  =  \big \{A  \in   \widetilde{\MX}_{n} \big \vert \, |A| =d \big\}, \qquad    \widetilde{\MX}_{n} =\sqcup_d   \widetilde{\MX}_{n,d}. 
\end{equation}
Also clearly we have ${\MX}_{n,d} \subset   \widetilde{\MX}_{n,d}$.

We define an equivalence relation $\approx$ on $\widetilde{\MX}_{n}^{ap}$ by
\begin{equation}
\label{eq:approx}
A \approx B \ \mbox{iff} \ A - B = p I_n, \ \mbox{for some even integer} \ p,
\end{equation}
where $I_n= \sum_{1\leq i\leq n} E^{i i}$, and let $\widehat A$ be  the equivalence class of $A$.
Whenever there causes no ambiguity, we write $I$ for $I_n$.
We define $\ro(\widehat A) = \widehat{\ro(A)}$ and $\co(\widehat A) = \widehat {\co(A)}$, and they are elements in $\mbb Z^{\fc}_n/\approx$.
We can then define the element $\zeta_{\widehat A} $ in $\dot \bU^{\fc}_n$ by
$(\zeta_{\widehat A})_d =0$ unless $d = |A|$ mod $2n$, and if $|A| = d+ p/2 n$ for some even integer $p$,
$
(\zeta_{\widehat A})_d  = \zeta_{A+pI},
$
where $\zeta_{A+pI}$ is the monomial basis attached to $A+pI$ in  Theorem ~\ref{CB-Udn}.
Since $\phi^{\fc}_{d, d-n} (\zeta_{A+pI}) = \zeta_{A+(p-2)I}$, we see that $\zeta_{\widehat A} \in \dot \bU^{\fc}_n$.
 
 The following linear independence is reduced to the counterpart at
 the Schur algebra level, by an argument similar to ~\cite[Theorem 5.5]{LW15}.

\begin{prop}
The set $\{ \zeta_{\widehat A} \big \vert  \widehat A \in \widetilde{\MX}_{n}^{ap}/\approx \} $ is linearly independent.
\end{prop}

To show that $\zeta_{\widehat A}$ is indeed a basis for $\dot \bU^{\fc}_n$,
let us take a closer look  at  the behavior of the monomials at the Schur algebra level.
For simplicity, we write $\f_{n-(i+ 1)} $ for $\e_i$ for all $i\in [0, r]$.
For $ \lambda \in \Lambda^{\fc}_{n,d}$ and
a pair $(\mbf i, \mbf a)$ where $\mbf i =( i_1, \ldots, i_s)$ and $\mbf a = (a_1, \ldots, a_s)$ with $0 \leq i_j \leq n$  and $a_j \in \mbb{N}$
for all $j$, we set
\[
_d \texttt{M}_{\mbf i, \mbf a, \lambda} = \f_{i_1}^{(a_1)} \f_{i_2}^{(a_2)} \cdots \f_{i_s}^{(a_s)} 1_{\lambda} \in \bU^{\fc}_{n,d},
\]
where $1_{\lambda}$ is the standard basis element of the diagonal matrix whose diagonal is $\lambda$.
Then $_d \texttt{M}_{\mbf i, \mbf a, \lambda} $ exhaust all possible monomials in $\bU^{\fc}_{n,d}$.
The following proposition is crucial  in showing that the various $\xi_{\widehat A}$ forms a basis for $\dot \U^{\C}_n$. 
Recall that $I = \sum_{1\leq i\leq n} E^{ii}_{\theta}$.

\begin{prop}
\label{m-c-p}
Fix a triple $(\mbf i, \mbf a, \lambda)$ with $|\lambda|=d$.
There is a finite subset $\mathcal I_{\mbf i, \mbf a, \lambda}$ of $\{ A \in \widetilde{\MX}_{n}^{ap} \big \vert |A|=d\}$ such that
\[
_{d+ pn} \texttt{M}_{\mbf i, \mbf a, \lambda + 2pI} = \sum_{A \in \mathcal I_{\mbf i, \mbf a, \lambda}} c_{A}   \zeta_{_{2p} A}, \; \forall p,
\ \mbox{where} \  c_{A} \in \mathcal A \ \mbox{is independent of} \ p.
\]
\end{prop}

\begin{proof}
By the multiplication formula for simple generators, we see that the standard basis element $[A]$, possibly periodic,
appearing in $_{d+pn} \texttt{M}_{\mbf i, \mbf a, \lambda+2pI}$ is  stabilized for $p\gg0$.
In other words, there is a finite set $\mathcal J_{\mbf i, \mbf a, \lambda}$ in $\widetilde{\MX}_{n}$ consisting of certain $A$ subject to  $|A|=d$ and
\[
_{d+ pn} \texttt{M}_{\mbf i, \mbf a, \lambda + 2pI} = \sum_{A \in \mathcal J_{\mbf i, \mbf a, \lambda}} g_{A, p}  [_{2p} A], \quad  \forall p
\]
where $g_{A, p} \in \mathcal A$  depends on $p$ in general.

Note that $\mathcal J_{\mbf i, \mbf a, \lambda}$ can be constructed in the following way.
Fix a $p$ large enough, so that when we multiply out the monomial
$_{d+pn}\texttt{M}_{\mbf i, \mbf a, \lambda + 2pI}$ in terms of standard basis,
we do not miss a term because that term has a negative entry in its diagonal.
Collect all the matrices, say $A$,  parametrizing   the standard basis element appearing in  $_{d+pn}\texttt{M}_{\mbf i, \mbf a, \lambda + 2pI}$, and further  throwing into this set all matrices $B$ such that $B<_{\text{alg}} A$.  This resulting set is again finite.
$\mathcal J_{\mbf i, \mbf a, \lambda}$ is then defined to be  the set of matrices obtained by subtracting the matrices in the previous set by $2pI$.

Let $\mathcal I_{\mbf i, \mbf a, \lambda}$  be the subset of $\mathcal J_{\mbf i, \mbf a, \lambda}$  consisting of aperiodic elements.
It follows by Theorem ~\ref{CB-Udn} that
\[
_{d+ pn} \texttt{M}_{\mbf i, \mbf a, \lambda + 2pI} = \sum_{A \in \mathcal I_{\mbf i, \mbf a, \lambda}} c_{A, p}  \zeta_{_{2p} A}, \quad  \forall p,
\]
where $c_{A, p} \in \mathcal A$ depends on $p$ in general.

By definition, we have
\begin{align*}
\phi^{\fc}_{d+pn, d+pn-n} ( _{d+pn} \texttt{M}_{\mbf i, \mbf a, \lambda + 2pI}) &= \ _{d+pn-n} \texttt{M}_{\mbf i, \mbf a, \lambda +2pI - 2I},
 \\
\phi^{\fc}_{d+pn, d+pn-n} ( \zeta_{_{2p} A}) &= \zeta_{_{p-2} A}, \quad \forall p.
\end{align*}
This implies that
\[
c_{A, p} = c_{A, p-1}, \quad \mbox{if} \ \zeta_{_{2(p-1)} A} \neq 0  \in \bU^{\fc}_{n, d+(p-1)n}.
\]
For large enough $p$, $\zeta_{_{2(p-1)} A}$ is obviously nonzero, and so $c_{A, p} =c_A$ is independent of  $p\gg 0$.
Recall that the set $\mathcal I_{\mbf i, \mbf a, \lambda}$ is finite. So we can find a $p_0$ such that $c_{A, p} =c_A$ 
for all $p\geq p_0$ and for all $A \in \mathcal I_{\mbf i, \mbf a, \lambda}$.
The proposition is thus proved.
\end{proof}

Now we return from Lusztig algebras to the algebra $\dot \bU^{\fc}_n$.

\begin{prop}
\label{m-basis}
The set $\{ \zeta_{\widehat A} | \widehat A \in \widetilde{\MX}_{n}^{ap}/\approx \} $ forms a basis for $\dot \bU^{\fc}_n$ and an $\mathcal A$-basis for ${}_\cA \dot{\bU}^{\fc}_n$.
\end{prop}

\begin{proof}
Similar to the element $_d \texttt{M}_{\mbf i, \mbf a, \lambda}$, we can define its limit version  $\texttt{M}_{\mbf i, \mbf a, \widehat \lambda}$ in $\dot \bU^{\fc}_n$. Moreover, these monomials exhaust all the possible monomials in $\dot \bU^{\fc}_n$.
The proposition now follows from Proposition ~\ref{m-c-p}.
\end{proof}

\section{Bilinear form on $\dot{\bU}^{\fc}_n$}
\label{bilinear-form}

Recall that for $i \in [0,r]$,
$\e_i = \sum [A]$ where $A - E^{i+1, i}_{\theta}$ is diagonal,
$\mbf f_i =\sum [A]$ where $A - E^{i, i+1}_{\theta}$ is diagonal,
and $\mbf k_i =\sum_{\lambda\in \Lambda^{\C}_{n, d}}  v^{\lambda_{i+1} -\lambda_i} 1_{\lambda}$.

Imitating McGerty~\cite{Mc12} in affine type $A$, we define a bilinear form $\langle \cdot, \cdot \rangle_d$ on ${\mbf S}^{\fc}_{n,d}$  as follows:
\[
\langle [A], [A'] \rangle_d = \delta_{A, A'} v^{-2 d_{A^t}}    \# X^{L'}_{A^t},
\]
where $L' \in \X^{\fc}_{n, d}(\ro(A^t))$. With the help of  the identity \eqref{dAdA},  the same argument as in \cite[Proposition~ 3.2]{Mc12} gives us the following.

\begin{prop}
\label{prop:adjoint}
We have $\langle [A] * [B], [C] \rangle_d =   \langle [B], v^{d_A - d_{A^t}} [A^t] *[C] \rangle_d$.
\end{prop}

\begin{cor}
For all $i\in [0, r]$, we have the following:
\begin{enumerate}
\item $ \langle \e_i  [A_1], [A_2] \rangle_{d} = \langle [A_1], v \mbf k_i   \mbf f_i [A_2] \rangle_{d}$.
\item $\langle \mbf f_i [A_1], [A_2] \rangle_{d} = \langle [A_1], v^{-1} \e_i \mbf k_i^{-1}   [ A_2 ] \rangle_{d}$.
\item $\langle \mbf k_i [A_1], [A_2] \rangle_{d} = \langle [A_1], \mbf k_i [A_2] \rangle_{d}$.
\end{enumerate}
\end{cor}

\begin{proof}
We prove (1).
If $A - E^{i+1, i}_{\theta}$ is diagonal for some $i \in [0, r]$,
then
$$
d_A = \co(A)_{i+1} \quad \mbox{and} \quad d_{A^t} =  \ro(A)_i = \co(A)_i  -1.
$$
Hence $d_A - d_{A^t}  = \co(A)_{i+1} - \co(A)_i +1$.
Thus, we have
$$
v\mbf k_i (L, L') = \delta_{L, L'} v^{1 +\co(A)_{i+1} - \co(A)_i }= \delta_{L,L'} v^{d_A-d_{A^t}},  \forall L, L' \in \X^{\fc}_{n, d}(\co(A)),
$$
which implies (1).

We now prove (2).  If  $A - E^{i, i+1}_{\theta}$ is diagonal for some $i\in [0, r]$,
then
$$
d_A =  \co(A)_i
\quad \mbox{and} \quad
d_{A^t} = \ro(A)_{i+1} = \co(A)_{i+1} -1 - \delta_{i,0} -  \delta_{i, n}.
$$
So
$d_A - d_{A^t} = \co(A)_i - \co(A)_{i+1} + 1  + \delta_{i,0} + \delta_{i,n}$.
Hence, if $(L, L')$ subject to  $L\in \X^{\fc}_{n, d}(\co(A))$, $L_i \subseteq L'_i$, $L_j = L'_j$ for all $j \in [0, r]\backslash \{i\}$, then
\begin{align*}
v^{d_A - d_{A^t}} \e_i (L, L') &= v^{1 + \delta_{i,0} + \delta_{i, n}} \mbf k_i^{-1} \e_i (L, L')
  \\
&=  v^{1 + \delta_{i,0} + \delta_{i, n}}  v^{- 2 -\delta_{i,0} - \delta_{i, n}}  \e_i \mbf k_i^{-1} (L, L') = v^{-1} \e_i \mbf k_i^{-1} (L, L').
\end{align*}
Part~(2) follows.

Part~ (3) follows from the fact that $d_A = d_{A^t} =0$ if $A$ is diagonal.
\end{proof}

The same argument as in~\cite{Mc12} shows that there is a well-defined bilinear form $\langle \cdot, \cdot \rangle$ on $\dot \bU^{\fc}_n$ given by
\[
\langle x, y \rangle  = \sum_{d=1}^n \lim_{p\to \infty} \langle  x_{d+ pn}, y_{d+pn} \rangle_{d+pn}, \quad \forall x=(x_d), y=(y_d) \in \dot \bU^{\fc}_n.
\]

\begin{rem}
The  same adjointness property as in Proposition~\ref{prop:adjoint} holds 
 for the bilinear form $\langle \cdot, \cdot \rangle$ on $\dot \bU^{\fc}_n$.
\end{rem}

\section{The canonical basis of $\dot{\bU}^{\fc}_n$ and positivity}
\label{sec:CB-Unc}

As we have set up all the preliminary preparation,
the constructions and properties of the canonical basis for $\dot \bU^{\fc}_n$ can be established without further difficulty.
Actually as the technical proofs for the intermediate steps are literally the same as in the affine type $A$ setting \cite{Mc12} and/or 
in the finite type $B/C$ setting \cite{LW15, FL15}, we will formulate the statements while referring to those papers for detailed proofs. 

With the help of the bilinear form $\langle \cdot, \cdot \rangle$ and Theorem~\ref{CB-Udn}, 
the same arguments as in ~\cite{Mc12}, or ~\cite{LW15} prove the following. 

\begin{prop}
\label{prop:CBstable}
For any $A \in \MX^{\ap}_{n,d}$, we have
\[
\phi^{\fc}_{d+pn, d+(p-1)n}  (\{_{2p}A \}_{d+pn} ) =\{ _{2p-2} A\}_{d+(p-1)n}, \quad \forall p\gg 0.
\]
Moreover, we have
$$
\{_{2p}A\}_{d + pn} = \zeta_{_{2p}A} +\sum_{B \in \widetilde{\MX}_{n}^{ap}: B < A} c_{A, B, p}\, \zeta_{_{2p}B}
$$
with $c_{A, B, p} \in \cA$ independent of $p$ for $p\gg 0$.
\end{prop}

Recall $\widetilde{\MX}_{n}$ and $\widetilde{\MX}_{n}^{\ap}$ from \eqref{eq:Mtilde}.
\begin{Def} 
For  any $\widehat A \in \widetilde{\MX}_{n}^{ap}/\!\approx$,  an element
$b_{\widehat A} \in \dot \bU^{\fc}_n$ is defined as follows:
$(b_{\widehat A})_d =0$ if $d \neq |A| $ mod $2n$;  If $ |A| =d + s n$ for some  integer $s$, we set
\[
(b_{\widehat A})_{d + sn+ pn} = \{ _{2p} A \}_{d+sn+pn}, \quad \forall p \geq p_0, \ \mbox{for some fixed} \ p_0,
\]
and for general $q<p_0$, we set
$
(b_{\widehat A})_{d+sn+qn} = \phi_{d+sn+p_0n, d+rn+qn}^{\fc} (\{_{2 p_0} A\}_{d+sn+p_0n}).
$
\end{Def}
The fact that $b_{\widehat A}$ as defined above lies in $\dot \bU^{\fc}_n$ follows from Proposition~\ref{prop:CBstable}. 
Moreover, $\zeta_{\widehat A} = b_{\widehat A} +$ lower terms.
The next theorem now follows from the existence of the monomial basis $\{\zeta_{\widehat A}\}$ for $\dot \bU^{\fc}_n$; cf. Proposition ~\ref{m-basis}. 

\begin{thm}
 \label{thm:iCB-Unc}
The set $\dot{\mbf B}^{\fc}_n  :=\{b_{\widehat A} \big \vert  \widehat A \in \widetilde{\MX}_{n}^{ap}/\approx\}$  
forms a basis for $\dot \bU^{\fc}_n$.
\end{thm}
The basis $\dot{\mbf B}^{\fc}_n$ is called the {\em canonical basis} of $\dot \bU^{\fc}_n$.

As a consequence, we deduce formally the following results by the same arguments in ~\cite{LW15} and ~\cite{FL15}.

\begin{prop}
The signed canonical basis $\{ \pm b_{\widehat A}\big \vert  \widehat A \in \widetilde{\MX}_{n}^{ap}/\approx\}$
is characterized by the bar-invariance, integrality (i.e. $b_{\widehat A} \in {}_\cA \dot \bU^{\fc}_n$), and almost orthonormality
(i.e., $\langle b_{\widehat A}, b_{\widehat A'} \rangle =\delta_{{\widehat A},  {\widehat A'}} \mod v^{-1}\mbb Z[[v^{-1}]] $). 
\end{prop}

The canonical basis of $\dot \bU^{\fc}_n$ enjoys several remarkable positivity properties as follows.
The proofs use the same arguments as in ~\cite{LW15} and ~\cite{FL15}. In particular, for the positivity with respect to
comultiplication, the positivity of the canonical basis in the Lusztig algebra $\bU_{n,d}^{\fc}$ as in Proposition~\ref{comult+U} is used.

\begin{thm}
  \label{thm:positivity-Unc}
The structure constants of the canonical basis $\dot{\mbf B}^{\fc}_n$ lie in $\mbb N[v, v^{-1}]$
with respect to the multiplication and  comultiplication, and in $v^{-1} \mbb N[[v^{-1}]]$ with respect to the bilinear pairing.
\end{thm}

\section{Another presentation of  the algebra $\dot{\bU}^{\fc}_n$}
\label{another}

We shall give a more familiar description of the algebra $\dot{\bU}^{\fc}_n$. 
We start with introducing the limit version of the imbeddings $\jmath_{n,d}$.

Recall $\mbb Z_n$ from \eqref{Zn} and $\mbb Z_n^{\fc}$ from  \eqref{Znc}, and there is an inclusion $\mbb Z_n^{\fc} \subset \mbb Z_n$. 
Recall the notation $\models$ from (\ref{XiA}), and we extend it further to $\mbb Z^{\C}_n \times \mbb Z_n$ as follows.
Given a pair $( \lambda,  \lambda' ) \in \mbb Z_n^{\fc} \times \mbb Z_n$, we write $\lambda'  \models \lambda$ if
\[
\lambda_i = \lambda_i' + \lambda_{n-i}' + \delta_{i, 0} + \delta_{i, r+1}, \quad \forall i.
\]
We write $\bar{\lambda}'  \models \widehat{\lambda}$ if $\lambda'  \models \tilde \lambda$ for some $\tilde \lambda $ in the equivalence class $\widehat \lambda$ and the notation $\bar{\lambda}$ is defined in Section ~\ref{sec:Un-A}.
(In this case, we shall assume that $\lambda'  \models \lambda$.)

Recall  $_{\bar{\mu} } (\dot \bU_n)_{\bar{\lambda} }$
from Section ~\ref{sec:Un-A}.
We set 
$\ _{\widehat \mu} (\dot \bU^{\fc}_n)_{\widehat \lambda}
=1_{\widehat \mu} \dot \bU^{\fc}_n 1_{\widehat \lambda}$.
For a quadruple $(\widehat \lambda, \widehat \mu, \bar{\lambda}' , \bar{\mu}' )$ such that
$\bar{\lambda}'  \models \widehat \lambda$ and $ \bar{\mu}'  \models \widehat \mu$,
we define a linear map
\[
\jmath_{ \bar{\lambda}' , \bar{\mu}' , \widehat \lambda, \widehat \mu} :
\ _{\widehat \mu} (\dot \bU^{\fc}_n)_{\widehat \lambda}
\longrightarrow
\ _{\bar{\mu}' } (\dot \bU_n)_{\bar{\lambda}' },
\]
to be the composition
\[
\ _{\widehat \mu} (\dot \bU^{\fc}_n)_{\widehat \lambda}
\hookrightarrow
\dot \bU^{\fc}_n
\overset{\jmath_n}{\longrightarrow}
\bU_{n,\infty}
\twoheadrightarrow
\ _{\bar{\mu}' } (\dot \bU_n)_{\bar{\lambda}' },
\]
where the first map is a natural inclusion and the third one is the projection.
Set
\[
\jmath_{\widehat \lambda, \widehat \mu}=
\prod_{ \bar{\lambda}'  \models \widehat \lambda,\, \bar{\mu}'  \models \widehat \mu}
\jmath_{ \bar{\lambda}' , \bar{\mu}' , \widehat \lambda, \widehat \mu}:
\ _{\widehat \mu} (\dot \bU^{\fc}_n)_{\widehat \lambda}
\longrightarrow
\prod_{ \bar{\lambda}'  \models \widehat \lambda, \, \bar{\mu}'  \models \widehat \mu}
\ _{\bar{\mu}' } (\dot \bU_n)_{\bar{\lambda}' }.
\]

Recall the imbedding $\jmath_{n} : \bU^{\fc}_{n,\infty} \rightarrow \bU_{n,\infty}$ from Section~\ref{sec:Unc}. 
We have
$$
\jmath_{n} ( \ _{\widehat \mu} (\dot \bU^{\fc}_n)_{\widehat \lambda}) \subseteq
\prod_{\bar{\lambda}'  \models \widehat \lambda, \, \bar{\mu}'  \models \widehat \mu}
\ _{{\bar{\mu}'}} (\dot \bU_n)_{\bar{\lambda}' }.
$$
The injectivity of $\jmath_n$ implies that 
the homomorphism $\jmath_{\widehat \lambda, \widehat \mu}$ is injective.

Now a modified form of $\bU^{\fc}_{n}$, denoted by $ \Ua$, can be defined algebraically in a standard way as
\[
\Ua = \oplus_{\widehat{\mu}, \widehat{\lambda} \in \mbb Z^{\fc}_n / \approx} \   _{\widehat{\mu}} (\bU^{\fc}_{n, \text{alg}})_{\widehat{\lambda}},
\]
where
\[
_{\widehat{\mu}} (\Ua)_{\widehat{\lambda}}
= \bU^{\fc}_n\Big / \Big(\sum_{0 \leq i \leq r} (\mbf k_i - v^{-\mu_i + \mu_{i+1}} ) \bU^{\fc}_n + \sum_{0\leq i \leq r} \bU^{\fc}_n (\mbf k_i - v^{-\lambda_i + \lambda_{i+1}})\Big).
\]
The algebra homomorphism $\jmath_n: \bU^{\fc}_n \to \bU_n$ then induces a linear map
\[
\tilde \jmath_{ \bar{\lambda}' , \bar{\mu}' , \widehat \lambda, \widehat \mu} :
\ _{\widehat \mu} (\dot \bU^{\fc}_{n, \text{alg}} )_{\widehat \lambda}
\longrightarrow
\ _{\bar{\mu}' } (\dot \bU_n)_{\bar{\lambda}' }
\]
such that the following diagram commutes:
\[
\begin{CD}
\bU^{\fc}_n @> \jmath_n >> \bU_n \\
@VVV @VVV\\
\ _{\widehat \mu} (\dot \bU^{\fc}_{n, \text{alg}} )_{\widehat \lambda}
@> \tilde \jmath_{ \bar{\lambda}' , \bar{\mu}' , \widehat \lambda, \widehat \mu} >>
\ _{{\bar{\mu}' }} (\dot \bU_n)_{{\bar{\lambda}' }}
\end{CD}
\]
Set
\[
\tilde \jmath_{\widehat \lambda, \widehat \mu}=
\prod_{\substack{\bar{\lambda}'  \models \widehat \lambda \\ \bar{\mu}'  \models \widehat \mu}}
\tilde \jmath_{ \bar{\lambda}' , \bar{\mu}' , \widehat \lambda, \widehat \mu}:
\ _{\widehat \mu} (\dot \bU^{\fc}_{n, \text{alg}})_{\widehat \lambda}
\longrightarrow
\prod_{\substack{\bar{\lambda}'  \models \widehat \lambda \\ \bar{\mu}'  \models \widehat \mu}}
\ _{\bar{\mu}' } (\dot \bU_n)_{\bar{\lambda}' }.
\]
Since $\jmath_n$ is injective,  
$\tilde \jmath_{\widehat \lambda, \widehat \mu}$ is injective.

By definition, there exists a unique linear map
\[
\phi_{d, \text{alg},\widehat{\mu}, \widehat{\lambda}}: 
\ _{\widehat{\mu}} (\dot \bU^{\fc}_{n, \text{alg}})_{\widehat \lambda} \longrightarrow \oplus_{\mbf b \in \widehat{\mu},
 \mbf a\in \widehat{\lambda}} \bU^{\fc}_{n,d} (\mbf b, \mbf a),
\]
where $\bU^{\fc}_{n,d} (\mbf b, \mbf a) = 1_{\mbf b}(\bU^{\fc}_{n,d} )1_{\mbf a}$, such that the following diagram commutes:
\[
\begin{CD}
\bU^{\fc}_n @>>> _{\widehat{\mu}} (\dot \bU^{\fc}_{n, \text{alg}})_{\widehat \lambda} \\
@V\phi_d VV  @VV\phi_{d, \text{alg}, \widehat \mu, \widehat \lambda} V \\
\bU^{\fc}_{n,d} @>>> \oplus_{\mbf b \in \widehat{\mu}, \mbf a\in \widehat{\lambda}} \bU^{\fc}_{n,d} (\mbf b, \mbf a)
\end{CD}
\]
From this we have constructed an algebra homomorphism
$$
\phi_{d, \text{alg}} := \oplus_{\widehat \mu, \widehat \lambda \in \mbb Z^{\fc}_n/\approx} \phi_{d, \text{alg}, \widehat \mu, \widehat \lambda}:
\dot \bU^{\fc}_{n, \text{alg}} \longrightarrow \bU^{\fc}_{n,d}.
$$
Since $\phi_{d, \text{alg}}$ commutes with the transfer maps, i.e.,
$
\phi_{d-n, \text{alg}} = \phi_{d, d-n} \phi_{d, \text{alg}}.
$
we obtain an algebra homomorphism
$\psi : \dot \bU^{\fc}_{n, \text{alg}} \to   \bU^{\fc}_{n,\infty}$.
Observe that the image of this homomorphism is exactly  $\dot \bU^{\fc}_n$ by considering the image of the idempotents 
$1_{\widehat \lambda}$. Therefore, we have a surjective algebra homomorphism:
$
\psi: \Ua \longrightarrow \dot \bU^{\fc}_n.
$
By restriction, we have
$
\psi_{\widehat \mu, \widehat \lambda}: \ _{\widehat \mu} (\Ua)_{\widehat \lambda}  \rightarrow \ _{\widehat \mu} (\dot \bU^{\fc}_n)_{\widehat \lambda},
$
for various $\widehat \mu, \widehat \lambda$.
Since $\jmath_{\widehat \lambda, \widehat \mu}$ and $\tilde \jmath_{\widehat \lambda, \widehat \mu}$  are injective, and 
$
\tilde \jmath_{\widehat \lambda, \widehat \mu} = \jmath_{\widehat \lambda, \widehat \mu} \circ \psi_{\widehat \mu, \widehat \lambda},
$
we conclude that $\psi_{\widehat \mu, \widehat \lambda}$ and hence $\psi$ is injective.  
Summarizing, we have established the following.

\begin{prop}
The map $\psi: \Ua \to \dot \bU^{\fc}_n$ is an algebra isomorphism.
\end{prop}

Therefore, a presentation of $\dot \bU^{\fc}_n$ is reduced to finding a presentation of $\Ua$, 
and the latter can be obtained
by modifying the definition/presentation of $\bU^{\fc}_n$ as given in Propositions ~\ref{present:Ujj} and \ref{present:Ujj-r=0}, 
in a way similar to Lusztig's presentation for modified quantum groups \cite{Lu93}. 
The finite type counterpart of a presentation of $\Ua$ can be found in \cite{BKLW14}.

\chapter{A second coideal subalgebra of quantum affine $\mathfrak{sl}_\nn$}
 \label{chap:coideal2}
 
 In this chapter, setting $\nn=n-1$ we consider a subvariety of $\X_{n,d}^{\fc}$, and study its corresponding convolution algebra 
 ${\mbf S}^{\ji}_{\nn,d}$ which is a subalgebra of $\Sj$. 
 We introduce Lusztig subalgebra ${\mbf U}^{\ji}_{\nn,d}$ of the $\jiw$-Schur algebra ${\mbf S}^{\ji}_{\nn,d}$. 
 We study the properties of a comultiplication  on ${\mbf U}^{\ji}_{\nn,d}$,
 which allow us to form a projective system $\{({\mbf U}^{\ji}_{\nn,d},  \phi^{\ji}_{d, d-\nn})\}_{d\ge 0}$ and then
 two distinguished algebras ${\mbf U}^{\ji}_{\nn}$ and $\dot{\mbf U}^{\ji}_{\nn}$. 
We show that $(\bU(\slh_{\nn}),  {\mbf U}^{\ji}_{\nn})$ forms an affine quantum symmetric pair. 
The canonical basis of $\dot{\mbf U}^{\ji}_{\nn}$ is established and shown to admit positivity with respect to
multiplication, comultiplication, and a bilinear pairing. 

Recall $n =2r+2$, and we now set 
$$\nn =n-1=2r+1 \quad (r\geq 1).
$$

\section{The Schur algebras of type $\ji$}
   \label{ji-version}

We shall construct Schur algebras ${\mbf S}^{\ji}_{\nn,d}$ and Lusztig algebras $\bU^{\ji}_{\nn,d}$.
These algebras are defined as the affine counterpart of \cite{BKLW14}, and many basic properties of these algebras are established following~\cite[Section 5]{FL15}.

Recall 
the set $\MX_{n,d}$ from \eqref{Mdn}.
We introduce a subset $\MX^{\ji}_{\nn,d}$ which consists of matrices $A \in \MX_{n,d}$ whose $(r+1)$st row 
and $(r+1)$st column entries are all zero except $a_{r+1, r+1} =1$, i.e., 
\begin{equation}
 \label{Mjid}
\MX^{\ji}_{\nn,d} =\{ A\in \MX_{n,d} \big \vert  a_{r+1, j} = \delta_{r+1, j} ,  a_{i, r+1} = \delta_{i, r+1}, \forall i, j \in \mbb Z\}.
\end{equation}
Introduce the following idempotent in the algebra $\Sj$:
\begin{equation}
  \label{jr}
\bj_r = \sum_{A \in \MX^{\ji}_{\nn,d}: A \text{ diagonal}} [A],
\end{equation}
and form the following subalgebra of $\Sj$:
\begin{align}
\label{Sji}
{\mbf S}^{\ji}_{\nn,d} = \bj_r \Sj \bj_r.
\end{align}
Then $\bj_r$ becomes the identity of ${\mbf S}^{\ji}_{\nn,d}$, which will sometimes be denoted by $1$ when there is no ambiguity.
Note that the algebra ${\mbf S}^{\ji}_{\nn,d}$ is the generic version of the convolution algebra on pairs of lattice chains in
the set $\X^{\ji}_{\nn,d} : =\{ L \in \X^{\C}_{n,d} | L_r = L_{r+1}\}$.
The set $\{ [A] | A\in \MX^{\ji}_{\nn,d}\}$ forms a basis of ${\mbf S}^{\ji}_{\nn,d}$.

Introduce the following elements in ${\mbf S}^{\ji}_{\nn,d}$:
\begin{align}
\begin{split}
\eji_{i} & = \bj_r \e_{i} \bj_r,  \quad
\fji_i  = \bj_r \f_i \bj_r, \\
\kji^{\pm 1}_i & = \bj_r \bk^{\pm 1}_i \bj_r, \quad \forall i\in [0, r-1], \\
\hji^{\pm 1}_a & =  \bj_r \bh^{\pm 1}_a \bj_r, \quad \forall a \in [0, r], \\
\tji_r & = \bj_r \Big ( \f_r \e_r + \frac{\bk_r - \bk_r^{-1}}{v - v^{-1}} \Big ) \bj_r.
\end{split}
\end{align}
We note that
\[
\bj_r \e_r \bj_r =0,  \quad
\bj_r \f_r \bj_r =0, \quad
\bj_r \e_r \f_r \bj_r =0.
\]
Lusztig algebra (of type $\jiw$) ${\mbf U}^{\ji}_{\nn,d}$ is defined to be the subalgebra of ${\mbf S}^{\ji}_{\nn,d}$ generated by 
the Chevalley generators 
$\eji_i$, $\fji_i$, $\kji^{\pm 1}_i$, for all $i\in [0, r-1]$, and $\tji_r$.

Now let us present the type $A$ analogue of the above construction.
Recall from Section~\ref{sec:MB-A} in Chapter~\ref{chap:A}
 that $\Theta_{n,d}$ parametrizes a basis of ${\mbf S}_{n,d}$.
We set
\[
\Theta^{\ji}_{\nn,d} =\{ A\in \Theta_{n,d} \vert a_{i, r+1}=0, a_{r+1, j}=0, \quad \forall i, j \in \mbb Z\}.
\]
Similar to $\bj_r$, we define the following idempotent in ${\mbf S}_{n,d}$:
\[
\bJ_r = \sum_{A \in \Theta^{\ji}_{\nn,d} : A \ \text{diagonal}} [A].
\]
As the algebra $\bJ_r {\mbf S}_{n,d} \bJ_r$
 is canonically isomorphic to ${\mbf S}_{\nn,d}$ (recall $\nn =n-1$), we shall
 simply identify ${\mbf S}_{\nn,d} \equiv \bJ_r {\mbf S}_{n,d} \bJ_r$ below. Let
\begin{align*}
\Eji_{i} & =
\begin{cases}
\bJ_r \E_i \bJ_r, & \mbox{if} \  i \in [0, r-1],\\
\bJ_r \E_{r+1} \E_{r} \bJ_r, &\mbox{if} \ i = r, \\
\bJ_r \E_{i+1} \bJ_r, & \mbox{if} \ i \in [r+1,  \nn -1 ].
\end{cases}
\\
\Fji_{i} & =
\begin{cases}
\bJ_r \F_i \bJ_r, & \mbox{if} \  i \in [0, r-1],\\
\bJ_r \F_r \F_{r+1} \bJ_r, &\mbox{if} \ i = r, \\
\bJ_r \F_{i+1} \bJ_r, & \mbox{if} \ i \in [r+1, \nn -1].
\end{cases}
\\
\Kji^{\pm 1}_{i} & =
\begin{cases}
\bJ_r \bK^{\pm 1}_i \bJ_r,  & \mbox{if} \  i\in [0, r-1],\\
\bJ_r \bK^{\pm 1}_r \bK^{\pm 1}_{r+1} \bJ_r, & \mbox{if} \ i = r, \\
\bJ_r \bK^{\pm 1}_{i+1} \bJ_r , & \mbox{if} \ i \in [r+1, \nn -1].
\end{cases}
\\
\Hji^{\pm 1}_{a}  & =
\begin{cases}
\bJ_r \bH^{\pm 1}_a \bJ_r, & \mbox{if} \  a \in  [0, r],\\
\bJ_r \bH^{\pm 1}_{a+1} \bJ_r, & \mbox{if} \ a \in [r+1, \nn].
\end{cases}
\end{align*}
For convenience, one can extend the range of index $i$ from the  interval $[0, \nn-1]$ to $\mbb Z$ by setting $\Eji_i=\Eji_{i + \nn}$ for all $i\in \mbb Z$, etc.
We shall identify ${\mbf U}_{\nn,d}$ with the subalgebra of ${\mbf S}_{\nn,d}$ generated by $\Eji_i$, $\Fji_i$ and $\Kji^{\pm 1}_i$ for all $i\in [0, \nn -1]$.

\section{The comultiplication}

Recall the algebra homomorphism $\widetilde \Delta^{\fc} : {\mbf S}_{n,d}^{\fc} \to  {\mbf S}_{n,d'}^{\fc} \otimes  {\mbf S}_{n,d''}$ from (\ref{tDj}),  
for  $d', d''$ such that $d=d'+d''$.
We shall show its restriction to the subalgebra ${\mbf U}^{\ji}_{\nn,d}$ (denoted by the same notation) relates to the constructions above in Section~\ref{ji-version}. 

\begin{lem}
\label{tDj-ji}
We have an algebra homomorphism $\widetilde \Delta^{\fc} : {\mbf U}^{\ji}_{\nn,d} \longrightarrow {\mbf U}^{\ji}_{\nn,d'} \otimes {\mbf U}_{\nn,d''} $.
More explicitly, for $i\in [0, r-1]$, we have
\begin{align*}
\begin{split}
\widetilde \Delta^{\fc} ( \eji_i)
&= \eji_i' \otimes \Hji_{i+1}'' \Hji''^{-1}_{\nn -1  -i} + \hji'^{-1}_{i+1} \otimes \Eji_i''  \Hji''^{-1}_{\nn -1 -i} +  \hji'_{i+1} \otimes \Fji''_{\nn -1 -i} \Hji''_{i+1}.  \\
\widetilde \Delta^{\fc} (\fji_i)
 & = \fji'_i \otimes \Hji''^{-1}_{i} \Hji''_{\nn -i} + \hji'_i \otimes \Fji''_i \Hji''_{\nn -1 - i} + \hji'^{-1}_{i} \otimes \Eji''_{\nn -1 - i} \Hji''^{-1}_{i}. \\
\widetilde \Delta^{\fc} (\kji_i) & = \kji'_i \otimes \Kji''_i \Kji''^{-1}_{\nn -1 -i}.\\
\widetilde \Delta^{\fc} (\tji_{r}) & = \tji'_r \otimes \Kji''_r + v^2 \kji'^{-1}_r \otimes \Hji''_{r+1} \Fji''_r + v^{-2} \kji'_r \otimes \Hji''^{-1}_r \Eji''_r.
\end{split}
\end{align*}
\end{lem}

\begin{proof}
The fact $\widetilde \Delta^{\fc} ({\mbf U}^{\ji}_{\nn,d}) \subseteq {\mbf U}^{\ji}_{\nn,d'} \otimes {\mbf U}_{\nn,d''} $  
follows once we establish these explicit formulas. 

We observe that $\widetilde \Delta^{\fc} (\bj_r) = \bj'_r \otimes \bJ''_r$.
So,  by Proposition ~\ref{tDj-form},
\begin{align*}
\begin{split}
\widetilde \Delta^{\fc} ( \eji_i) =
\bj'_r \otimes \bJ''_r
\left (
\e_i' \otimes \mbf H_{i+1}'' \mbf H''^{-1}_{n -1  -i} + \mbf h'^{-1}_{i+1} \otimes \mbf E_i''  \mbf H''^{-1}_{n -1 -i}
+  \mathbf h'_{i+1} \otimes \mbf F''_{n -1 -i} \mbf H''_{i+1}
\right )
\bj'_r \otimes \bJ''_r \\
= \eji_i' \otimes \Hji_{i+1}'' \Hji''^{-1}_{\nn -1  -i} + \hji'^{-1}_{i+1} \otimes \Eji_i''  \Hji''^{-1}_{\nn -1 -i} +  \hji'_{i+1} \otimes \Fji''_{\nn -1 -i} \Hji''_{i+1}.
\end{split}
\end{align*}
The formulas for $\widetilde \Delta^{\fc} (\fji_i)$ and $\widetilde \Delta^{\fc} (\kji_i)$ are similarly proved.
The last formula can be proved in exactly the same manner as that of ~\cite[Lemma 5.1.1]{FL15}.
\end{proof}

Following the definition of $\phi^{\fc}_{d, d-n} $ in \eqref{transfer},
we define the transfer map
$$
\phi^{\ji}_{d, d-\nn} : {\mbf S}^{\ji}_{\nn,d} \longrightarrow {\mbf S}^{\ji}_{\nn,d-\nn}
$$
to be the composition
\[
\begin{CD}
\phi^{\ji}_{d, d-\nn} : {\mbf S}^{\ji}_{\nn,d}
@>\widetilde \Delta^{\fc}>>
{\mbf S}^{\ji}_{\nn,d-\nn} \otimes {\mbf S}_{\nn, \nn}
@>1\otimes \chi_{\nn} >>
{\mbf S}^{\ji}_{\nn,d-\nn}
\end{CD}
\]
where the homomorphism $\chi_{\nn} : \mbf S_{\nn, \nn} \longrightarrow \mbb Q(v)$ is the generalized signed representation of $ {\mbf S}_{\nn, \nn}$.
We have
$\chi_{\nn} (\Eji_i)=\chi_{\nn} (\Fji_i)=0$ and $\chi_{\nn} (\Hji_a)=v$ for all $i\in [0, \nn -1]$, $a\in [0, \nn]$.
Thus by Lemma~\ref{tDj-ji}, we have for all $i\in [0, r-1]$, 
\begin{equation}
\phi^{\ji}_{d, d-\nn} (\eji_i) = \eji'_i, \;\;
\phi^{\ji}_{d, d-\nn} (\fji_i) = \fji'_i, \;\;
\phi^{\ji}_{d, d-\nn} (\kji_i) = \kji'_i, \;\;
 \phi^{\ji}_{d, d-\nn} (\tji_r) = \tji'_r.
\end{equation}

Recall $\Delta^{\fc}: {\mbf S}_{n,d}^{\fc} \longrightarrow  {\mbf S}_{n,d'}^{\fc} \otimes  {\mbf S}_{n,d''}$ from (\ref{Dj}). 
Let us consider the restriction $\Delta^{\fc} |_{{\mbf S}^{\ji}_{\nn,d}}$, which will be denoted by $\Dji$.

\begin{prop}
\label{Dc-ji}
We have an algebra homomorphism $\Dji: {\mbf S}^{\ji}_{\nn,d} \longrightarrow {\mbf S}^{\ji}_{\nn,d'} \otimes {\mbf S}_{\nn,d''}$,
and by restriction, a homomorphism $\Dji: \bU^{\ji}_{\nn,d} \longrightarrow \bU^{\ji}_{\nn,d'} \otimes \bU_{\nn,d''}$.
More explicitly, for all $i\in [0, r-1]$, we have
\begin{align}
\begin{split}
\Dji (\eji_i)
& = \eji'_i \otimes \Kji''_i + 1 \otimes  \Eji''_i
+ \kji'_i \otimes v^{\delta_{i,0}}   \Fji''_{\nn -1-i}  \Kji''_i, \\
\Dji (\fji_i)
& =  \fji'_i \otimes \Kji''_{\nn -1-i} + \kji'^{-1}_i \otimes  v^{-\delta_{i,0}} \Kji''_{\nn -1-i} \Fji''_i
+ 1\otimes  \Eji''_{\nn -1-i}, \\
\Dji (\kji_i) &= \kji'_i \otimes \Kji''_i \Kji''^{-1}_{\nn -1-i}, \\
\Dji(\tji_r) & = \tji'_r \otimes \Kji''_{r} + 1 \otimes \Eji''_r + 1\otimes v \Kji''_r \Fji''_r.
\end{split}
\end{align}
\end{prop}

\begin{proof}
Since $\Dji (\bj_r) = \bj'_r \otimes \bJ''_r$,  we see that $\Dji ({\mbf S}^{\ji}_{\nn,d} ) \subseteq {\mbf S}^{\ji}_{\nn,d'} \otimes {\mbf S}_{\nn,d''} $.

So it remains to establish the formulas.
The first three follow by $\Dji (\bj_r) = \bj'_r \otimes \bJ''_r$ and Proposition ~\ref{Dc}.
We now prove the last one on $\Dji(\tji_r)$.
The superscripts $'$ and $''$ are dropped for simplicity  for the rest of the proof.

By applying Proposition ~\ref{Dc} and using that $\bj_r \e_r \bj_r =0$ and  $\bj_r \f_r \bj_r =0$,  we have
\begin{align}
\label{Dji-1}
\begin{split}
&\Dji (\bj_r \f_r \e_r \bj_r ) = \bj_r \f_r \e_r \bj_r \otimes \bJ_r \bK_{r+1} \bK_r \bJ_r + \bj_r \bk^{-1}_r \bj_r \otimes \bJ_r \bK_{r+1} \F_r \E_r \bJ_r \\
& + 1 \otimes \bJ_r \E_{r+1} \E_r \bJ_r + 1 \otimes \bJ_r \bK_{r+1} \F_r \F_{r+1} \bK_r \bJ_r + \bj_r \bk_r \bj_r \otimes \bJ_r \E_{r+1} \F_{r+1} \bK_r \bJ_r.
\end{split}
\end{align}
By using the fact that
\[
\bJ_r \bK_r \bJ_r = \Hji^{-1}_r, \bJ_r \bK_{r+1} \bJ_r = \Hji_{r+1}, \bJ_r \E_r \F_r \bJ_r =0,  \bJ_r \F_{r+1} \E_{r+1} \bJ_r =0,
\]
we have
\begin{align*}
\begin{split}
\bJ_r \bK_{r+1} \F_r \E_r \bJ_r & = \Hji_{r+1} \bJ_r ( \E_r \F_r - \frac{\bK_r - \bK^{-1}_r}{v-v^{-1}}) \bJ_r = \Hji_{r+1} \frac{\Hji_r - \Hji^{-1}_r}{v - v^{-1}}, \\
\bJ_r \bK_{r+1} \F_r \F_{r+1} \bK_r \bJ_r & = \Hji_{r+1} \Fji_r \Hji^{-1}_r = v \Kji_r \Fji_r, \\
\bJ_r \E_{r+1} \F_{r+1} \bK_r \bJ_r & =  \bJ_r ( \F_{r+1} \E_{r+1} + \frac{\bK_{r+1} - \bK^{-1}_{r+1}}{v - v^{-1}} ) \bJ_r \Hji^{-1}_r
= \frac{\Hji_{r+1} - \Hji^{-1}_{r+1} }{v -v^{-1}} \Hji^{-1}_r.
\end{split}
\end{align*}
So we can  rewrite (\ref{Dji-1}) as follows:
\begin{align*}
\begin{split}
\Dji &  (\bj_r \f_r \e_r \bj_r )
 = \bj_r \f_r \e_r \bj_r \otimes \Kji_r + 1 \otimes \Eji_r + 1 \otimes v \Kji_r \Fji_r  \\
&\qquad\qquad + \bj_r \bk^{-1}_r \bj_r \otimes \Hji_{r+1} \frac{\Hji_r - \Hji^{-1}_r}{v - v^{-1}}
 + \bj_r \bk_r \bj_r \otimes \frac{\Hji_{r+1} - \Hji^{-1}_{r+1} }{v -v^{-1}} \Hji^{-1}_r \\
& =  \tji_r  \otimes \Kji_r + 1 \otimes \Eji_r + 1 \otimes v \Kji_r \Fji_r
+ \bj_r \bk^{-1}_r \bj_r \otimes \frac{\Hji_r \Hji_{r+1}}{v - v^{-1}} - \bj_r \bk_r \bj_r \otimes \frac{\Hji^{-1}_r \Hji^{-1}_{r+1}}{v - v^{-1}}.
 \end{split}
\end{align*}
Finally, we have
\begin{align*}
\Dji (\bj_r \frac{\bk_r - \bk^{-1}_r}{v - v^{-1}} \bj_r ) =
\bj_r \bk_r \bj_r \otimes \frac{\Hji^{-1}_r \Hji^{-1}_{r+1}}{v - v^{-1}} - \bj_r \bk^{-1}_r \bj_r \otimes \frac{\Hji_r \Hji_{r+1}}{v - v^{-1}}.
\end{align*}
The formula for $\Dji(\tji_r)$ follows by adding the above two equations.
\end{proof}

Now set $d'=0$. Since ${\mbf S}^{\ji}_{\nn, 0} = \mbb Q(v)$, we obtain an algebra homomorphism
\begin{align}
\label{iji-2}
\iji_{\nn,d}:= \Dji|_{d'=0}:  {\mbf S}^{\ji}_{\nn,d} \longrightarrow {\mbf S}_{\nn,d}, 
\end{align}
which is injective by Proposition ~\ref{j-inj}.
Moreover,
\[
\eji_i = \fji_i =0, \kji_i = v^{-\delta_{i, 0}}, \tji_r =1 \in {\mbf S}^{\ji}_{\nn, 0}, \quad \forall i\in [0, r-1].
\]
The following can now be read off from Proposition ~\ref{Dc-ji}, while the injectivity of $\iji_{\nn,d}$ follows from
a similar argument in \cite{FL15}. 

\begin{prop}
\label{iji}
We have an imbedding of algebras
$$
\iji_{\nn,d} :  {\mbf S}^{\ji}_{\nn,d} \longrightarrow {\mbf S}_{\nn,d}.
$$
Moreover, for all $i\in [0, r-1]$, we have 
\begin{align}
\begin{split}
\iji_{\nn,d} (\eji_i)  = \Eji_i + v^{-\delta_{i, 0}} \Kji_i \Fji_{\nn - 1 - i}, \quad
& \iji_{\nn,d} (\fji_i)  = \Eji_{\nn -1 -i} + v^{\delta_{i, 0}} \Fji_i \Kji_{\nn -1 -i}, \\
\iji_{\nn,d} (\kji_i)  = v^{-\delta_{i, 0} } \Kji_i \Kji^{-1}_{\nn -1 -i}, \quad
& \iji_{\nn,d} (\tji_r)  = \Eji_r + v \Kji_r \Fji_r + \Kji_r.
\end{split}
\end{align}
In particular, we have an imbedding of algebras $\iji_{\nn,d} :  \bU^{\ji}_{\nn,d} \rightarrow \bU_{\nn,d}.$
\end{prop}

\section{The monomial basis of ${\mbf U}^{\ji}_{\nn,d}$}
 \label{sec:coideal2-CB}

Next, we shall construct a $\ji$-monomial basis of ${\mbf U}^{\ji}_{\nn,d}$,  which is bar invariant and preserved by $\phi^{\ji}_{d, d-\nn}$.  
The compatibility of  a monomial basis with $\phi^{\ji}_{d, d-\nn}$ requires additional work in the current $\ji$ setting
than the previous $\fc$-case (compare Theorem ~\ref{CB-Udn}),
and this will be carried out by a similar procedure as in finite type $\imath$-version  in~\cite{LW15}.

Let $A$ be a matrix in $\MX^{\ji}_{\nn,d}$. 
Let $\mrm{dlt}_{r+1} (A)$ be the $\ZZ\times \ZZ$ matrix obtained from $A$ by deleting the $k$th rows and columns for all $k \equiv r+1$ mod $n$, 
so that  $\mrm{dlt}_{r+1} (A)$  and $A$ share the same $[-r,r] \times [-r,r]$-minors.
The resulting matrix $\mrm{dlt}_{r+1} (A) =(b_{ij})$ satisfies
\begin{align}
 \label{eq:Mji2}
b_{-i, - j}  = b_{ij} =b_{i + \nn, j + \nn}, \;
\sum_{i\in [1, \nn], j\in \mbb Z} b_{ij}  = 2d +1, \;
b_{00}  \in 2 \mbb Z + 1.
\end{align}
We shall denote by 
\begin{equation}
 \label{Mjicheck}
\check{\MX}^{\ji}_{\nn,d} = \{B=(b_{ij}) \in \text{Mat}_{\ZZ\times \ZZ}(\mbb{N}) \big \vert B \text{ satisfies } \eqref{eq:Mji2} \}. 
\end{equation}
In particular, we have a bijection
\begin{equation}
  \label{bijection:ji}
\mrm{dlt}_{r+1} : \MX^{\ji}_{\nn,d} \longrightarrow \check{\MX}^{\ji}_{\nn,d}, \quad
A\mapsto \mrm{dlt}_{r+1} (A).
\end{equation}

\begin{Def}
\label{def:ji-mon}
A matrix $A$ in $\MX^{\ji}_{\nn,d}$ is called {\em $\ji$-aperiodic} 
if $\mrm{dlt}_{r+1} (A)$ is aperiodic.
\end{Def}
Toward the construction of a suitable monomial basis, it is convenient for us 
to freely use parametrization of standard basis for ${\mbf S}^{\ji}_{\nn,d}$ by matrices in $\check{\MX}^{\ji}_{\nn,d}$ or ${\MX}^{\ji}_{\nn,d}$
under such a bijection,  and thus it makes sense to say things like  ``$[A] \in {\mbf S}^{\ji}_{\nn,d}$ for $A\in \check{\MX}^{\ji}_{\nn,d}$''.  
We shall add the index $\nn$ to the old notation to denote $E^{h, h+1}_{\nn}$, $E^{h, h+1}_{\theta, \nn}=E^{h, h+1}_{\nn} + E^{-h, -(h+1)}_{\nn}$ 
corresponding to $E^{h, h+1}$, $E^{h, h+1}_{\theta}$, and so on, under the bijection.
(Note that the former has period $\nn$, while the latter has period $n$.) 

\begin{lem} 
\label{ji-estimate}
Let $A, B, C \in \check{\MX}^{\ji}_{d}$. Let $R$ be a positive integer.
\begin{enumerate}
\item
Assume that  $B-R E_{\theta, \nn}^{h, h+1}$ is diagonal for some $ h \in [0, r]$ and $\co(B) = \ro(A)$.
Assume further that $R=R_0 + \cdots + R_l$ and the matrix $A$ satisfy  one of  the following conditions:
\begin{align*}
\begin{cases}
a_{0 m} =0 ,  \ a_{1, k+i} = R_i,\ a_{1k} \geq R_0,\  a_{1j}  =0, &   \mbox{if} \ h =0,\ k \geq 1;\\
a_{h m}=0, \    a_{ h +1,k+i}= R_i, \ a_{h+1,k} \geq R_0 ,\  a_{h +1,j}=0,&  \mbox{if} \ h\in [1,r-1];\\
 a_{r m}=0, \   a_{r + 1,k+i}= R_i, \ a_{r + 1,k} \geq  2 R_0,\ a_{r + 1, j}=0,&  \mbox{if} \  h =r,\ k \geq   r+1;\\
\end{cases}
\end{align*}
for all $m \geq k$, $i\in [1,l]$ and $j > k+l$.
Then  we have
\[
[B] * [A] = [ A+ \sum_{i=0}^l R_i(E^{h, k+i}_{\theta, \nn}- E^{h +1, k+i}_{\theta, \nn})] + \mbox{lower terms}.
\]

\item 
Assume  that $C-R E_{\theta, \nn}^{h+1,h} $ is diagonal for some $h\in [0, r-1]$ and $\co(C ) =\ro(A)$.
Assume further  that $R= R_0 + \cdots + R_l$ and   $A$ satisfy one of the following conditions:
\begin{align*}
\begin{cases}
 a_{1m} =0,  \ a_{0,k+i} = R_i,\ a_{0,k+l} \geq  R_l,\ a_{0j} = 0,  &    \mbox{if} \ h =0,\ k + l < 0;\\
 a_{1m} =0, \  a_{0,k+i} = R_i,\ a_{00} \geq 2R_l  ,\ a_{0j} = 0, &  \mbox{if} \ h =0, k +l = 0;\\
 a_{h+1,m} =0, \ a_{h,k+i}= R_i,\ a_{h,k+l} \geq R_l, \ a_{hj} =0,  & \mbox{if} \ h\in [1,r-1];\\
\end{cases}
\end{align*}
for all $m \leq k+l$, $ i\in [0,l-1]$ and $j < k$.
Then we have
\[
[C] * [A] = [ A-\sum_{i=0}^{l} R_i(E^{h, k+i}_{\theta, \nn}- E^{h+1,k+i}_{\theta, \nn}) ] + \mbox{lower terms}.
\]
\end{enumerate}
\end{lem}
Note the above multiplication formula for $h=r$ corresponds to multiplication with the new generator $\tji_r$ in ${\mbf S}^{\ji}_{\nn,d}$.

\begin{proof}
All cases are directly taken from Lemma~\ref{BKLW3.9-b}, except the third case in (1), 
which can be obtained by applying Lemma~\ref{BKLW3.9-b}(1), Cases 3-4, and Lemma~\ref{BKLW3.9-b}(2), Case 3.
\end{proof}

A {\em $\ji$-aperiodic monomial} is by definition of the form $[X_1]* \cdots * [X_m]$ in ${\mbf S}^{\ji}_{\nn,d}$ where  $X_i \in \check{\MX}^{\ji}_{\nn,d}$
satisfies the conditions that either $X_i - R E^{h, h+1}_{\theta, \nn}$ for $h \in [0, r]$ 
or $X_i - R E^{h+1, h}_{\theta, \nn}$ for $h \in [0, r-1]$ is diagonal for each $i$. 
The same argument as for Proposition~\ref{Astandard-basis} (or Theorem~\ref{CB-Udn}) gives us the following.

\begin{prop}
 \label{prop:MB:ji}
For each aperiodic matrix $A$ in $\check{\MX}^{\ji}_{\nn,d}$, there exists a $\ji$-aperiodic monomial $y_A$ in ${\mbf S}^{\ji}_{\nn,d}$ such that
$
y_A = [A] + \mbox{lower terms}.
$
\end{prop}
We freely switch the index set for $\{y_A\}$ back to $A\in {\MX}^{\ji}_{\nn,d}$ under the bijection \eqref{bijection:ji}. 
By Proposition~\ref{iji}, $\iji_{\nn, d}: {\mbf U}^{\ji}_{\nn,d} \to {\mbf U}_{\nn,d}$ is an imbedding, and we shall regard ${\mbf U}^{\ji}_{\nn,d} \subseteq {\mbf U}_{\nn,d}$
by identifying ${\mbf U}^{\ji}_{\nn,d}$ with its image under $\iji_{\nn, d}$. Recall that ${\mbf U}_{\nn,d}$ admits a canonical basis
 $\{\{A\}_d  \big \vert A\in \MX_{\nn,d} \; \ji\text{-aperiodic} \}$.
 
 We have the following analogue of Theorem~\ref{CB-Udn}. 

\begin{prop}
\label{prop:ji-aperiodic}
The set $\{\{A\}_d \big \vert A\in \MX^{\ji}_{\nn,d} \; \ji\text{-aperiodic} \}$ forms a (canonical) basis for ${\mbf U}^{\ji}_{\nn,d}$.
Also, $\{ y_A \big \vert A\in {\MX}^{\ji}_{\nn,d} \; \ji\text{-aperiodic} \}$ forms a monomial basis for ${\mbf U}^{\ji}_{\nn,d}$.
\end{prop}

\begin{proof}
We have an imbedding $\iji_{\nn, d}: {\mbf U}^{\ji}_{\nn,d} \to {\mbf U}_{\nn,d}$ by Proposition~\ref{iji}.
A counterpart of Proposition ~\ref{monomial-aperiodic} makes sense in our setting.
We also have Proposition~\ref{prop:MB:ji}. 
Therefore we have all the three key ingredients available to rerun the argument for Theorem~\ref{CB-Udn}. The proposition is proved. 
\end{proof}
 
 Note that $y_A$ is not bar invariant in general.
As in the finite $\imath$-setting \cite{LW15}, this monomial basis $\{y_A\}$ is   not preserved by the transfer map $\phi_{\nn, \nn-d}^{\ji}$, 
and thus this basis is not directly applicable for studying the limiting algebra $\bU_{\nn}^{\ji}$ in the following Section~\ref{sec:coideal2}. 
To overcome this obstacle, we introduce the $hybrid$ $monomial$  ${h}_A$ obtained from $y_A$ by replacing every factor $[X_i]$ in the monomial $y_A$
by its associated canonical basis element $\{X_i\}_d $ if $X_i$ is of the form 
$X_i = X(D, R):= D + R E^{r, r+1}_{\theta, \nn}$ for some diagonal matrix $D$ and for some $R$.
We still have $\{ X_i\}_d \in {\mbf S}^{\ji}_{\nn,d}$ thanks to the fact that $\{X_i\}_d \in [X_i] +\sum_{0\le k<R} \cA X (D_k, k)$ for some diagonal matrices $D_k$;
see \cite{LW15}. 
Hence
we have 
${h}_A = y_A + \mbox{lower terms}  \in {\mbf S}^{\ji}_{\nn,d}. 
$
Thus we have obtained the following. 

\begin{prop}
  \label{prop:hMB:ji}
For each aperiodic matrix $A$ in $\check{\MX}^{\ji}_{\nn,d}$, 
there exists a $\ji$-aperiodic hybrid monomial ${h}_A$ in ${\mbf S}^{\ji}_{\nn,d}$ such that
$
{h}_A = [A] + \mbox{lower terms}, \,
\overline{{h}_A} = {h}_A,
$ and $
\phi^{\ji}_{d, d-\nn} ({h}_A) = {h}_{A-  2I_{\nn}}$, with $I_{\nn} = \sum_{1\leq i\leq \nn} E^{ii}_{\nn}$.
Moreover,  $\{ {h}_A \big \vert A\in \check{\MX}^{\ji}_{\nn,d} \; \ji\text{-aperiodic} \}$ forms a (hybrid) monomial basis for ${\mbf U}^{\ji}_{\nn,d}$.
\end{prop}
(It is understood above that ${h}_{A-  2I_{\nn}}=0$ if ${A-  2I_{\nn}}$ contains some negative entry.)

\begin{example}
Set $r=2$ and so $\nn=5$. Consider the  $\ji$-aperiodic matrix $A\in \check{\MX}^{\ji}_{5, d}$:

\[A =
\begin{tabular}{   c | c | c | c | c| c| c| c| c | c | c | c} 

 & c-3 & c-2 & c-1  & c0 & c1 & c2 & c3 & c4 & c5 & c6 & c7  \\
 \hline 
 & & && && && && & \\
\hline 
 r0 & & 3& 1 & * & 1 & 3& && &&  \\
 \hline
 r1 & &  & 0 & 6 & * & 0 & 4 & && & \\
\hline 
r2 & & && 8 & 7 & * & 2 & 5 && & \\
\hline 
r3 & & && & 5 & 2 &*  & 7  & 8 & & \\
\hline 
r4 & & && &&4 &0 &*  & 6 & 0 & \\
\hline 
r5 & & && && &3 & 1&* & 1 & 3\\
\hline 
& & && && && && & \\
\end{tabular}
\]
where `ri' and  `cj' in the table  indicate the $i$-th row and $j$-th column of the matrix $A$, respectively.
We have 
\begin{align*}
\begin{split}
y_A & = \eji_1^{(8)} * \tji_2^{[5]} * \fji_1^{(4)} * \fji^{(4)}_0 * \eji^{(14)}_0 * \eji^{(12)}_1 * \tji_2^{[6]}  * \fji^{(3)}_0  * 1_{\co(A)}
=[A] + \mbox{lower terms}, \\
h_A & = \eji_1^{(8)} * \tji_2^{\{5\}} * \fji_1^{(4)} * \fji^{(4)}_0 * \eji^{(14)}_0 * \eji^{(12)}_1 * \tji_2^{\{6\}}  * \fji^{(3)}_0  * 1_{\co(A)}
=[A] + \mbox{lower terms}, \\
\end{split}
\end{align*}
where $\tji_2^{[R]} $  and $\tji_2^{\{R\}}$ denote $\sum_{X} [X]$ and $\sum_X \{X\}$, respectively, 
with the sum  taken over $X$ such that $X- R E^{r, r+1}_{\theta, \nn}$ is diagonal.
\end{example}

\section{The coideal subalgebra of type ${\ji}$}
 \label{sec:coideal2}

Now that the results at the $\ji$-Schur algebra level are established (which is the counterpart of Chapter~\ref{chap:schur}), 
we will formulate the $\ji$-analogue of Chapter~\ref{chap:coideal}. As most of these are straightforward, we will skip some of the details.

Starting with the projective system $\{({\mbf U}^{\ji}_{\nn,d}, \phi^{\ji}_{d, d- \nn} )\}_{d\in \mbb{N}}$, we construct
two distinguished algebras  ${\mbf U}^{\ji}_{\nn}$ and $\dot{{\mbf U}}^{\ji}_{\nn}$ out of the associated projective limit algebra ${\mbf U}^{\ji}_{\nn, \infty}$;
the Chevalley generators of  ${\mbf U}^{\ji}_{\nn}$ will be denoted again  by 
$\eji_i$, $\fji_i$, $\kji^{\pm 1}_i$ $(i\in [0, r-1])$, and $\tji_r$. 
The family of imbeddings $\{\iji_{\nn,d}: {\mbf U}^{\ji}_{\nn,d} \to \bU_{\nn,d} \}_{d\ge 0}$ induces  an algebra imbedding 
$\iji_{\nn}: {\mbf U}^{\ji}_{\nn} \to \bU_{\nn}$. 
The family of $\Dji$ (for various $d'+d'' =d$) induce an algebra homomorphism (which is coassociative in a suitable sense)
$\Dji : {\mbf U}^{\ji}_{\nn} \rightarrow {\mbf U}^{\ji}_{\nn} \otimes \bU_{\nn},
$
whose action on the Chevalley generators can be presented explicitly. Recall the algebra isomorphism $\bU_{\nn} \cong \bU (\slh_{\nn})$.
Summarizing we have established the following.

\begin{thm}
  \label{thm:QSP2}
The pair $(\bU (\slh_{\nn}), {\mbf U}^{\ji}_{\nn})$ forms an affine quantum symmetric pair. 
$($see Figure~\ref{figure:ji} for the relevant involution.$)$
\end{thm}


Recall the Cartan integers $\texttt{c}_{ij}$ from \eqref{aij}. 
We give a presentation for the algebra ${\mbf U}_{\nn}^{\ji}$, which is a counterpart of Proposition~\ref{present:Ujj} for $\bU_n^\C$. 
This presentation  is a variant of \cite[Theorem 7.1]{Ko14} in our setting and our notation. 
Recall we always assume $r\geq 1$ so $\nn \ge 3$. 
\begin{prop}
  \label{present:Uji}
The $\mbb Q(v)$-algebra ${\mbf U}_{\nn}^{\ji}$ has a presentation with generators
$\eji_i, \fji_i$, and $\kji^{\pm 1}_i$ $(i\in [0, r-1])$ and $\tji_r$, and the following relations\footnote{Remove the relation $ \kji_0  (\kji^2_1 \cdots \kji^2_{r-1} )   = v^{-1},$}: for all $i, j \in [0, r-1]$,
\begin{align*}
 \kji_0 & (\kji^2_1 \cdots \kji^2_{r-1} )   = v^{-1}, 
\\
\kji_i \kji_i^{-1} & = 1, \quad
\kji_i \kji_j   = \kji_j \kji_i, \quad 
\kji \tji_r = \tji \kji_r,  \\
\kji_i \eji_j \kji_i^{-1}& =  v^{\texttt{c}_{ij} + \delta_{i,0} \delta_{j,0}}  \eji_j , \\
\kji_i \fji_j \kji_i^{-1} & =  v^{-\texttt{c}_{ij} - \delta_{i,0} \delta_{j,0} }  \fji_j , \\
\eji_i \eji_j & = \eji_j \eji_i , \quad 
\fji_i \fji_j  = \fji_j \fji_i , \quad \forall |i-j|> 1, \\
\eji_i \tji_r & = \tji_r \eji_i, \quad  
\fji_i \tji_r  = \tji_r \fji_i, \quad \forall i \leq r-2, 
\end{align*}
\begin{align*}
\eji_i^2 \eji_j  + \eji_j \eji_i^2 &= (v + v^{-1}) \eji_i \eji_j \eji_i, \quad \forall |i-j|=1,\\
\fji_i^2 \fji_j  + \fji_j \fji_i^2 &= (v + v^{-1}) \fji_i \fji_j \fji_i, \quad \forall |i-j|=1,\\
\eji^2_{r-1} \tji_r + \tji_r \eji^2_{r-1} & = (v + v^{-1}) \eji_{r-1} \tji_r \eji_{r-1}, \\
\fji^2_{r-1} \tji_r + \tji_r \fji^2_{r-1} & = (v + v^{-1}) \fji_{r-1} \tji_r \fji_{r-1}, \\
\tji_r^2 \eji_{r-1} + \eji_{r-1} \tji_r^2 & = (v + v^{-1}) \eji_{r-1} \tji_r \eji_{r-1} + \eji_{r-1}, \\
\tji_r^2 \fji_{r-1} + \fji_{r-1} \tji_r^2 & = (v + v^{-1}) \fji_{r-1} \tji_r \fji_{r-1} + \fji_{r-1}, \\
\eji_i \fji_j - \fji_j \eji_i & = \delta_{ij} \frac{\kji_i - \kji_i^{-1}}{v - v^{-1}} , \quad \forall (i, j) \neq (0,0),\\
\eji_0^2 \fji_0 + \fji_0 \eji_0^2 
&= (v + v^{-1}) \big( \eji_0 \fji_0 \eji_0 - (v \kji_0 + v^{-1} \kji_0^{-1}) \eji_0 \big),\\
\fji_0^2 \eji_0   + \eji_0 \fji_0^2 
& = (v + v^{-1}) \big ( \fji_0 \eji_0 \fji_0 - \fji_0 (v \kji_0 + v^{-1} \kji_0^{-1}) \big).\\
\end{align*}
\end{prop}

\begin{proof}
We verify directly the above relations for Lusztig algebras $\bU_{\nn,d}^{\ji}$, and it follows that the relations hold for ${\mbf U}_{\nn}^{\ji}$
by construction. Then we use Theorem~\ref{thm:QSP2} and \cite[Theorem~ 7.1]{Ko14} to conclude that we do not need additional relations. 
\end{proof}

Now the construction of canonical basis with positivity for the coideal algebra in Section~\ref{sec:CB-Unc} can be repeated.
Recalling $\widetilde{\MX}_{n}$ from \eqref{eq:Mtilde},
we introduce the following subsets of $\widetilde{\MX}_{n}$:
\begin{align}
  \label{eq:Mtildeji}
  \begin{split}
  \widetilde{\MX}_{\nn}^{\ji}
  &= \big\{A=(a_{ij}) \in \widetilde{\MX}_{n} \big \vert \, a_{r+1, j} = \delta_{r+1, j} ,  a_{i, r+1} = \delta_{i, r+1} \big\},
 \\
\widetilde{\MX}_{\nn}^{\ji,ap}
&= \{A \in \widetilde{\MX}_{\nn}^{\ji} \big \vert A \text{ is $\ji$-aperiodic} \}.
\end{split}
\end{align}
Recalling $\widetilde{\MX}_{n,d}$ from \eqref{tXid}, we further introduce, for $d\in \ZZ$, 
\begin{align}
  \label{tXid:ji}
\widetilde{\MX}_{\nn,d}^{\ji}
&= \{A \in \widetilde{\MX}_{\nn}^{\ji} \big \vert |A|=d \},
\qquad
\widetilde{\MX}_{\nn,d}^{\ji} = \bigcup_d \widetilde{\MX}_{\nn,d}^{\ji}.
\end{align}
We define an equivalence relation $\approx$ on $\widetilde{\MX}_{\nn}^{\ji,ap}$ as in \eqref{eq:approx} and let $\widehat A$ be  the equivalence class of $A$.
The hybrid monomial basis $\{{h}_A\}$ for ${\mbf S}^{\ji}_{\nn,d}$ (cf. Proposition~\ref{prop:hMB:ji}) gives rise to a monomial basis 
$\{{h}_{\widehat A} \big \vert \widehat A \in \widetilde{\MX}_{\nn}^{\ji,\ap} /\!\approx \}$ for the algebra $\dot {\mbf U}^{\ji}_{\nn}$.
A bilinear form $\langle \cdot, \cdot \rangle$ on $\dot{{\mbf U}}^{\ji}_{\nn}$ can be defined similarly as in Section~\ref{sec:CB-Unc} and shown to be non-degenerate.
The following is a $\ji$-analogue of Theorems~\ref{thm:iCB-Unc} and \ref{thm:positivity-Unc}.
\begin{thm}
 \label{thm:ji-positive}
There exists a canonical basis $\dot{\mbf B}^{\ji}_{\nn}  =\{b_{\widehat A} \big \vert  \widehat A \in \widetilde{\MX}_{\nn}^{\ji,ap}/\approx\}$  
for $\dot {\mbf U}^{\ji}_{\nn}$, whose transition matrix with respect to the monomial basis 
$\{{h}_{\widehat A} \big \vert \widehat A \in \widetilde{\MX}_{\nn}^{\ji,ap} /\!\approx \}$ is uni-triangular.
Moreover, the structure constants of the canonical basis $\dot{\mbf B}^{\ji}_{\nn}$ lie in ${\mathbb N}[v,v^{-1}]$
with respect to the multiplication and comultiplication, and  in $v^{-1} \mbb N[[v^{-1}]]$ with respect to 
the bilinear pairing.
\end{thm}

\chapter{More variants of coideal subalgebras of quantum affine   $\mathfrak{sl}_n$}
  \label{chap:coideal34}
  
  This chapter offers two more variants of geometric origin (denoted by types $\ijw$ and $\ii$) 
  of the constructions in Chapters~\ref{chap:coideal} and \ref{chap:coideal2}.
  Set 
  \[
  \mm =\nn-1 =n-2 =2r \; (r\geq 1).
  \] 
Schur algebras ${\mbf S}^{\ij}_{\nn,d}$ and Lusztig  algebras ${\mbf U}^{\ij}_{\nn,d}$ are introduced,
and the family of Lusztig algebras gives rise to algebras ${\mbf U}^{\ij}_{\nn}$ and $\dot{\mbf U}^{\ij}_{\nn}$. 
We show that $(\bU(\slh_{\nn}),  {\mbf U}^{\ij}_{\nn})$ forms an affine quantum symmetric pair.  
In addition, a family of algebras ${\mbf U}^{\ii}_{\mm,d} \subset {\mbf S}^{\ii}_{\mm,d}$ is introduced and gives rise to algebras 
  ${\mbf U}^{\ii}_{\mm}$ and $\dot{\mbf U}^{\ii}_{\mm}$. 
  Then $(\bU(\slh_{\mm}),  {\mbf U}^{\ii}_{\mm})$ forms an affine quantum symmetric pair.  
The canonical bases of both algebras $\dot{\mbf U}^{\ij}_{\nn}$ and $\dot{\mbf U}^{\ii}_{\mm}$  admit positivity with respect to
multiplication, comultiplication, and a bilinear pairing.

\section{The Schur algebras of type $\ijw$}

Recall the set $\MX_{n,d}$ from \eqref{Mdn}.
We set
\begin{equation}
  \label{Mijd}
\MX^{\ij}_{\nn,d} = \{ A \in \MX_{n,d} | a_{0, j} = \delta_{0, j}, a_{i, 0} = \delta_{i, 0}, \forall i, j \in \mbb Z\}.
\end{equation}
Introduce the idempotent $\bj_0$ in the algebra $\Sj$ given by 
$
\bj_0 = \sum_{A\in \MX^{\ij}_{\nn,d} \ \text{diagonal}}  [A],
$
and form the following subalgebra of $\Sj$ (called Schur algebra of type $\ijw$):
\begin{equation} 
  \label{Sij}
{\mbf S}^{\ij}_{\nn,d} = \bj_0 \Sj \bj_0.
\end{equation}
We further introduce the following elements  in ${\mbf S}^{\ij}_{\nn,d}$:
\begin{align}
\begin{split}
\eij_i & = \bj_0 \e_i \bj_0, \quad 
\fij_i  = \bj_0 \f_i \bj_0, \\
\kij^{\pm 1}_i & = \bj_0 \bk^{\pm 1}_i \bj_0, \quad \forall i\in [1, r], \\
\hij^{\pm 1}_a & = \bj_0 \bh^{\pm 1}_a \bj_0, \quad \forall a \in [0, r], \\
\tij_0 & = \bj_0 \Big ( \e_0 \f_0 + \frac{ \bk^{-1}_0 - \bk_0}{v - v^{-1}} \Big ) \bj_0.
\end{split}
\end{align}

We also have the following vanishing results in ${\mbf S}^{\ij}_{\nn,d}$, which will be used freely:
\[
\bj_0 \e_0 \bj_0 =0, \quad
\bj_0 \f_0 \bj_0 = 0, \quad
\bj_0 \f_0 \e_0 \bj_0 =0.
\]
The Lusztig algebra ${\mbf U}^{\ij}_{\nn,d}$ is defined to be the subalgebra of ${\mbf S}^{\ij}_{\nn,d}$ generated by 
the Chevalley generators $\eij_i$, $\fij_i$, $\kij^{\pm 1}_i$, for  $i\in [1, r]$, and $\tij_0$.

Let us also formulate a type $A$ version which is compatible with the above construction.
Let 
\begin{equation}
\Theta^{\ij}_{\nn,d} = \{ A \in \Theta_{n,d} \vert a_{0, j}=0, a_{i, 0} =0, \quad \forall i, j \in \mbb Z\}.
\end{equation}
Using the  idempotent $\bJ_0$ in ${\mbf S}_{n,d}$ given by 
$\bJ_0 = \sum_{A \in \Theta^{\ij}_{\nn,d} \ \text{diagonal}} [A],$
we form the subalgebra $\bJ_0 {\mbf S}_{n,d} \bJ_0$  of ${\mbf S}_{n,d}$,
which is isomorphic to the algebra ${\mbf S}_{\nn,d}$ defined earlier.
We shall always identify ${\mbf S}_{\nn,d} \equiv \bJ_0 {\mbf S}_{n,d} \bJ_0$ below.
We introduce the following elements in ${\mbf S}_{\nn,d}$: 
\begin{align}
\begin{split}
\Eij_i & =
\begin{cases}
\bJ_0 \E_0 \E_{-1} \bJ_0, & \mbox{if} \ i=0, \\
\bJ_0 \E_i \bJ_0, & \mbox{if} \ i \in [1, \nn -1].
\end{cases} \\
\Fij_i & =
\begin{cases}
\bJ_0 \F_{-1} \F_{0} \bJ_0, & \mbox{if} \ i=0, \\
\bJ_0 \F_i \bJ_0, & \mbox{if} \ i \in [1, \nn -1].
\end{cases} \\
\Kij^{\pm 1}_i & =
\begin{cases}
\bJ_0 \bK^{\pm 1}_0 \bK^{\pm 1}_{-1} \bJ_0, & \mbox{if} \ i=0, \\
\bJ_0 \bK^{\pm 1}_i \bJ_0, & \mbox{if} \ i \in [1, \nn -1].
\end{cases}\\
\Hij^{\pm 1}_a & =
\begin{cases}
\bJ_0 \bH^{\pm 1}_{-1} \bJ_0, & \mbox{if} \ a=0, \\
\bJ_0 \bH^{\pm 1}_i \bJ_0, & \mbox{if} \ a \in [1, \nn ].
\end{cases}
\end{split}
\end{align}
We can extend the interval $i\in [0, \nn -1]$ to $i \in \mbb Z$ by setting $\Eij_i = \Eij_{i+\nn}$, etc.
We observe
\[
\bJ_0 \E_{-1} \bJ_0 =0, \quad
\bJ_0 \F_{-1} \bJ_0 =0, \quad
\bJ_0 \E_{-1} \F_{-1} \bJ_0 =0, \quad
\bJ_0 \F_0 \E_0 \bJ_0 =0.
\]
We identify ${\mbf U}_{\nn,d}$ with the subalgebra generated by the Chevalley generators $\Eij_i$, $\Fij_i$ and $\Kij^{\pm 1}_i$ for all $i\in [0, \nn -1]$.

\section{Comultiplication and transfer map of type $\ijw$} 

We shall study the restriction to Lusztig algebra ${\mbf U}_{\nn,d}^{\ij}$ (denoted by the same notation) of
$\widetilde \Delta^{\fc} : {\mbf S}_{n,d}^{\fc} \to  {\mbf S}_{n,d'}^{\fc} \otimes  {\mbf S}_{n,d''}$ from (\ref{tDj}),  for arbitrary $d', d''$ such that $d=d'+d''$.

\begin{prop}
\label{tDj-ij}
We have an algebra homomorphism 
$$
\widetilde \Delta^{\fc} : {\mbf U}^{\ij}_{\nn,d} \longrightarrow {\mbf U}^{\ij}_{\nn,d'} \otimes {\mbf U}_{\nn,d''}.
$$
More explicitly, for all $i\in [1, r]$, we have
\begin{align*}
\begin{split}
\widetilde \Delta^{\fc} ( \eij_i) & = \eij'_i \otimes \Hij''_{i+1} \Hij''^{-1}_{n-1-i} + \hij'^{-1}_{i+1} \otimes \Eij''_i \Hij''^{-1}_{n-1-i}
+ \hij'_{i+1} \otimes \Fij''_{n-1-i} \Hij''_{i+1}, \\
\widetilde \Delta^{\fc} ( \fij_i)  & = \fij'_i \otimes \Hij''^{-1}_i \Hij''_{n-i} + \hij'_i \otimes \Fij''_i \Hij''_{n-i}  + \hij'^{-1}_i \otimes \Eij''_{n-1-i} \Hij''^{-1}_i, \\
\widetilde \Delta^{\fc} ( \kij_i) & = \kij'_i \otimes \Kij''_i \Kij''^{-1}_{n-1-i}, \\
\widetilde \Delta^{\fc} ( \tij_0)  & = \tij'_0 \otimes \Kij''_0 + v^2 \kij'_0 \otimes \Hij''_1 \Fij''_0 + v^{-2} \kij'^{-1}_0 \otimes \Hij''^{-1}_0 \Eij''_0.
\end{split}
\end{align*}
\end{prop}

\begin{proof}
The inclusion $\widetilde \Delta^{\fc}  ( {\mbf U}^{\ij}_{\nn,d} ) \subseteq {\mbf U}^{\ij}_{\nn,d'} \otimes {\mbf U}_{\nn,d''} $ follows once the formulas are established. 
The superscripts $'$ and $''$ will be dropped throughout the proof for the sake of simplicity.
The first three formulas follow from Proposition ~\ref{tDj-form}.
To prove the last one, we proceed similarly as in the $\ji$-version.
By using $\bj_0 \e_0 \bj_0 =0$ and $\bj_0 \f_0 \bj_0=0$, we have
\begin{align*}
\begin{split}
\widetilde \Delta^{\fc}  (\bj_0 \e_0 \f_0 \bj_0 )
& = \bj_0 \otimes \bJ_0 \widetilde \Delta^{\fc} (\e_0) \widetilde \Delta^{\fc} (\f_0) \bj_0 \otimes \bJ_0 \\
& =  \bj_0 \e_0 \f_0 \bj_0 \otimes \bJ_0 \bH_1 \bH^{-1}_{-1} \bJ_0
+ \bj_0 \bh^{-1}_1 \bh_0 \bj_0 \otimes \bJ_0 \E_0 \bH^{-1}_{-1} \F_0 \bH_0 \bJ_0   \\
& + \bj_0 \bh_1 \bh_0 \bj_0 \otimes \bJ_0 \F_{-1} \bH_1 \F_0 \bH_0 \bJ_0
+ \bj_0 \bh^{-1}_1 \bh^{-1}_0 \bj_0 \otimes \bJ_0 \E_0 \bH^{-1}_{-1} \E_{-1} \bH^{-1}_0 \bJ_0 \\
& + \bj_0 \bh_1 \bh^{-1}_0 \bj_0 \otimes \bJ_0 \F_{-1} \bH_1 \E_{-1} \bH^{-1}_0 \bJ_0.
\end{split}
\end{align*}
By using $\bJ_0 \bH_0 \bJ_0 =1$, $\bJ_0 \bK_0 \bJ_0 = \Hij_1$, and $\bJ_0 \F_0 \E_0 \bJ_0 =0$, we have
\begin{align*}
\begin{split}
\bJ_0 \E_0 \bH^{-1}_{-1} \F_0 \bH_0 \bJ_0
& =
\bJ_0 \E_0 \bH^{-1}_{-1} \F_0 \bJ_0
= \bJ_0 \bH^{-1} \E_0 \F_0 \bJ_0 \\
& =\Hij^{-1}_0 \bJ_0 \E_0 \F_0 \bJ_0 = \Hij^{-1}_0 \bJ_0 \frac{\bK_0 - \bK^{-1}_0}{v - v^{-1}} \bJ_0
= \Hij^{-1}_0  \frac{\Hij_1 - \Hij^{-1}_1}{v - v^{-1}}.
\end{split}
\end{align*}
We also have
\begin{align*}
 \bJ_0 \F_{-1} \bH_1 \F_0 \bH_0 \bJ_0 & = \bJ_0 \F_{-1} \bH_1 \F_0 \bJ_0 = \Hij_1 \Fij_0, \\
 \bJ_0 \E_0 \bH^{-1}_{-1} \E_{-1} \bH^{-1}_0 \bJ_0 &= \Hij^{-1}_0 \Eij_0, \\
 \bJ_0 \F_{-1} \bH_1 \E_{-1} \bH^{-1}_0 \bJ_0 & =
 \bJ_0 \bH_1 \bJ_0 \bJ_0 \F_{-1} \E_{-1} \bJ_0  = \Hij_1 \frac{\Hij_0 - \Hij^{-1}_0}{ v - v^{-1}}.
\end{align*}
Observe that $\bj_0 \bh_1 \bh_0 \bj_0 = v^2 \kij_0$. By the above analysis, we have
\begin{align*}
\widetilde \Delta^{\fc}  (\bj_0 \e_0 \f_0 \bj_0 )
& =
\bj_0 \e_0 \f_0 \bj_0 \otimes \Kij_0 + \kij^{-1}_0 \otimes \Hij^{-1}_0  \frac{\Hij_1 - \Hij^{-1}_1}{v - v^{-1}}\\
&+ v^2 \kij_0 \otimes \Hij_1 \Fij_0 + v^{-2} \kij^{-1}_0 \otimes \Hij^{-1}_0 \Eij_0 + \kij_0 \otimes \Hij_1 \frac{\Hij_0 - \Hij^{-1}_0}{ v - v^{-1}}.
\end{align*}
By definition, we also have
\begin{align*}
\widetilde \Delta^{\fc} (\bj_0 \frac{\bk^{-1}_0 - \bk_0}{ v -v^{-1}} \bj_0) = - \kij_0 \otimes \frac{\Hij_1 \Hij_0}{v-v^{-1}} + \kij^{-1}_0 \otimes \frac{\Hij^{-1}_1 \Hij^{-1}_0 }{v - v^{-1}}.
\end{align*}
By adding the last two equalities, we have established the formula for $\widetilde \Delta^{\fc} ( \tij_0)$.
\end{proof}

We define the transfer map
$
\phi^{\ij}_{d, d - \nn} : {\mbf S}^{\ij}_{\nn,d} \longrightarrow {\mbf S}^{\ij}_{\nn, d-\nn}
$
to be the composition
\[
\begin{CD}
\phi^{\ij}_{d, d - \nn} : {\mbf S}^{\ij}_{\nn,d} @>\widetilde \Delta^{\fc} >>  {\mbf S}^{\ij}_{\nn, d-\nn} \otimes {\mbf S}_{\nn, \nn} @>1\otimes \chi_{\nn} >> {\mbf S}^{\ij}_{\nn, d-\nn}.
\end{CD}
\]
Recall that the ``signed" homomorphism $\chi_{\nn}: {\mbf S}_{\nn, \nn} \rightarrow \mbb Q(v)$ satisfies that
$\chi_{\nn} (\Eij_i)=0$, $\chi_{\nn} (\Fij_i)=0$ and $\chi_{\nn} (\Hij_i)=v$.
It follows by Proposition~\ref{tDj-ij} that, for all $i\in [1, r]$,
\begin{equation}
\phi^{\ij}_{d, d-\nn} ( \eij_i) = \eij'_i, \;\;
\phi^{\ij}_{d, d - \nn} (\fij_i) = \fij'_i,  \;\;
\phi^{\ij}_{d, d - \nn} (\kij^{\pm 1}_i) = \kij'^{\pm 1}_i, \;\;
\phi^{\ij}_{d, d-\nn} (\tij_0) = \tij'_0.
\end{equation}
Hence we have constructed projective systems $\{({\mbf S}^{\ij}_{\nn,d}, \phi^{\ij}_{d, d-\nn})\}_{d\ge 0}$ and $\{({\mbf U}^{\ij}_{\nn,d}, \phi^{\ij}_{d, d-\nn})\}_{d\ge 0}$.

We now describe the restriction of $\Delta^{\fc}: {\mbf S}_{n,d}^{\fc} \longrightarrow  {\mbf S}_{n,d'}^{\fc} \otimes  {\mbf S}_{n,d''}$ 
defined in  (\ref{Dj}) to the subalgebra ${\mbf S}^{\ij}_{\nn,d}$, which shall be denoted by $\Dij$. 
\begin{prop}
\label{Dc-ij}
We have a homomorphism $\Dij : {\mbf S}^{\ij}_{\nn,d} \longrightarrow {\mbf S}^{\ij}_{\nn,d'} \otimes {\mbf S}_{\nn,d''} $
and by restriction  $\Dij : \bU^{\ij}_{\nn,d} \longrightarrow \bU^{\ij}_{\nn,d'} \otimes \bU_{\nn,d''} $. 
More precisely, for all $i\in [1, r]$, we have
\begin{align}
\begin{split}
\Dij ( \eij_i) & =  \eij'_i \otimes \Kij''_i + 1\otimes \Eij''_i + \kij'_i \otimes \Fij''_{n-1-i} \Kij''_i, \\
\Dij (\fij_i) & = \fij'_i \otimes \Kij''_{n-1-i} + \kij'^{-1}_i \otimes \Kij''_{n-1-i} \Fij''_i + 1\otimes \Eij''_{n-1-i}, \\
\Dij (\kij_i) & = \kij'_i \otimes \Kij''_i \Kij''^{-1}_{n - 1 -i}, \\
\Dij (\tij_0) & =  \tij'_0 \otimes \Kij''_0 + 1 \otimes v \Kij''_0 \Fij''_0 + 1 \otimes \Eij''_0.
\end{split}
\end{align}
\end{prop}

\begin{proof}
The first three formulas follow by $\Dij (\bj_0 ) = \bj'_0 \otimes \bJ''_0$ and Proposition ~\ref{Dc}.
The last one can be obtained as that of Proposition ~\ref{tDj-ij},  and we skip the detail.
%
\end{proof}

Since
$
\eij_i=\fij_i =0, \tij_0 =1, \kij_i = v^{\delta_{i, r}} \in {\mbf S}^{\ij}_{\nn, 0}$ for all $i\in [1, r],$ 
we have the following degenerate version of Proposition \ref{Dc-ij}.
\begin{prop}
We have an imbedding of algebras
$$
\iij_{\nn,d} = \Dij|_{d'=0}: {\mbf S}^{\ij}_{\nn,d} \longrightarrow {\mbf S}_{\nn,d}
$$ 
such that, for all $i\in [1, r]$,  
\begin{align}
\begin{split}
\iij_{\nn,d} (\eij_i)  = \Eij_i + v^{\delta_{i, r}} \Fij_{n-1-i} \Kij_i, \quad &
\iij_{\nn,d} (\fij_i)  = \Eij_{n-1-i} + v^{- \delta_{i, r}} \Kij_{n-1-i} \Fij_i, \\
\iij_{\nn,d} (\kij_i)   = v^{\delta_{i, r}} \Kij_i \Kij^{-1}_{n -1 -i}, \quad &
\iij_{\nn,d} (\tij_0)  = \Eij_0 + v \Kij_0 \Fij_0 + \Kij_0.
\end{split}
\end{align}
In particular, we have by restriction an imbedding of algebras $\iij_{\nn,d} : \bU^{\ij}_{\nn,d} \longrightarrow \bU_{\nn,d}$.
\end{prop}

Following Definition~\ref{def:ji-mon},  a notation of a $\ijw$-aperiodic matrix in $\MX^{\ij}_{\nn,d}$ is self-explanatory.
The following is a counterpart of Proposition~\ref{prop:ji-aperiodic}, whose proof will be skipped.

\begin{prop}
 \label{prop:ij-aperiodic}
The algebra ${\mbf U}^{\ij}_{\nn,d}$ has a canonical basis $\{ \{A\} \big \vert \, A \in \MX^{\ij}_{\nn,d} \ \ij\text{-aperiodic} \}$.
\end{prop}

\section{Quantum symmetric pair $(\bU (\slh_{\nn}), {\mbf U}^{\ij}_{\nn})$ and canonical basis on $\dot{{\mbf U}}^{\ij}_{\nn}$}
  \label{sec:coideal3}

The results in Chapter~\ref{chap:coideal2}, in particular those in Sections~\ref{sec:coideal2-CB}--\ref{sec:coideal2}, 
admit $\ijw$-counterparts with basically identical proofs; we shall outline these below.

Starting with the projective system $\{({\mbf U}^{\ij}_{\nn,d}, \phi^{\ij}_{d, d- \nn} )\}_{d\in \mbb{N}}$, we construct
two distinguished algebras  ${\mbf U}^{\ij}_{\nn}$ and $\dot{{\mbf U}}^{\ij}_{\nn}$ 
out of its associated  projective limit algebra ${\mbf U}^{\ij}_{\nn, \infty}$;
the Chevalley generators of ${\mbf U}^{\ij}_{\nn}$ are denoted again by $\eij_i$, $\fij_i$, $\kij^{\pm 1}_i$, for  $i\in [1, r]$, and $\tij_0$.
The family of imbeddings $\{\iij_{\nn,d}: {\mbf U}^{\ij}_{\nn,d} \to \bU_{\nn,d} \}_{d\ge 1}$ induces  an algebra imbedding  
$\iij_{\nn}: {\mbf U}^{\ij}_{\nn} \to \bU_{\nn}$. 
The family of $\Dij$ (for various $d', d''$) induces an algebra homomorphism 
$
\Dij : {\mbf U}^{\ij}_{\nn} \rightarrow {\mbf U}^{\ij}_{\nn} \otimes \bU_{\nn},
$
whose action on the Chevalley generators can be presented explicitly. Recall the algebra isomorphism $\bU_{\nn} \cong \bU (\slh_{\nn})$.
Summarizing we have established the following.

\begin{thm}
  \label{thm:QSP3}
The pair $(\bU (\slh_{\nn}), {\mbf U}^{\ij}_{\nn})$ forms a quantum symmetric pair of affine type.  
 (see Figure~\ref{figure:ij} for the relevant involution.)
\end{thm}


Recalling $\widetilde{\MX}_{n}$ from \eqref{eq:Mtilde},
we introduce the following subsets of $\widetilde{\MX}_{n}$:
\begin{align}
  \label{eq:Mtildeij}
  \begin{split}
  \widetilde{\MX}_{\nn}^{\ij}
  &= \big\{A=(a_{ij}) \in \widetilde{\MX}_{n} \big \vert \, a_{0, j} = \delta_{0, j} ,  a_{i, 0} = \delta_{i, 0} \big\},
 \\
\widetilde{\MX}_{\nn}^{\ij,ap}
&= \{A \in \widetilde{\MX}_{\nn}^{\ij} \big \vert A \text{ is $\ij$-aperiodic} \}.
\end{split}
\end{align}
We define an equivalence relation $\approx$ on $\widetilde{\MX}_{\nn}^{\ij,ap}$ 
as in \eqref{eq:approx} and let $\widehat A$ be  the equivalence class of $A$.
A hybrid monomial basis $\{{h}_A\}$ for ${\mbf S}^{\ij}_{\nn,d}$ can be constructed 
(similar to Proposition~\ref{prop:hMB:ji} in $\ji$ type), and it gives rise to a monomial basis 
$\{{h}_{\widehat A} \big \vert  \widehat A \in \widetilde{\MX}_{\nn}^{\ij,ap} /\!\approx \}$ for the algebra $\dot {\mbf U}^{\ij}_{\nn}$.
A bilinear form $\langle \cdot, \cdot \rangle$ on $\dot{{\mbf U}}^{\ij}_{\nn}$ can be defined similarly as in Section~\ref{bilinear-form} 
and shown to be non-degenerate.
We have the following analogue of Theorem~\ref{thm:ji-positive} (and also of Theorems~\ref{thm:iCB-Unc} and \ref{thm:positivity-Unc}).

\begin{thm}
 \label{thm:CBij}
There exists a canonical basis $\dot{\mbf B}^{\ij}_{\nn}  =\{b_{\widehat A} \big \vert  \widehat A \in \widetilde{\MX}_{\nn}^{\ij,ap}/\approx\}$  
for $\dot {\mbf U}^{\ij}_{\nn}$, whose transition matrix with respect to the monomial basis is uni-triangular.
Moreover, the structure constants of the canonical basis $\dot{\mbf B}^{\ij}_{\nn}$ 
are positive integral, i.e., they all lie in ${\mathbb N}[v,v^{-1}]$
with respect to the multiplication and comultiplication, and 
lie in $v^{-1}\mbb N[[v^{-1}]]$ with respect to the bilinear pairing .
\end{thm}

\section{The Schur algebras of type $\ii$}
   \label{ii-version}

Recall $\mm= n-2$, and so $\mm=\nn -1=2r$ for $r\geq 1$.
We set
\begin{equation}
  \label{Miid}
\MX^{\ii}_{\mm,d} = \MX^{\ji}_{\nn,d} \cap \MX^{\ij}_{\nn,d}, \qquad
\bj_{r, 0} =\bj_r \bj_0.
\end{equation}
The idempotent $\bj_{r, 0}$ gives rise to the subalgebra ${\mbf S}^{\ii}_{\mm,d}$:
\begin{equation}
  \label{Sii}
{\mbf S}^{\ii}_{\mm,d} = \bj_{r, 0} \Sj \bj_{r, 0} = {\mbf S}^{\ji}_{\nn,d} \cap {\mbf S}^{\ij}_{\nn,d}.
\end{equation}
Let ${\mbf U}^{\ii}_{\mm,d}$ be the subalgebra of ${\mbf S}^{\ii}_{\mm,d}$ generated by the following Chevalley generators:
\begin{align}
\begin{split}
\eii_i & = \bj_{r, 0}  \eji_i \bj_{r, 0}, \quad
\fii_i  = \bj_{r, 0} \fji_i \bj_{r, 0}, \\
\kii^{\pm 1}_i & = \bj_{r, 0} \kji^{\pm 1}_i \bj_{r, 0}, \quad \forall i\in [1, r-1],\\
\hii^{\pm 1}_a & = \bj_{r, 0} \hji^{\pm 1}_a \bj_{r, 0}, \quad \forall a \in [0, r], \\
\tii_0 & = \bj_{r, 0} \big( \eji_0 \fji_0 + \frac{\kji^{-1}_0 - \kji_0}{v - v^{-1}} \big) \bj_{r, 0} = \bj_{r, 0} \tij_0 \bj_{r, 0}, \\
\tii_r & = \bj_{r, 0} \tji_r \bj_{r, 0}.
\end{split}
\end{align}
Note that $\eii_i = \bj_{r, 0} \e_i \bj_{r, 0} = \bj_{r, 0} \eij_i \bj_{r, 0}$, etc.

We shall also need a type $A$ counterpart of the above construction as follows. 
We set
\begin{align*}
\Theta^{\ii}_{\mm, d} = \Theta^{\ji}_{\nn,d} \cap \Theta^{\ij}_{\nn,d},
\quad 
\bJ_{r, 0} = \bJ_r \bJ_0, \quad
{\mbf S}_{\mm,d} = \bJ_{r, 0} {\mbf S}_{n,d} \bJ_{r, 0} .
\end{align*}
Let ${\mbf U}_{\mm,d}$ be the subalgebra of ${\mbf S}_{\mm,d}$ generated by the following Chevalley generators:
\begin{align}
\begin{split}
\Eii_i & =
\begin{cases}
\bJ_{r, 0} \Eji_0 \Eji_{-1} \bJ_{r, 0}, & \mbox{if} \ i=0, \\
\bJ_{r, 0} \Eji_i \bJ_{r, 0}, & \mbox{if} \ i\in [1, \eta -1].
\end{cases} \\
\Fii_i & =
\begin{cases}
\bJ_{r, 0} \Fji_{-1} \Fji_{0} \bJ_{r, 0}, & \mbox{if} \ i=0, \\
\bJ_{r, 0} \Fji_i \bJ_{r, 0}, & \mbox{if} \ i\in [1, \eta -1].
\end{cases} \\
\Kii^{\pm 1}_i & =
\begin{cases}
\bJ_{r, 0} \Kji^{\pm 1}_0 \Kji^{\pm 1}_{-1} \bJ_{r, 0}, & \mbox{if} \ i=0, \\
\bJ_{r, 0} \Kji^{\pm 1}_i \bJ_{r, 0}, & \mbox{if} \ i\in [1, \eta -1].
\end{cases} \\
\Hii^{\pm 1}_a & =
\begin{cases}
\bJ_{r, 0} \Hji^{\pm 1}_{-1} \bJ_{r, 0}, & \mbox{if} \ a=0, \\
\bJ_{r, 0} \Hji^{\pm 1}_a \bJ_{r, 0}, & \mbox{if} \ a\in [1, \eta].
\end{cases}
\end{split}
\end{align}
We can make the indices periodic by setting $\Eii_i = \Eii_{i + \mm}$, etc, i.e., $i\in \mbb Z/\eta \mbb Z$.

Let us describe the restriction to the subalgebra ${\mbf U}_{\mm,d}^{\ii}$ (denoted by the same notation) 
of $\widetilde \Delta^{\fc} : {\mbf S}_{n,d}^{\fc} \to  {\mbf S}_{n,d'}^{\fc} \otimes  {\mbf S}_{n,d''}$ from (\ref{tDj}),  for arbitrary $d', d''$ such that $d=d'+d''$.
The proof is similar to that for Proposition ~\ref{tDj-ij} and will be skipped.
\begin{prop}
\label{tDj-ii}
We have an algebra homomorphism $\widetilde \Delta^{\fc} : {\mbf U}_{\mm,d}^{\ii} \to  {\mbf U}_{\mm,d'}^{\ii} \otimes  \bU_{\mm,d''}$.
More precisely, for all $i\in [1, r-1]$, we have
\begin{align*}
\begin{split}
\widetilde \Delta^{\fc} ( \eii_i) & = \eii'_i \otimes \Hii''_{i+1} \Hii''^{-1}_{\nn - 1-i} + \hii'^{-1}_{i+1} \otimes \Eii''_i \Hij''^{-1}_{\nn -1-i}
+ \hii'_{i+1} \otimes \Fii''_{\nn -1-i} \Hii''_{i+1}, \\
\widetilde \Delta^{\fc} ( \fii_i)  & = \fii'_i \otimes \Hii''^{-1}_i \Hii''_{\nn -i} + \hii'_i \otimes \Fii''_i \Hii''_{\nn  -i}  + \hii'^{-1}_i \otimes \Eii''_{\nn -1-i} \Hii''^{-1}_i, \\
\widetilde \Delta^{\fc} ( \kii_i) & = \kii'_i \otimes \Kii''_i \Kii''^{-1}_{\nn -1-i}, \\
\widetilde \Delta^{\fc} ( \tii_0)  & = \tii'_0 \otimes \Kii''_0 + v^2 \kii'_0 \otimes \Hii''_1 \Fii''_0 + v^{-2} \kii'^{-1}_0 \otimes \Hii''^{-1}_0 \Eii''_0, \\
\widetilde \Delta^{\fc} ( \tii_r)  & = \tii'_r \otimes \Kii''_r + v^2 \kii'^{-1}_r \otimes \Hii''_{r+1} \Fii''_r + v^{-2} \kii'_r \otimes \Hii''^{-1}_r \Eii''_r.
\end{split}
\end{align*}
\end{prop}

We define the transfer map
$
\phi^{\ii}_{d, d - \mm} : {\mbf S}^{\ii}_{\mm,d} \longrightarrow {\mbf S}^{\ii}_{\mm, d - \mm}
$
to be the composition of the following homomorphisms
\[
\begin{CD}
\phi^{\ii}_{d, d - \mm} : {\mbf S}^{\ii}_{\mm,d} @>\widetilde \Delta^{\fc} >>  {\mbf S}^{\ii}_{\mm, d - \mm} \otimes {\mbf S}_{\mm, \mm} @>1\otimes \chi_{\mm} >> {\mbf S}^{\ii}_{\mm, d - \mm}.
\end{CD}
\]
Noting that $\chi_\mm (\Eii_i)=0$, $\chi_\mm(\Fii_i)=0$ and $\chi_\mm (\Hii_i)=v$, we have, for all $i\in [1, r-1]$, 
\begin{align}
  \begin{split}
 \phi^{\ii}_{d, d-\mm} ( \eii_i) = \eii'_i, \quad & \phi^{\ii}_{d, d - \mm} (\fii_i) = \fii'_i, \quad  \phi^{\ii}_{d, d - \mm} (\kii^{\pm 1}_i) = \kii'^{\pm 1}_i,
 \\  
 \phi^{\ii}_{d, d-\mm} (\tii_0) = \tii'_0, \quad & \phi^{\ii}_{d, d-\mm} (\tii_r) = \tii'_r.
  \end{split}
\end{align}

We now describe the restriction of $\Delta^{\fc}$   (\ref{Dj}) to the subalgebra ${\mbf S}^{\ii}_{\mm,d}$, which shall be denoted by $\Dii$. 
We shall skip the proof, as it is similar to earlier cases.

\begin{prop}
\label{Dc-ii}
We have a homomorphism $\Dii : {\mbf S}^{\ii}_{\mm,d} \rightarrow {\mbf S}^{\ii}_{\mm,d'} \otimes {\mbf S}_{\mm,d''}$, and by restriction,
 a homomorphism $\Dii : \bU^{\ii}_{\mm,d} \rightarrow \bU^{\ii}_{\mm,d'} \otimes \bU_{\mm,d''}$. 
More precisely, for all $i\in [1, r-1]$, we have
\begin{align}
\begin{split}
\Delta^{\ii} ( \eii_i) & =  \eii'_i \otimes \Kii''_i + 1\otimes \Eii''_i + \kii'_i \otimes \Fii''_{\nn -1-i} \Kii''_i, \\
\Delta^{\ii} (\fii_i) & = \fii'_i \otimes \Kii''_{\nn -1-i} + \kii'^{-1}_i \otimes \Kii''_{\nn -1-i} \Fii''_i + 1\otimes \Eii''_{\nn -1-i}, \\
\Delta^{\ii} (\kii_i) & = \kii'_i \otimes \Kii''_i \Kii''^{-1}_{\nn - 1 -i}, \\
\Delta^{\ii} (\tii_0) & =  \tii'_0 \otimes \Kii''_0 + 1 \otimes v \Kii''_0 \Fii''_0 + 1 \otimes \Eii''_0, \\
\Delta^{\ii} (\tii_r) & =  \tii'_r \otimes \Kii''_r + 1 \otimes v \Kii''_r \Fii''_r + 1 \otimes \Eii''_r.
\end{split}
\end{align}
\end{prop}

A degenerate version of Proposition~\ref{Dc-ii} gives us the following description for the homomorphism
$ \iii_{\mm,d} = \Dii|_{d'=0}: {\mbf S}^{\ii}_{\mm,d} \longrightarrow {\mbf S}_{\mm,d}$.

\begin{prop}
We have imbeddings of algebras 
$$ 
\iii_{\mm,d} : {\mbf S}^{\ii}_{\mm,d} \longrightarrow {\mbf S}_{\mm,d},
\qquad
\iii_{\mm,d} : \bU^{\ii}_{\mm,d} \longrightarrow \bU_{\mm,d}.
$$
Moreover, for all $i\in [1, r-1]$, we have
\begin{align*}
\iii_{\mm,d} (\eii_i) & = \Eii_i +  \Fii_{\nn -1-i} \Kii_i = \Eii_i +  \Fii_{-i} \Kii_i , \\
\iii_{\mm,d} (\fii_i) & = \Eii_{\nn - 1-i} +  \Kii_{\nn -1-i} \Fii_i = \Eii_{-i} +  \Kii_{-i} \Fii_i, \\
\iii_{\mm,d} (\kii_i) &  =  \Kii_i \Kii^{-1}_{\nn -1 -i} =  \Kii_i \Kii^{-1}_{ -i}, \\
\iii_{\mm,d} (\tii_0) & = \Eii_0 + v \Kii_0 \Fii_0 + \Kii_0, \\
\iii_{\mm,d} (\tii_r ) & = \Eii_r + v \Kii_r \Fii_r + \Kii_r.
\end{align*}
\end{prop}

\section{Realization of a new coideal subalgebra ${\mbf U}^{\ii}_{\mm}$}

We first formulate quickly results on monomial and canonical bases for ${\mbf U}^{\ii}_{\mm,d}$ analogous to Lusztig algebras of types $\jj, \ji, \ij$ treated earlier.
Recall $\MX^{\ii}_{\mm,d}$ from \eqref{Miid}.
Following Definition~\ref{def:ji-mon},  a notation of a $\iiw$-aperiodic matrix $A$ in $\MX^{\ii}_{\mm,d}$ is self-explanatory.
Similar to Proposition~\ref{prop:ji-aperiodic} (also see Proposition~\ref{prop:ij-aperiodic}) we can establish the canonical basis for ${\mbf U}^{\ii}_{\mm,d}$.
This is again based on the existence of a monomial basis $\{ y_A\}$
for ${\mbf U}^{\ii}_{\mm,d}$, which can be established in a way similar to Proposition~\ref{prop:MB:ji}. 
A hybrid monomial basis $\{ h_A\}$ for ${\mbf U}^{\ii}_{\mm,d}$ can also be established in a way
similar to  Proposition~ \ref{prop:hMB:ji}. We summarize these as follows.

\begin{prop}
The algebra ${\mbf U}^{\ii}_{\mm,d}$ admits a monomial basis $\{ y_A \big \vert A \in \MX^{\ii}_{\mm,d} \; \iiw\text{-aperiodic} \}$
as well as a hybrid monomial basis $\{ h_A \big \vert A \in \MX^{\ii}_{\mm,d} \; \iiw\text{-aperiodic} \}$.
Also, the set $\{ \{A\} \big \vert A \in \MX^{\ii}_{\mm,d} \; \iiw\text{-aperiodic} \}$ forms a canonical basis
for ${\mbf U}^{\ii}_{\mm,d}$.
\end{prop}

\begin{example}
Let $r=1$, hence $\mm=2$. Consider the following matrix $A$ in $\MX^{\ii}_{d}$ after deleting zero and second row and columns.

\[
\begin{tabular}{   c | c | c | c | c| c| c| c| c | c | c } 

 &c-3&  c-2 & c-1  & c0 & c1 & c2 & c3 & c4 & c5 & c6   \\
 \hline 
  &&  && && && &&  \\
\hline 
 r0 &0 &0 & 2 & * & 0 & 3& 4&& &  \\
 \hline
 r1  & & 4 & 3 & 0 & * &  2 & 0 & 0 &&  \\
\hline 
r2 & & & 0& 0 & 2 & * & 0 & 3 & 4&  \\
\hline 
r3 & & && 4 & 3 & 0 &*  & 2 & 0 & 0 \\
\hline 
 && &&&& && &&  \\
\end{tabular}
\]
Then we have 
\[
\tii_0^{\langle 7\rangle} * \tii_1^{\langle 9\rangle} * \tii_0^{\langle 4\rangle} * 1_{\co(A)} = [A] +\mbox{lower terms},
\]
where
\begin{align*}
\tii_0^{\langle R\rangle} = \sum_{X: X - R E^{0, 1}_{\theta, \mm} \ \mbox{diagonal}} [X],\qquad 
\tii_1^{\langle R\rangle} = \sum_{X: X - R E^{2, 1}_{\theta, \mm} \ \mbox{diagonal}} [X].
\end{align*}
This is a typical monomial  appearing in a monomial basis of ${\mbf S}_{2, d}^{\ii}$. 
\end{example}


Now we shall formulate the $\ii$-counterparts of the results on coideal algebras
arising from families of Lusztig algebras in Sections~\ref{sec:coideal2} and \ref{sec:coideal3}.
Again we skip the proofs as they are analogous to the earlier cases. 

Starting with the projective system $\{({\mbf U}^{\ii}_{\mm,d}, \phi^{\ii}_{d, d- \mm} )\}_{d\in \mbb{N}}$, we construct
two distinguished algebras  ${\mbf U}^{\ii}_{\mm}$ and $\dot{{\mbf U}}^{\ii}_{\mm}$ 
out of its associated  limit algebra ${\mbf U}^{\ii}_{\mm, \infty}$;
the Chevalley generators of ${\mbf U}^{\ii}_{\mm}$ are denoted again by $\tii_0$, $\tii_r$, $\eii_i$, $\fii_i$, $\kii^{\pm 1}_i$, for  $i\in [1, r-1]$.
The family of imbeddings $\{\iii_{\mm,d}: {\mbf U}^{\ii}_{\mm,d} \to \bU_{\mm,d} \}_{d\ge 1}$ induces  an algebra imbedding  
$\iii_{\mm}: {\mbf U}^{\ii}_{\mm} \to \bU_{\mm}$. 
The family of $\Dii$ (for various $d', d''$) induces an algebra homomorphism 
$
\Dii : {\mbf U}^{\ii}_{\mm} \rightarrow {\mbf U}^{\ii}_{\mm} \otimes \bU_{\mm},
$
whose action on the Chevalley generators can be presented explicitly. Recall the algebra isomorphism $\bU_{\mm} \cong \bU (\slh_{\mm})$.
Summarizing we have established the following.

\begin{thm}
  \label{thm:QSP4}
The pair $(\bU (\slh_{\mm}), {\mbf U}^{\ii}_{\mm})$ forms a quantum symmetric pair of affine type.  
(see Figure~\ref{figure:ii} for the relevant involution.)
\end{thm}


Recalling $\widetilde{\MX}_{\nn}^{\ji}$ from \eqref{eq:Mtildeji} and  $\widetilde{\MX}_{\nn}^{\ij}$ from \eqref{eq:Mtildeij},
we introduce the following subsets of $\widetilde{\MX}_{n}$:
\begin{align}
  \label{eq:Mtildeii}
  \widetilde{\MX}_{\mm}^{\ii}
  =  \widetilde{\MX}_{\nn}^{\ji} \cap \widetilde{\MX}_{\nn}^{\ij},
 \qquad
\widetilde{\MX}_{\mm}^{\ii,ap}
= \{A \in \widetilde{\MX}_{\mm}^{\ii} \big \vert A \text{ is $\ii$-aperiodic} \}.
\end{align}
We have the following $\ii$-analogue of Theorem~\ref{thm:ji-positive} and  Theorems~\ref{thm:CBij}.

\begin{thm}
 \label{thm:CBii}
There exists a canonical basis $\dot{\mbf B}^{\ii}_{\mm}  =\{b_{\widehat A} \big \vert  \widehat A \in \widetilde{\MX}_{\mm}^{\ii,\ap}/\approx\}$  
for $\dot {\mbf U}^{\ii}_{\mm}$, whose transition matrix with respect to the monomial basis is uni-triangular.
Moreover, the structure constants of the canonical basis $\dot{\mbf B}^{\ii}_{\mm}$ 
all lie in ${\mathbb N}[v,v^{-1}]$ with respect to the multiplication and comultiplication, 
and  in $v^{-1} \mbb N[[v^{-1}]]$ with respect to the bilinear pairing .
\end{thm}

Recall the Cartan integers $\texttt{c}_{ij}$ from \eqref{aij}. 
We now give a presentation for the algebra ${\mbf U}_{ \mm}^{\ii}$, 
which is a counterpart of Propositions~\ref{present:Ujj} and \ref{present:Uji}. 
This presentation  is again a variant of \cite[Theorem~ 7.1]{Ko14} in our setting and our notation. 
\begin{prop}
 \label{present:Uii}
Let $r\geq 2$ and so $\mm =2r \ge 4$. 
The $\mbb Q(v)$-algebra ${\mbf U}_{ \mm}^{\ii}$ has a presentation with generators
 $\eii_i, \fii_i$,  $\kii^{\pm 1}_i$  for $i\in [1, r-1]$ and  $\tji_k$ for $k=0, r$ and the following relations\footnote{Replace the relation $\kii^2_1  \cdots \kii^2_{r-1}    = 1$ by $\kii_1  \cdots \kii_{r-1}    = 1$.} for all $i, j \in [1, r-1]$,  $k\in \{ 0, r\}$:
\begin{align*}
\kii^2_1 & \cdots \kii^2_{r-1}    = 1,\\
\kii_i \kii_i^{-1} & = 1, \quad
\kii_i \kii_j   = \kii_j \kii_i, \\
\kii_i \eii_j \kii_i^{-1} &=  v^{\texttt{c}_{ij}}  \eii_j,    \quad
\kii_i \fii_j \kii_i^{-1}  =  v^{-\texttt{c}_{ij}}  \fii_j, \\
\kii_i \tii_k & = \tii_k \kii_i, 
\quad \tii_0 \tii_r =\tii_r \tii_0,
\end{align*}
\begin{align*}
\eii_i \eii_j  &= \eii_j \eii_i,  \quad
\fii_i \fii_j  = \fii_j \fii_i, \quad \forall |i-j|> 1, \\
\eii_i \tii_k  & = \tii_k \eii_i,    \quad
\fii_i \tii_k  = \tii_k \fii_i, \quad \forall |i-k|>1, 
\end{align*}
\begin{align*}
\eii_i^2 \eii_j  + \eii_j \eii_i^2 &= (v + v^{-1}) \eii_i \eii_j \eii_i, \quad \forall |i-j|=1,\\
\fii_i^2 \fii_j  + \fii_j \fii_i^2 &= (v + v^{-1}) \fii_i \fii_j \fii_i, \quad \forall |i-j|=1,\\
\eii^2_{i} \tii_k  + \tii_k \eii^2_{i} & = (v + v^{-1}) \eii_{i } \tii_k \eii_{i}, \quad \forall |i-k|=1,\\
\fii^2_{i} \tii_k + \tii_k \fii^2_{i } & = (v + v^{-1}) \fii_{i } \tii_k \fii_{i}, \quad \forall |i-k|=1,  \\
\tii_k^2 \eii_{j} + \eii_{j} \tii_k^2 & = (v + v^{-1}) \eii_{j} \tii_k \eii_{j} + \eii_{j}, \forall |k-j|=1, \\
\tii_k^2 \fii_{j} + \fii_{j} \tii_k^2 & = (v + v^{-1}) \fii_{j} \tii_k \fii_{j} + \fii_{j}, \forall  |k-j|=1,  \\
\eii_i \fii_j - \fii_j \eii_i & = \delta_{ij} \frac{\kii_i - \kii_i^{-1}}{v - v^{-1}} .
\end{align*}
\end{prop}

The case for $\mm=2$ is excluded from Proposition~\ref{present:Uii} above. 
The algebra ${\mbf U}_{2}^{\ii}$ is generated by $\tii_0$ and $\tii_1$, and we have an imbedding
$
\iii_d: {\mbf U}_{2}^{\ii} \rightarrow \bU(\widehat{\mathfrak{sl}_2})
$
defined by
\[
\tii_0 \mapsto \Eii_0 + v \Kii_0 \Fii_0 + \Kii_0, \quad 
\tii_1 \mapsto \Eii_1 + v \Kii_1 \Fii_1 + \Kii_1. 
\]
\begin{prop}
The $\mbb Q(v)$-algebra ${\mbf U}^{\ii}_2$ has a presentation with generators $\tii_0$ and $\tii_1$, and
 the following relations:
\begin{align}
\label{t0t1}
\tii_0^3 \tii_1 - \llbracket 3\rrbracket \tii_0^2 \tii_1 \tii_0 + \llbracket 3\rrbracket \tii_0 \tii_1 \tii_0^2 - \tii_1 \tii_0^3 
 = \llbracket 2\rrbracket^2 (\tii_0 \tii_1 - \tii_1 \tii_0), \\
\label{t1t0}
\tii_1^3 \tii_0 - \llbracket 3\rrbracket \tii_1^2 \tii_0 \tii_1 + \llbracket 3\rrbracket \tii_1 \tii_0 \tii_1^2 - \tii_0 \tii_1^3 
= \llbracket 2\rrbracket^2 (\tii_1 \tii_0 - \tii_0 \tii_1).
\end{align}
Here $\llbracket n \rrbracket = \frac{v^n - v^{-n}}{v - v^{-1}}$.
\end{prop}

\begin{proof}
We first prove (\ref{t0t1}). 
Since $\iii_2$ is injective, it suffices to show that (\ref{t0t1}) holds  in $\bU (\widehat{\mathfrak{sl}_2})$
after applying $\iii_2$.
So we can assume that we are working in $\bU (\widehat{\mathfrak{sl}_2})$.
Let $S(\tii_0, \tii_1)$ denote the term on the left-hand side of (\ref{t0t1}). 
Similarly, we can define $S(\tii_0, \Eii_1)$, etc., so that we have
\[
S(\tii_0, \tii_1) = S(\tii_0, \Eii_1) + S(\tii_0, v \Kii_1 \Fii_1 ) + S(\tii_0, \Kii_1). 
\]
By expanding out $S(\tii_0, \Eii_1)$, which has $4 \times 3^4=324$ terms in total,  and using the defining relations of $\bU (\widehat{\mathfrak{sl}_2})$, we have
\begin{align}
\label{t0t1-1}
S(\tii_0, \Eii_1) 
= \llbracket 2\rrbracket^2 ( \Eii_0 \Eii_1 - \Eii_1 \Eii_0 - (v^2 -1) \Kii_0 \Eii_1 - (v^3 - v) \Kii_0 \Fii_0 \Eii_1). 
\end{align}
More precisely, the term $\Eii_0 \Eii_1 - \Eii_1 \Eii_0$ comes from simplifying the sum of   the terms in $S(\tii_0, \Eii_1)$ involving 
$\Kii_0 \Fii_0 \Eii_0^2 \Eii_1$ or its variants such as $\Kii_0 \Eii_0 \Fii_0 \Eii_0 \Eii_1$.
The term $\Kii_0 \Eii_1$  comes from simplifying the sums of 
$\Kii_0^2 \Fii_0 \Eii_0 \Eii_1$, $\Kii_0^2 \Fii_0 \Eii_1 \Eii_0$ and theirs variants.
The term $\Kii_0 \Fii_0 \Eii_1$ is a result of simplifying the sums of 
$\Kii_0^2 \Fii_0^2 \Eii_0 \Eii_1$, $\Kii_0^2 \Fii_0^2 \Eii_1 \Eii_0$ and their variants.
The rest of the terms in $S(\tii_0, \Eii_1)$ sums to  zero. 

Similarly, with a very lengthy calculation as above, we obtain
\begin{align}
\label{t0t1-2}
S(\tii_0, v \Kii_1 \Fii_1) 
&= \llbracket 2\rrbracket^2 ( (v^3 -v) \Kii_1 \Eii_0 \Fii_1  + \Kii_0 \Kii_1 (\Fii_0 \Fii_1 - \Fii_1 \Fii_0)  + (v - v^{-1})  \Kii_0 \Kii_1 \Fii_1) ,\\
\label{t0t1-3}
S(\tii_0, \Kii_1) 
&= \llbracket 2\rrbracket^2 ( (v^2 -1) \Kii_1 \Eii_0 - (v - v^{-1}) \Kii_0 \Kii_1 \Fii_0).
\end{align}
From (\ref{t0t1-1})-(\ref{t0t1-3}), it is straightforward to observe that 
$S(\tii_0, \tii_1)$ is equal to  the right-hand side of (\ref{t0t1}).

The equality (\ref{t1t0}) can be proved similarly.
By Theorem~\ref{thm:QSP4} and 
\cite[Theorem 7.1]{Ko14}, we do not need more relations for the coideal subalgebra ${\mbf U}_{2}^{\ii}$ of $\bU (\widehat{\mathfrak{sl}_2})$. 
\end{proof}

\begin{rem}
The algebra ${\mbf U}^{\ii}_2$ is the so-called $q$-Onsager algebra in the literature, see ~\cite[Example 7.6]{Ko14} and the references therein.
\end{rem}

\newpage
\part{Schur algebras and coideal  subalgebras of  $\bU (\widehat{\mathfrak{gl}}_n)$}  
  \label{part3}

\chapter{The stabilization algebra $\K^{\C}_n$ arising from Schur algebras}
  \label{chap:Kc}

In this chapter we study the stabilization of the family of Schur algebras $\Sj$ (as $d$ varies),
which leads to the formulation of the stabilization algebra $\K^{\C}_n$ as well as its monomial and stably canonical bases.
One difficulty of working with the Schur algebra $\Sj$ directly is that it does not have a good generating set.
We overcome the difficulty by embedding $\Sj$ into a Lusztig algebra of higher rank. 
This allows us to understand monomial bases, multiplication, comultiplication and bar operators of the Schur algebras and their stabilization properties
in  a conceptual way and lift these structures to $\K^{\C}_n$. 
We show that the pair $(\K_n, \Kc_n)$ forms a quantum symmetric pair in an idempotented form,  
where $\K_n$ is isomorphic to the idempotented quantum affine $\gl_n$.


\section{Monomial bases for Schur algebras}

Recall $n=2r+2$ for $r\ge 0$. We set 
\[
\breve r = r+1, \quad
\breve n = 2 \breve r +2.
\]
We consider the subset $\Xi_{\breve n, d}^{\jmath \jmath}$ which consists of all matrices $A \in \Xi_{\breve n, d}$ such that
$a_{1, j} = a_{i, 1}  =0$ for all $i, j\in \mbb Z$. 
Then the deleting operator  $\mbox{dlt}_1$ of the row and column $\pm 1$ mod $\breve n$ defines a bijective map   
$\Xi_{\breve n, d}^{\jmath \jmath} \to  \Xi_{n, d}$. 
We denote by $\ddot \empty : \Xi_{n, d} \rightarrow \Xi_{\breve n, d}$
the inverse map to $\mbox{dlt}_1$. 
More generally, we may regard $\ddot \empty$ as  an imbedding 
\begin{equation}
\label{eq:dd}
\ddot \empty : \Xi_{n, d} \longrightarrow \Xi_{\breve n, d}, \quad A \mapsto \ddot A,
\end{equation}
by  adding suitable rows and columns of zeros.

As we will study the behavior of the various bases in $\Sj$ and $\mbf S_{\breve n, d}$ under stabilization, we shall
put a subscript $d$ to  emphasize the dependence of $d$, e.g., $[A]_d$.

Just like our study of $\ij$, $\ji$ and $\ii$ versions, we consider the following idempotent in $\mbf S^{\C}_{\breve n, d}$ and its associated subalgebra:
\begin{align}
\begin{split}
\ddot {\mbf S}^{\C}_{n, d}  = \breve \bj_1 \mbf S^{\C}_{\breve n,d} \breve \bj_1,
\quad \mbox{where} \quad
\breve \bj_{1} = \sum_{A \in \Xi_{n, d}: A \ \mbox{diagonal}} [\ddot A]_d.
\end{split}
\end{align}

\begin{prop} 
\label{prop:rho}
There is  an algebra imbedding $\rho: \Sj \to \mbf S^{\C}_{\breve n, d}$, $[A]_d \mapsto [\ddot A]_d$, for $A \in \Xi_{n, d}$, 
and  an induced algebra isomorphism   $\rho: \Sj \overset{\simeq}{\longrightarrow} \ddot{\mbf S}^{\C}_{n, d}$,   
which are compatible with the  canonical bases.
\end{prop}

\begin{proof}
We define an imbedding $\X^{\C}_{n, d} \to \X^{\C}_{\breve n, d}$,  $L \mapsto \breve L$ 
by adding to  $L$ an extra copy of $L_1$ and $L_{-2}$ mod $n$. Specifically,  the lattice chains from $0$ to $\breve n$ in $\breve L$ 
are 
\[
(L_0, L_0, L_1, \cdots, L_{n-1} , L_{n-1} , L_{n}).
\] 
This imbedding clearly induces an injective algebra homomorphism
$ \Sj \to \mbf S^{\C}_{\breve n, d}$,  with image being $\ddot{\mbf S}^{\C}_{n, d}$.
\end{proof}

By Proposition  \ref{prop:rho}, we can study Schur algebra $\Sj$ through $\ddot{\mbf S}^{\C}_{n, d}$, which has an advantage that it admits
an inclusion
 \begin{equation}
  \label{eq:SU}
\ddot{\mbf S}^{\C}_{n, d} \subseteq \U^{\C}_{\breve n, d},
\end{equation}
since $\ddot{\mbf S}^{\C}_{n, d}$ is spanned by canonical basis elements 
parametrized by matrices  whose second columns are zero; such matrices are automatically aperiodic. Hence we have
\[
\rho: \Sj \longrightarrow  \U^{\C}_{\breve n, d}.
\]
The pair $(\Sj, \U^{\C}_{\breve n, d})$ for $\Sj$ plays 
a similar role as what  the pairs $(\mbf S^{\ji}_{\nn, d}, \Sj)$, $(\mbf S^{\ij}_{\nn, d}, \Sj)$ and $(\mbf S^{\ii}_{\eta, d}, \Sj)$ 
do for $\mbf S^{\ji}_{\nn, d}$, $\mbf S^{\ij}_{\nn, d}$, and $\mbf S^{\ii}_{\eta, d}$, respectively.

We shall put a superscript $\breve \empty$ on the Chevalley generators of $\mbf S^{\C}_{\breve n, d}$.
For convenience, let $\breve \f_i = \breve \e_{\breve n-(i+1)}$ if $r+1 \leq i \leq \breve n-1$
and $\breve \f_i = \breve \f_{\breve n+i}$ for all $i\in \mbb Z$. 
To each $tridiagonal$ matrix $A \in \Xi_{n, d}$ such that  $A - \sum_{1\leq i\leq n} \alpha_i E^{i, i+1}_{\theta, n}$ is diagonal, 
we set $\alpha_0 =\alpha_n$ and 
\begin{align}
\label{fA}
\ddot \f_{A; d} = \breve \f_0^{(\alpha_0)} * 
\breve \f_{n}^{(\alpha_{n-1})} * \breve \f_{n+1}^{(\alpha_{n-1})} 
* 
\left (\breve \f_{n-1}^{(\alpha_{n-2})} * \breve \f_{n-2}^{(\alpha_{n-3})} * \cdots * \breve \f_1^{(\alpha_0)} \right ) 
*
1_{\co (\ddot A)},
\end{align}
where the idempotent $1_{\co (\ddot A)}$ is the standard basis element 
attached to the diagonal matrix in $ \Xi_{\breve n, d}$ with diagonal  $\co (\ddot A)$.
Note that  the product is taken in $\mbf S^{\C}_{\breve n, d}$. 
Since it lies in the component $\mbf S^{\C}_{\breve n, d} (\ro (\ddot A) , \co (\ddot A))$ and hence
lies in the image of $\rho$, we can define an element  $\f_{A; d}$ in  $\Sj$ to be its preimage under $\rho$, i.e.,  
\begin{equation}
  \label{eq:fA}
\f_{A; d} = \rho^{-1} (\ddot \f_{A; d}).
\end{equation} 

\begin{lem}
\label{fA-leading}
For each tridiagonal matrix $A$ in $\Xi_{n, d}$, we have 
$
[A]_d = \f_{A; d} + \text{lower terms}. 
$
\end{lem}

\begin{proof}
It is reduced to showing a similar  statement for $\ddot \f_{A; d}$ in  $\mbf S^{\C}_{\breve n, d}$ via $\rho$. 
We first observe that the monomial 
$\breve \f_{n+1}^{(\alpha_{n-1})} 
* 
\left (\breve \f_{n-1}^{(\alpha_{n-2})} * \breve \f_{n-2}^{(\alpha_{n-3})} * \cdots * \breve \f_1^{(\alpha_0)} \right ) 
*
1_{\co (\ddot A)}$ 
(a part of \eqref{fA}) has a leading term  $[A']_d$ of a certain tridiagonal matrix $A'$ such that 
$A' - \sum_{1\leq i\leq n-1} \alpha_{i-1}E^{i, i+1}_{\theta, \breve n} - 
\alpha_{n-1} E^{n+1, n+2}_{\theta, \breve n}$ is diagonal. 
In particular, the off-diagonal upper triangular entries of $A'$  are the same as those of $\ddot A$ except at
$(0, 2)$, $(1, 2)$, $(n, n+2)$, $(n+1, n+2)$ mod $\breve n$. 
After composing with 
$\breve \f_0^{(\alpha_0)} * \breve \f_{n}^{(\alpha_{n-1})}$ and using Lemma ~\ref{BKLW3.9}, 
we see that the leading term of $\ddot \f_{A; d}$ is exactly $[\ddot A]_d$.
Transporting back via $\rho^{-1}$, the lemma is thus proved.
\end{proof}

As  a product of bar-invariant Chevalley generators in $\mbf S^{\C}_{\breve n, d}$, 
$\ddot \f_{A; d}$ is bar invariant in $\mbf S^{\C}_{\breve n, d}$. 
Since the imbedding $\rho$ is compatible with the bar operators in $\Sj$ and $\mbf S^{\C}_{\breve n, d}$, 
the preimage $\rho^{-1}(\f_{A; d})$  must be bar invariant in $\Sj$. Thus we have the following.

\begin{lem}
One has
$\overline{\f_{A; d}} = \f_{A; d}$ for all tridiagonal $A\in \Xi_{n, d}$.
\end{lem}

To a matrix, we define the $depth$ of $A$ by
\begin{equation}
  \label{eq:depth}
\dep (A) = \max \{l \in \mbb N | a_{i, i+l} \neq 0 \ \mbox{for some} \ i \}.
\end{equation}

The following description of leading terms leads to the determination of a set of  multiplicative generators for the Schur algebra $\Sj$.

\begin{prop}
\label{new-10.1.2}
Let $A, B \in \Xi_{n, d}$ such that  $\ro (A) = \co (B)$ and $\dep(A) \le m$ for some positive integer $m$. 
Assume further that $B - \sum_{1\leq i \leq n} \beta_i E^{i, i+1}_{\theta}$ is diagonal for some $\beta \in \mbb Z_n$
and $a_{i+1, i+m}\geq \beta_i\geq 0$ for all $i$.
Then  we have 
\[
[B]_d * [A]_d = \Big [A + \sum_{1\leq i \leq n} \beta_i (E^{i, i+m}_{\theta} - E^{i+1, i+m}_{\theta}) \Big ]_d + \text{lower terms}.
\]
\end{prop}

\begin{proof}
It is enough to show a similar statement with $[B]_d$ replaced by $\f_{B; d}$ by Lemma ~\ref{fA-leading}. 
We then transport this problem to the setting of $\mbf S^{\C}_{\breve n, d}$ and use  Lemma \ref{BKLW3.9-b}. 
Now the order in (\ref{fA}) allows us to push  $\beta_0$ and $\beta_{n-1}$ across rows $1$ and $n+1$ respectively to the desired positions. 
The statement then follows by pulling back to $\mbf S^{\C}_{n, d}$ via $\rho$. 
\end{proof}

Let us present an example explaining  the proof of Proposition ~\ref{new-10.1.2}.

\begin{example}
Let $A$ be the following matrix in $\Xi_{4, d}$ with $n=4$ and $\breve n =6$.
\[
\begin{tabular}{   c | c | c | c | c| c| c| c| c | c | c | c} 

 & c-3 & c-2 & c-1  & c0 & c1 & c2 & c3 & c4 & c5 & c6 & c7  \\
 \hline 
 & & && && && && & \\
\hline 
 r0 & & 7& * & $d_0$& * & 7& && &&  \\
 \hline
 r1 & &  & 6 & * & $d_1$ & * & 5 & && & \\
\hline 
r2 & & && 4 & * & $d_2$ & * & 4 && & \\
\hline 
r3 & & && & 5 & * & $d_3$  & *  & 6 & & \\
\hline 
r4 & & && &&7 &* & $d_0$  & * & 7& \\
\hline 
& & && && && && & \\
\end{tabular}
\]
where $d_i$ is the diagonal entries of $A$ and $*$ are some nonnegative integers irrelevant to the discussion.
Now let $(\beta_i)_{1\leq i \leq 4} = (2, 4, 5, 3)$ such that $B - \sum_{1\leq i\leq 4} \beta_i E^{i, i+1}_{\theta}$ is diagonal and 
$\ro(A) = \co (B)$.  We want to determine the leading term of $\f_{B; d} * [A]_d$.
By definition, we have 
\[
\ddot \f_{B; d} = \breve \f^{(3)}_0* \breve \f^{(5)}_4 * \breve \f^{(5)}_5 * \breve \f^{(4)}_3 * \breve \f^{(2)}_{2} * \breve \f^{(3)}_1 * 1_{\co(\ddot B)}
\]
Now we expand $A$ at row/column $\pm 1$ to get the  matrix $\ddot A$ in $\Xi_{6, d}$, which is completely determined by its upper triangular part  as follows.
\[
\begin{tabular}{     c | c | c| c| c| c| c | c | c | c | c | c} 

  & c-1  & c0 & c1 & c2 & c3 & c4 & c5 & c6 & c7 &c8 & c9 \\
 \hline 
  && && && && & & \\
\hline 
 r0  &  & $d_0$& 0 & * &7 && && & \\
 \hline
 r1   &  &  & $0$ & 0 & 0 &0 & 0& & & \\
\hline 
 r2   &  &  &  & $d_1$ & * & 5 &0& & &  \\
\hline 
r3  &&  &  & & $d_2$& * & 0 & 4 & & \\
\hline 
r4  && &  &  & & $d_3$  & 0 & * & 0 &6 \\
\hline 
r5  && && & &  & $0$ & 0& 0 & 0 & \\
\hline 
r6  && && & &  & & $d_0$ & 0 & * & 7\\
\hline 
 && && && && & & \\
\end{tabular}
\]
Then we apply $\ddot \f_{B; d}$ to $[\ddot A]_d$ to get the following leading term.
\[
\begin{tabular}{     c | c | c| c| c| c| c | c | c | c | c | c | c | c } 

  & c-1  & c0 & c1 & c2 & c3 & c4 & c5 & c6 & c7 &c8 & c9 & c10 & c11 \\
 \hline 
  && && && && & & &&  \\
\hline 
 r0  &  & $d_0$& 0 & 2&3 && && & &&\\
 \hline
 r1   &  &  & 0 & 0 & 0 &0 & & & & && \\
\hline 
 r2   &  &  &  & $d_1$ & * & 2 &0 &2 & &  &&\\
\hline 
r3  &&  &  & & $d_2$& * & 0 & 2 &0 & 4&& \\
\hline 
r4  && &  &  & & $d_3$  & 0 & * & 0 &2&5& \\
\hline 
r5  && && & &  & 0 & 0& 0 & 0 &0 & \\
\hline 
r6  && && & &  & & $d_0$ & 0 & * & 2 &3& \\
\hline 
 && && && && & &  &&\\
\end{tabular}
\]
This leading term is corresponding to the expected matrix in $\Xi_{4, d}$ whose upper triangular part is as follows.
\[
\begin{tabular}{c | c | c| c| c| c| c | c | c | c  } 

   & c-1  & c0 & c1 & c2 & c3 & c4 & c5 & c6 & c7   \\
 \hline 
  && && && &&  &  \\
\hline 
 r0  &  & $d_0$& * & 2&3  & & & &\\
 \hline
 r1   &  &  & $d_1$ & * & 2 &2 & & & \\
\hline 
 r2   &  &  &  & $d_2$ & * & 2&4& & \\
\hline 
r3  &&  &  & & $d_3$& * & 2 & 5&  \\
\hline 
r4  && &  &  & & $d_0$  & * & 2  & 3\\
\hline 
&& && && && &\\
\end{tabular}
\]
\end{example}

The following theorem is obtained by applying  Proposition ~\ref{new-10.1.2}  repeatedly.

\begin{thm}
\label{new-10.1.3}
For any matrix  $A=(a_{ij}) \in \Xi_{n, d}$ of depth $m$, there exist unique tridiagonal matrices $A_1$, $A_2, \ldots, A_m \in \Xi_{n, d}$  
satisfying $\ro (A_m) = \ro (A)$, $\co (A_1)= \co (A)$, $\ro (A_i) = \co (A_{i+1})$ for $1\leq i\leq m-1$ 
and 
$A_i - \sum_{1 \leq j \leq n} (\sum_{k \leq j + i -1} a_{k, j+1}) E^{j, j+1}_{\theta}$ is  diagonal 
for all $1\leq i \leq m$ such that the following formulas hold in $\Sj$:
\begin{align}
[A_m]_d * [A_{m-1}]_d * \cdots  * [A_1]_d &= [A]_d + \mbox{lower terms},
\label{[A]}
\\
\f_{A; d}:= \f_{A_m; d} * \f_{A_{m-1}; d} * \cdots  * \f_{A_1; d} 
&= [A]_d + \mbox{lower terms}.
\label{basis:fA}
\end{align}
\end{thm}

\begin{proof}
We prove (\ref{[A]}) by induction with respect to the depth of $A$. If $\dep (A) =0$, 
the matrix $A$ is diagonal, and the statement is clearly true.
Now assume that $\dep (A) =m>0$ and the statement holds for all matrices of depth $<m$.
Set 
$$
A' = A - \sum_{1\leq i \leq n} a_{i, i+m}  (E^{i, i+m}_{\theta} - E^{i+1, i+m}_{\theta}).
$$
Let $B$ be the unique tridiagonal matrix in $\Xi_{n, d}$ such that 
$B-  \sum_{1\leq i\leq n} a_{i, i+m} E^{i, i+1}_{\theta} $ 
is diagonal and $\co(B) = \ro (A')$.
By Proposition ~\ref{new-10.1.2}, we have $[B]_d * [A']_d = [A]_d + \text{lower terms}$.
Now observing that $\dep (A') < m$, we complete the proof of   \eqref{[A]}  by induction.

The second formula \eqref{basis:fA} for $\f_{A; d}$ follows from \eqref{[A]} and Lemma ~\ref{fA-leading}. 
\end{proof}


\begin{cor}
The set $\{\f_{A; d} \vert A\in \Xi_{n, d}  \}$ forms a basis for $\Sj$
(called a {\em monomial basis}). 
\end{cor}

\begin{cor}
\label{cor:fa-standard}
The set
 $\{\f_{A; d} \vert A\in \Xi_{n, d} \text{ tridiagonal} \}$ (respectively,  $\{ [A]_d \vert A\in \Xi_{n, d} \text{ tridiagonal} \}$)
 forms a generating set for the algebra $\Sj$.
\end{cor}

\section{Stabilization of the Schur algebras}
 \label{sec:SSSc}

In this section, we study the stabilization of the multiplication and bar operator of  the Schur algebras $\Sj$.

Recall that 
$I_n = \sum_{1\leq i\leq n} E^{ii}_{n}$.
Recall the operation $\ddot \empty$ from \eqref{eq:dd} so that  
$\ddot I_n = I_{\breve n} - E^{1,1}_{\theta, \breve n}$.
We set 
\[
_{\ddot p} A = A + p \ddot I_n, \quad  \forall A \in \widetilde \Xi_{\breve n}.
\]

Introduce the algebra
$\mathscr R = \mbb Q(v) [ v', v'^{-1}]$ with a bar involution such that $\overline{v}=v^{-1}$ and $\overline{v'} =v'^{-1}$.
For $a \in \mathbb{Z}$ and $b \in \mathbb N$, we define  the following polynomials in $\mathscr R$:
\begin{equation*}  
\begin{bmatrix}
a\\
b
\end{bmatrix}_{v, v'}
=\prod_{1\leq i\leq b} \frac{v^{2(a-i+1)} v'^{-2} - 1}{v^{2i}-1}, \quad \text{ and } \quad [a]_{v, v'} = \begin{bmatrix}
a\\
1
\end{bmatrix}_{v, v'}.
\end{equation*}
For $0\leq i\leq \breve n-1$, $A\in \widetilde \Xi_{\breve n, d}$ with $a_{1,j}=0$ for all $j\in \mbb Z$  and $t=(t_u)_{u\in \mbb Z} \in \mbb N^{\mbb Z}$ such that $\sum_{j\in \mbb Z} t_u =R$, we define a polynomial  $Q_{i, R; A}^{ t}(v, v')$ in $\mathscr R$ as follows.
For any $i\in [0, \breve n-1]\backslash \{0, \breve r+1, 1, n+1\}$,  we define 
\begin{align}
\begin{split}
Q_{i, R; A}^{ t} (v, v') = 
v^{\beta_t}
\prod_{u \in \mbb Z, u\neq i}
\overline{\begin{bmatrix} a_{iu} + t_u \\ t_u \end{bmatrix}}
\cdot 
v'^{(\delta_{i, 1} + \delta_{i, \breve n - 1}) \sum_{i+1 \geq u} t_u}
\overline{\begin{bmatrix} a_{ii} + t_i \\ t_i \end{bmatrix}}_{v, v'},
\end{split}
\end{align}
where
\[
\beta_t =
 \sum_{j \geq  u} a_{i j} t_u - \sum_{j > u} a_{i+1, j} t_u + \sum_{j < u} t_j t_u
 + \frac{1} {2} (\delta_{i, \breve r} + \delta_{i, \breve n-1})   \left ( \sum_{j + u < 2(i+1)} t_j t_u + \sum_{j < i + 1} t_j\right ).
\]
For $i= 1$ or $n+1$, we define 
\begin{align}
\begin{split}
Q_{i, R; A}^{ t}(v, v') =
v^{\beta_t} 
\prod_{u \in \mbb Z, u\neq i}
\overline{\begin{bmatrix} a_{iu} + t_u \\ t_u \end{bmatrix}}
\cdot 
v'^{-  \sum_{i \geq u} t_u}.
\end{split}
\end{align}
For $i=0$ or  $\breve r+1$,  we define 
\begin{align}
\begin{split}
Q_{i, R; A}^{ t} (v, v') = v^{\beta'_t} 
\prod_{u > i}
\overline{\begin{bmatrix} a_{i u} + t_u + t_{2i -u} \\ t_u \end{bmatrix} }
\prod_{u < i}
\overline {\begin{bmatrix} a_{i u} + t_u \\ t_u \end{bmatrix}} 
\cdot
\prod_{u=0}^{t_i -1}
 \frac{\overline{[a_{ii} +1 + 2u]}_{v, v'} } { \overline{ [u +1]} },
\end{split}
\end{align}
where
\[
\beta'_t
=
\sum_{j \geq u} a_{ij} t_u - \sum_{j > u} a_{i+1, j} t_u
 + \sum_{j<u, j+u \leq 2i } t_j t_u -  \sum_{j> i} \frac{t_j^2 - t_j}{2}
+ \frac{R^2 - R}{2}.
\]

The following lemma follows directly from the definition.

\begin{lem}
\label{v'=1}
We have $Q_{i, R; \  \! _{\ddot p} A}^{t} (v, 1) = Q_{i, R; A}^{t}(v, v^{-p})$, for all $p\in 2\mbb Z$ and all admissible $i, t, R, A$.
\end{lem}

Given the same data $(i, A, t)$ as above, we define
\begin{equation}
 \label{AiRt}
A_{i, R, t} =  A + \sum_{u\in \mbb Z} t_u (E^{i, u}_{\theta, \breve n} - E^{i+1, u}_{\theta, \breve n}).
\end{equation}
It is  convenient  to introduce the following notations for later use.
\begin{align}
\label{con}
\f_{A; d} =[A]_d =0,  \quad \forall A\not \in \Xi_{ n, d}.
\end{align}
The following lemma describes the stabilization behavior of the multiplication formulas in $\mbf S^{\C}_{\breve n, d}$ after adding $p \ddot I_n$.

\begin{lem}
\label{4.2.3-stab}
Assume  $A, B \in \widetilde \Xi_{\breve n, d}$ and $R\in \mbb N$ satisfy the following properties:
$\ro (A) = \co (B)$, $B - R E^{i, i+1}_{\theta, \breve n}$ is diagonal for some $1 \leq i \leq \breve n$ and $a_{1,j} =0$ for all $j\in \mbb Z$. 
Then we have 
\[
[_{\ddot p} B]_{d+ \frac{p}{2} n}  *[_{\ddot p} A]_{d+\frac{p}{2} n} = \sum_{t} Q_{i, R; A}^{t} (v, v^{-p}) [_{\ddot p} A_{i,R, t}]_{d+\frac{p}{2}n}, 
\quad \forall  p \in 2 \mbb Z.
\] 
where the sum runs over all sequences $t=(t_u) \in \mbb N^{\mbb Z}$ 
such that $\sum_{u\in \mbb Z} t_u = R$ and  $A_{i,R, t} \in \widetilde \Xi_{\breve n, d}$, independent of $p$.
\end{lem}

\begin{proof}
We observe that the specialization $Q_{i,R;  A}^{ t}(v, 1)$ of $Q_{i, R; A}^{ t}(v, v')$ at $v'=1$ 
is exactly the structure constant of $[A_{i, R, t}]_d$ in the multiplication formulas in Proposition ~\ref{prop-stand-mult},
modulo some changes of indexes for $\breve r +1 \leq i \leq \breve n -1$.
The lemma follows then from Lemma ~\ref{v'=1} and the convention (\ref{con}).
\end{proof}

We shall  need a stronger version of   Lemma ~\ref{4.2.3-stab}.
Given tuples $\mbf i=(i_1, \ldots, i_s)$ and $\mbf a =(a_1, \ldots, a_s) \in \mbb N^s$, we introduce the notation
\[
\mbf i_{\geq l}=(i_l, i_{l+1}, \ldots, i_s), \quad 
\mbf a_{\geq l} = (a_l, a_{l+1}, \ldots, a_s), \quad \forall 1\leq l \leq s.
\]
Given a tuple $\mbf t = (t_1,\ldots, t_s)$ of sequences such that 
\begin{equation}
  \label{eq:tt}
\text{the $l$-th component $t_l=(t_{l, j})_{j\in \mbb Z} \in \mathbb N^{\ZZ}$ satisfies
$\sum_{j\in \mbb Z} t_{l, j} = a_l$ for all $1\leq l \leq s$
}
\end{equation}
 and a matrix $A \in \widetilde \Xi_{\breve n, d}$ such that $a_{1,j}=0$ for all $j\in \mbb Z$, 
we define inductively the matrix $A_{\mbf i, \mbf a, \mbf t}$ and  the polynomial $Q^{\mbf t}_{\mbf i, \mbf a; A} (v, v')$  in $\mathscr R$ 
via \eqref{AiRt} as follows:
\begin{align}
   \label{QiaA}
\begin{split}
A_{\mbf i, \mbf a, \mbf t} &= (A_{\mbf i_{\geq 2}, \mbf a_{\geq 2},\mbf t_{\geq 2}})_{i_1, a_1, t_1},\\
Q^{\mbf t}_{\mbf i, \mbf a; A} (v, v') 
&= Q_{i_1, a_1;  A_{\mbf i_{\geq 2}, \mbf a_{\geq 2}, \mbf t_{\geq 2}}}^{t_1} (v, v') \cdot 
Q_{\mbf i_{\geq 2}, \mbf a_{\geq 2}; A}^{\mbf t_{\geq 2}}(v, v').
\end{split}
\end{align}

By Lemma \ref{v'=1} and by induction on the length of $\mbf i$,   we have
\begin{align}
\label{v'-p}
Q^{\mbf t}_{\mbf i, \mbf a; A} (v, v^{-p}) = Q^{\mbf t}_{\mbf i, \mbf a; {}_{\ddot p} A} (v, 1).  
\end{align}

Given a pair $(\mbf i, \mbf a)$ and  $A \in \widetilde \Xi_{\breve n, d}$ such that $a_{1,j}=0$ for all $j\in \mbb Z$, we define the set
$\mathcal T_{\mbf i, \mbf a, A}$ to be 
the set of all tuples  $\mbf t =(t_1,\ldots, t_s)$ of sequences in $\mbb N^{\mbb Z}$ 
such that the $l$-th component $t_l=(t_{l, j})_{j\in \mbb Z}$ satisfies
$\sum_{j\in \mbb Z} t_{l, j} = a_j$ for all $1\leq j \leq s$, $A_{\mbf i_{\geq l}, \mbf a_{\geq l}, \mbf t_{\geq l}} \in \widetilde \Xi_{\breve n, d}$ for all 
$1\leq l \leq s$.
Clearly, we have $\mathcal T_{\mbf i, \mbf a, A} = \mathcal T_{\mbf i, \mbf a, {}_{\ddot p} A}$ for all $p$.

\begin{prop}
\label{4.2.3-stab-v3}
Assume  $A, B_j \in \widetilde \Xi_{\breve n, d}$, for all $1\leq j \leq s$ and pairs of tuples $(\mbf i, \mbf a)$ satisfy the following properties:
$\ro (A) = \co (B_s)$, $\ro (B_u) = \co (B_{u-1}), \forall  1< u \leq s$,  $B_u - a_u E^{i_u, i_u+1}_{\theta, \breve n}$ is diagonal  for $1\leq u\leq s$, 
and $a_{1,u} =0$ for all $j\in \mbb Z$.  Then we have 
\[
[_{\ddot p} B_1]_{d+ \frac{p}{2} n} * \cdots * [_{\ddot p} B_s]_{d+ \frac{p}{2} n}  *[_{\ddot p} A]_{d+\frac{p}{2} n} 
= \sum_{\mbf t \in \mathcal T_{\mbf i, \mbf a, A}} Q_{\mbf i, \mbf a; A}^{\mbf t} (v, v^{-p}) [_{\ddot p} A_{\mbf i,\mbf a, \mbf t}]_{d+\frac{p}{2} n}, 
\quad \forall  p \in 2 \mbb Z.
\] 
\end{prop}

\begin{proof}
Let $\mathcal T_{\mbf i, \mbf a, A ; d}$ be the subset of $\mathcal T_{\mbf i, \mbf a, A}$
consisting of all $\mbf t$ such that 
$A_{\mbf i_{\geq l}, \mbf a_{\geq l}, \mbf t_{\geq l}} \in \Xi_{\breve n, d}$ for all $1\leq l \leq s$, 
where $s$ is the length of $\mbf i$.
In view of Lemma ~\ref{4.2.3-stab}, the left-hand side of the equality in the lemma is equal to
\[
\sum_{\mbf t \in \mathcal T_{\mbf i, \mbf a, A; d+pn}} Q_{\mbf i, \mbf a; A}^{\mbf t} (v, v^{-p}) 
[_{\ddot p} A_{\mbf i,\mbf a, \mbf t}]_{d+\frac{p}{2} n}.
\]
It is reduced to showing that
if $_{\ddot p} A_{\mbf i, \mbf a, \mbf t} \in \Xi_{\breve n, d+\frac{p}{2}n}$ and 
$_{\ddot p} A_{\mbf i_{\geq l}, \mbf a_{\geq l}, \mbf t_{\geq l}} \not \in \Xi_{\breve n, d+\frac{p}{2}n}$ for some $l$, 
then the  structure constant of $[_{\ddot p} A_{\mbf i, \mbf a, \mbf t} ]_{d+\frac{p}{2} n}$  is zero.
In such a case, there is an $l_0$ such that 
$_{\ddot p}A_{\mbf i_{\geq l_0}, \mbf a_{\geq l_0}, \mbf t_{\geq l_0}} \not \in \Xi_{\breve n, d+\frac{p}{2}n}$
and $_{\ddot p} A_{\mbf i_{\geq l_0 + 1}, \mbf a_{\geq l_0 + 1}, \mbf t_{\geq l_0+ 1}}  \in \Xi_{\breve n, d+\frac{p}{2}n}$;
this implies that  the $i_{l_0+1} $-th diagonal entry of $_{\ddot p}A_{\mbf i_{\geq l_0}, \mbf a_{\geq l_0}, \mbf t_{\geq l_0}}$
is negative, while nonnegative after adding the $i_{l_0+1}$-th entry of the tuple $t_{l_0+1}$.
The latter condition further yields that 
the  factor 
$Q^{t_{l_0 + 1}}_{i_{l_0+1}, a_{l_0+1}; {}_{\ddot p} A_{\mbf i_{\geq l_0 }, \mbf a_{\geq l_0}, \mbf t_{\geq l_0}}} (v, 1)$
of $Q^{\mbf t}_{\mbf i, \mbf a;  {}_{\ddot p}A}(v, 1) $, and hence itself,  is zero (see ~\cite[Lemma A.20]{BLW14}).
Now the proposition follows by applying (\ref{v'-p}).
\end{proof}

Now we discuss the stabilization of $\Sj$. Recall that $_p A = A + p I_n$.
The following proposition 
describes the relationship between the standard basis elements $[A]_d$ and 
the elements $\f_{A; d}$ under the stabilization with respect to $p I_n$.
Note that the partial orders $\leq_{\text{alg}}$ and $\leq$ on $\Xi_{n, d}$ can be defined 
on $\tilde \Xi_{n, d}$ as well in exactly the same way.

\begin{prop}
\label{prop-A-f}
Let $A\in \widetilde \Xi_{n, d}$.
There exist $Z_i \in \widetilde \Xi_{n, d}$, for $1\leq i\leq m$,  with $Z_i < A$, $Q_i (v, v') \in \mathscr R$ 
and $p_0\in \mbb N$ such that
\begin{align}
\label{A-f}
[_p A]_{d+\frac{p}{2} n} = \f_{_p A; d + \frac{p}{2} n} + \sum_{i=1}^m Q_i(v, v^{-p}) [_pZ_i]_{d+\frac{p}{2} n}, 
\quad \forall p \geq p_0, p\in 2\mbb N.
\end{align}
\end{prop}

\begin{proof}
We transport the statement via $\rho$ to a similar one for $\ddot \f_{_{p}A; d+\frac{p}{2} n}$ in $\mbf S^{\C}_{\breve n, d+\frac{p}{2}n}$. 
The existence of $p_0$, $Z_i$ and $Q_i(v, v')$ follows by Proposition ~\ref{4.2.3-stab-v3}.
The claim on the leading term  follows from  Theorem~\ref{new-10.1.3}.
\end{proof}

Now we can formulate the stabilization of the multiplication of $\Sj$.

\begin{prop}
\label{new-10.2.1}
Assume that  $A_1, \ldots, A_l \in \widetilde \Xi_{n, d}$ satisfy $\co (A_i) = \ro (A_{i+1})$ for all $1\leq i\leq l-1$.
There exist $Z_1, \ldots, Z_m \in \widetilde \Xi_{n, d}$, $G_1(v,v'), \ldots, G_m(v, v') \in \mathscr R$, and $p_0\in \mbb N$
such that 
\begin{align}
\label{10.2.1-a}
[_p A_1]_{d+\frac{p}{2} n} * [_p A_2]_{d+\frac{p}{2} n} * \cdots * [_p A_l]_{d+\frac{p}{2} n} 
= \sum_{i=1}^m G_i(v, v^{-p}) [_p Z_i]_{d+\frac{p}{2} n}, 
\quad  \forall p\geq p_0, p\in 2\mbb N.
\end{align}
\end{prop}

\begin{proof}
By Proposition ~\ref{4.2.3-stab-v3}, we have a formula similar to (\ref{10.2.1-a})  with $_p A_i$ replaced by $\f_{_pA_i; d+\frac{p}{2}n}$.
The proposition now follows by using Proposition ~\ref{prop-A-f} 
and an induction with respect to the partial order $\leq$ on the $A_i$'s.
\end{proof}

We have the following corollary to Theorem ~\ref{new-10.1.3} and Proposition ~\ref{new-10.2.1}. 

\begin{cor}
\label{K-A-leading}
For any matrix  $A\in \widetilde \Xi_{n, d}$ of depth $m$, there exist 
tridiagonal matrices $A_1$, $A_2, \ldots, A_m$  
in $\widetilde \Xi_{n, d}$
satisfying $\ro (A_m) = \ro (A)$, $\co (A_1)= \co (A)$, $\ro (A_i) = \co (A_{i+1})$ for $1\leq i\leq m-1$ 
and 
$A_i - \sum_{1 \leq j \leq n} (\sum_{k \leq j + i -1} a_{k, j+1}) E^{j, j+1}_{\theta}$ is  diagonal 
for all $1\leq i \leq m$ such that
\[
[_pA_m]_{d+\frac{p}{2} n}  *  [_pA_{m-1}]_{d+\frac{p}{2} n}   * \dots * [_pA_1]_{d+\frac{p}{2}n} = [_pA]_{d+\frac{p}{2} n} +\sum_{i=1}^l G_i(v, v^{-p}) [_p Z_i]_{d+ \frac{p}{2} n}, \ \forall p\in 2\mbb N, p\geq p_0,
\]
where $p_0$,  $G_i(v, v')\in \mathscr R$ and $Z_1,\ldots, Z_l \in \widetilde \Xi_{n, d}$ are as in Proposition ~\ref{new-10.2.1} such that $Z_i < A$.
\end{cor}

The following stabilization of the bar operator on $\Sj$ is a counterpart of  \cite[Proposition~4.3]{BLM90}.
It can be proved in the same way by induction with respect to the partial order $\leq$ on $A$, with the help of \eqref{A-f} and Corollary ~\ref{K-A-leading};
we skip the detail. 

\begin{prop}
\label{bar-stab}
Assume that $A \in \widetilde \Xi_{n, d}$. 
Then there exist $Y_i \in \widetilde \Xi_{n, d}$ with $Y_i < A$, $H_i (v, v') \in \mathscr R$ 
for all  $1\leq i \leq s$ and $p_0\in \mbb N$ such  that
\begin{align}
\overline{[_p A]}_{d+\frac{p}{2}n}  
= [_p A]_{d+\frac{p}{2}n}  + \sum_{i=1}^s H_i(v, v^{-p}) [_pY_i]_{d+\frac{p}{2} n}, \quad \forall p \geq p_0, p\in 2 \mbb N.
\end{align}
\end{prop}

%

\section{Comultiplication and stabilization} 
 
In the section, we take advantage of the embedding $\rho: \Sj \rightarrow \U^{\C}_{\breve n, d}$ to study the coassociativity 
and stability behavior of the comultiplication $\Delta^{\C}: {\mbf S}_{n,d}^{\fc} \longrightarrow  {\mbf S}_{n,d'}^{\fc} \otimes  {\mbf S}_{n,d''}$ 
(recall $\Delta^{\C}$ was defined in \eqref{Dj}). 

To avoid any ambiguity,  we put a subscript $n$ to the comultiplication $\Delta^{\C}$ of $\Sj$, 
and use  $\Delta^{\C}_{\breve n}$ for that on $\mbf S^{\C}_{\breve n, d}$.
We apply the same convention to the imbedding $\rho$ too.
Note that exactly the same definition gives rise to an imbedding $\mbf S_{n, d} \to \mbf S_{\breve n, d}$, 
which we shall again denote  by $\rho_d$.
The following lemma shows the compatibility of the comultiplications and the imbedding $\rho$.

\begin{lem}
\label{rho}
The following diagram is commutative:
\[
\begin{CD}
\Sj @>\Delta^{\C}_n >> \mbf S^{\C}_{n, d'} \otimes \mbf S_{n, d''} \\
@V \rho_d VV @VV\rho_{d'} \otimes \rho_{d''} V\\
\U^{\C}_{\breve n, d} @>  \Delta^{\C}_{\breve n} >>  \U^{\C}_{\breve n, d'} \otimes \U_{\breve n, d''}
\end{CD}
\]
\end{lem}

\begin{proof}
By definitions, we have a similar commutative diagram with the  $\Delta^{\C}$'s replaced by the raw ones $\widetilde \Delta^{\C}$; cf. \eqref{tDj}.
Now the twists $s(\mbf b', \mbf  a', \mbf b'', \mbf a'')$ and $u(\mbf b'', \mbf a'')$ remain unchanged under the obvious imbeddings
$\Lambda_{n, d} \to \Lambda_{\breve n, d}$ and $\Lambda^{\C}_{n, d} \to \Lambda^{\C}_{\breve n, d}$. This immediately shows that the commutative diagram for $\widetilde \Delta^{\C}$'s can be extended to the one in the lemma. 
\end{proof}

\begin{prop}
\label{Sj-coass}
The comultiplication  $\Delta^{\C}$  on $\Sj$ is coassociative, that is, 
$$(1\otimes \Delta) \Delta^{\C} = (\Delta^{\C} \otimes 1) \Delta^{\C}.$$
\end{prop}

\begin{proof}
By Lemma ~\ref{rho}, 
this is a consequence of the fact that the restriction of $\Delta^{\C}_{\breve n}$ to $\U^{\C}_{\breve n, d}$ is coassociative in 
Proposition ~\ref{coassociative}.
\end{proof}

\begin{rem}
 \label{S-coass}
Recall the comultiplication $\Delta$ on $\mbf S_{n, d}$ of affine type $A$ from (\ref{A-coass}). 
It follows by the same argument as above that
the comultiplication $\Delta$ on $\mbf S_{n, d}$ is coassociative, that is,
$(1\otimes \Delta) \Delta = (\Delta\otimes 1) \Delta$.
\end{rem}

Now, we  study the stabilization behavior of the comultiplication $\Delta^{\C}_n$ (cf. \eqref{Dj}) as $d$ varies.
Recall the notation $\models$ from Section ~\ref{another}. We generalize it as follows.
For any $\lambda'$, $\lambda$ in $\mbb Z_n^{\C}$ and $\lambda'' \in \mbb Z_n$, we say that
$(\lambda', \lambda'') \models \lambda$ if $\lambda_i = \lambda'_i + \lambda'_{-i} +\lambda''_i + \delta_{i, 0} + \delta_{i, r+1}$
for all $i\in \mbb Z$. 
Let
\[
\Delta^{\C}_{\mbf b', \mbf a', \mbf b'', \mbf a''} : \Sj(\mbf b, \mbf a) \longrightarrow \Sj \overset{\Delta^{\C}}{\longrightarrow} \mbf S^{\C}_{n, d'} \otimes \mbf S_{n, d''}
\longrightarrow
\Sj(\mbf b', \mbf a') \otimes \mbf S_{n, d''}(\mbf b'', \mbf a'')
\]
be a component of $\Delta^{\C}$ with $(\mbf b', \mbf b'') \models \mbf b$, $(\mbf a', \mbf a'') \models \mbf a$, 
where the first and third maps are the natural inclusion and projection, respectively.
Similar to the notation $_p A$, we can define the notation $_p \mbf a = \mbf a + ( \ldots, p, p, p, \ldots)$. 
We put 
\[
{}_p \Delta^{\C}_{\mbf b', \mbf a', \mbf b'', \mbf a''} =\Delta^{\C}_{_p\mbf b', {}_p \mbf a', {}_p \mbf b'',  {}_p\mbf a''}.
\]

\begin{prop}
\label{cop-stab}
Assume that $d'+d''=d$ and  let $\mbf b', \mbf a'  \in \mbb Z^{\C}_n $ and $\mbf b'', \mbf a'' \in \mbb Z_n$ be so that 
$_p\Delta^{\C}_{\mbf b', \mbf a', \mbf b'', \mbf a''}$ is defined.
Fix $A\in \widetilde \Xi_{n,d}$.
There exist  matrices $A'_i \in \widetilde \Xi_{n, d'}$ where $1\leq i\leq l$ for some $l$,  
matrices $A''_j$ in $\widetilde{\Theta}_{n, d''}$  where $1\leq j \leq m$ for some $m$,
$C_{i,j}(v, v') \in \mathscr R$ for $1\leq i \leq l$, $1\leq j\leq m$, and $p_0\in \mbb N$ such that 
\[
_{p}\Delta^{\C}_{\mbf b', \mbf a', \mbf b'', \mbf a''} ([_{2p} A]_{d + pn}) 
= \sum_{1\leq i\leq l, 1\leq j\leq m} 
C_{i, j} (v, v^{-p}) [_p A'_i]_{d'+\frac{p}{2} n} \otimes \ ^{\mathfrak a}[_{p}A''_j]_{d''+pn}, 
\quad \forall p \geq p_0, p\in 2\mbb N.
\]
\end{prop}

\begin{proof}
We prove this by induction with respect to the partial order on $A$. By Proposition ~\ref{prop-A-f},  we have
\[
[_{2p} A]_{d+pn} = \f_{_{2p} A; d+pn} + \sum_{i=1}^m Q_i(v, v^{-2p}) [_{2p}Z_i]_{d+pn}, \quad \forall p \geq p_0, p\in 2\mbb N.
\]
If we define $\tilde Q_i (v, v') = Q_i(v, v'^2)$ for all $i$, then we can rewrite the above equality as 
\[
[_{2p} A]_{d+pn} = \f_{_{2p} A; d+pn} + \sum_{i=1}^m \tilde Q_i(v, v^{- p}) [_{2p}Z_i]_{d+pn}, \quad \forall p \geq p_0, p\in 2\mbb N.
\]
With this equality and by induction,
it is reduced to proving a similar  statement with $[A]_d$ replaced by $\f_{A; d}$. 
By Lemma ~\ref{rho}, this is in turn reduced to proving a similar result for $\ddot \f_{A; d}$, which is then a consequence of 
Proposition ~\ref{4.2.3-stab-v3} and Proposition ~\ref{Dc}.
The proposition follows.
\end{proof}

\section{The algebra $\K^{\C}_n$ and its stably canonical basis}
\label{Kc}

Let $\K^{\C}_n$ be the vector space over $\mbb Q(v)$ spanned by the formal symbols $[A]$ where $A \in \widetilde \Xi_n$. 
By Proposition ~\ref{new-10.2.1} and applying a standard argument, 
the space $\K^{\C}_n$ becomes an associative algebra without unit  with the product
\begin{align}
\label{K-mult}
[A_1] \cdot [ A_2]  = \sum_{i=1}^m G_i(v, 1) [ Z_i], \quad \forall A_1, A_2 \in \widetilde \Xi_n,
\end{align}
where $G_i(v, v')$ and $Z_i$ for all $1\leq i\leq m$ are as in Proposition ~\ref{new-10.2.1}.
Corollary ~\ref{K-A-leading} implies the following. 

\begin{prop}
\label{new-11.2.1}
For any matrix  $A$  in $\widetilde \Xi_{n}$ of depth $m$, there exist tridiagonal matrices $A_1$, $A_2, \ldots, A_m \in \widetilde \Xi_n$
satisfying $\ro (A_m) = \ro (A)$, $\co (A_1)= \co (A)$, $\ro (A_i) = \co (A_{i+1})$ for $1\leq i\leq m-1$ 
and 
$A_i - \sum_{1 \leq j \leq n} (\sum_{k \leq j + i -1} a_{k, j+1}) E^{j, j+1}_{\theta}$ is  diagonal 
for all $1\leq i \leq m$ such that
\[
m'_A := [A_m]  \cdot  [A_{m-1}]   \cdot . . .  \cdot  [A_1]= [A] + \mbox{lower terms}.
\]
\end{prop}
Thus $\{m'_A | A\in  \widetilde \Xi_n\}$ forms a basis for $\K_n^\C$ (called {\em a semi-monomial basis}). 
Notice that the element  $m'_A$ is not necessarily bar-invariant.

For each matrix $A\in \widetilde \Xi_n$, we define the element $\f_A \in \K^{\C}_n$  to be 
\[
\f_A = [A] + \sum_{i=1}^m Q_i(v, 1) [Z_i],
\]
where $Q_i(v, v')$ and $Z_i$ are in (\ref{A-f}). 
In particular, we have
\[
\f_A   = [A] + \mbox{lower terms}.
\]
 
Moreover, we can give a more precise description of $Q_i(v, v')$ and $Z_i$.
By the definition of $\f_A$ in  \eqref{basis:fA} (also see \eqref{eq:fA})
and Proposition ~\ref{4.2.3-stab-v3}, we have the following. 

\begin{prop}
For any matrix $A \in \widetilde \Xi_{n, d}$, there exists a pair of tuples $(\mbf i, \mbf a)$ such that 
\begin{align}
\label{fa-explicit}
\f_A = \sum_{\mbf t \in \mathcal T_{\mbf i, \mbf a, D_{\co (\ddot A)}}} Q^{\mbf t}_{\mbf i, \mbf a; D_{\co (\ddot A)}} (v, 1) 
\big[ \mrm{dlt}_1 (D_{\co (\ddot A)})_{\mbf i, \mbf a, \mbf t} \big],
\end{align}
where $D_{\co (\ddot A)}$ is the diagonal matrix in $\widetilde \Xi_{\breve n, d}$ with diagonal $\co (\ddot A)$ 
and $\mrm{dlt}_1$ is the deleting operation inverse to the operation $\ddot {\phantom{x}}$.
\end{prop}

Assume that $B\in \widetilde \Xi_{n, d}$ and $B - \sum_{1\leq i\leq n} \beta_i E^{i, i+1}_{\theta}$ is diagonal. 
Let $\mbf i_0 $ and $\mbf b_0$ denote the sequences of subscripts and superscripts 
in the left hand side of (\ref{fA}) (with $A$ replaced by $B$), respectively, that is,
\[
\mbf i_0 =( 0, n, n+1, n-1, n, n-2,  \cdots, 1), \quad 
\mbf b_0 = ( \beta_0, \beta_{n-1}, \beta_{n-1}, \beta_{n-2}, \cdots, \beta_0),
\] 
where $\beta_0=\beta_n$
The following multiplication formula in $\K^\C_n$ follows by Proposition ~\ref{4.2.3-stab-v3}.

\begin{prop}
Let $A, B \in \widetilde \Xi_n$ be such that $\co (B) = \ro (A)$ and $B - \sum_{1\leq i\leq n} \beta_i E^{i, i+1}_{\theta}$ is diagonal.
Then the following multiplication formula holds in $\K^\C_n$:
\begin{align}
\f_B \cdot [A] 
= \sum_{\mbf t \in \mathcal T_{\mbf i_0, \mbf b_0, \ddot A}} 
Q^{\mbf t}_{\mbf i_0, \mbf b_0; \ddot A} (v, 1) \big[ \mrm{dlt}_1 (\ddot A)_{\mbf i_0, \mbf b_0, \mbf t} \big].
\end{align}
\end{prop} 

Now we show that the element $\f_A \in \K^\C_n$ can be expressed as a monomial in $\f_{A_i}$ for various tridiagonal matrices $A_i$ 
(similar to the Schur algebra case).
 
\begin{prop}
\label{fA-monomial}
Let $A\in \widetilde \Xi_n$, and we retain the notations of tridiagonal matrices $A_i$ from Proposition ~\ref{new-11.2.1}. Then we have 
\[
\f_A  = \f_{A_m}  \cdot  \f_{A_{m-1}}   \cdot . . .  \cdot  \f_{A_1}.
\]
Moreover, we have $\overline{\f_A} = \f_A$. 
\end{prop}

\begin{proof}
Let $\K^{\C}_{n, \mathscr R}$ be the free $\mathscr R$-module  spanned by the matrices in $\widetilde \Xi_n$.
Similar to (\ref{K-mult}), we can define an associative algebra over $\mathscr R$ by
\[
A_1 \cdot'  A_2  = \sum_{i=1}^m G_i(v, v') \  Z_i, \quad \forall A_1, A_2 \in \widetilde \Xi_n.
\]
Similar to (\ref{fa-explicit}), we can define 
\begin{align}
\label{f'A}
\f'_A = \sum_{\mbf t \in \mathcal T_{\mbf i, \mbf a, D_{\co (\ddot A)}}} Q^{\mbf t}_{\mbf i, \mbf a; D_{\co (\ddot A)}} (v, v') \
\mrm{dlt}_1 (D_{\co (\ddot A)})_{\mbf i, \mbf a, \mbf t}.
\end{align}
Then by Proposition ~\ref{4.2.3-stab-v3}, we have
\[
\f'_A  = \f'_{A_m}  \cdot'  \f'_{A_{m-1}}   \cdot' . . .  \cdot'  \f'_{A_1}.
\]
By specializing $v'$ at $v'=1$, we obtain the product formula for $\f_A$.

The bar invariance of $\f_A$ follows 
from the same fact on the Schur algebra level and the formal stabilization procedure as above. We skip the detail. 
\end{proof}

By Proposition ~\ref{bar-stab}, we  can define a bar involution on $\K^{\C}_n$ by $\bar v = v^{-1}$ and letting
\[
\overline{[A]} = [A] + \sum_{i=1}^s H_i(v, 1) [Y_i], \quad \forall A\in \widetilde \Xi_n,
\]
where $H_i(v, v')$ and $Y_i < A$ are as in Proposition ~\ref{bar-stab}.
The next proposition follows by a standard argument.

\begin{prop}
\label{stable}
For any $A \in \widetilde \Xi_{n}$, there exists a unique element $\{A\}$ in $\K^{\C}_n$ such that 
\[
\overline{\{A\}} = \{A\}, 
\quad
\{A\} = [A] + \sum_{A'<A} \pi_{A, A'}[A'], \quad \pi_{A, A'}\in v^{-1} \mbb Z [v^{-1}].
\]
Moreover,  $\{ \{A\} \vert A \in \widetilde \Xi_{n} \}$ forms a basis for $\K^\C_n$ (called the {\em stably canonical basis}).
\end{prop}

Let us summarize the main results of this section. 

\begin{thm}
\label{Kj-bases}
The algebra $\K^{\C}_n$ admits a standard basis $\{[A] | A\in \widetilde \Xi_n\}$, 
a semi-monomial basis $\{m'_A | A\in  \widetilde \Xi_n\}$, 
a monomial basis $\{ \f_A | A\in \widetilde \Xi_n\}$,
and a stably canonical basis
$\{\{A \}| A\in  \widetilde \Xi_n\}$.
\end{thm}

\section{The  algebra $\K_n$ of affine type $A$ and its comultiplication}

In this section, we revisit the Schur algebras of affine type $A$ and study its stabilization algebra $\K_n$.
The  constructions  in this section will serve as 
a prerequisites for  the constructions  of  the comultiplication of  the algebra $\K^{\C}_n$ in the following section.

Recall the comultiplication $\Delta$ from (\ref{A-coass}) of affine type $A$. 
The following stabilization for the comultiplication $\Delta$  at the Schur algebra level is the counterpart of
Proposition~\ref{cop-stab} which can be proved in the same way.

\begin{prop}
\label{cop-stab-A}
Assume that $d'+d''=d$, and let $\mbf b', \mbf a' , \mbf b'', \mbf a'' \in \mbb Z_n$ be so that 
$_p \Delta_{\mbf b', \mbf a', \mbf b'', \mbf a''}$ is defined.
For each $A\in \widetilde{\Theta}_{n,d}$,
there exist  $A'_i \in \widetilde{\Theta}_{n, d'}$ where $1\leq i\leq l$ for some $l$,  
$A''_j \in \widetilde{\Theta}_{n, d''}$  where $1\leq j \leq m$ for some $m$,
$^{\mathfrak a}C_{i,j}(v, v') \in \mathscr R$ for $1\leq i \leq l$, $1\leq j\leq m$, and $p_0\in \mbb N$ such that 
\[
_{p}\Delta_{\mbf b', \mbf a', \mbf b'', \mbf a''} (^{\mathfrak a}[_{2p} A]_{d+2pn}) 
= \sum_{1\leq i\leq l, 1\leq j\leq m} 
\ ^{\mathfrak a}C_{i, j} (v, v^{-p}) \ ^{\mathfrak a} [_p A'_i]_{d'+pn} \otimes \ ^{\mathfrak a}[_{p}A''_j]_{d''+pn}, 
\quad \forall p \geq p_0.
\]
\end{prop}

Let  
\[
\widetilde{\Theta}_n = \{ A=(a_{ij})_{i, j\in \mbb Z} | a_{ij} \in \mbb N, \forall i\neq j, a_{ii}\in \mbb Z, \forall i\in \mbb Z\}.
\]
Let $\K_n$ be the vector space over $\mbb Q(v)$ spanned by the symbols $^{\mathfrak a}[A]$ for all $A \in \widetilde \Theta_n$.
Replacing $\Sj$ by the Schur algebra $\mbf S_{n,d}$ from Chapter \ref{chap:A} and repeating the  constructions in the preceding sections, 
we can endow $\K_n$ with an associative algebra structure, a bar involution, 
a canonical basis $\{ ^{\mathfrak a}\{A\} | A\in \widetilde \Theta_n\}$.
Indeed the treatment is much simpler in the current type $A$ setting since the analogous basis elements $^{\mathfrak a} \f_A$ and $^{\mathfrak a}[A]$ 
coincides when $A$ is tridiagonal.

\begin{rem}
 \label{rem:gln}
The associative algebra structure on $\K_n$ and its stably canonical basis were first introduced in \cite{DF13} by a completely different and  Hecke algebraic approach,
(also see \cite{LL15}).
Moreover, they showed that $\K_n$ is isomorphic to the idempotented quantum affine $\gl_n$, $\dot\U (\glh_n)$. 
\end{rem}

Moreover, by Proposition ~\ref{cop-stab-A}, we can define a comultiplication for $\K_n$ as follows.
Let ${}_{\mbf b}\K_{\mbf a}$, for any $\mbf b, \mbf a\in \mbb Z_n$,  be the subspace of $\K_n$ spanned by the standard basis elements $^{\mathfrak a}[A]$ such 
that $\ro(A) = \mbf b$ and $\co (A) = \mbf a$.
For any $\mbf b, \mbf a, \mbf b', \mbf a', \mbf b'', \mbf a'' \in \mbb Z_n$ such that 
$\mbf b' + \mbf b'' = \mbf b$ and $\mbf a' + \mbf a'' =\mbf a$, we define a linear map
\begin{align*}
\Delta_{\mbf b', \mbf a', \mbf b'', \mbf a''}  &:  {}_{\mbf b} \K_{\mbf a} \longrightarrow  {}_{\mbf b'}\K_{\mbf a'} \otimes {}_{\mbf b''}\K_{\mbf a''},
\\
\Delta_{\mbf b', \mbf a', \mbf b'', \mbf a''} ( ^{\mathfrak a} [A] ) 
& = \sum_{1\leq i\leq l, 1\leq j \leq m} \ ^{\mathfrak a} C_{i, j} (v, 1) \ ^{\mathfrak a}[A'_i] \otimes \ ^{\mathfrak a} [A''_j],
\end{align*}
where $A'_i, A''_j$, $^{\mathfrak a} C_{ij}(v, v')$ are from Proposition ~\ref{cop-stab-A}.
We shall call the collection 
\[
\dot \Delta = (\Delta_{\mbf b', \mbf a', \mbf b'', \mbf a''})_{\mbf b', \mbf a', \mbf b'', \mbf a'' \in \mbb Z_n}
\]
the $comultiplication$ of $\K_n$.
Let ${}^{\mathfrak a} g^C_{A, B}$ and ${}^{\mathfrak a} h_A^{B, C}$ denote  
the structure constants with respect to the multiplication and comultiplication in $\K_n$, respectively, i.e., 
\begin{align*}
^{\mathfrak a}[A] \cdot {}^{\mathfrak a} [B] 
&= \sum_{C \in \widetilde \Theta_n} {}^{\mathfrak a} g^C_{A, B} {}^{\mathfrak a} [C],
\\
 \Delta_{\mbf b', \mbf a', \mbf b'', \mbf a''}  (^{\mathfrak a} [A])
 &=\sum_{B, C \in \widetilde \Theta_n } {}^{\mathfrak a} h^{B, C}_A \ {}^{\mathfrak a} [B] \otimes {}^{\mathfrak a} [C].
 \end{align*} 

\begin{prop}
\label{prop-cop-alg-hom-A}
The comultiplication $\dot \Delta$  is an algebra homomorphism in the following sense:
for all matrices $A, B, C', C'' \in \widetilde \Theta_n$,  one has
\begin{align}
\label{cop-alg-hom-A}
\sum_{C \in \widetilde \Theta_n} {}^{\mathfrak a} g^C_{A, B}  {}^{\mathfrak a} h^{C', C''}_C
= \sum_{ A', A'', B', B'' \in \widetilde \Theta_n} {}^{\mathfrak a} h^{A', A''}_A {}^{\mathfrak a} h^{B', B''}_B  {}^{\mathfrak a} g_{A', B'}^{C'} 
{}^{\mathfrak a} g_{A'', B''}^{C''}.
\end{align}
\end{prop}

\begin{proof}
We first show that the sums in the two sides  of the equation (\ref{cop-alg-hom-A}) are finite.
For two fixed matrices $A, B$ in $\widetilde \Theta_n$, 
there are only finitely many $C \in \widetilde \Theta_n$ such that $g^C_{A, B}\neq 0$ by definition.
So  the sum on the left-hand side  is finite.
To see that the sum on the righthand side of  (\ref{cop-alg-hom-A}) is finite,
we first observe that for each $A \in \widetilde \Theta_n$, 
if the structure constant  $^{\mathfrak a}h_{A}^{A', A''}$ is nonzero, then $A', A'' \leq_{\text{alg}}  A$.  
Next we observe that for $\mbf c', \mbf c'' \in \mbb Z_n$,  
 the set $\{ (A', A'') | {}^{\mathfrak a} h_A^{A', A''} \neq 0, \ro (A') =\mbf c', \co (A'') = \mbf c'' \} $ is finite. 
This is because  if $A', A'' \leq_{\text{alg}} A$, then the number of the choices for the entry $(i, j)$ for $i\neq j$ of $A'$ and $A''$ is finite. 
Now the  row and column vectors of $A'$ and $A''$ are fixed respectively, 
forcing the choice of  the diagonal entries of $A'$ and $A''$ to be finite.
In the sum of the right-hand side of (\ref{cop-alg-hom-A}), we must have 
that $\ro(A') = \ro (C')$, $\co (B') = \co (C')$,  $\ro (A'') = \ro (C'')$ and $\co (B'') = \co (C'')$, which are fixed.  
So the sum on the righthand side of  (\ref{cop-alg-hom-A}) is indeed finite.

Once we observe that both sums in (\ref{cop-alg-hom-A}) are finite, 
the proof of the equation is reduced to showing a similar equation on the level of the Schur algebra $\mbf S_{n, d}$ for very large $d$, 
which is in turn equivalent to the fact that the comultiplication $\Delta$ on $\mbf S_{n, d}$ is an algebra homomorphism in ~\cite{FL15}.
The proposition is thus proved.
\end{proof} 

Proposition ~\ref{prop-cop-alg-hom-A} can be equivalently reformulated as the following commutative diagram:
for all 
tuples $\mbf a, \mbf a', \mbf a'', \mbf b, \mbf b', \mbf b'' , \mbf c \in \mbb Z_n$ 
such that $\mbf b' +  \mbf b'' =\mbf b$ and $\mbf a' + \mbf a'' =\mbf a$, we have
\begin{align}
\label{Diag-cop-alg-hom-A}
\xymatrix{ 
 & {}_{\mbf b} \K_{\mbf a} \ar@<0ex>[dr]^{\Delta_{\mbf b', \mbf a', \mbf b'', \mbf a''}} &   && \\
%
%
_{\mbf b} \K_{\mbf c} \otimes {}_{\mbf c} \K_{ \mbf a}  \ar@<0ex>[ur]^{m}  \ar@<0ex>[d]_{\prod \Delta\otimes \Delta}   
  & & 
{}_{\mbf b'} \K_{\mbf a'} \otimes {}_{\mbf b''} \K_{\mbf a''}  \\
\prod
{}_{\mbf b'} \K_{\mbf c'} \otimes {}_{\mbf b''}\K_{\mbf c''} \otimes {}_{\mbf c'} \K_{\mbf a'} \otimes {}_{\mbf c''}\K_{\mbf a''} 
\ar@<0ex>[rr]^{P_{23}}
 &&  
\prod
{}_{\mbf b'} \K_{\mbf c'} \otimes {}_{\mbf c'} \K_{\mbf a'} \otimes {}_{\mbf b''}\K_{\mbf c''}  \otimes {}_{\mbf c''}\K_{\mbf a''} 
\ar@<0ex>[u]_{m \otimes m}
}
\end{align}
where $m$ represents the multiplication of $\K_n$, all products run  over all 
tuples $(\mbf c',  \mbf c'')$ such that  $\mbf c' + \mbf c'' =\mbf c$, 
$\prod \Delta\otimes \Delta$ stands for the product of 
$\Delta_{\mbf b', \mbf c', \mbf b'', \mbf c''} \otimes \Delta_{\mbf c', \mbf a', \mbf c'', \mbf a''}$
and $P_{23}$ permutes the second and third entries.

\begin{prop}
\label{prop-cop-coass-A}
The comultiplication $\dot \Delta$ is coassociative in the following sense:
for any matrices $A, A', A'', A''' \in \widetilde \Theta_n$, we have
\begin{align}
\label{cop-coass-A}
\sum_{B\in \widetilde \Theta_n} {}^{\mathfrak a} h^{B, A'''}_A {}^{\mathfrak a} h^{A' ,  A''}_B 
=
\sum_{B\in \widetilde \Theta_n} {}^{\mathfrak a} h^{A', B}_A \  {}^{\mathfrak a} h^{A'', A'''}_B.
\end{align}
\end{prop}

\begin{proof}
By arguing in a similar way as in the proof of Proposition ~\ref{prop-cop-alg-hom-A}, we see that both sums in  (\ref{cop-coass-A}) are finite. 
The equality can then be reduced to proving a similar equation on the Schur algebra level  
as in the proof of Proposition ~\ref{prop-cop-alg-hom-A}, 
which in turn follows by the coassociativity in  Remark~\ref{S-coass}.
\end{proof}

Proposition ~\ref{prop-cop-coass-A} can be equivalently reformulated as the following commutative diagram:
for all sequences $\mbf a, \mbf a', \mbf a'', \mbf a''', \mbf b, \mbf b', \mbf b'', \mbf b''' \in \mbb Z_n$, we have 
\begin{align}
\label{Diag-cop-coass-A}
\begin{CD}
{}_{\mbf b} \K_{\mbf a} 
@>\Delta_{\mbf b'+\mbf b'', \mbf a'+\mbf a'', \mbf b''', \mbf a'''} >> 
{}_{\mbf b'+\mbf b''} \K_{\mbf a'+\mbf a''} \otimes {}_{\mbf b'''}\K_{\mbf a'''} \\
@V\Delta_{\mbf b', \mbf a', \mbf b''+\mbf b''', \mbf a''+\mbf a'''}VV @VV\Delta_{\mbf b', \mbf a', \mbf b'', \mbf a''}\otimes 1V \\
{}_{\mbf b'}\K_{\mbf a'} \otimes {}_{\mbf b'' + \mbf b'''} \K_{\mbf a'' +\mbf a'''} 
@>>1\otimes \Delta_{\mbf b'', \mbf a'', \mbf b''', \mbf a'''}> 
{}_{\mbf b'}\K_{\mbf a'} \otimes {}_{\mbf b''}\K_{\mbf a''} \otimes {}_{\mbf b'''}\K_{\mbf a'''}
\end{CD}
\end{align}

\section{The comultiplication on $\K^{\C}_n$}

Recall $\K^{\C}_n$ from Section ~\ref{Kc}.
For any $\mbf a, \mbf b\in \mbb Z^{\C}_n$, let $_{\mbf b}\K^{\C}_{\mbf a}$ denote the subspace of $\K^{\C}_n$ spanned by
the standard basis element $[A]$ such that $\ro(A)=\mbf b$ and $\co (A) =\mbf a$.
For any  $\mbf b, \mbf a, \mbf b', \mbf a' \in \mbb Z_n^{\C}$ and $\mbf b'', \mbf a'' \in \mbb Z_n$ such that  
$(\mbf b', \mbf b'') \models \mbf b$ and $(\mbf b'', \mbf a'') \models \mbf a$, we define a linear map 
\begin{align*}
\Delta^{\C}_{\mbf b', \mbf a', \mbf b'', \mbf a''}
& : {}_{\mbf b} \K^{\C}_{\mbf a} \longrightarrow
 {}_{\mbf b'}\K^{\C}_{\mbf a'} \otimes {}_{\mbf b''}\K_{\mbf a''},
\\
\Delta^{\C}_{\mbf b', \mbf a', \mbf b'', \mbf a''} ([A]) 
&=
\sum_{i, j} C_{i,j} (v, 1) [A'_i]\otimes \ ^{\mathfrak a} [A''_j],
\end{align*}
where $C_{i, j} (v, v')$, $A'_i$ and $A''_j$ are as in Proposition ~\ref{cop-stab}.
We shall call the collection 
\[
\dot \Delta^{\C} = (\Delta^{\C}_{\mbf b', \mbf a', \mbf b'', \mbf a''})_{\mbf b', \mbf a'\in \mbb Z_n^{\C}, \mbf b'', \mbf a'' \in \mbb Z_n}
\]
the $comultiplication$ of $\K^{\C}_n$.
Let $g_{A, B}^C$ and $h_A^{B, C}$ be the structure constants of the multiplication and comultiplication of $\K^{\C}_n$, respectively, with respect to the standard bases.

\begin{prop}
\label{prop-cop-alg-hom}
The comultiplication $\dot \Delta^{\C}$ on $\K^\C_n$ is  an algebra homomorphism in the following sense:
for all matrices $A, B, C' \in \widetilde \Xi_n$, $C'' \in \widetilde  \Theta_n$  one has
\begin{align}
\label{cop-alg-hom}
\sum_{C \in \widetilde \Xi_n}  g^C_{A, B}   h^{C', C''}_C
= \sum_{ A', B' \in \widetilde \Xi_n, A'',  B'' \in \widetilde \Theta_n}  h^{A', A''}_A  h^{B', B''}_B  g_{A', B'}^{C'} 
{}^{\mathfrak a} g_{A'', B''}^{C''}.
\end{align}
\end{prop}

\begin{proof}
The proof is the same as that of Proposition ~\ref{prop-cop-alg-hom}, and shall not be repeated.
\end{proof}

Proposition ~\ref{Diag-cop-alg-hom} can be equivalently presented in terms of the following commutative diagram:
for all  sequences $\mbf a, \mbf a', \mbf b, \mbf b',  \mbf c \in \mbb Z^{\C}_n$ and $\mbf a'', \mbf b'' \in \mbb Z_n$ such that $\mbf a' + \mbf a'' \models \mbf a$ and $\mbf b' +\mbf b'' \models \mbf b$, we have
\begin{align}
\label{Diag-cop-alg-hom}
\xymatrix{ 
 & {}_{\mbf b} \K^{\C}_{\mbf a} \ar@<0ex>[dr]^{\Delta^{\C}_{\mbf b', \mbf a', \mbf b'', \mbf a''}} &   && \\
%
%
_{\mbf b} \K^{\C}_{\mbf c} \otimes {}_{\mbf c} \K^{\C}_{ \mbf a}  \ar@<0ex>[ur]^{m^{\C}}  \ar@<0ex>[d]_{\prod \Delta^{\C} \otimes \Delta^{\C}}   
  & & 
{}_{\mbf b'} \K^{\C}_{\mbf a'} \otimes {}_{\mbf b''} \K_{\mbf a''}  \\
\prod \;
{}_{\mbf b'} \K^{\C}_{\mbf c'} \otimes {}_{\mbf b''}\K_{\mbf c''} \otimes {}_{\mbf c'} \K^{\C}_{\mbf a'} \otimes {}_{\mbf c''}\K_{\mbf a''} 
\ar@<0ex>[rr]^{\omega_{23}}
 &&  
\prod \;
{}_{\mbf b'} \K^{\C}_{\mbf c'} \otimes {}_{\mbf c'} \K^{\C}_{\mbf a'} \otimes {}_{\mbf b''}\K_{\mbf c''}  \otimes {}_{\mbf c''}\K_{\mbf a''} 
\ar@<0ex>[u]_{m^{\C} \otimes m}
}
\end{align}
Here $m^{\C}$ stands for the multiplication in $\K^{\C}_n$,
all products run over all sequences $\mbf c' \in \mbb Z_n^{\C}$ and  $ \mbf c'' \in \mbb Z_n$ such that 
$\mbf c' + \mbf c'' \models \mbf c$,  and
$\prod \Delta^{\C} \otimes \Delta^{\C}$ stands for the product of 
$\Delta^{\C}_{\mbf b', \mbf c', \mbf b'', \mbf c''} \otimes \Delta^{\C}_{\mbf c', \mbf a', \mbf c'', \mbf a''}$.

\begin{prop}
\label{prop-cop-coass}
The comultiplication $\dot \Delta^{\C}$  is coassociative in the following sense:
for any matrices $A, A' \in \widetilde \Xi_n,  A'', A''' \in \widetilde \Theta_n$, we have
\begin{align}
\label{cop-coass}
\sum_{C\in \widetilde \Xi_n}  h^{C, A'''}_A h^{A' ,  A''}_C 
=
\sum_{B \in \widetilde \Theta_n} h^{A', B}_A \  {}^{\mathfrak a} h^{A'', A'''}_B.
\end{align}
\end{prop}

\begin{proof}
The proof is similar to that of Proposition ~\ref{prop-cop-coass-A}, where we use Proposition ~\ref{Sj-coass} instead of Remark~\ref{S-coass}.
\end{proof}


Proposition \ref{prop-cop-coass} can be equivalently formulated as the following commutative diagram:
for $\tilde{\mbf a}, \mbf a, \mbf a', \tilde{\mbf b}, \mbf b, \mbf b' \in \mbb Z^{\C}_n$, $\mbf a'', \mbf a''', \mbf b'', \mbf b''' \in \mbb Z_n$ such that 
$(\mbf a', \mbf a'') \models \mbf a$, $(\mbf b',\mbf b'')\models \mbf b$, 
$(\mbf a, \mbf a''')\models \tilde{\mbf a}$ and $(\mbf b, \mbf b''')\models \tilde{\mbf b}$, 
we have
\begin{align}
\label{Diag-cop-coass}
\begin{CD}
{}_{\tilde{ \mbf b}} \K^{\C}_{\tilde{ \mbf a}} 
@>\Delta^{\C}_{\mbf b, \mbf a, \mbf b''', \mbf a'''} >> 
{}_{\mbf b} \K^{\C}_{\mbf a} \otimes {}_{\mbf b'''}\K_{\mbf a'''} \\
@V\Delta^{\C}_{\mbf b', \mbf a', \mbf b''+\mbf b''', \mbf a''+\mbf a'''}VV @VV\Delta^{\C}_{\mbf b', \mbf a', \mbf b'', \mbf a''}\otimes 1V \\
{}_{\mbf b'}\K^{\C}_{\mbf a'} \otimes {}_{\mbf b'' + \mbf b'''} \K_{\mbf a'' +\mbf a'''} 
@>>1\otimes \Delta_{\mbf b'', \mbf a'', \mbf b''', \mbf a'''}> 
{}_{\mbf b'}\K^{\C}_{\mbf a'} \otimes {}_{\mbf b''}\K_{\mbf a''} \otimes {}_{\mbf b'''}\K_{\mbf a'''}.
\end{CD}
\end{align}

\begin{rem}
\label{rem:KnQSP}
In light of Propositions ~\ref{prop-cop-alg-hom} and ~\ref{prop-cop-coass}, 
we say that the pair $(\K_n, \K^{\C}_n)$ forms an $idempotented$ quantum symmetric pair.
Recall from Remark~\ref{rem:gln}
that $\K_n$ is isomorphic to the idempotented quantum affine $\gl_n$, $\dot\U (\glh_n)$. 
\end{rem}

\section{A homomorphism  from $\K^{\C}_n$ to $\Sj$}

Recall that we set $[A]_d =0$ and $\f_{A; d} =0$ in $\Sj$ if $A\not \in \Xi_{n, d}$.
We define a linear map 
\begin{align*}
\Psi_{n, d} &: \K^{\C}_n \longrightarrow  \Sj,
\\
[A] & \mapsto [A]_d, \quad \text{for } A\in \widetilde \Xi_n.
\end{align*}

\begin{lem}
\label{Psi-fa}
For all  $A\in \widetilde \Xi_n$, we have
$\Psi_{n, d} (\f_A) = \f_{A; d}$.
In particular, the map  $\Psi_{n, d}$ commutes with the bar involutions.
\end{lem}

\begin{proof}
By Proposition ~\ref{4.2.3-stab-v3},  we have
\begin{align}
\label{fad-explicit}
\f_{A; d} = \sum_{\mbf t \in \mathcal T_{\mbf i, \mbf a, D_{c(\ddot A)}}} 
Q^{\mbf t}_{\mbf i, \mbf a; D_{c(\ddot A)}} (v, 1) \big[ \mrm{dlt}_1  (D_{\co (\ddot A)})_{\mbf i, \mbf a, \mbf t} \big]_d.
\end{align}
The equality $\Psi_{n, d} (\f_A) = \f_{A; d}$ follows readily by comparing (\ref{fad-explicit}) and (\ref{fa-explicit}).

Since $\f_A$ and $\f_{A; d}$ are bar-invariant, it follows that $\Psi_{n,d }$ commutes with the bar maps.
\end{proof}

\begin{prop}
\label{Psi-algebra}
The map $\Psi_{n, d}$ is a surjective algebra homomorphism.
\end{prop}

\begin{proof}
By Theorem ~\ref{Kj-bases} and Lemma ~\ref{Psi-fa}, 
it suffices to show that 
\[
\Psi_{n, d} (\f_{A_1} \cdot \f_{A_2}) = \f_{A_1; d} * \f_{A_2; d},
\quad \forall A_1, A_2 \in \widetilde \Xi_n.
\]
Let $(\mbf i_1, \mbf a_1)$ and $(\mbf i_2, \mbf a_2)$ be the pairs of tuples associated to 
$\f_{A_1}$ and $\f_{A_2}$, respectively,  in (\ref{fa-explicit}).
The product $\f_{A_1} \cdot \f_{A_2}$ can then be written in a similar form as (\ref{fa-explicit}) with $(\mbf i, \mbf a)$ 
replaced by $(\mbf i_1 \mbf i_2, \mbf a_1 \mbf a_2)$, by Proposition ~\ref{fA-monomial}.
Similarly, the product $ \f_{A_1; d} * \f_{A_2; d}$ admits a similar form of (\ref{fad-explicit}) with $(\mbf i, \mbf a)$ replaced by
$(\mbf i_1\mbf i_2, \mbf a_1 \mbf a_2)$.
By arguing in a similar fashion as 
the proof of Lemma ~\ref{Psi-fa}, 
we see that $\Psi_{n, d}$ sends the product $\f_{A_1} \cdot \f_{A_2}$ to $ \f_{A_1; d} * \f_{A_2; d}$.  
\end{proof}

By a standard argument such as the proof of  ~\cite[Theorem A.21]{BKLW14}, we reach at the following  result. 

\begin{thm}
\label{thm:Psi:c}
We have $\Psi_{n, d} (\{A\}) = \{A\}_d$ if $A\in \Xi_{n, d}$, and zero otherwise.
\end{thm}

\section{The algebra $\K^{\C}_n$ as a subquotient of $\K^{\C}_{\breve n}$}
\label{sec:subquot}

Let $\K^{\C}_{\breve n,  1, 0}$ be the subalgebra of $\K^{\C}_{\breve n}$ spanned by the elements
$[A]$ such that $\ro(A)_{1} =\co (A)_1 =0$.
Let $\mathcal I$ be the subspace  of $\K^{\C}_{\breve n, 1, 0}$ spanned by those elements $[A]$ 
such that  $a_{1, 1} < 0$.
Then a similar argument for Lemma ~\ref{Psi-fa} (see also ~\cite[A.3]{BLW14}) gives us the following. 

\begin{lem}
\label{I-ideal}
The subspace $\mathcal I$ 
is a two-sided ideal of $\K^{\C}_{\breve n, 1, 0}$ and $\mathcal I \cap   \{ \{A\} | A\in \widetilde \Xi_{\breve n} \}$ forms a basis of $\mathcal I$.
\end{lem}

Let  $\K^{\C}_{\breve n, 1, 0}/\mathcal I$ 
be the quotient algebra of  $\K^{\C}_{\breve n, 1, 0}$ by  $\mathcal I$.
By Lemma ~\ref{I-ideal}, the set 
\[
\{ \{A\} + \mathcal I | A\in \widetilde \Xi_{\breve n},  a_{1,i}=a_{i, 1}=0,\forall i\in \mbb Z \}
\]
is a  $stably$ $canonical$ basis of $\K^{\C}_{\breve n, 1, 0}/\mathcal I$.
We shall identify the subquotient $\K^{\C}_{\breve n, 1, 0}/\mathcal I$ with the stabilization algebra $\K^{\C}_n$.

\begin{prop}
 \label{prop:sq}
The assignment $\tilde \rho: [A] \mapsto[\ddot A] + \mathcal I$, for all $A\in \widetilde \Xi_{n}$, defines an isomorphism from 
the algebra $\K^{\C}_n$ to the  subquotient $\K^{\C}_{\breve n,  1, 0}/\mathcal I$  
of $\K^{\C}_{\breve n}$ with compatible stably canonical bases.
\end{prop}

\begin{proof}
By a similar argument as in the proof of Lemma ~\ref{Psi-fa}, we have 
\[
\tilde \rho (\f_A) = \f_{\ddot A} + \mathcal I, \quad \forall A\in \widetilde \Xi_{n}. 
\]
A similar argument as in the proof of Proposition ~\ref{Psi-algebra} shows that $\tilde \rho$ is an algebra homomorphism
by showing that $\rho(\f_{A_1} \cdot \f_{A_2}) = \f_{\ddot A_1} \cdot \f_{\ddot A_2} + \mathcal I$ for all $A_1, A_2\in \widetilde \Xi_n$.
By Lemma ~\ref{I-ideal} we know that $\tilde \rho$ is an algebra isomorphism.
A standard argument shows the compatibility with the canonical bases. The proposition is thus proved.
\end{proof}

Clearly, the projection $\Psi_{\breve n, d}: \K^{\C}_{\breve n} \to \mbf S^{\C}_{\breve n,d}$ 
induces a projection 
$\Psi_{\breve n, d}: \K^{\C}_{\breve n, 1, 0}/\mathcal I  \to \mbf S^{\C}_{\breve n, d}$.
We have the following commutative diagram: 
\[
\begin{CD}
\K^{\C}_n @> \tilde \rho>> \K^{\C}_{\breve n,  1, 0}/\mathcal I \\
@V\Psi_{n, d} VV @VV\Psi_{\breve n, d}V \\
\Sj @>\rho>> \mbf S^{\C}_{\breve n, d} 
\end{CD}
\]

\begin{rem}
The construction of $\K^{\C}_n$ as a subquotient of $\K^{\C}_{\breve n}$ here is modeled on the construction
in \cite{BLW14} (see also ~\cite{FL14}), where an algebra $\dot\U^\imath$  is realized as a subquotient of an algebra $\dot\U^\jmath$
with compatible stably canonical bases. 
\end{rem}

\chapter{Stabilization algebras  arising from other Schur algebras}
  \label{chap:Kc:ji}

In this chapter, the approach to the stabilization of the family of Schur algebras $\Sj$ (as $d$ varies)  in the preceding Chapter~\ref{chap:Kc} will be 
adapted with modifications to study the remaining 3 families of Schur algebras of types $\ji$, $\ij$ and $\ii$.
We will present more details for the type $\ji$ while merely formulating the main statements for types $\ij$ and $\ii$. 

\section{A monomial basis for Schur algebra $\Sji^{\ji}_{\nn, d}$}

Recall that  $\nn =n-1 =2r+1$.
Recall the set $\MX^{\ji}_{\nn, d}$ from \eqref{Mjid}, 
the set $\check{\MX}^{\ji}_{\nn,d}$ from \eqref{Mjicheck}, and the bijection
from \eqref{bijection:ji} 
$$
\mrm{dlt}_{r+1} : \MX^{\ji}_{\nn, d} \longrightarrow \check{\MX}^{\ji}_{\nn, d}.
$$
We also set $\ddot A = \mrm{dlt}_{r+1}^{-1} (A)$ for all $A \in \check{\MX}^{\ji}_{\nn, d}$.

Recall the subalgebra $\Sji^{\ji}_{\nn, d}$ of $\Sj$ from (\ref{Sji}).
Since the coproduct on $\Sj$ is coassociative,   so is
the comultiplication  $\Delta^{\ji}$ on $\Sji_{\nn, d}^{\ji}$.

For each tridiagonal matrix $A \in \check{\Xi}^{\ji}_{\nn, d}$ such that 
$\mrm{dlt}_{r+1} (A) - \sum_{1\leq i\leq \nn} \alpha_i E^{i, i+1}_{\theta, \nn}$ is diagonal,
we define 
\begin{align}
\label{ji-fA}
\f^{\ji}_{A; d} = \f^{(\alpha_r)}_{r} * \f^{(\alpha_{r-1})}_{r-1} * \cdots * \f^{(\alpha_{-(r+1)})}_{-(r+1)} * 1_{\co (A)} \in \Sji^{\ji}_{\nn, d}.
\end{align}

We call a matrix $A \in \MX^{\ji}_{\nn,d}$ $\ji$-tridiagonal, if the associated matrix $\mrm{dlt}_{r+1} (A)$ is tridiagonal. 
Given any matrix  $A=(a_{ij})$  in $\MX^{\ji}_{\nn, d}$ of depth $m \ge 1$ and $\mrm{dlt}_{r+1}(A) = (a'_{ij})$, 
we define $\ji$-tridiagonal matrices $A_1$, $A_2, \ldots, A_m \in \Xi^{\ji}_{\nn, d}$  
by the conditions that $\ro (A_m) = \ro (A)$, $\co (A_1)= \co (A)$, $\ro (A_i) = \co (A_{i+1})$ for $1\leq i\leq m-1$ 
and 
$\mrm{dlt}_{r+1}(A_i) - \sum_{1 \leq j \leq n} (\sum_{k \leq j + i -1} a'_{k, j+1}) E^{j, j+1}_{\theta, \nn}$ is  diagonal 
for all $1\leq i \leq m$.
Then we set 
\begin{align}
\label{ji-fA-monomial}
\f^{\ji}_{A; d} =
\f^{\ji}_{A_m; d} * \f^{\ji}_{A_{m-1}; d} *\cdots * \f^{\ji}_{A_1; d}. 
\end{align}
By definition, the element $\f^{\ji}_{A; d}$ is bar-invariant.

By an argument similar to Theorem ~\ref{new-10.1.3}, we have the following.

\begin{prop}
\begin{enumerate}
\item
We have $\f^{\ji}_{A; d} = [ A]_d + \mbox{lower terms}$, for all $A\in \MX^{\ji}_{\nn, d}$. 

\item
The set $\{ \f^{\ji}_{A; d} | A \in  \MX^{\ji}_{\nn, d}\}$ forms a bar-invariant basis of $\Sji^{\ji}_{\nn, d}$
(called a {\em monomial basis}).
\end{enumerate}
\end{prop}

\section{Stabilization of Schur algebras of type $\ji$}

Now we shall formulate the stabilization of the family of Schur algebras $\{\Sji^{\ji}_{\nn, d}\}_{d \ge 1}$,
analogous to the family of Schur algebras $\{\Sji^{\C}_{n, d}\}_{d \ge 1}$ treated in Section~\ref{sec:SSSc}.

Recall $\widetilde \Xi^{\ji}_{\nn, d}$ in \eqref{tXid:ji} is a variant of $\Xi^{\ji}_{\nn, d}$ 
which does not require the diagonal entries to be nonnegative.

Recall the set $\widetilde \Xi_{n, d}$ from \eqref{tXid}. 
For $0 \leq i \leq n-1$, $A \in \widetilde \Xi_{n, d}$   for all $j\in \mbb Z$, 
$t=(t_u)_{u\in \mbb Z} \in \mbb N^{\mbb Z}$ such that $\sum_{j\in \mbb Z} t_u = R$, we define 
the polynomials $Q^{\ji, t}_{i, R; A} \in \mathscr R$ as follows.
For any $i\in [1,  n-1]\backslash \{r, r+1\}$,  we define 
\begin{align}
\label{Qji}
Q_{i, R; A}^{\ji,  t} (v, v') = 
v^{\beta_t}
\prod_{u \in \mbb Z, u\neq i}
\overline{\begin{bmatrix} a_{iu} + t_u \\ t_u \end{bmatrix}}
\cdot 
v'^{(\delta_{i, 1} + \delta_{i,  n - 1}) \sum_{i+1 \geq u} t_u}
\overline{\begin{bmatrix} a_{ii} + t_i \\ t_i \end{bmatrix}}_{v, v'},
\end{align}
where
\[
\beta_t =
 \sum_{j \geq  u} a_{i j} t_u - \sum_{j > u} a_{i+1, j} t_u + \sum_{j < u} t_j t_u
 + \frac{1} {2} (\delta_{i, r} + \delta_{i,  n-1})   \left ( \sum_{j + u < 2(i+1)} t_j t_u + \sum_{j < i + 1} t_j\right ).
\]

We further define 
\begin{align}
Q_{r, R; A}^{\ji,  t}(v, v') 
&=
v^{\beta_t} 
\prod_{u \in \mbb Z, u\neq i}
\overline{\begin{bmatrix} a_{iu} + t_u \\ t_u \end{bmatrix}}
\cdot  v'^{- \sum_{i \geq u} t_u},   & \text{for } i=r, 
\label{Qji2}
 \\
Q_{r+1, R; A}^{\ji,  t} (v, v') 
&= v^{\beta'_t} 
\prod_{u > i}
\overline{\begin{bmatrix} a_{i u} + t_u + t_{2i -u} \\ t_u \end{bmatrix} }
\prod_{u < i}
\overline {\begin{bmatrix} a_{i u} + t_u \\ t_u \end{bmatrix}} 
\cdot v'^{\sum_{i+1\geq u} t_u},  & \text{for } i=r+1,
\label{Qji3}
\end{align}
where
\[
\beta'_t
=
\sum_{j \geq u} a_{ij} t_u - \sum_{j > u} a_{i+1, j} t_u
 + \sum_{j<u, j+u \leq 2i } t_j t_u -  \sum_{j> i} \frac{t_j^2 - t_j}{2}
+ \frac{R^2 - R}{2}.
\]

Given tuples $\mbf i=(i_1, \ldots, i_s)$ and $\mbf a =(a_1, \ldots, a_s) \in \mbb N^s$  and a tuple $\mbf t = (t_1,\ldots, t_s)$ 
satisfying \eqref{eq:tt}, we defined 
the polynomials $Q^{\mbf t}_{\mbf i, \mbf a; A}$ in \eqref{QiaA}.
We can similarly define the polynomials $Q^{\ji, \mbf t}_{\mbf i, \mbf a; A}(v, v')$ in $\mathscr R$, inductively on $s$ 
starting with \eqref{Qji}-\eqref{Qji3},  for $A\in \widetilde \Xi^{\ji}_{\nn, d}$.

Propositions~\ref{ji-stab-v3}--\ref{ji-bar-stab} are the $\ji$-counterparts of 
Propositions~\ref{4.2.3-stab-v3}--\ref{bar-stab}. We skip the similar proofs. 
The notations are understood in this section that 
$\ddot I_{\nn} = I_n - E^{r+1, r+1}_{n}$, 
and $_{\ddot p} A = A + p \ddot I_{\nn}$.

\begin{prop}
\label{ji-stab-v3}
Assume $A, B_j \in \widetilde \Xi_{n, d}$, for $1\leq j \leq s$, and a pair of tuples $(\mbf i, \mbf a)$ satisfy the following properties:
$\ro (A) = \co (B_s)$, $\ro (B_j) = \co (B_{i-1}), \forall  1< i\leq s$,  $B_j - a_j E^{i_j, i_j+1}_{\theta,  n}$ is diagonal  
and $a_{r+1, j} =\delta_{j, r+1}$ for all $j\in \mbb Z$.  Then we have 
\[
[_{\ddot p} B_1]_{d+ \frac{p}{2} \nn} * \cdots * [_{\ddot p} B_s]_{d+ \frac{p}{2} \nn}  *[_{\ddot p} A]_{d+\frac{p}{2} \nn} 
= \sum_{\mbf t \in \mathcal T_{\mbf i, \mbf a, A}} Q_{\mbf i, \mbf a; A}^{\ji, \mbf t} (v, v^{-p}) [_{\ddot p} A_{\mbf i,\mbf a, \mbf t}]_{d+\frac{p}{2} \nn}, 
\quad \forall  p \in 2 \mbb Z.
\] 
\end{prop}

\begin{prop}
\label{prop:ji-A-f}
Let $A\in \widetilde \Xi^{\ji}_{\nn, d}$.
There exist $Z_i \in \widetilde \Xi^{\ji}_{\nn, d}$, for $1\leq i\leq m$,  with $Z_i < A$, $Q_i (v, v') \in \mathscr R$ 
and $p_0\in \mbb N$ such that
\begin{align}
\label{ji-A-f}
[_{\ddot p} A]_{d+\frac{p}{2} \nn} = \f^{\ji}_{_{\ddot p} A; d + \frac{p}{2} \nn} + \sum_{i=1}^m Q_i(v, v^{-p}) [_{\ddot p}Z_i]_{d+\frac{p}{2} \nn}, 
\quad \forall p \geq p_0, p\in 2\mbb N.
\end{align}
\end{prop}

\begin{prop}
\label{ji-10.2.1}
Assume that  $A_1, \ldots, A_l \in \widetilde \Xi^{\ji}_{\nn, d}$  satisfy $\co (A_i) = \ro (A_{i+1})$ for all $1\leq i\leq l-1$.
There exist $Z_1, \ldots, Z_m \in \widetilde \Xi_{\nn, d}^{\ji}$, $G_1(v,v'), \ldots, G_m(v, v') \in \mathscr R$, and $p_0\in \mbb N$
such that 
\begin{align}
\label{ji-10.2.1-a}
[_{\ddot p} A_1]_{d+\frac{p}{2} \nn} * [_{\ddot p} A_2]_{d+\frac{p}{2} \nn} * \cdots * [_{\ddot p} A_l]_{d+\frac{p}{2} \nn} 
= \sum_{i=1}^m G_i(v, v^{-p}) [_{\ddot p} Z_i]_{d+\frac{p}{2} \nn}, 
\quad  \forall p\geq p_0, p\in 2\mbb N.
\end{align}
\end{prop}

\begin{cor}
\label{ji-K-A-leading}
For any matrix  $A \in \widetilde \Xi^{\ji}_{\nn, d}$  of depth $m$ and $\mrm{dlt}_{r+1}(A) = (a'_{ij})$, there exist  unique
$\ji$-tridiagonal matrices $A_1$, $A_2, \ldots, A_m \in \widetilde \Xi^{\ji}_{\nn, d}$
satisfying $\ro (A_m) = \ro (A)$, $\co (A_1)= \co (A)$, $\ro (A_i) = \co (A_{i+1})$ for $1\leq i\leq m-1$ 
and
$\mrm{dlt}_{r+1} (A_i) - \sum_{1 \leq j \leq n} (\sum_{k \leq j + i -1} a'_{k, j+1}) E^{j, j+1}_{\theta, \nn}$ is  diagonal 
for all $1\leq i \leq m$ such that
\[
[_{\ddot p}A_m]_{d+\frac{p}{2} \nn}  *  [_{\ddot p}A_{m-1}]_{d+\frac{p}{2} \nn}   * \dots * [_{\ddot p}A_1]_{d+\frac{p}{2} \nn} 
= [_{\ddot p}A]_{d+\frac{p}{2} \nn} +\sum_{i=1}^l G_i(v, v^{-p}) [_{\ddot p} Z_i]_{d+ \frac{p}{2} \nn}, \ \forall p\in 2\mbb N, p\geq p_0,
\]
where $p_0$,  $G_i(v, v')\in \mathscr R$ and $Z_1,\ldots, Z_l \in \widetilde \Xi_{\nn, d}^{\ji}$ are given in Proposition ~\ref{ji-10.2.1} such that $Z_i < A$.
\end{cor}

\begin{prop}
\label{ji-bar-stab}
Assume that $A \in \widetilde \Xi^{\ji}_{\nn, d}$. 
Then there exist $Y_i \in \widetilde \Xi^{\ji}_{\nn, d}$ with $Y_i < A$, $H_i (v, v') \in \mathscr R$ 
for all  $1\leq i \leq s$ and $p_0\in \mbb N$ such  that
\begin{align}
\overline{[_{\ddot p} A]}_{d+\frac{p}{2} \nn}  
= [_{\ddot p} A]_{d+\frac{p}{2}\nn}  + \sum_{i=1}^s H_i(v, v^{-p}) [_{\ddot p}Y_i]_{d+\frac{p}{2} \nn}, \quad \forall p \geq p_0, p\in 2 \mbb N.
\end{align}
\end{prop}

The following is a counterpart of Proposition~\ref{cop-stab}.

\begin{prop}
\label{ji-cop-stab}
Assume that $d'+d''=d$ and  that $\mbf b', \mbf a'  \in \mbb Z^{\C}_n $ and $\mbf b'', \mbf a'' \in \mbb Z_{\nn}$ so that 
$_p\Delta^{\ji}_{\mbf b', \mbf a', \mbf b'', \mbf a''}$ is defined.
Let  $A\in \widetilde \Xi^{\ji}_{\nn,d}$.
There exist  $A'_i \in \widetilde \Xi^{\ji}_{\nn, d'}$ where $1\leq i\leq l$ for  some $l$,  
$A''_j \in \widetilde{\Theta}^{\ji}_{\nn, d''}$  where $1\leq j \leq m$ for some $m$,
$C_{i,j}(v, v') \in \mathscr R$ for  $1\leq i \leq l$, $1\leq j\leq m$, and $p_0\in \mbb N$ such that 
\[
_{\ddot p}\Delta^{\ji}_{\mbf b', \mbf a', \mbf b'', \mbf a''} ([_{\ddot{2p}} A]_{d + p\nn}) 
= \sum_{1\leq i\leq l, 1\leq j\leq m} 
C_{i, j} (v, v^{-p}) [_{\ddot p} A'_i]_{d'+\frac{p}{2} \nn} \otimes \ ^{\mathfrak a}[_{\ddot p}A''_j]_{d''+p \nn}, 
\quad \forall p \geq p_0, p\in 2\mbb N.
\]
\end{prop}

\section{The stabilization algebra $\K_n^{\ji}$}

Recall the set $\widetilde \Xi^{\ji}_{\nn}$ and $\widetilde \Xi^{\ji}_{\nn, d}$ from \eqref{tXid:ji}.
Consider the $\mbb Q(v)$-space $\K^{\ji}_{\nn}$ spanned by the formal symbols $[A]$ for all $A \in \widetilde \Xi^{\ji}_{\nn}$.
We define an associative algebra structure on $\K^{\ji}_{\nn}$ by
\begin{align}
[A_1] \cdot [A_2] = \sum_{i=1}^m G_i (v, 1) [Z_i], \quad \forall A_1, A_2 \in \widetilde \Xi^{\ji}_{\nn},
\end{align}
where $G_i(v, v') \in \mathscr R$ and $Z_i$ are from Proposition ~\ref{ji-10.2.1}.

For each $A\in \widetilde \Xi^{\ji}_{\nn}$, we define
\begin{align}
\f^{\ji}_{A} = [A] + \sum_{i_1}^m G_i(v, 1) [Z_i],
\end{align}
where $G_i(v, v') \in \mathscr R$ and $Z_i$ are from Corollary ~\ref{ji-K-A-leading}.
It follows by definition that  
$\{ \f^{\ji}_A | A\in \widetilde \Xi^{\ji}_{\nn} \}$ forms a basis of $\K^{\ji}_{\nn}$ (called {\em a monomial basis}).

By Proposition ~\ref{ji-stab-v3}, we can establish the following. 

\begin{prop}
\begin{enumerate}
\item
For any $A \in \widetilde \Xi^{\ji}_{\nn}$, there exists a pair $(\mbf i, \mbf a)$ of tuples such that 
\begin{align}
\label{ji-fa-explicit}
\f^{\ji}_A = \sum_{\mbf t \in \mathcal T_{\mbf i, \mbf a, D_{\co (A)}}} Q^{\ji, \mbf t}_{\mbf i, \mbf a; D_{\co ( A)}} (v, 1) 
\big[(D_{\co (A)})_{\mbf i, \mbf a, \mbf t} \big],
\end{align}
where $D_{\co (A)}$ is the diagonal matrix in $\widetilde \Xi^{\ji}_{\nn}$ with diagonal $\co (A)$.

\item
The element $\f^{\ji}_A$ can be written in a product form as
\begin{equation}
\label{fA:ji}
\f^{\ji}_A =
\f^{\ji}_{A_m} \cdot  \f^{\ji}_{A_{m-1}}  \cdot \ldots \cdot \f^{\ji}_{A_1}, 
\end{equation}
where $A_i$ are $\ji$-tridiagonal matrices  defined similarly as in  (\ref{fA-monomial}).

\item
We have $\overline{\f^{\ji}_A} = \f^{\ji}_A$, for $A\in \widetilde \Xi^{\ji}_{\nn}$.
\end{enumerate}
\end{prop}

Similarly, for $A\in \widetilde \Xi^{\ji}_{\nn}$, we set
\[
m^{\ji}_A = [A_m] \cdot [A_{m-1}] \cdots [A_1],
\]
where $A_i$'s are the same as in \eqref{fA:ji}. 
One also has $m^{\ji}_A =\f^{\ji}_A +\text{lower terms}$. 
Thus $\{m_A^{\ji} \vert A\in  \widetilde \Xi^{\ji}_{\nn}\}$ forms a basis for $\K_{\nn}^{\ji}$ (called {\em a semi-monomial basis}). 
Just like its $\jj$-sibling, the monomial $m^{\ji}_A$ is not necessarily bar-invariant.

The following multiplication formula on $\K^{\ji}_{\nn}$  follows from Proposition~\ref{ji-stab-v3}.

\begin{prop}
Assume  the matrices  $A, B \in \widetilde \Xi^{\ji}_{\nn}$
satisfy that 
$\co (B) = \ro (A)$ and $\mrm{dlt}_{r+1} (B) - \sum_{1\leq i\leq n} \beta_i E^{i, i+1}_{\theta, \nn}$ is diagonal.
Then we have a multiplication formula of the form 
\begin{align}
\f^{\ji}_B \cdot [A] 
= \sum_{\mbf t \in \mathcal T_{\mbf i^{\ji}_0, \mbf b^{\ji}_0,  A}} 
Q^{\ji, \mbf t}_{\mbf i^{\ji}_0, \mbf b^{\ji}_0; A} (v, 1) 
\big[A_{\mbf i^{\ji}_0, \mbf b^{\ji}_0, \mbf t} \big],
\end{align}
where
$\mbf i^{\ji}_0 = (r, r-1, \ldots, 1-r) $
and
$\mbf b^{\ji}_0 = (\beta_r, \beta_{r-1}, \ldots, \beta_{-r}).$
\end{prop}

We define a bar involution on $\K^{\ji}_{\nn}$  by
\begin{align}
\overline{[A]} = [A] + \sum_{i=1}^s H_i(v, 1) [Y_i],  \quad \forall A \in \widetilde \Xi^{\ji}_{\nn},
\end{align}
where $H_i(v, v') $ and $Y_i<A$ are from Proposition ~\ref{ji-bar-stab}.
By a standard argument, we can now establish the existence of the stably canonical basis for $\K^{\ji}_{\nn}$.

\begin{prop}
\label{ji-stable}
\begin{enumerate}
\item 
For any $A \in \widetilde \Xi^{\ji}_{\nn}$, 
there exists a unique element $\{A\}$ in $\K^{\ji}_{\nn}$ such that 
\[
\overline{\{A\}} = \{A\}, 
\quad
\{A\} = [A] + \sum_{A'<A} \pi^{\ji}_{A, A'}[A'], \quad \pi^{\ji}_{A, A'}\in v^{-1} \mbb Z [v^{-1}].
\]
\item
The set $\{ \{A\} \vert A \in \widetilde \Xi^{\ji}_{\nn} \}$ forms a basis for $\K^{\ji}_{\nn}$ 
(called the {\em stably canonical basis}). 
\end{enumerate}
\end{prop}

Let us summarize the main results of this section as follows. 

\begin{thm}
\label{Kji-bases}
The algebra $\K^{\ji}_{\nn}$ admits a standard basis $\{[A] | A\in \widetilde \Xi^{\ji}_{\nn} \}$, 
a semi-monomial  basis $\{m^{\ji}_A | A\in  \widetilde \Xi^{\ji}_{\nn}\}$, 
a monomial basis $\{ \f^{\ji}_A | A\in \widetilde \Xi^{\ji}_{\nn} \}$,
and a canonical basis
$\{\{A \}| A\in  \widetilde \Xi^{\ji}_{\nn}\}$.
\end{thm}

Recall our convention that $[A]_d =0$ in $\Sji^{\ji}_{\nn,d}$, for all $A\in \widetilde \Xi^{\ji}_{\nn} \backslash \Xi^{\ji}_{\nn,d}.$
The following is a counterpart of Theorem~\ref{thm:Psi:c}.

\begin{thm}
The assignment $[A] \mapsto [A]_d$, for all $A\in \widetilde \Xi^{\ji}_{\nn}$,  defines a surjective algebra homomorphism
$\Psi^{\ji}_{\nn, d}: \K^{\ji}_{\nn} \to \Sji^{\ji}_{\nn, d}$.
Moreover, we have $\Psi^{\ji}_{\nn, d} (\{A\}) = \{A\}_d$ if $A\in \Xi^{\ji}_{\nn,d}$ and zero otherwise.
\end{thm}

We have developed the current Chapter~\ref{chap:Kc:ji} on the stabilization algebra $\K_{\nn}^{\ji}$
which is based on the imbeddings $\Sji_{\nn,d}^{\ji} \to \Sj$,
in analogy to the stabilization algebra $\Kc_n$ in Chapter~\ref{chap:Kc} which was based on the imbeddings $\Sj \to \Sji_{\breve n,d}^\C$.
Just as the imbeddings  $\Sj \to \Sji_{\breve n,d}^\C$ lead to a realization of $\Kc_n$ as a subquotient of $\K^\C_{\breve n}$
(see Proposition~\ref{prop:sq}), 
the imbeddings $\Sji_{\nn,d}^{\ji} \to \Sj$ lead to a realization of $\K_{\nn}^{\ji}$ as a subquotient of $\Kc_{n}$.

We shall simply formulate the statement below and skip the detail (compare with \cite{BLW14}).
Let $\mathcal J_{<}^{\ji}$ be the $\mathbb Q(v)$-subspace of $\K^{\C}_{n}$ spanned by 
$[A]$ for $A =(a_{ij})  \in \widetilde \Xi^{\ji}_{\nn}$ with $a_{r+1, r+1}<0$.
Then one shows that $\mathcal J_{<}^{\ji}$ is a two-sided ideal of $\K^{\C}_{n}$
with a stably canonical basis 
$$\{ \{A\} \big \vert A =(a_{ij})  \in  \widetilde \Xi^{\ji}_{\nn}, a_{r+1, r+1}<0\}. 
$$
Moreover,
the natural linear map  
$$
\K^{\ji}_{\nn} \longrightarrow \K^\C_n /\mathcal J_{<}^{\ji},  
\qquad 
[A] \mapsto [A] + \mathcal J_{<}^{\ji}
$$ 
is an  algebra isomorphism, and
it preserves the stably canonical bases. We summarize these as follows. 

\begin{thm}
\label{Kji-sq}
The algebra $\K^{\ji}_{\nn}$ is a subquotient of the algebra $\K^{\C}_n$ with compatible stably canonical bases.
\end{thm}

We finally discuss the comultiplication on $\K^{\ji}_{\nn}$. 
Let
\[
\mbb Z^{\ji}_{\nn} =
\{
\lambda = (\lambda_i)_{i\in \mbb Z} \in\mbb Z^{\C}_{n} | \lambda_{r+1}=1\},
\quad
\mbb Z^{\ji}_{\nn} =\{ \lambda \in \mbb Z_n | \lambda_{r+1} =0\}.
\]
Note that there is a canonical bijection $\mbb Z^{\ji}_{\nn} \simeq \mbb Z_{\nn}$, which we shall identify.
For any $\mbf a, \mbf b\in \mbb Z^{\ji}_{\nn}$, 
let $_{\mbf b}\K^{\ji}_{\mbf a}$ denote the subspace of $\K^{\ji}_{\nn}$ spanned by
the standard basis element $[A]$ such that $\ro(A)=\mbf b$ and $\co (A) =\mbf a$.
For any  $\mbf b, \mbf a, \mbf b', \mbf a' \in \mbb Z_{\nn}^{\ji}$ and $\mbf b'', \mbf a'' \in \mbb Z^{\ji}_{\nn}$ such that  
$(\mbf b', \mbf b'') \models \mbf b$ and $(\mbf b'', \mbf a'') \models \mbf a$, we define a linear map 
\begin{align}
\label{Kji-cop}
\Delta^{\ji}_{\mbf b', \mbf a', \mbf b'', \mbf a''} : {}_{\mbf b} \K^{\ji}_{\mbf a} \longrightarrow  {}_{\mbf b'}\K^{\ji}_{\mbf a'} \otimes {}_{\mbf b''}\K_{\mbf a''},
\end{align}
by
\[
\Delta^{\ji}_{\mbf b', \mbf a', \mbf b'', \mbf a''} ([A]) 
=
\sum_{i, j} C_{i,j} (v, 1) [A'_i]\otimes \ ^{\mathfrak a} [A''_j],
\]
where $ {}_{\mbf b''}\K_{\mbf a''}$ is a component of $\K_{\nn}$, $C_{i, j} (v, v')$, $A'_i$ and $A''_j$ are given in Proposition ~\ref{ji-cop-stab}.
We shall call the collection 
$$
\dot \Delta^{\ji} = (\Delta^{\ji}_{\mbf b', \mbf a', \mbf b'', \mbf a''})_{\mbf b', \mbf a'\in \mbb Z_{\nn}^{\ji}, \mbf b'', \mbf a'' \in \mbb Z_{\nn}}
$$
the $comultiplication$ of $\K^{\ji}_{\nn}$.
Let $g_{A, B}^C$ and $h^{C', C''}_C$ be the structure constants of $\K^{\ji}_{\nn}$ of the multiplication and comultiplication, respectively, with respect to the standard bases.
We have the following $\ji$-counterparts of the commutative diagrams \eqref{Diag-cop-alg-hom} and \eqref{Diag-cop-coass} for the comultiplication $\dot\Delta^\C$.

\begin{prop}
\label{ji-cop-alg-hom}
\begin{enumerate}
\item
The $\dot \Delta^{\ji}$ is an algebra homomorphism in the following sense:
for all $A, B, C' \in \widetilde \Xi^{\ji}_{\nn}$, $C'' \in \widetilde  \Theta_{\nn}$  one has
\begin{align}
\sum_{C \in \widetilde \Xi^{\ji}_{\nn}}  g^C_{A, B}   h^{C', C''}_C
= \sum_{ A', B' \in \widetilde \Xi^{\ji}_{\nn}, A'',  B'' \in \widetilde \Theta_{\nn}}  h^{A', A''}_A  h^{B', B''}_B  g_{A', B'}^{C'} 
{}^{\mathfrak a} g_{A'', B''}^{C''}.
\end{align}

\item
The $\dot \Delta^{\ji}$ is coassociative in the following sense: 
for all $A, A' \in \widetilde \Xi^{\ji}_{\nn},  A'', A''' \in \widetilde \Theta_{\nn}$, one has
\begin{align}
\sum_{C\in \widetilde \Xi^{\ji}_{\nn}}  h^{C, A'''}_A h^{A' ,  A''}_C 
=
\sum_{B \in \widetilde \Theta_{\nn}} h^{A', B}_A \  {}^{\mathfrak a} h^{A'', A'''}_B.
\end{align}
\end{enumerate}
\end{prop}
Recall from Remark~\ref{rem:gln} that $\K_{\nn}$ is isomorphic to an idempotented quantum $\glh_{\nn}$. 
\begin{prop}
The pair $(\K_{\nn}, \K^{\ji}_{\nn})$ forms a quantum symmetric pair. 
\end{prop}

\section{Stabilization algebra of type $\ijw$}

Recall the subalgebra $\Sij^{\ij}_{\nn, d}$ of $\Sj$ from (\ref{Sij}).
In analogue with the operator $\mrm{dlt}_{r+1}$, we can define the operator $\mrm{dlt}_0$.
For each $\ij$-tridiagonal matrix $A \in \Xi^{\ij}_{\nn, d}$ (cf. \eqref{Mijd})
such that $\mrm{dlt}_0(A) - \sum_{1\leq i\leq \nn} \alpha_i E^{i, i+1}_{\theta, \nn}$ is diagonal, we introduce the following element in $\Sij^{\ij}_{\nn, d}$:
\begin{align}
\label{ij-fA}
\ddot \f^{\ij}_{A; d} = \f^{(\alpha_{\nn})}_{\nn} * \f^{(\alpha_{\nn-1})}_{\nn -1} * \cdots * \f^{(\alpha_1)}_1 * \f^{(\alpha_{\nn})}_{0} 1_{\co (A )} \in \Sij^{\ij}_{\nn, d}.
\end{align}

Now repeat the process of the $\ji$-version. We obtain an associative algebra $\K^{\ij}_{\nn}$ with a basis $[A]$ 
parametrized by the matrices $A$ in $\widetilde \Xi^{\ij}_{\nn,}$ (which is defined  exactly the same as $\widetilde \Xi^{\ji}_{\nn,}$ with the roles of $r+1$ and $0$ switched).
Moreover, to each matrix $A$ in $\widetilde \Xi^{\ij}_{\nn, d}$, 
we can define elements $ \f^{\ij}_A$, $m^{\ij}_A$ and $\{A\}$ in $\K^{\ij}_{\nn}$, similar to those elements indexed by $\ji$ in $\K^{\ji}_{\nn}$, 
now starting with (\ref{ij-fA}).
Then all the main results for $\K^{\ji}_{\nn}$ admit counterparts for the algebra $\K^{\ij}_{\nn}$.

\begin{thm}
\label{Kij-bases}
\begin{enumerate}
\item The algebra $\K^{\ij}_{\nn}$ admits a standard basis 
$\{[A] | A\in \widetilde \Xi^{\ij}_{\nn} \}$, 
a semi-monomial basis $\{m^{\ij}_A | A\in  \widetilde \Xi^{\ij}_{\nn}\}$, 
a monomial basis $\{ \f^{\ij}_A | A\in \widetilde \Xi^{\ij}_{\nn} \}$,
and a canonical basis $\{\{A \}| A\in  \widetilde \Xi^{\ij}_{\nn}\}$.

\item 
The assignment $[A] \mapsto [A]_d$, for all $A\in \widetilde \Xi^{\ij}_{\nn}$,  defines a surjective algebra homomorphism
$\Psi^{\ij}_{\nn, d}: \K^{\ij}_{\nn} \to \Sij^{\ij}_{\nn, d}$ such that $\Psi^{\ij}_{\nn, d} (\{A\}) = \{A\}_d$ if $A\in \Xi^{\ij}_{\nn,d}$ and zero otherwise.

\item $\K^{\ij}_{\nn}$ is a subquotient of $\K^{\C}_n$ with compatible stably canonical bases.

\item The pair $(\K_{\nn}, \K^{\ij}_{\nn})$ forms an idempotented quantum symmetric pair.
\end{enumerate}
\end{thm}

\section{Stabilization algebra of type $\ii$}
  \label{sec:diag}

Recall the subalgebra $\Sii^{\ii}_{\eta, d}$ of $\Sj$ from (\ref{Sii}).
For each $\ii$-tridiagonal matrix $A \in \Xi^{\ii}_{\eta, d}$ (cf. \eqref{Mijd}) such that 
the matrix $\mrm{dlt}_{0, r} (A) - \sum_{1\leq i\leq \eta} \alpha_i E^{i, i+1}_{\theta, \eta}$ is diagonal, 
we introduce the following element $ \f^{\ii}_{A; d}$ in $\Sii^{\ii}_{\eta, d}$:
\begin{align}
\f^{\ii}_{A; d}
= 
\f^{(\alpha_{\eta})}_{\eta+1} * \f^{(\alpha_{r-1})}_r * 
\left (\f^{(\alpha_{\eta-1})}_{\eta} * \cdots * \f^{(\alpha_r)}_{r+1} \right ) * 
\left (\f^{(\alpha_{r-1})}_{r-1} * \cdots * \f^{(\alpha_0)}_0  \right ) * 1_{\co (A)} \in \Sii^{\ii}_{\eta, d}.
\end{align}

We collect the main results of $\K^{\ii}_{\eta}$ in the following.
The proofs are very similar to the previous cases, 
and so we shall skip them to avoid redundancy.

\begin{thm}
\label{Kii-bases}
\begin{enumerate}
\item The algebra $\K^{\ii}_{\eta}$ admits a standard basis 
$\{[A] | A\in \widetilde \Xi^{\ii}_{\eta} \}$, 
a semi-monomial basis $\{m^{\ii}_A | A\in  \widetilde \Xi^{\ii}_{\eta}\}$, 
a monomial basis $\{ \f^{\ii}_A | A\in \widetilde \Xi^{\ii}_{\eta} \}$,
and a canonical basis
$\{\{A \}| A\in  \widetilde \Xi^{\ii}_{\eta}\}$.

\item 
The assignment $[A] \mapsto [A]_d$, for all $A\in \widetilde \Xi^{\ii}_{\eta}$,  defines a surjective algebra homomorphism
$\Psi^{\ii}_{\eta, d}: \K^{\ii}_{\eta} \to \Sii^{\ii}_{\eta, d}$ such that $\Psi^{\ii}_{\eta, d} (\{A\}) = \{A\}_d$ if $A\in \Xi^{\ii}_{\eta,d}$ and zero otherwise.

\item $\K^{\ii}_{\eta}$ is a subquotient of $\K^{\ji}_{\nn}$ and $\K^{\ij}_{\nn}$, with compatible stably canonical bases.

\item The pair $(\K_{\eta}, \K^{\ii}_{\eta})$ is an idempotented quantum symmetric pair.
\end{enumerate}
\end{thm}


Let us summarize the interrelations among different family of Schur algebras, 
as well as the interrelations among different family of stabilization algebras of types $\jj, \ji, \ij, \ii$.

Recall $\breve n =n+2$, $n =\nn +1$, and $\nn =\eta+1$, where $n$ is even.
On the Schur algebra level, we have the following commutative diagram for natural inclusions of Schur algebras:
\begin{equation}
\xymatrix{
&\Sji^{\ji}_{\nn, d}  \ar@{^{(}->}[dr]  
&&\\
\Sii^{\ii}_{\eta, d}  \ar@{^{(}->}[ur]  \ar@{^{(}->}[dr]  && \Sj  \ar@{^{(}->}[r] & \mbf S^{\C}_{\breve n, d} \\
& \Sij^{\ij}_{\nn, d} \ar@{^{(}->}[ur] &&
}
\end{equation}

On the stabilization algebra level, we have the following diagram of subquotients:
\begin{equation}
\xymatrix{
&\K^{\ji}_{\nn}  \ar@{->>}[dl]_{\mathfrak{sq}} 
&&\\
\K^{\ii}_{\eta}  && \K^{\C}_{n}  \ar@{->>}[ul]_{\mathfrak{sq}}  \ar@{->>}[dl]^{\mathfrak{sq}} & \K^{\C}_{\breve n} \ar@{->>}[l]^{\mathfrak{sq}} \\
& \K^{\ij}_{\nn} \ar@{->>}[ul]^{\mathfrak{sq}} &&
}
\end{equation}
where the notation $\mbf K_1 \overset{\mathfrak{sq}}{\twoheadrightarrow} \mbf K_2$ stands for 
the statement that $\mbf K_2$ is a subquotient of $\mbf K_1$.
Remarkably, all the subquotients between various pairs of algebras preserve the stably canonical bases.

\begin{rem}
One can show that the Schur algebras $\Sji^{\ji}_{\nn, d}$  and $\Sij^{\ij}_{\nn, d}$ 
are isomorphic with compatible standard and canonical bases.
This isomorphism can be further lifted to the stabilization level. 
The proofs of these isomorphisms will be given elsewhere.
\end{rem}

\appendix
\chapter{Constructions in  finite type $C$}
 \label{chap:finiteC}

We shall present more details on results in finite type $C$ which was only sketched in~\cite{BKLW14}.
In addition, we will present details on comultiplications and transfer maps in finite type $C$,  adapting
the finite type $B$ formulation in \cite{FL15}. 
This will serve as a  helpful preparation for formulation and computations in affine type $C$ which are
presented in the main text.

\section{Multiplication formulas}

Recall that $\nn =2r+1$.  
We fix a non-degenerate skew-symmetric bilinear form $Q: \mbb F^{2d}_q \times \mbb F^{2d}_q \rightarrow \mbb F_q$.
Let ${\SP}({2d})$ be the symplectic subgroup of $GL({2d})$ which consists of all elements $g$ such that
$Q(gu, gu') = Q(u, u'), \forall u, u' \in \mbb F^{2d}_q$.
Consider the following sets
\begin{equation*}
  \begin{split}
    X_{\C} &= \{ 0= L_0 \subset L_1\subset \ldots \subset L_{\nn} = \mbb F^{2d}_q \vert L_{\nn -i} = L_i^{\perp}\},
     \\
    Y_{\C} &= \{ 0= L_0 \overset{1}{\subset} L_1\overset{1}{\subset} \ldots \overset{1}{\subset} L_{{2d}} = \mbb F^{2d}_q \big\vert L_{{2d} -i} = L_i^{\perp}\},
     \\
   {}^{\C}\Xi &= \Big\{A = (a_{ij}) \in {\rm Mat}_{\nn \times \nn}(\mbb N) \Big \vert \sum_{i,j\in [1,\nn]} a_{ij}= {2d},\ a_{ij} = a_{\nn+1-i, \nn+1-j}, \forall i, j\in [1,\nn] \Big
    \},
    \\
    {}^{\C}\Pi & = \Big\{B = (b_{ij}) \in {\rm Mat}_{\nn \times {2d}}(\mbb N)
      \Big\vert \sum_{i\in [1,\nn]} b_{ij}= 1,\ b_{ij} = b_{\nn+1-i, {2d}+1-j}, \forall i\in [1,\nn],j \in [1,{2d}] \Big\},\\
    {}^{\C}\Sigma & = \Big\{\sigma = (\sigma_{ij}) \in {\rm Mat}_{{2d} \times {2d}}(\mbb N) \Big \vert \sum_{i\in [1,{2d}]} \sigma_{ij}= 1=\sum_{j\in [1,{2d}]}\sigma_{ij},
      \\   &\qquad \qquad \qquad \qquad\qquad \qquad \qquad \qquad \sigma_{ij} = \sigma_{{2d}+1-i, {2d}+1-j}, \forall i,j \in [1,{2d}] \Big\}.
  \end{split}
\end{equation*}
The notation  $\overset{1}{\subset}$ above denotes inclusion of codimension $1$ as before. 
The action of ${\SP}({2d})$ on $\mbb F^{2d}_q$ induces a well-defined action of ${\SP}({2d})$ on $X_{\C}$ and $Y_{\C}$.
Let ${\SP}({2d})$ act diagonally on $X_{\C} \times X_{\C}$ 
and $Y_{\C} \times Y_{\C}$.

\begin{lem}~\cite[Lemma 6.5]{BKLW14}
There are natural bijections
 ${\SP}({2d}) \backslash X_{\C} \times X_{\C} \longleftrightarrow {}^{\C}\Xi$, 
 ${\SP}({2d}) \backslash X_{\C} \times Y_{\C} \longleftrightarrow {}^{\C}\Pi$,   and  ${\SP}({2d}) \backslash Y_{\C} \times Y_{\C} \longleftrightarrow {}^{\C}\Sigma.$
\end{lem}

Let ${}^{\C}\mbf S_d^{\jmath} = \mathcal A_{{\SP}({2d})}(X_{\C}\times X_{\C})$ be the 
algebra of ${\SP}({2d})$-invariant $\mathcal A$-valued functions on $X_{\C}\times X_{\C}$, where $\mathcal A =\mbb Z[v,v^{-1}]$
and the multiplication is given by a convolution product.

The most typical phenomenon of type $C$ already shows up when $\nn=5$, and so let us consider this case in detail.
Let $V_k$ be a ${2d}$-dimensional vector space over $k= \mbb F_q$ equipped with a non-degenerate symplectic form.
Let $ (L_i| 0\leq i\leq 5)$ be a flag of vector subspaces in  $V_k$ such that $L_i^{\perp} = L_{5-i}$ for $i\in [0, 5]$. Consider the set
\[
Z_i = \{ U \subseteq V_k | \dim_k U =1, U \subseteq L_i, U \not \subseteq L_{i-1}\}, \quad \forall i\in [1, 4].
\]
The following lemma is an analogue of  ~\cite[Lemma 3.1.3]{FL14} with an easier proof.

\begin{lem}
  \label{typeC-counting}
For $\nn=5$, we have
$
\# Z_3 = q^{\dim L_2} \frac{ q^{\dim L_3/L_2} -1} {q-1}
$
and
$
\# Z_4 = q^{\dim L_3} \frac{ q^{\dim L_4/L_3} -1} {q-1}.
$
\end{lem}

\begin{proof}
Because all lines in $V_k$  are isotropic, we have
\[
\# Z_3 = \frac{q^{\dim L_3} -1}{q-1} - \frac{q^{\dim L_2} -1}{q-1}
= q^{\dim L_2} \frac{ q^{\dim L_3/L_2} -1}{q-1}.
\]
The counting for $Z_4$ is the same.
\end{proof}

We have the following multiplication formula in finite type $C$. 
Let $E_{i,j}$ for all $1\leq i, j\leq \nn$, the standard basis of the space of $\nn$ by $\nn$ matrices.
We set $E^{\theta}_{i, j} = E_{i, j} + E_{\nn -i, \nn-j}$ for all $1\leq i, j\leq \nn $.

\begin{prop}
 \label{prop:Cmulti}
  Suppose that $h\in [1, r]$ and $R \in \mbb N$.
\begin{enumerate}
\item
For $A, B\in {}^{\C}\Xi$ such that $\ro (A) = \co (B)$ and $B - R E^{\theta}_{h, h+ 1}$ is diagonal,  we have
\begin{align}
\begin{split}
e_B * e_A
 = \sum_{t} v^{2 \sum_{j > u} a_{hj}t_u } \prod_{u=1}^{\nn} \begin{bmatrix}
   a_{hu}+t_u\\
   t_u
 \end{bmatrix} e_{A + \sum_{u=1}^{\nn}t_u(E^{\theta}_{h u} - E^{\theta}_{h + 1, u})},
\end{split}
\end{align}
where $t =(t_u) \in \mbb N^{\nn}$ such that $\sum_{u=1}^{\nn} t_u =R$ and 
$
\begin{cases}
t_u\leq a_{h+1,u}, & {\rm if}\ h<r\\
t_u+t_{\nn+1-u} \leq a_{h+1,u}, &{\rm if}\ h=r.
\end{cases}$

\item
For $A, C\in {}^{\C}\Xi$ such that $\ro (A) = \co (C)$ and $C - R E^{\theta}_{h + 1, h}$ is diagonal,  we have
\begin{align}
\begin{split}
e_C * e_A
& = \sum_{t} v^{2 \sum_{j< u} a_{h + 1,j}t_u } \prod_{u=1}^{\nn} \begin{bmatrix}
  a_{h+1,u}+t_u\\
  t_u
\end{bmatrix} e_{A - \sum_{u=1}^{\nn} t_u( E^{\theta}_{h u} - E^{\theta}_{h + 1, u})},\ {\rm if}\ h<r;
 \\
e_C * e_A
& = \sum_{t} v^{2 \sum_{j<u} a_{r + 1,j} t_u + 2\sum_{\nn +1-j <u <j}t_ut_j + \sum_{u>r+1} t_u(t_u+1) } \prod_{u<r+1} \begin{bmatrix}
  a_{r+1,u}+t_u\\
  t_u
\end{bmatrix}\\
& \cdot \prod_{u> r+1} \begin{bmatrix}
  a_{r+1,u}+t_u+t_{\nn+1-u}\\
  t_u
\end{bmatrix}\prod_{i=1}^{t_{r+1}}\frac{[a_{r+1,r+1}+2i]}{[i]}  e_{A - \sum_{u=1}^{\nn} t_u(E^{\theta}_{ru} + E^{\theta}_{r + 1, u})} \ {\rm if}\ h=r,
\end{split}
\end{align}
where $t=(t_u) \in \mbb N^{\nn}$ such that $\sum_{u=1}^{\nn} t_u =R$ and $t_u \leq a_{hu}$.
\end{enumerate}
\end{prop}

\begin{proof}
We only give a sketch as it is similar to \cite{BKLW14, FL14}. First the proposition is proved for $R=1$ with the help of
Lemma~\ref{typeC-counting} (which takes care of a genuine type $C$ counting). Then a similar argument using induction as
in \cite[Proposition 3.3]{BKLW14} or \cite[Corollary 4.3.4]{FL14}  proves the general case. 
\end{proof}

For  $A = (a_{ij}) \in {}^{\C}\Xi$, we set 
\begin{equation*}
  d(A) = \dim \mathcal O_A\quad {\rm and} \quad d_A = d(A) - d(B),
\end{equation*}
where $B = (b_{ij}) $ is the diagonal matrix such that $b_{ii} = \sum_k a_{ik}$. 


\begin{lem} \label{finite-dimension}
For any $A= (a_{ij}) \in {}^{\C}\Xi$, we have 
\begin{align}
d_A  & =\frac{1}{2} \left (\sum_{i \geq k, j < l} a_{ij} a_{kl}  + \sum_{i\geq r + 1>j} a_{ij} \right ).
\end{align}
\end{lem}

\begin{proof}
The proof is similar to the proof of \cite[Lemma 3.5]{BKLW14} or \cite[Lemma 4.5.1]{FL14}.
See also the proof of  Lemma ~\ref{dimension}.
\end{proof}

We set
\begin{equation}
  \label{AeA}
  [A] = v^{-d_A} e_A.\quad \forall A\in {}^{\C}\Xi.
\end{equation}
It is clear that $\{[A]|A\in {}^{\C}\Xi\}$ form an $\mathcal A$-basis of ${}^{\C}\mbf S_d^{\jmath}$, which is called a standard basis.


By a direct calculation using \eqref{AeA}, we have the following reformulation of Proposition~\ref{prop:Cmulti} in terms of $[A]$. 

\begin{prop}
 \label{mul[A]}
 Suppose that $A, B,C \in {}^{\C}\Xi$, $h\in [1, r]$ and $R \in \mbb N$.
\begin{enumerate}
\item
If $\ro (A) = \co (B)$ and $B - R E^{\theta}_{h, h+ 1}$ is diagonal, then we have
\begin{align}
\begin{split}
[B] * [A]
 = \sum_{t} v^{\beta(t) } \prod_{u=1}^{\nn} \overline{\begin{bmatrix}
   a_{hu}+t_u\\
   t_u
 \end{bmatrix}} [A + \sum_{u=1}^{\nn}t_u(E^{\theta}_{h u} - E^{\theta}_{h + 1, u})],
\end{split}
\end{align}
where the sum over $t $ is as in Proposition~ \ref{prop:Cmulti}(1) and
\[\beta(t) = \sum_{u\leq j} a_{hj}t_u -\sum_{u<j} a_{h+1,j}t_u + \sum_{u<j} t_ut_j +\delta_{hn}
\Big(\sum_{\overset{u<j}{u+j<\nn+1}}t_ut_j+\sum_{u<r+1} \frac{t_u(t_u-1)}{2}\Big).
\]

\item
Assume that $\ro (A) = \co (C)$ and $C - R E^{\theta}_{h + 1, h}$ is diagonal.
Then for $h<r$ we have
\begin{align}
\begin{split}
[C] * [A]
& = \sum_{t} v^{\beta'(t) } \prod_{u=1}^{\nn} \overline{\begin{bmatrix}
  a_{h+1,u}+t_u\\
  t_u
\end{bmatrix}} [A - \sum_{u=1}^{\nn} t_u( E^{\theta}_{h u} - E^{\theta}_{h + 1, u})],
\end{split}
\end{align}
where the sum over $t$ is as  in Proposition~ \ref{prop:Cmulti}(2) and
$$\beta'(t) = \sum_{u\geq j}a_{h+1,j}t_u -\sum_{u>j}a_{hj}t_u +\sum_{u>j} t_ut_j;
$$
For $h=r$, we have
\begin{align}
\begin{split}
[C] * [A]
& = \sum_{t} v^{\gamma(t)} \prod_{u > r+1}\overline{ \begin{bmatrix}
  a_{r+1,u}+t_u\\
  t_u
\end{bmatrix}}
\prod_{u< r+1} \overline{\begin{bmatrix}
  a_{r+1,u}+t_u+t_{\nn+1-u}\\
  t_u
\end{bmatrix}}\\
&\quad  \cdot\prod_{i=1}^{t_{r+1}}\frac{\overline{[a_{r+1,r+1}+2i]}}{\overline{[i]}}  [A - \sum_{u=1}^{\nn} t_u(E^{\theta}_{ru} + E^{\theta}_{r + 1, u})],
\end{split}
\end{align}
where $\gamma(t)= \sum_{u\leq j} a_{r+1,j}t_u -\sum_{u> j}a_{hj}t_u + \sum_{\nn+1-j\leq u <j}t_ut_j - \sum_{u<r+1} \frac{t_u^2}{2} +\frac{R^2}{2} + \sum_{u\geq r+1}\frac{t_u}{2} $.
\end{enumerate}
\end{prop}

Let 
\begin{equation}
  \begin{split}
  \widetilde{\Xi}_{\C} =\{ A=(a_{ij}) \in & {\rm Mat}_{\nn \times \nn}(\mbb Z) \big \vert  a_{ij} \geq 0\ {\rm if}\ i\neq j, \\
  &  a_{ij}= a_{\nn+1-i,\nn+1-j},\ \forall i,j,\ {\rm and}\ a_{r+1,r+1}\in 2\mbb Z\}.
  \end{split}
\end{equation}
Denoted by ${}^{\C}\mbf K^{\jmath}$ the free $\mathcal A$-module spanned by $\{[A] \vert A\in \widetilde{\Xi}_{\C}\}$.
For any matrix $A$, we set 
\[{}_{2p}A= A +2 p I.\]
Here $I = \sum_{1\leq i \leq \nn} E_{ii}$.
By a similar argument as that for Proposition 4.2 in \cite{BLM90}, we have the following proposition.

\begin{prop}\label{prodlem}
  Suppose that $A_1,  \ldots, A_s \in \widetilde{\Xi}_{\C}$ $(s\geq 2)$ satisfy that ${\rm co}(A_i) = {\rm ro}(A_{i+1})$ for all $i$.
  Then there exist $Z_1,\ldots, Z_m \in  \widetilde{\Xi}_{\C}$, ${}^{\C}\!G_i(v, v') \in \mbb Q(v)[v']$ such that
  \[[{}_{2p}A_1] * [{}_{2p}A_2]* \cdots *[{}_{2p}A_s] = \sum_{i=1}^m {}^{\C}G_i (v, v^{-2p}) [{}_{2p}Z_i], \quad \text{for }p\gg 0.
  \]
\end{prop}

By specialization at $v'=1$, we have the following corollary.
\begin{cor}
\label{cor:stable}
Retain the assumption in Proposition~\ref{prodlem}. 
  There is a unique associative $\mathcal A$-algebra structure on ${}^{\C}\mbf K^{\jmath}$ given by
\begin{equation*}
  [A_1] * [A_2]* \cdots *[A_s] = \sum_{i=1}^m {}^{\C}\!G_i (v, 1) [Z_i].
\end{equation*} 
\end{cor}

\section{Isomorphisms between type C and type B}

Recall that $\mbf S^{\jmath} =\mbf S_d^{\jmath}$ is the convolution algebra on $\nn$-step type $B$ flags defined in \cite{BKLW14}, 
and it admits a standard basis $\{[A] \vert A \in  {}^{\mathfrak b}\Xi\}$, where (${}^{\mathfrak b}\Xi$ is denoted by $\Xi_d$ in {\em loc. cit.})
\[   
{}^{\mathfrak b}\Xi= \Big\{A = (a_{ij}) \in {\rm Mat}_{\nn \times \nn}(\mbb N) \Big \vert \sum_{i,j\in [1,\nn]} a_{ij}= {2d+1},\ a_{ij} = a_{\nn+1-i, \nn+1-j}, \forall i, j\in [1,\nn] \Big\}.
\]
Clearly sending $A \mapsto A-E_{r+1,r+1}$ defines a bijection ${}^{\mathfrak b}\Xi \overset{\sim}{\longrightarrow} {}^{\C}\Xi$.
Let $\psi: \mbf S_d^{\jmath}\rightarrow {}^{\C}\mbf S_d^{\jmath}$ be the $\cA$-linear map sending $[A] \mapsto [A-E_{r+1,r+1}]$ for all $A \in {}^{\C}\Xi$.
It is clear that $\psi$ is an $\cA$-linear isomorphism.

\begin{prop}\label{stablem}
  The map $\psi: \mbf S_d^{\jmath}\longrightarrow {}^{\C}\mbf S_d^{\jmath}$ is an $\mathcal A$-algebra isomorphism.
\end{prop}

\begin{proof}
Since the structure of the two algebras are completely determined by the multiplication formulas in
Proposition~\ref{mul[A]} and  \cite[Proposition~ 3.7]{BKLW14}, we only need to see if they match under the correspondence
$[A] \mapsto [A-E_{r+1,r+1}]$, which can be checked directly. 
\end{proof}

Let $\tilde{\psi}: \mbf K^{\jmath}\rightarrow {}^{\C}\mbf K^{\jmath}$ be the 
$\cA$-linear map sending $[A] \mapsto [A-E_{r+1,r+1}]$ for all $A \in \widetilde{\Xi}_{\C}$, where $\mbf K^{\jmath}$ 
is the algebra defined in \cite[Section 4]{BKLW14}, a finite type $B$ counterpart of ${}^{\C}\mbf K^{\jmath}$.
The algebra isomorphisms $\psi: \mbf S_d^{\jmath}\longrightarrow {}^{\C}\mbf S_d^{\jmath}$ (for varies $d$) 
and the stabilization procedure (Proposition~\ref{prodlem} and Corollary~\ref{cor:stable}) which 
defines the algebra ${}^{\C}\mbf K^{\jmath}$ (and similar for $\mbf K^{\jmath}$) lead readily to the following identification.

\begin{prop}
   The  map $\tilde{\psi}: \mbf K^{\jmath}\longrightarrow {}^{\C}\mbf K^{\jmath}$ is an $\mathcal A$-algebra isomorphism.
\end{prop}

\section{The comultiplication} 

We define $\e_i$, $\f_i$, $\mbf h^{\pm 1}_a \in {}^{\C}\mbf S_d^{\jmath}$, for  $i\in [1,r]$ and $a \in [1,r+1]$, as follows:
for all $L, L' \in X_{\C}$, 
\begin{align}
\mbf e_i (L, L') &=
\begin{cases}
v^{-|L'_{i+1}/L'_i| - \delta_{i,r}}, & \mbox{if}\; L_i \overset{1}{\subset} L_i', L_j = L_j',\forall j\in [1, r ] \backslash \{i\}; \\
0, &\mbox{otherwise}.
\end{cases}\\
\mbf f_i (L, L') &=
\begin{cases}
v^{-|L'_i / L'_{i-1}|}, &\mbox{if}\; L_i \overset{1}{\supset} L_i', L_j = L_j', \forall j \in [1, r]  \backslash \{i\}; \\
0, &\mbox{otherwise}.
\end{cases}\\
\mbf h_{a}^{\pm 1} (L, L')
&= v^{\pm (|L_{a}'/L_{a-1}'| +\delta_{a,r+1})} \delta_{L, L'}.
\end{align}
Also set $\mbf k_i =\mbf h_{i+1} \mbf h^{-1}_i$ in $ \mbf S^{\jmath}_{d}$. 
Note that our $\mbf h_a$ corresponds to $\mbf d_a^{-1}$ in \cite[(3.3)]{BKLW14}, 
and the definitions of $\e_i$, $\f_i$, $\mbf h_a$ above for finite type $C$ formally
coincide with those for finite type $B$ \cite[(3.1)-(3.3)]{BKLW14} (except $\e_r$, $\f_r$, $\mbf h_{r+1}$). 

\begin{prop}
The isomorphism $\psi: \mbf S_d^{\jmath}\longrightarrow {}^{\C}\mbf S_d^{\jmath}$ 
sends $\e_i$, $\f_i$, $\mbf h_a:= \mbf d_a^{-1}$  for $i\in [1,r]$ and $a \in [1,r+1]$ in $\mbf S_d^{\jmath}$
to the elements in ${}^{\C}\mbf S_d^{\jmath}$ in the same notations, respectively.
\end{prop}

\begin{proof}
The element $\e_i$ on both sides is a sum of all standard matrices $A$ such that $A- E^{\theta}_{i+1, i}$ is diagonal.
Hence we have the result for $\e_i$ by the definition of $\psi$. Similarly, one can prove the results for $\f_i$ and $\mbf h_a$. 
\end{proof}


We shall denote by $\mbf S_{d}$ the Schur algebra of finite type $A$ arising from $\nn$-step flags in an $d$-dimensional space.
For any $i\in [1, \nn-1]$, $a\in [1, \nn]$, we define the following elements in $\mbf S_{d}$:
\begin{equation}\label{generatorA}
\begin{split}
\mbf E_i (V, V') &=
\begin{cases}
v^{-|V'_{i+1}/V'_i|}, &\mbox{if}\; V_i\overset{1}{\subset} V_i', V_j=V_{j'},\forall j \neq i; \\
0, &\mbox{otherwise}.
\end{cases} \\
\mbf F_i (V, V') &=
\begin{cases}
v^{-|V'_i/V'_{i-1}|}, &\mbox{if}\; V_i\overset{1}{\supset} V_i', V_j=V_{j'},\forall j\neq i; \\
0, &\mbox{otherwise},
\end{cases} \\
\mbf H_{ a}^{\pm 1} (V, V') & = v^{\pm |V_a/V_{a-1}|}  \delta_{V, V'},  \quad \forall V, V' \in X_d.\\
\mbf K_{i}^{\pm 1}  &  = \mbf H_{i+1}^{\pm 1}  \mbf H_{i}^{\mp 1}.\\
\end{split}
\end{equation}

In a completely analogous way to the definition $\widetilde \Delta^{\jmath}$ in \cite[\S3.2]{FL15}, for a composition $d=d' +d''$, we 
have a comultiplication    
\begin{align*} 
\widetilde \Delta^{\C}: \mbf S^{\C}_d \longrightarrow \mbf S^{\C}_{d'} \otimes \mbf S_{d''}.
\end{align*}
Then we have the following proposition, similar to \cite[Proposition~ 3.2.4]{FL15}.

\begin{prop}
\label{finiteD}
For any $i\in [1, r]$, we have
\begin{align*}
\begin{split}
\widetilde \Delta^{\C} ( \e_i)
&= \e_i' \otimes \mbf H_{i+1}'' \mbf H''^{-1}_{\nn-i} + \mbf h'^{-1}_{i+1} \otimes \mbf E_i''  \mbf H''^{-1}_{\nn-i} +  \mathbf h'_{i+1} \otimes \mbf F''_{\nn-i} \mbf H''_{i+1}.  \\
\widetilde \Delta^{\C} (\f_i)
 & = \f'_i \otimes \mbf H''^{-1}_{i} \mbf H''_{\nn+1-i} + \mbf h'_i\otimes \mbf F''_i \mbf H''_{\nn+1-i} + \mbf h'^{-1}_{i} \otimes \mbf E''_{\nn-i} \mbf H''^{-1}_{i}. \\
\widetilde \Delta^{\C} (\mbf k_i) & = \mbf k'_i \otimes \mbf K''_i \mbf K''^{-1}_{\nn-i}.
\end{split}
\end{align*}
\end{prop}

\begin{proof}
With the help of Lemma ~\ref{typeC-counting}, the proof of \cite[Proposition~ 3.2.4]{FL15} can be essentially repeated here.
\end{proof}

By checking the image of algebra generators of $\mbf S_d^{\jmath}$, we have the following proposition.

\begin{prop}\label{com-coproduct}
  The following diagram is commutative:
  \[
  \xymatrix{
  \mbf S_d^{\jmath} \ar[rr]^{\tilde{\Delta}^{\jmath}} \ar[d]_{\psi} && \mbf S_d^{\jmath} \otimes \mbf S_{d} \ar[d]^{\psi \otimes 1} \\
  {}^{\C}\mbf S_d^{\jmath} \ar[rr]^{\tilde{\Delta}^{\C}}  && {}^{\C}\mbf S_d^{\jmath} \otimes \mbf S_{d}
  }
  \]
\end{prop}

Following \cite{FL15}, we introduce the following notation
\[
\Lambda^{\jmath}_{\nn,d}
=\big \{ \mbf a = (a_i) \in \mbb{N}^\nn \big \vert  \sum a_i = 2d+1, a_i = a_{\nn+1-i} \big\}.
\]
An isotropic flag $L$ of type $C$ defines a unique element $\alpha(L) \in \Lambda^{\jmath}_{\nn,d}$ by 
\[
\alpha(L)_i = \dim L_i/L_{i-1} + \delta_{i, r+1}, \forall i.
\]
Then we have the following partition: 
\[
X_{\C} =\bigsqcup_{\mbf a \in \Lambda^{\jmath}_{\nn,d}}  X_{\C} (\mbf a), \quad X_{\C}(\mbf a) = \{ L \vert  \alpha(L) = \mbf a\}.
\]

For any $\mbf a, \mbf b \in \Lambda^{\jmath}_{\nn,d}$, 
let ${}^{\C}\mbf S_d^{\jmath}(\mbf b, \mbf a)$ be the subspace of ${}^{\C}\mbf S_d^{\jmath}$ spanned by all functions supported on
$X_{\C}(\mbf b) \times X_{\C}(\mbf a)$. Then we have
\[ {}^{\C}\mbf S_d^{\jmath} = \oplus_{\mbf b,\mbf a\in \Lambda^{\jmath}_{\nn,d}} {}^{\C}\mbf S_d^{\jmath}(\mbf b, \mbf a).\]
We shall denote $\iota_{\mbf b, \mbf a}$ and $p_{\mbf b, \mbf a}$ the embedding of 
${}^{\C}\mbf S_d^{\jmath}(\mbf b, \mbf a)$ into ${}^{\C}\mbf S_d^{\jmath}$ and the projection of ${}^{\C}\mbf S_d^{\jmath}$ to ${}^{\C}\mbf S_d^{\jmath}(\mbf b, \mbf a)$, respectively.
By abuse of notations, the projection from $\mbf S_{d}$ to $\mbf S_{d}(\mbf b, \mbf a)$ is still denoted by $p_{\mbf b, \mbf a}$.
For any $\mbf b, \mbf a, \mbf b', \mbf a', \mbf b''$ and $\mbf a''$ satisfying that 
\[b_i = b_i'+ b_i''+ b_{\nn +1-i}'',\quad {\rm and}\quad a_i = a_i'+ a_i''+ a_{\nn +1-i}'',\ \forall i\in [1,\nn],
\] 
we set $\tilde{\Delta}^{\C}_{\mbf b', \mbf a',\mbf b'',\mbf a''} = (p_{\mbf b',\mbf a'} \otimes p_{\mbf b'',\mbf a''}) \circ \tilde{\Delta}^{\C} \circ \iota_{\mbf b, \mbf a}$. 
Let 
\[\Delta^{\C} = \bigoplus_{\mbf b, \mbf a, \mbf b', \mbf a', \mbf b'', \mbf a''} \Delta^{\C}_{\mbf b', \mbf a', \mbf b'', \mbf a''},\]
where $\Delta^{\C}_{\mbf b', \mbf a', \mbf b'', \mbf a''} 
= v^{\sum_{1\leq i \leq j\leq \nn} b_i'b_j''-a_i'a_j''}v^{u(\mbf b'', \mbf a'')}
\tilde{\Delta}^{\C}_{\mbf b', \mbf a', \mbf b'', \mbf a''}$,
and $u(\mbf b, \mbf a)$ is the function defined in \cite[(44)]{FL15} in finite type $B$ setting.
The definition of $\Delta^{\C}$ is completely analogous to the definition of  $\Delta^{\jmath}_{\mbf v}$ in \cite[(45)]{FL15}. 
The following proposition follows by comparing the definitions.

\begin{prop}
Given $d=d'+d''$, we have the following commutative diagram:
  \begin{equation}
    \label{comm-diagram}
\xymatrix{\mbf S_d^{\jmath} \ar[rr]^{\Delta^{\jmath}_{\mbf v}} \ar[d]_{\psi} && \mbf S_{d'}^{\jmath} \otimes \mbf S_{d''} \ar[d]^{\psi \otimes 1} \\
  {}^{\C}\mbf S_d^{\jmath} \ar[rr]^{\Delta^{\C}}  && {}^{\C}\mbf S_{d'}^{\jmath} \otimes \mbf S_{d''}.
  }\end{equation}
\end{prop}

The transfer map 
\[\phi^{\C}_{d,d-\nn}: {}^{\C}\mbf S_d \longrightarrow {}^{\C}\mbf S_{d-\nn}\]
is defined to be the composition 
$\xymatrix{{}^{\C}\mbf S_d \ar[r]^-{\tilde{\Delta}^{\C}} & {}^{\C}\mbf S_d \otimes \mbf S_{\nn} \ar[r]^-{1\otimes \chi}  & {}^{\C}\mbf S_{d-\nn} \otimes \mathcal A= {}^{\C}\mbf S_{d-\nn}}$,
where $\chi(\eta_A)= {\rm det}(A)$ for any $A\in {}^{\C}\Xi$ and $\eta_A$ is the characteristic function on the orbit corresponding the matrix $A$.
This is analogous to  the transfer map $\phi^{\jmath}_{d,d-\nn}:  \mbf S_d^{\jmath} \rightarrow \mbf S_{d-\nn}^{\jmath}$ 
defined in \cite[\S3.6]{FL15} in the finite type $B$ setting.
By Proposition \ref{com-coproduct}, we have the following proposition.
\begin{prop}
  The following diagram is commutative:
  \[
  \xymatrix{
   \mbf S_d^{\jmath} \ar[rr]^-{\phi^{\jmath}_{d,d-\nn}} \ar[d]^{\psi} && \mbf S_{d-\nn}^{\jmath} 
   \ar[d]_{\psi}\\
  {}^{\C}\mbf S_{d}^{\jmath} \ar[rr]^-{\phi^{\C}_{d,d-\nn}}  && {}^{\C}\mbf S_{d-\nn}^{\jmath}
  }
  \]
  \end{prop}

Finally, we address the $\imath$-version.
Recall that $\mm=\nn-1=2r$.
Let
\[ 
X^{\imath}_{\C} = \{ 0= V_0 \subseteq V_1\subseteq \ldots \subseteq V_{\mm} = \mbb F^{2d}_q \big \vert V_{\mm -i} = V_i^{\perp}\}.
\]
The convolution algebra on $X^{\imath}_{\C} \times X^{\imath}_{\C}$ is denoted by ${}^{\C}\mbf  S_d^{\imath}$.
We shall naturally embed $X^{\imath}_{\C}$ into $X_{\C}$ by sending a $\mm$-step flag in $X^{\imath}_{\C}$ as above to
an $\nn$-step flag 
\[
0= V_0 \subseteq V_1\subseteq \ldots \subseteq V_r \subseteq V_r \subseteq \ldots \subseteq V_{\mm}= \mbb F^{2d}_q
\]
(where the maximal isotropic subspace $V_r$ in the middle is repeated).
Therefore, ${}^{\C}\mbf  S_d^{\imath}$ is naturally a subalgebra of ${}^{\C}\mbf  S_d$.
Consider the following set
\begin{equation*}
  \begin{split}
    \Xi^{\imath}_{\C} &= \{A =(a_{ij} \in {}^{\C}\Xi \big \vert  a_{r+1,j}=0 = a_{i, r+1},\ \forall i, j\}.
  \end{split}
\end{equation*}
By \cite[Lemma 6.1]{BKLW14}, we have a natural bijection
${\SP}({2d}) \backslash X_{\C}^{\imath} \times X_{\C}^{\imath} \leftrightarrow {}^{\C}\Xi^{\imath}$,
and moreover, $\{[A]| A\in \Xi^{\imath}_{\C}\}$ forms a basis of  ${}^{\C}\mbf  S_d^{\imath}$.
Recall a completely analogous subalgebra $\mbf  S_d^{\imath}$ of $\mbf  S_d^{\jmath}$ was defined in ~\cite[\S5]{BKLW14}.
The standard basis of $\mbf  S_d^{\imath}$ is parametrized by a subset $\Xi^\imath \subset \Xi$, and there is a natural bijection
$\Xi^\imath \longrightarrow  \Xi^{\imath}_\C$,  $A \mapsto A-E_{r+1,r+1}$.
The following proposition follows by the definition of $\psi$.
\begin{prop}
  The restriction of $\psi: \mbf S_d^{\jmath}\longrightarrow {}^{\C}\mbf S_d^{\jmath}$ induces an algebra isomorphism ${}^{\C}\mbf  S_d^{\imath}\simeq \mbf  S_d^{\imath}$.
\end{prop}

\begin{rem}
It should be clear for the reader that
the various canonical bases from finite type B/C geometries are compatible under the isomorphism $\psi$.
\end{rem}

\clearpage


\begin{thebibliography}{DDPW08}\frenchspacing



\bibitem[BBD82]{BBD82}
A.A. Beilinson, J. Bernstein and P. Deligne,
{\em Faisceaux pervers},
Ast\'{e}risque {\bf 100} (1982).

\bibitem[BLM90]{BLM90} 
A. Beilinson, G. Lusztig and R. McPherson,
          {\em A geometric setting for the quantum deformation of $GL_n$}, 
Duke Math. J., {\bf 61} (1990), 655--677.

\bibitem[Bao16]{Bao16} 
H. Bao, 
{\em Kazhdan-Lusztig theory of super type $D$ and quantum symmetric pairs}, Preprint, 2016.

\bibitem[BKLW14]{BKLW14}
H.~Bao, J.~Kujawa, Y.~Li, and W.~Wang,
{\em Geometric {S}chur duality of classical type}, 
\href{http://arxiv.org/abs/1404.4000}{arXiv:1404.4000v3}.

\bibitem[BLW14]{BLW14}
H.~Bao,  Y.~Li, and W.~Wang, 
{\em A geometric setting for the coideal algebra $\dot{\bU}^\imath$  and compatibility of canonical bases}, 
Appendix to \cite{BKLW14}, 15pp.

\bibitem[Br03]{Br03}
T. Braden,
{\em Hyperbolic localization of intersection cohomology},
Transform. Groups {\bf 8} (2003),  209--216.

\bibitem[BW13]{BW13}
H. Bao and W. Wang,
{\em A new approach to Kazhdan-Lusztig theory of type $B$ via quantum symmetric pairs},
\href{http://arxiv.org/abs/1310.0103}{	arXiv:1310.0103}.

\bibitem[BW16]{BW16} H. Bao and W. Wang,
{\em Canonical bases arising from quantum symmetric pairs},   
in preparation.

\bibitem[CG97]{CG97} N. Chriss and V. Ginzburg,
        {\em Representation Theory and Complex Geometry},
        Birkh\"auser, Boston (1997).


\bibitem[DD05]{DD05}
B. Deng and J. Du, 
{\em Monomial bases for quantum affine $sl_n$}, 
Adv. in  Math.  {\bf 191} (2005), 276--304.

\bibitem[DDF12]{DDF12}
B. Deng, J. Du and Q.  Fu,
{\em A double Hall algebra approach to affine quantum Schur-Weyl theory.}
London Mathematical Society Lecture Note Series, {\bf 401}.
Cambridge University Press, Cambridge, 2012.

\bibitem[DDPW08]{DDPW08}
B.~Deng, J.~Du, B.~Parshall and J.~Wang, 
{\em Finite dimensional algebras and quantum groups}, 
Mathematical Surveys and Monographs {\bf 150}.
American Mathematical Society, Providence, RI, 2008.

\bibitem[DF13]{DF13} 
J. Du and Q. Fu,
{\em  Quantum affine $\frak{gl}_n$ via Hecke algebras},
\href{http://arxiv.org/abs/1311.1868}{arXiv:1311.1868}, 
Adv. in Math. (to appear).

\bibitem[DF14]{DF14} 
J. Du and Q. Fu,
{\em The integral quantum loop algebra of $\mathfrak{gl}_n$}, 
\href{http://arxiv.org/abs/1404.5679}{arXiv:1404.5679}.   

\bibitem[Dr86]{Dr86} V. Drinfeld,
{\em Quantum groups}, Proc. Int. Congr. Math. Berkeley 1986, vol. {\bf 1}, Amer. Math. Soc. 1988, 798--820. 
 
\bibitem[ES13]{ES13} 
M.~Ehrig and C.~Stroppel,
{\em Nazarov-Wenzl algebras, coideal subalgebras and categorified skew Howe duality}, 
\href{http://arxiv.org/abs/1310.1972}{arXiv:1310.1972}.


\bibitem[FLLLW]{FLLLW}
Z. Fan, C. Lai,  Y. Li, L. Luo and W. Wang,
{\em Affine Hecke algebras and quantum symmetric pairs}, in preparation, 2016. 

\bibitem[FL14]{FL14}
Z. Fan and Y. Li,
{\em Geometric Schur duality of classical type, II},
Trans. Amer. Math. Soc.,
Series {\bf B 2} (2015), 51--92.  

\bibitem[FL15]{FL15}
Z. Fan and Y. Li,
{\em Positivity of canonical basis under comultiplication},  
\href{http://arxiv.org/abs/1511.02434}{arXiv:1511.02434}.

\bibitem[G97]{G97} 
R. Green, 
{\em Hyperoctahedral Schur algebras}, 
J. Algebra  {\bf 192} (1997), 418--438.

\bibitem[G99]{G99}
R. Green,
{\em The affine q-Schur algebra},
J. Algebra {\bf 215} (1999), 379--411.

\bibitem[GL92]{GL92} 
I. Grojnowski and G. Lusztig, 
{\em  On bases of irreducible representations of quantum $GL_n$}. 
In:  {\em Kazhdan-Lusztig theory and related topics}  
(Chicago, IL, 1989), 167-174, Contemp. Math. {\bf 139}, Amer. Math. Soc., Providence, RI, 1992.

\bibitem[GV93]{GV93} 
V. Ginzburg and E. Vasserot,  
{\em Langlands reciprocity for affine quantum groups of type $A_n$},
Internat. Math. Res. Notices {\bf 3} (1993), 67--85. 

\bibitem[H99]{H99} 
R. Howe, 
{\em Affine-like Hecke algebras and p-adic representation theory.} in
{\em  Iwahori-Hecke algebras and their representation theory} 
(Martina-Franca, 1999), 27-69, Lecture Notes in Math., 1804, Springer, Berlin, 2002.

\bibitem[I64]{I64} N. Iwahori,
{\em On the structure of a Hecke ring of a Chevalley group over a finite field}, 
J. Fac. Sci. Univ. Tokyo Sect. I {\bf 10} (1964), 215--236.

\bibitem[IM65]{IM65} N. Iwahori and H. Matsumoto,
{\em On some Bruhat decomposition and the structure of the
Hecke rings of $p$-adic Chevalley groups}, 
Publications math. I.H.E.S. {\bf 25} (1965), 5--48.

\bibitem[Jim86]{Jim86} M. Jimbo,
{\em A $q$-analogue of $U({\mathfrak g\mathfrak l}(N+1))$, Hecke
algebra, and the Yang-Baxter equation}, Lett. Math. Phys. {\bf 11}
(1986), 247--252.

\bibitem[K91]{K91} 
M. Kashiwara,
{\em On crystal bases of the $Q$-analogue of universal enveloping algebras}, 
Duke Math.~J.~{\bf 63} (1991), 456--516.

\bibitem[K94]{K94} M. Kashiwara,
{\em Crystal bases of modified quantized enveloping algebra,}
Duke Math. J. {\bf 73} (1994) 383-413.


\bibitem[Ko14]{Ko14} 
S. Kolb,
{\em Quantum symmetric Kac-Moody pairs}, 
Adv. in Math. {\bf 267} (2014), 395--469.

\bibitem[KL79]{KL79}
D.~Kazhdan and G.~Lusztig,
{\em Representations of {C}oxeter groups and {H}ecke algebras},
Invent. Math. {\bf 53} (1979), 165--184.

          
\bibitem[Le02]{Le02}
G.~Letzter, {\em Coideal subalgebras and quantum symmetric pairs}, 
New directions in Hopf algebras (Cambridge), MSRI publications, vol. {\bf 43}, Cambridge Univ. Press, 2002, pp. 117--166.

\bibitem[LL15]{LL15}
C. Lai and L. Luo, 
{\em An elementary construction of monomial bases of quantum affine $\mathfrak{gl}_n$},
\href{http://arxiv.org/abs/1506.07263}{arXiv:1506.07263}.

\bibitem[Lu90]{Lu90} G. Lusztig, 
{\em Canonical bases arising from
quantized enveloping algebras}, J.~Amer.~Math.~Soc.~{\bf 3} (1990), 447--498.

\bibitem[Lu93]{Lu93}
G.~Lusztig, 
{\em Introduction to quantum groups}, 
Progress in Mathematics {\bf 10},
Birkh\"auser Boston, Inc., Boston, MA, 1993.

\bibitem[L97]{L97}
G.  Lusztig,
{\em Cells in affine Weyl groups and tensor categories,}
Adv. Math. 129 (1997), no. 1, 85-98.

\bibitem[Lu99]{Lu99} 
G.~Lusztig,
{\em Aperiodicity in quantum affine $\mathfrak{gl}_n$}, Asian J. Math. {\bf 3} (1999), 147--177.

\bibitem[Lu00]{Lu00} 
G.~Lusztig, {\em Transfer maps for quantum affine $\mathfrak{sl}_n$,} 
in 
{\em Representations and quantizations}, (ed.  J.~Wang et. al.), 
China Higher Education Press and Springer Verlag 2000, 341--356. 

\bibitem[Lu03]{Lu03}
G.~Lusztig,
{\em Hecke algebras with unequal parameters},
CRM Monograph Series, {\bf 18}. American Mathematical Society, Providence, RI, 2003.

\bibitem[LW15]{LW15}
Y. Li and W. Wang,
{\em Positivity vs negativity of canonical basis},
\href{http://arxiv.org/abs/1501.00688}{arXiv:1501.00688v3}.

\bibitem[Mc12]{Mc12} 
K.~McGerty,
{\em On the geometric realization of the inner product and canonical basis for quantum affine $\mathfrak{sl}_n$},
Alg. and Number Theory {\bf 6} (2012), 1097--1131.


\bibitem[P09]{P09} 
G. Pouchin,
{\em A geometric Schur-Weyl duality for quotients of affine Hecke algebras,} 
J. Algebra {\bf  321}  (2009),   230--247. 

\bibitem[R90]{R90} 
C. M. Ringel,
{\em Hall algebras and quantum groups}, 
Invent. Math. {\bf 101} (1990), 583-592.
        
\bibitem[Sa99]{Sa99}
D. Sage,
{\em The geometry of fixed point varieties of affine flag manifolds},
Transactions AMS, {\bf 352} (1999), 2087--2119.


\bibitem[Sch06]{Sch06}
O. Schiffmann,
{\em Lectures on Hall algebras},
\href{http://arxiv.org/abs/math/0611617}{arXiv:math/0611617}.



        
\bibitem[SV00]{SV00} 
O. Schiffmann and E. Vasserot, {\em Geometric construction of the 
global base of the quantum modified algebra of $\widehat{ \mathfrak{gl}}_n$}, Transform. Groups {\bf 5} (2000), 351--360.

\bibitem[VV99]{VV99} 
M. Varagnolo and E. Vasserot,  
{\em On the decomposition matrices of the quantized Schur algebra}, 
Duke Math. J. {\bf 100} (1999), 267--297.

\end{thebibliography}
\end{document}